\documentclass[12pt,twoside]{book}
\usepackage{graphics,color}      
\usepackage{verbatim}
\usepackage{amsthm}
\usepackage{setspace}
\usepackage{fancyhdr}
\usepackage{amsmath}  
\usepackage{html}
\usepackage{hyperref}
\usepackage{latexsym}

\newcommand{\norm}[1]{\Vert #1 \Vert}
\newcommand{\Norm}[1]{\vert #1 \vert}

\newtheorem{prop}{Proposition}
\newtheorem{thm}{Theorem}[section]
\newtheorem{cor}[thm]{Corollary}
\newtheorem{lem}[thm]{Lemma}

\newtheorem*{conj}{Conjecture}
\newtheorem*{quest}{Question}
\theoremstyle{definition}
\newtheorem{dfn}{Definition}
\theoremstyle{remark}
\newtheorem*{rmk}{Remark}
\theoremstyle{remark}

\begin{document}


\titlepage
\pagenumbering{alph}    
\pagestyle{empty}
\begin{center}
\LARGE{A treatise on information geometry}\\
\end{center}
\vspace{2cm}
\begin{center}  
\end{center}
\begin{center} \includegraphics{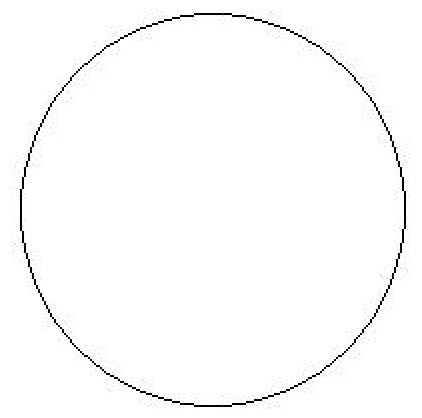} \end{center}
\vspace{5cm}
\begin{center}
Chris Goddard, The University of Melbourne\\
\end{center}

\newpage
\frontmatter
\pagenumbering{roman}   
\setcounter{page}{1}
\pagestyle{plain}
\tableofcontents
\newpage
\pagestyle{fancy}
\lhead{}
\rhead{}
\renewcommand{\chaptermark}[1]{\markboth{#1}{}}
\renewcommand{\sectionmark}[1]{\markright{#1}}
\chead[\rightmark]{\leftmark}

\chapter{Preface}

During my time at the University of Melbourne I looked at a number of different topics.  In a rough semblance of order, these were: Characteristic Classes, Pseudo-Riemannian Geometry, Comparison Geometry (together with the diameter sphere problem), and Geometric Measure Theory (with a focus towards generalising a particular theorem in minimal surface theory).  I will discuss each of these in turn, with the exception of comparison geometry, since that is to be the topic of my PhD thesis.

In July 2006 I attended the AMSI winter school at the University of Queensland, and was most fortunate to be given an extremely accessible account of geometric evolution equations, and in particular, the Ricci Flow by Ben Andrews of the ANU.  In the same year I also studied in some detail the theory of Information Geometry, and its application to what I call the theory of Physical Manifolds via an information measure called the Fisher information.  I constructed a mathematics which I think could best be described as statistical geometry for this purpose.  This, and various generalisations, forms the focus of the latter part of this manuscript.

The general premise of statistical geometry is fairly simple.  Take a differential $n$ manifold, $M$.  Consider the natural space of inner products on $R^{n}$, $A$.  Associate to each point in $M$ a distribution $f(m) : A \rightarrow R^{+}$ which assigns in some sense a weighting to each potential geodesic direction from $m$, possibly favouring certain directions over others.  In particular, we would like $\int_{A}f(m,a) = 1$ for all $m \in M$.  There is a natural statistical derivative $\nabla_{f}$ induced by $f$.

We might use this to force $f$ to satisfy an additional conservation condition, which essentially translates to conservation of probabilistic flux - that $\Delta_{f}f = 0$.  This was in fact roughly the original approach I took towards the matter.  However it turns out that this is roughly equivalent to criticality of the Fisher information.  In particular, for the special case of a \emph{sharp} or Cartan-Riemannian manifold, where $f(m,a) = \delta(\sigma(m) - a)$ for some metric $\sigma$, we have roughly that $\Delta_{f}f = R_{\sigma}$, and also that the Fisher information is $\int_{M}R_{\sigma}$; in particular this is critical when $R_{\sigma} = 0$.

Originally I also made some attempt to make concrete various ideas that do not fit in any of the above categories.  In particular, I wanted to somehow make rigorous the idea of curvature on sets of fractal dimension.  Early in 2008 I made some progress with respect to these ideas, which became part of a project I came to call turbulent geometry.

Following my preliminary investigations into turbulent geometry I began to focus on the idea of building statistical spaces on top of statistical spaces.  In particular this, and various hybrid models including also turbulence and ideas from statistics, have application to number theory and more practically towards constructing theoretical frameworks for condensed matter physics.  One of the more exciting developments here, in my opinion, was my discovery of the \emph{correspondence principle}, which in its simplest incarnation essentially states that there is a 1-1 correspondence between completely general statistical manifolds and sharp turbulent geometries of scalar type.  This allows a physical interpretation of the former in a way that has considerable aesthetic appeal, as of a Cartan-Riemannian manifold wherein the points are of  generalised measure.

However, with a few minor initial exceptions, much of the theory I develop in the later sections of this book still lacks the degree of rigour which its scope would otherwise dictate.  An indication of the level of detail that is required is present in my proof of both Stoke's theorem and the Cramer-Rao inequality for statistical manifolds.  Unfortunately my generalisations of these theorems to the rather broader classes of geometrical objects I deal with further along the path of my researches are slightly rushed, and probably do require some inspection.

Indeed the main part of the dissertation I had in mind for this disclaimer is the chapter on turbulent geometry (chapter 9).  I have made a good faith attempt to prove many results in this chapter, but I fear that many of my "proofs" here are wanting in precision.  However I have decided to retain this work because of its potential promise to solve many extremely difficult and interesting problems.  These include what I call the Lorentz problem (why the preferred geometrical structure of the universe should be Lorentzian), indications of how to construct a geometric theory of condensed matter physics, and indications of how to develop a powerful and predictive theory for entanglement physics.  Others still are increased understanding of turbulence in fluid flow, the theory of purely formal structures of fractal measure (such as the Mandelbrot set), and questions about the structure of the prime numbers.


\chapter{Acknowledgements}

First and foremost, I would like to acknowledge my PhD supervisor, Professor Hyam Rubinstein.  Not only did he manage to continue to find me things to examine and learn, but his boundless patience with me when I mentioned various ideas of mine in incomplete form, and his tireless capacity for impressing upon me the need for rigour was something which I have greatly appreciated.  I would also like to thank my parents; my father, for his interest in my research, for listening to me, and helping guide my thought process to meaningful results, and my mother, for being a polite listener.

Thanks to the popular scientific magazine the New Scientist, I would be otherwise unaware of Frieden's book on the subject of Fisher information, which was reviewed a number of years back.  I should also thank my father once more for pointing out the review to me, and impressing upon me the potential importance of the work.  

I should also thank the New Scientist for bringing to my attention the notion of scale free physics, an idea advocated first by Howard Georgi in \cite{[Ge]}, though, taking a similar role as Fisher to Frieden, the original instigators were Banks and Zaks in their 1982 paper \cite{[BZ]}.  This motivated much of my interest in investigating (smooth) fractional geometry and its connection to turbulence in early 2008.

Additionally, I owe a debt of gratitude to Ben Andrews for a passing dinnertime conversation during the 2006 IAGSM winter school in Brisbane in which he mentioned the work of Amari, Nagaoka, Murray and Rice on Information Geometry, without which I would not have been able to make rigorous my justification for the variational principle that I invoke later on.  There are various incarnations of this principle, but the first steps along the path of extension and generalisation of the result known as the Cramer-Rao inequality would have been impossible without reference to the work of other information geometers.  I might add that this principle underlies Roy Frieden's work.

Of course, Frieden's work would not have been possible, or at least would certainly have been much more difficult, without various ideas, not at least that of the Fisher information, which are attributed to Ronald Fisher (1890 - 1962).

Finally an additional word of thanks to Professor Frieden who suggested during an email exchange in late December 2008 that one of the best tools for proving the optimality of the Fisher information over the space of positive functionals for a given geometric model is the Cencov inequality, which is closely related to the Cencov uniqueness theorem.  This is still something that I need to more fully investigate.

\chapter{Organisation and attribution of work}

The organisation of this manuscript will be as follows.  In the first chapter I survey J. Milnor's book on characteristic classes.  Chapter two is also a survey, this time of Riemannian geometry.  No particular source has been used here, rather it being a synthesis of a number of different works.  Towards the end of the chapter some original results are presented.

In chapter three I survey Leon Simon's and Frank Morgan's books on geometric measure theory, as well as incorporating some notes from a course on the same that Marty Ross gave back in 2006.

Chapter four is a survey mainly of the work in Gilbarg and Tr\"udinger's book "Elliptic Partial Differential Equations", together with an extensive commentary on techniques from Shatah and Struwe's book on geometric wave equations.

Following this in chapter five I give the notes rewritten almost verbatim from a course that Ben Andrews gave at the AMSI winter school of 2006, together with a few minor personal modifications.  I also cover a few lectures given by Gerard Huisken and one given by Nick Sheridan.  The common thread to all of these materials is that of geometric evolution equations and the Hamilton-Perelman approach to the solution of the geometrisation conjecture, which is the topic of this chapter.  I also draw upon J. Morgan and Gang Tian's detailed and extensive work on the same subject.

Chapter six is the first truly original component of this work, and introduces the notion of statistical geometry.  Chapter seven draws heavily upon Roy Frieden's program described in his book "Physics from Fisher Information" and attempts to rigorise it.  Chapter eight is the application of the techniques developed in the previous two chapters to the derivation of the equations of geometrodynamics and also the equations of the standard model.  This chapter is also more or less completely original.

Chapter nine is an extension of the ideas in the preceeding three chapters to the idea of a turbulent manifold.  The ideas presented in this and the last chapter are almost entirely due to the author.

\mainmatter
\pagenumbering{arabic}  
\setcounter{page}{1}
\chapter{Characteristic Classes}


The idea of a characteristic class is to construct some cohomological framework on a manifold and look at the things that generate it, which are known as \emph{characteristic classes}.  These things are used as invariants to differentiate between different manifolds.  The most well known examples of such things are the Stiefel-Whitney classes, Euler classes, Chern classes, and Pontrjagin classes, pertaining to $Z_{2}$, integer, complex and quaternionic cohomology respectively.  I shall discuss these all in turn, and give a broad outline of their construction, following \cite{[MS]}.

A core element of the construction of all these various types of characteristic classes lies in understanding fibre bundles, in particular fibre bundles with $GL(n)$ as structure group and vector spaces as fibres, i.e. vector bundles.  A good general reference on fibre bundles can be found in \cite{[S]}.

The main reason that we are interested in characteristic class theory is that an extremely important application of this method is the ability to distinguish between different differentiable structures on a topological manifold.  In other words, it is possible to show that there are different ways of doing calculus on a given topology.  John Milnor, together with collaborator Michel Kervaire, used this theory to great effect in classifying the various allowable structures on $n$-spheres in the 50s, which has inspired much of the recent modern interest.   More recently, a particularly spectacular success was Donaldson's proof that there are uncountably many different differentiable structures imposable upon $R^{4}$, for instance.

In particular the fact that there may be various different ways of doing calculus on a given topological space was in fact not known until Milnor's seminal paper, "On Manifolds Homeomorphic to the 7-sphere".  This was a major advance and demonstrated the counterintuitive notion that exotic forms of calculus do and can exist on manifolds.  However, the method of characteristic classes is not totally constructive, and only serves as a crude indicator of how to distinguish between different differentiable structures.  It will be one of the primary purposes and aims of this dissertation to build a few examples of atypical geometric structures on manifolds, and to discuss, where possible, the associated insights into physics.

For now, however, I will sketch the preliminary ideas (and motivation) behind the development of characteristic classes.

\section{Characteristic Classes as pullbacks}

Before providing an axiomatic formulation of the concept of characteristic class, I will give the historical motivation for their development.  I will follow Chern \cite{[Ch]} closely here.

First of all, the \emph{Grassman manifold} $H(n,N)$ is the space of all $n$-dimensional linear spaces through a point in $R^{n + N}$. This can also be thought of as the classes of the rotation group $SO(n + N)$ modulo the rotation group $SO(n)$ of rotations of the base and modulo the rotation group $SO(N)$ of rotations of the fibre, ie $H(n,N) \cong SO(N) \backslash SO(n + N) /SO(n)$.

We may now state two theorems of great importance in the theory of sphere bundles:

\begin{thm} (Whitney imbedding theorem for sphere bundles, \cite{[Wh]}).  Every sphere bundle whose spheres are of dimension $n-1$ is equivalent to the bundle induced by mapping the base space $M$ into the manifold $H(n,N)$, provided $N \geq dim (M) + 1$.
\end{thm}

\begin{thm} (Steenrod, \cite{[St]}).  Two sphere bundles are equivalent iff the mappings of $M$ into $H(n,N)$ are homotopic. \end{thm}

Consider now the cohomology ring $K(H(n,N))$ of $H(n,N)$ relative to some coefficient ring $R$.  Then the class of homotopic mappings of $M$ into $H(n,N)$, ie, the equivalence classes of sphere bundles, induce a ring homomorphism of  $K(H(n,N))$ into $K(M)$.  The image $C(M,R)$ of this map in $K(M)$ is called the \emph{characteristic ring} of sphere bundles over $M$.  So, in a sense, we have "pulled back" classes in $H(n,N)$ to generate classes in $M$.  Finally, a cohomology class of $C(M,R)$ is called a \emph{characteristic cohomology class}.

Of course, we may want to understand characteristic classes for other bundles than just sphere bundles!  In particular, we are interested in \emph{vector bundles}, which are the primary focus of this chapter.  But these classes of bundles are more or less the same, since an $n$-plane bundle can easily be mapped to an $n-1$ sphere bundle via a map call it $g$.  Since cohomology functors reverse direction, we then clearly have an induced map $g^{*} : C(M,R) \rightarrow \bar{C}(M,R)$, where $\bar{C}$ is the characteristic ring of $n$-plane bundles over $M$ with respect to the ring $R$.

\section{Standard Classes and their Construction}

\subsection{The Stiefel Whitney Classes}

The definition of Stiefel Whitney classes can be put on an axiomatic footing presuming we know their existence.  Proving their existence is actually the most difficult part; once we know that they exist and what their properties are it is (relatively) easy to do calculations.  The following is directly from \cite{[MS]}.

\emph{AXIOM 1}. To each vector bundle $\chi$ there corresponds a sequence of cohomology classes

$w_{i}(\chi) \in H^{i}(B(\chi);Z/2), i \in {0,1,2,...}$

called the Stiefel Whitney classes of $\chi$.  The class $w_{0}(\chi)$ is equal to the unit element

$1 \in H^{0}(B(\chi);Z/2)$ and $w_{i}(\chi)$ is zero for $i \geq n$ if $\chi$ is an $n$-plane bundle.

(an n-plane bundle here basically means a bundle with an $n$-dimensional vector space for fibres).

For all intents and purposes, $B(\chi)$ will often be the tangent bundle of a manifold $M$ (i.e. $M$ is the base space, with fibres at each point corresponding to tangent spaces).  In this case we often refer to the above axiom/definition as defining the Stiefel-Whitney classes of $M$.

\emph{AXIOM 2}. (Naturality) If $f: B(\chi) \rightarrow B(\eta)$ is covered by a bundle map from $\chi$ to $\eta$, then

$w_{i}(\chi) = f^{*}w_{i}(\eta)$.

A bundle map from $\chi$ to $\eta$ is a continuous function $g: E(\chi) \rightarrow E(\eta)$ which carries each vector space $F_{b}(\chi)$ isomorphically onto one of the vector spaces $F_{b'}(\eta)$.

\emph{AXIOM 3}. (The Whitney Product Theorem) If $\chi$ and $\eta$ are vector bundles over the same base space, then

$w_{k}(\chi \oplus \eta) = \sum_{i=0}^{k}w_{i}(\chi) \cup w_{k-i}(\eta)$

Here $\cup$ is the standard cup product operation in cohomology, and $\chi \oplus \eta$ is the Whitney sum of $\chi$ and $\eta$, which I shall now proceed to define.

First of all, it is necessary to define the notion of an induced bundle.  Let $\chi$ be a vector bundle with projection $\pi : E \rightarrow B$ and $B_{1}$ an arbitrary topological space.  Given any map $f: B_{1} \rightarrow B$ define the induced bundle $f^{*}\chi$ over $B_{1}$ as follows.  The total space $E_{1}$ of $f^{*}\chi$ is the subset $E_{1} \subset B_{1} \times E$ consisting of all pairs $(b,e)$ with $f(b) = \pi(e)$.  The projection map $\pi_{1} : E_{1} \rightarrow B_{1}$ is defined by $\pi_{1}(b,e) = b$.  Then, if we define $\hat{f} : E_{1} \rightarrow E$, $\hat{f} : (b,e) \mapsto e$ then we have $\pi \circ \hat{f} = f \circ \pi_{1}$.

Now take two bundles $\chi_{1}$, $\chi_{2}$ over the same base space $B$.  Let $d : B \rightarrow B \times B$ denote the diagonal embedding.  Then the bundle $d^{*}(\chi_{1} \times \chi_{2})$ is referred to as the Whitney sum of $\chi_{1}$ and $\chi_{2}$.

Finally, we have

\emph{AXIOM 4}.  For the line bundle $\gamma^{1}_{1}$ over the circle $P^{1}$, the Stiefel-Whitney class $w_{1}(\gamma_{1}^{1})$ is non-zero.

These 4 axioms completely characterise Stiefel Whitney classes.


\subsection{The Euler Class}

Let $E_{0}$ be the set of all nonzero elements in the total space $E$ of a oriented n-plane bundle $\chi$.

We have the following important theorem:

\begin{thm}  The group $H^{i}(E,E_{0})$ is zero for $i < n$, and $H^{n}(E, E_{0})$ contains a unique class $u$ such that for each fibre $F = \pi^{-1}(b)$ the restriction

\begin{center}
$u | (F,F_{0}) \in H^{n}(F,F_{0})$
\end{center}

is the unique non-zero class in $H^{n}(F,F_{0})$.  Furthermore the correspondence $x \mapsto x \cup u$ defines an isomorphism $H^{k}(E) \rightarrow H^{k + n}(E, E_{0})$ for every $k$.  (We call $u$ the fundamental cohomology class.)
\end{thm}

In order to define the Euler class, we consider the inclusion $(E, \emptyset) \subset (E,E_{0})$.  This gives rise to a restriction homomorphism $H^{*}(E,E_{0};Z) \rightarrow H^{*}(E;Z)$, denoted by $y \mapsto y | E$.  Applying this to the fundamental class $u \in H^{n}(E,E_{0};Z)$ we obtain a new class $u | E \in H^{n}(E;Z)$.  But we have that $H^{n}(E;Z)$ is canonically isomorphic to $H^{n}(B;Z)$.  This isomorphism follows from a highly nontrivial theorem that I will not prove.

The Euler class of an oriented n-plane bundle $\chi$ is then defined to be the cohomology class $e(\chi) \in H^{n}(B;Z)$ which corresponds to $u | E$ under the canonical isomorphism $\pi* : H^{n}(B;Z) \rightarrow H^{n}(E;Z)$.


There is in fact a connection between the Euler class of $TM$ and the Euler characteristic of $M$, where $M$ is a manifold and $TM$ the corresponding tangent bundle.

First we need to define the notion of Kronecker index.

For $M$ a closed, possibly disconnected, smooth $n$-manifold, there is a unique fundamental homology class

\begin{center}
$\mu(M) \in H_{n}(M ;\mathcal{Z}/2)$
\end{center}

For any cohomology class $\nu \in H^{n}(M ; \mathcal{Z}/2)$, we define the Kronecker index as $<\nu, \mu_{M}> = \nu(\mu_{M}) \in \mathcal{Z}/2$.

\begin{thm} If $M$ is a smooth compact oriented manifold, then the Kronecker index $<e(\tau_{M}),\mu>$, using rational or integer coefficients, is equal to the Euler characteristic $\chi(M)$. Similarly, for a non-oriented manifold, the Stiefel-Whitney number $<w_{n}(\tau_{M}),\mu> = w_{n}[M]$ is congruent to $\chi(M)$ modulo $2$.
\end{thm}

\subsection{Chern Classes}

To construct Chern classes, one uses, instead of real vector bundles, complex vector bundles.  So let $\omega$ be a complex n-plane bundle.  Construct a new canonical (n-1)-plane bundle $\omega_{0}$ over the deleted total space $E_{0}$.  A point in $E_{0}$ is specified by a fibre $F$ of $\omega$ together with a non-zero vector $v$ in that fibre.  Suppose a Hermitian metric has already been specified on $\omega$.  Then the fiber of $\omega_{0}$ over $v$ is defined to be the orthogonal complement of $v$ in $F$.  These vector spaces can be considered as fibers of a new vector bundle $\omega_{0}$ over $E_{0}$.

Any real oriented $2n$-plane bundle possesses an exact Gysin sequence

\begin{center}
$... \rightarrow H^{i-2n}(B) \rightarrow^{\cup e} H^{i}(B) \rightarrow^{\pi^{*}_{0}} H^{i}(E_{0}) \rightarrow H^{i-2n+1}(B) \rightarrow ...$
\end{center}

with integer coefficients.  Here $\pi^{*}_{0}$ is the restriction of the projection map of $\omega$, $\pi$, to $E_{0}$.

The groups $H^{i-2n}(B)$ and $H^{i-2n+1}(B)$ are zero for $i < 2n - 1$, from which it follows that $\pi^{*}_{0} : H^{i}(B) \rightarrow H^{i}(E_{0})$ is an isomorphism.

Now define the Chern classes $c_{i}(\omega) \in H^{2i}(B;Z)$ by induction on the complex dimension $n$ of $\omega$.  Define the top Chern class $c_{n}(\omega)$ to coincide with the Euler class $e(\omega_{R})$.  For $i < n$ set

\begin{center}
$c_{i}(\omega) = \pi^{*-1}_{0}c_{i}(\omega_{0})$
\end{center}

Suppose $T$ is the total space of $\omega_{0}$.  Then what we are really doing here is computing the Chern classes of $\omega_{0}$ which lie in $E_{0}$ (the base space of $\omega_{0}$), and then defining the Chern classes of $\omega$ by pulling these back to $B$ via the map $\pi^{*-1}_{0}$.

$\pi^{*}_{0} : H^{2i}(B) \rightarrow H^{2i}(E_{0})$ is an isomorphism for $i < n$, so this is ok.  For $i > n$ the class $c_{i}(\omega)$ is defined to be zero.


There is in fact a link between Chern classes and Stiefel Whitney classes; Chern classes are more or less the even Stiefel Whitney classes of a bundle of even dimension, without the reduction to $Z_{2}$ coefficients.

\subsection{Pontrjagin Classes}

I conclude my summary of (some of) the ideas in [MS] with a construction of the Pontrjagin classes for a bundle.  The $i$-th Pontrjagin class of a real n-plane bundle $\chi$

\begin{center}
$p_{i}(\chi) \in H^{4i}(B;Z)$ .
\end{center}

is defined to be the integral cohomology class $(-1)^{i}c_{2i}(\chi \otimes \mathcal{C})$

Here $\chi \otimes \mathcal{C}$ is the so called complexification of $\chi$, where we take the tensor product of each fiber $V$ with $\mathcal{C}$, the complex numbers.



\section{Generalisations}

A direction of further investigation might be to examine the world of characteristic classes for general fibre bundles.  In other words, dropping the restriction that the structure group be the general linear group.  I am quite interested to see what the analogies of the above classes might be in the general setting, if in fact there are any.

\section{Generalised Invariants and application to Exotic Differentiable Structures}

\subsection{Definition}

The motivating question here is:

\begin{quest}
Given a set with a particular topological structure, ie a particular homeomorphism class, is there a systematic way towards classifying all possible differentiable structures that this set admits?
\end{quest}

There are of course many topologies that are somewhat troublesome to work with, and which may not even admit differentiable structures at all.  So we make the following assumption:

There is, about each point in our set $X$, an open set $U$ and a function $f$ such that the map $f : U \rightarrow R^{n}$ is a homeomorphism.  We call $n$ the \emph{dimension} of our topology.  If two such sets $U$ and $V$ overlap, we require the transition function from $U$ to $V$ to be a homeomorphism.

Note that there will always be at least one way of placing a differentiable structure on such a manifold (by taking finer charts in the original topology, and requiring transition maps to be differentiable).

If there is a map $g : X \rightarrow R^{m}$ such that $g$ is a bilinear mapping and $g(X)$ is a differentiable submanifold of $R^{m}$, we say that $X$ inherits the differentiable structure from $R^{m}$ from its embedding via $g$.  There may well be more than one way to embed $X$ in $R^{m}$, of course, which is kind of the whole point.

So-

What we would like to do is produce a class of invariants that, for a given differentiable manifold, will specify its diffeomorphism class uniquely, given of course that we know all of the invariants.  So consider homotopy classes of maps from a space $Y$ into our space $X$.  For example, the homotopy groups $\pi_{n}(X)$ arise in this way, for $Y_{n} = S^{n}$.

Now, let $Y$ be an arbitrary differentiable submanifold of infinite dimensional euclidean space.  We can represent such spaces in general by associating a symmetric bilinear (possibly degenerate) form $g(x)$ to each point $x$ of $R^{\infty}$ such that in a local chart the induced functions $g_{ij}(x)$ are smooth; in other words we are thinking of them as (pseudo)-Riemannian submanifolds of our infinite dimensional space.  Now, since for any space $X^{n}$ there is an embedding $k : X^{n} \rightarrow R^{2n+1}$, it seems reasonable to restrict to pseudo-Riemannian submanifolds of $R^{2n+1}$ and expect the following result to hold:

\begin{conj} Suppose we know, for all symmetric bilinear forms $g$ on $R^{2n+1}$, the homotopy groups $G(g,0,X)$ corresponding to the homotopy classes of continuous maps $f : (R^{2n+1},g) \rightarrow X$.  Then we know the homeomorphism class of $X$.
\end{conj}

Now suppose we know the homeomorphism class of $X$, and want to know how many possible diffeomorphism classes we have to choose from.  This leads us to make the following generalisation: instead of considering merely continuous maps $f$ from $(R^{2n+1},g)$ to $X$, require only that the maps be $\alpha$-H\"older continuous, that is:

\begin{center}
$lim_{x \rightarrow x_{0}}\frac{\vert f(x) - f(x_{0})\vert}{\vert x - x_{0} \vert^{\alpha}}$
\end{center}

is well defined, for every $x_{0} \in R^{2n+1}$.

We let $\alpha$ vary between $0$ and $1$.  Clearly $\alpha = 0$ corresponds to continuous maps and $\alpha = 1$ corresponds to differentiable maps.  We expect there to be greater variety in the latter than the former, since in general we will have more obstructions to being able to find a homotopy between disparate mappings if we restrict to differentiable functions.  This leads naturally to the following additional

\emph{Conjecture}: Let $0 \leq \alpha \leq \beta \leq 1$.  Then

\begin{center}
$G(g,\alpha,X) \subset G(g,\beta,X)$
\end{center}

for every bilinear form $g$ on $R^{2n+1}$, where $G(g,a,X)$ is the homotopy group corresponding to $g$, $X$, and $a$.  In fact, we might expect something even stronger than this:

\begin{center}
$G(g,\alpha,X) \unlhd G(g,\beta,X)$
\end{center}

\begin{cor}  It is possible to take the quotient of the space of our invariants corresponding to differentiable maps (diffeomorphism classes) by the space of our invariants corresponding to continuous maps (homeomorphism classes).
\end{cor}

So for our example, after we have gone through this process, we will know precisely how many ways there are to place differentiable structures on $X$ given that we know its homeomorphism class.  The quotient space may often be a point- in which case there will be only one way.

Of course, this may not tell us \emph{how} to constuct these diffeomorphism classes- but it may be a good first step towards such a goal.





\subsection{Geometric Interpretation of Generalised Invariants}

There is a natural way to think of the groups $G(g,1,X)$.

\emph{Claim}: Suppose $R_{ij}(g,f)$ is the induced Ricci curvature of the image of $(R^{2n+1},g)$ in $X$ via the differentiable map $f$.  Let $f_{*}g$ be the push forward of $g$.  Take the limit of the Ricci Flow in $X$,

\begin{center}
$\frac{\partial}{\partial t}f_{*}g(t) = - Ric(f_{*}g(t))$, $Ric(f_{*}g(0)) = Ric(g,f)$
\end{center}

Then the limits of the Ricci flow under this construction, counted with multiplicity ie as varifolds, correspond to the elements of the group $G(g,1,X)$, i.e. the limit of the flow above corresponds to the homotopy class of the map $f$ from $N = (R^{2n+1},g)$ to $X$.  If a flow may be perturbed from one limit to another, ie if we modify the image via a diffeomorphism and restart the flow we get a different limit, then we identify the two limits.

In other words, what we are doing, if we consider the Ricci flow as $\sim_{1}$, and the perturbation identification as $\sim_{2}$, is claiming that

\begin{center}
$G(g,1,X) \cong (\{$all maps from $N$ to $X \}/\sim_{1})/\sim_{2}$
\end{center}

Intuitively one should think of the fundamental group and the torus, for instance, and consider the action of this flow as causing the image $f(S^{1})$ to pinch around topological obstructions in the limit.

This is slightly messy; evidently we might want to remove the later consideration about perturbation identification.  So instead first volume normalised Ricci flow the manifold $X$ with the marking $f(N)$ within the ambient space $R^{2n+1}$; then flow $f(N)$ within the ambient space $X$.  Then the worst thing that can happen is for the Ricci flow limits to degenerate; for instance, consider the smooth torus.  Then any circle around its waist is a valid limit of the Ricci flow of $f(S^{1})$.  So we get a one dimensional family of solutions which smoothly vary one into the other.  This leads us to the following

\emph{Claim}: If we perform this modified procedure, then the Ricci flow limits will correspond to the elements of the group $G(g,1,X)$, with the implicit assumption that we identify Ricci flow limits that can be deformed smoothly from one to the other.

In fact we might expect:

\emph{Claim}: Degeneracy of limits will correspond to an underlying symmetry of the ambient space $X$.  (For example, for the torus this is circular symmetry).


\chapter{Pseudo-Riemannian Geometry}


\section{Overview}

Out of general interest, I acquired a very good book on Pseudo (or Semi) Riemannian Geometry \cite{[O]} and made a study of it towards the middle of 2005.  The essential difference between pseudo and standard Riemannian geometry, as is explained quite quickly, is the relaxing of the condition that the metric $g$ on the tangent bundle to a manifold be a symmetric positive definite bilinear form to being merely a \emph{nondegenerate} symmetric bilinear form with fixed index (dimension of the space of negative eigenvalues of $g$).  One recalls that nondegeneracy of a bilinear form means that its matrix in a given representation is invertible.

Many of the results that apply to Riemannian manifolds follow through for semi-Riemannian manifolds- though by no means all.  One also gets an interesting interplay between timelike, spacelike, and lightlike vectors (vectors where $g(v,v)$ is negative, positive, or zero respectively).  The 0 vector is defined to be spacelike.

A particularly important case of a semi-riemannian manifold is the case where the manifold has a metric with index one.  These are referred to as Lorentz manifolds, and are the core ingredient in the physical model known as general relativity.

Riemannian geometry can be generalised to what I shall call \emph{Riemann-Cartan} geometry.  For observe that if instead $g$ is an antisymmetric bilinear form, that the fundamental theorem of Riemannian geometry still goes through, and we get a unique connection, known as the \emph{Cartan connection}, corresponding to that metric $g$.  We may then conclude by linearity of the proof of this statement for both symmetric and antisymmetric bilinear forms that we may drop the assumption of any form of symmetry, and consider instead a merely nondegenerate form $g$.  Then, by splitting it into its symmetric and antisymmetric components, we establish the existence of unique connections corresponding to these parts; by linearity we establish the existence of a unique connection corresponding to $g$.

Everything else then carries through for Riemann-Cartan geometry.  The curvature tensor still remains a symmetric tensor, since it is induced by the second geometric derivative.  We still have geodesics, parallel transport, and injectivity radii- the entire language remains the same.

\section{The fundamental theorem of Riemannian geometry}

\subsection{The metric tensor}

In standard euclidean geometry $R^{n}$, in particular, vector calculus, we have the notion of the "dot product" $\cdot : R^{n} \times R^{n} \rightarrow R$ which is a generalisation in some sense of the angle between two vectors.  In particular, $a \cdot b := \norm{a}\norm{b}cos(\theta(a,b)) = \sum_{i}a_{i}b_{i}$.  This is an example in fact of an "inner product" on $R^{n}$.  There are others.  For instance we might define $a \bar{\cdot} b = 2a_{1}b_{1} + a_{2}b_{2} + a_{3}b_{3} + ... + a_{n}b_{n}$.  For this new inner product, we have an "induced angle" $\bar{\theta}(a,b)$, such that $a \bar{\cdot} b = \norm{a}\norm{b}cos(\bar{\theta}(a,b))$, where now $\norm{a}^{2} = a \bar{\cdot} a$.  So in a certain sense the geometry of our space, which is described with distances and angles, can be made to depend wholly on the choice of inner product.

Now consider a differentiable manifold $M$.  We generalise the idea of inner products on $R^{n}$ to inner products on $TM$.  In particular, we define a \emph{Riemannian metric} on $M$ to be a symmetric positive definite bilinear form $g : TM \times TM \rightarrow R$ such that $g_{ij} := g(x_{i},x_{j})$ are smooth functions for each and every chart $x_{i}$ for some element $U$ of an atlas for $M$.

\subsection{The euclidean covariant derivative}

In euclidean space, we may have the idea of a \emph{vector field}.  A vector field is a map $V : R^{n} \rightarrow R^{n}$ which assigns to each point a vector is a smoothly varying fashion.  We may take derivatives of vector fields in a standard sort of way.  In particular, in much that same sort of way one can take the directional derivative with respect to a function, $f$, one can take a directional derivative with respect to a vector field.

Recall that $\nabla_{v}f$, the directional derivative of $f$ in direction $v$, is $\nabla f \cdot v$.  We define in a similar sort of sense directional derivatives for a vector field $X$; write $X = (X_{1},...,X_{n})$ for functions $X_{i}$.  Then $\nabla_{v}X := (\nabla_{v}X_{1},...,\nabla_{v}X_{n})$.

This leads us to the notion of the covariant derivative for euclidean space.  Let $v = W(x)$ now just be an element of a vector field $W$.  Observe first of all that $W(f) := \nabla_{W}f = \nabla f \cdot W$ will no longer be constant but variable with position.  Then we define $\nabla_{W}X$ to be the vector field $(\nabla_{W}X_{1},...,\nabla_{W}X_{n})$.

\subsection{Axiomatic extension to arbitrary metrics}

In much the same way we extended the idea of inner product from euclidean space to Riemannian metrics on differentiable manifolds, we may axiomatically extend the idea of covariant derivative $\nabla$ from euclidean space to Riemannian manifolds.

First of all, observe that the Euclidean covariant derivative has the following properties:

\begin{itemize}
\item[(i)] $\nabla_{fX + gY}Z = f\nabla_{X}Z + g\nabla_{Y}Z$
\item[(ii)] $\nabla_{X}(Y + Z) = \nabla_{X}Y + \nabla_{X}Z$
\item[(iii)] $\nabla_{X}(fY) = f\nabla_{X}Y + X(f)Y$
\end{itemize}

In addition, we have that it is compatible with the Euclidean metric, ie if we write $D = exp^{-1} \circ d$ where $exp : R^{n} \rightarrow R^{n}$ is the identity for euclidean space, then

\begin{center} $\frac{d}{dt}g(V,W) = g(\frac{DV}{dt},W) + g(V,\frac{DW}{dt})$ \end{center}

Finally we have that it is symmetric, that is

\begin{center} $\nabla_{X}Y - \nabla_{Y}X = [X,Y]$ \end{center}

Then, it is simplicity to generalise from here.  If an operator $\nabla$ satisfies the first three conditions it is said to be an \emph{affine} connection. If in addition it satisfies the fourth and fifth where now $g$ is to be a Riemannian metric, we will call it a Riemannian, or Levi-Civita connection associated to that metric.  In fact, it turns out that there is a one to one correspondence between these two objects, which I will establish shortly.

\subsection{Uniqueness of the Levi-Civita connection}

In the previous section I described how one can abstract the idea of a connection to a Riemannian manifold starting from the properties of the standard connection on Euclidean space.  I will now prove that, up to choice of metric, such connections are unique.

\begin{thm} ("The fundamental theorem of Riemannian geometry").  Given a Riemannian manifold $M$, there exists a unique Riemannian connection associated to it.

\begin{proof} If we assume existence, then certainly

\begin{center} $X(Y,Z) = g(\nabla_{X}Y,Z) + g(Y,\nabla_{X}Z)$ \\
$Y(Z,X) = g(\nabla_{Y}Z,X) + g(Z,\nabla_{Y}X)$ \\
$Z(X,Y) = g(\nabla_{Z}X,Y) + g(X,\nabla_{Z}Y)$ \end{center}

Adding the first two of these statements together then subtracting the third gives

\begin{center} $Xg(Y,Z) + Yg(Z,X) - Zg(X,Y) = g([X,Z],Y) + g([Y,Z],X) + g([X,Y],Z) + 2(Z,\nabla_{Y}X)$ \end{center}

which is an expression that can easily be rearranged to give $\nabla$ uniquely in terms of the metric.  To demonstrate existence, define $\nabla$ to satisfy the above expression.  Then it is easy to show that this definition satisfies the axioms.
\end{proof}
\end{thm}

\begin{rmk} For Riemann-Cartan manifolds, the Levi-Civita uniqueness theorem still holds, but with the following minor modification. Instead of considering $Xg(Y,Z) - Yg(X,Z) + Zg(X,Y)$, it is necessary to consider the sum $(Xg(Y,Z) - Xg(Z,Y)) - (Yg(X,Z) - Yg(Z,X)) + (Zg(X,Y) - Zg(Y,X))$, due to the lack of symmetry.  One should then get an expression wherein all terms in the connection save one can be made to cancel, and uniqueness follows.  \end{rmk}

\section{Geodesics and the geodesic equation}

\subsection{Definition and variational formulation}

A geodesic in differential geometry is defined to be the shortest path between two points on a Riemannian manifold.  In this sense it is a generalisation of the notion of straight line in Euclidean space.  In particular, we are interested in measuring the length of a path.  In euclidean space, the distance from $a$ to $b$ along a path $\gamma : [0,1] \rightarrow R^{n}$ will be $\int_{0}^{1}\norm{\gamma '(t)}dt$; the integral of the velocity along the parametrised curve.

Similarly, in Riemannian geometry, we define the length to be

\begin{center} $L(\gamma) = \int_{0}^{1}g(\gamma '(t),\gamma '(t))^{1/2}dt$ \end{center}

where now one takes the norm with respect to the inner product $g$ on the manifold in question.

In order for the path to be a shortest path, we require that the first variation of the length vanish:

\begin{center} $\delta L(\gamma) = 0$ \end{center}

which is equivalent to requiring that

\begin{center} $\nabla_{\gamma '}\gamma ' = 0$ \end{center}

This above expression is known as the \emph{geodesic equation}.  It can be written in coordinates as

\begin{center} $\frac{d^{2}x_{k}}{dt^{2}} + \Gamma^{k}_{ij}\frac{dx_{i}}{dt}\frac{dx_{j}}{dt} = 0$ \end{center}

where $\Gamma^{k}_{ij} := (\nabla_{X_{i}}X_{j})_{k}$ are the \emph{Christoffel symbols} associated to the metric $g$.

A couple of examples of geodesics for instance are straight lines in Euclidean space or great circles on the sphere.

\subsection{The cut and conjugate loci}

The cut locus of a point $p$ in a space $M$ is defined to be the set of points $C(p)$ such that their preimage is the boundary of the critical ball about $0$ in $T_{p}M$.  The conjugate locus of the same point $p$ is the set of points $\bar{C}(p)$ such that the exponential map fails to be a diffeomorphism on their preimage in $T_{p}M$.

\section{Curvature}

The curvature is essentially the first correction term, or first obstruction from a Riemannian metric being locally flat.  In other words, provided $\norm{x} = o(\delta)$,

\begin{center} $g_{ij}(\frac{x}{\delta}) = \delta_{ij}(0) +  \delta^{2}R_{ijkl}(0)\frac{x_{k}}{\delta}\frac{x_{l}}{\delta} + o(\delta^{2})$ \end{center}

It can be thought of alternatively as the "geometric acceleration".  For instance a sphere has positive curvature, whereas a saddle has negative curvature.  Similarly, it is possible, given bounds on curvature, to deduce how quickly geodesics will converge or diverge.

\subsection{The curvature tensor and its symmetries}

It turns out that the curvature tensor can be more precisely described as

\begin{center} $R_{ijkl} := g(R(X_{i},X_{j})X_{k},X_{l})$ \end{center}

where

\begin{center} $R(U,V) :=  \nabla_{U}\nabla_{V} - \nabla_{V}\nabla_{U} - \nabla_{[U,V]}$ \end{center}

It can be checked that it has the properties of a tensor, ie linear in each argument.  Furthermore, it can be shown to have the following symmetries:

\begin{itemize}
\item[(i)] $R(u,v) = -R(v,u)$
\item[(ii)] $g(R(u,v)w,z) = -g(R(u,v)z,w)$
\item[(iii)] $R(u,v)w + R(v,w)u + R(w,u)v = 0$
\item[(iv)] $\nabla_{u}R(v,w) + \nabla_{v}R(w,u) + \nabla_{w}R(u,v) = 0$
\end{itemize}

The Ricci curvature is defined to be $R_{il} := R_{ijkl}g^{jk}$, the contraction of the Riemann curvature with respect to the second and third indices.

The scalar curvature is the contraction of the Ricci curvature, $R := R_{il}g^{il}$.

\subsection{Interpretation of the Ricci and Scalar curvature}

Consider a geodesic $\gamma$.  Define a metric tube $\gamma_{\epsilon}$ about $\gamma$ as

\begin{center}
$\gamma_{\epsilon} = \{x \vert d(x,\gamma) \leq \epsilon \}$
\end{center}

Then

\begin{center}
$Vol(\gamma_{\epsilon}) = \epsilon^{n-1}\norm{S^{n-2}}L[\gamma] - c(n)\epsilon^{n+1}\int_{\gamma}R(\gamma ',\gamma ')ds$
\end{center}

for $\epsilon << 1$.  This gives us a geometrical interpretation of the Ricci curvature.

Similarly,

\begin{center}
$Vol(B_{r}(x)) = r^{n}\norm{Vol B_{1}^{R^{n}}} - c(n)r^{n+2}R(x)$
\end{center}

for $r << 1$, giving us a geometrical idea of what the scalar curvature does.

\subsection{Some basic comparison results}

As I mentioned before, curvature bounds, and knowledge of the curvature in general, can be used to conclude various things about the topology of the space.  For instance, the celebrated Gauss-Bonnet formula states that the integral of the scalar curvature of a closed surface is $2\pi$ times the Euler characteristic of that surface.  So we are extracting purely topological information out from what we put in.  Higher dimensional analogues of this result are possible, but I shall neither sketch nor prove them here.

Another result of importance is the Toponogov comparison theorem, which essentially states that if one takes geodesic triangles with two sides of fixed length and angle in two different spaces, and the first has smaller curvature than the second, then the second will have a shorter third side.  An immediate consequence of this is that geodesics will  converge more rapidly in spaces of higher curvature.

Knowledge of curvature also allows one to bound certain things like the injectivity radius of a space from above. The injectivity radius of a point $p$ is the maximum number $r$ such that the exponential map is injective on all balls of radius $r$ in $T_{p}M$;  the injectivity radius of a manifold is the infinum over all such points.

\section{Stoke's Theorem}

In this section I will give two different types of generalisations of Stoke's theorem, which will prove relevant once I have finished with the preliminaries and begin to develop the information geometry in chapter 7.  The first of these I developed a number of years ago, back in 2002-3.  I did not derive a careful proof of this first result however until late 2007.


\begin{thm}
Let $N^{k+1}, M^{n+1}$ be differential manifolds.  Let $U^{k+1}(m) \subset N$, $V^{n+1} \subset M$ be compact subsets of $N$, $M$, where $U^{k+1}(m)$ is a family of subsets of $N$ specified by $m \in V$ such that the function $m \mapsto U^{k+1}(m)$ is smooth.  Let $\omega \in \Omega^{k}(N) \otimes \Omega^{n}(M)$.  Let $d_{N}$ be the exterior derivative or de Rham operator on $\Omega(N)$; let $d_{M}$ be likewise w.r.t. $\Omega(M)$.  Then we have

\begin{equation}
\label{Stokes}
\int_{V}\int_{U(m), m \in V}d_{N}d_{M}\omega = \int_{\partial V}\int_{\partial U(m), m \in V}\omega
\end{equation}
\end{thm}

\begin{rmk}
The nontrivial nature of this result is that $\omega$ is allowed to have the form

\begin{center}
$f(x,y)dx_{1}...dx_{k}dy_{1}...dy_{n}$
\end{center}

in local coordinates, though it is possible that in all the examples I have been working with the function $f(x,y)$ could be expressed in a power series expansion $\sum_{i = 0}^{\infty}\sum_{j = 0}^{\infty} a_{i}(x)b_{j}(y)$.  But it is still not obvious how to prove the above because we are dealing with a \underline{smooth family} of sets in $N$ specified by a single set in $M$.
\end{rmk}

\begin{rmk} The general idea of this result, or at least a generalisation, will prove useful later when I am looking into mathematical turbulence in chapter 9.
\end{rmk}

\begin{proof} Without loss of generality, assume that $\omega = f(x,y)dxdy$ in local coordinates as in the remark above.  Make the further assumption that $f$ and all its derivatives vanish at infinity.  Then we may split $f$ as a doubly infinite sum of a set of suitable eigenfunctions as in the remark above.  The result then follows by applying the standard form of Stoke's theorem twice for each term in the sum, and then bringing everything back together.
\end{proof}

The second type of generalisation of Stoke's theorem is as follows:

\begin{thm}  Let $M$ be a $n$-manifold.  Let $A$ be diffeomorphic to the space of inner products on $R^{n}$.  Consider a family of differential operators $d_{(M ; a)}$ on $M$ indexed smoothly by $A$.  Let $\omega(m,a)$ be a family of smooth $(n - 1)$ forms on $M$ indexed smoothly by $A$.  Then we have that

\begin{center} $\int_{M}\int_{A}d_{(M ; a)}\omega(m,a) = \int_{\partial M}\int_{A}\omega(m,a)$ \end{center}
\end{thm}

\begin{rmk} We can in fact view $A$ as a distribution of sorts over each point of $M$, so in fact this result has a statistical flavour.  Additionally, we are taking only one derivative, and not two.  It will turn out that this sort of result is core to the development of the theory of information for statistical manifolds, which is the subject matter of chapters 6 through 8.
\end{rmk}

\begin{proof} A bit of care is required here since we are dealing with statistical derivatives.  First observe that we may write $\omega(m,a) = F(m,a)dmda$, and $d_{(M;a)}(m) = \int_{b \in A}exp_{a}(m) \circ d_{M} \circ exp^{-1}_{b}(m)$ where we have a 1-1 correspondence between $a \in A$ and inner products at $m \in M$.  

Then 

\begin{center} $\int_{M}\int_{A}d_{(M ; a)}(m)\omega(m,a) = \int_{M}\int_{A}\int_{A}exp_{a}d_{M}(exp^{-1}_{b}(m)F(m,a))dbdadm$ \end{center}

By Stoke's theorem in the standard case this reduces to

\begin{center} $\int_{\partial M}\int_{A}\int_{A}exp_{a}exp^{-1}_{b}F(m,a)dbdadm$ \end{center}

since $d_{M}exp_{a}(m) = 0$ via conservation of probabilistic flux.

This finally can be simplified, after observing that since $d_{(M;a)}$ is a statistical derivative we must have

\begin{center} $\int_{A}exp_{a}exp^{-1}_{b}db = Id_{M}$ \end{center}

to

\begin{center} $\int_{\partial M}\int_{A}F(m,a)dadm$ \end{center}

\end{proof}

\begin{rmk} So essentially the theorem more or less works, but we have to make sure that our family of differential operators is chosen in such a way to make "physical sense".  In other words, the probabilistic flux must be conserved, and also we need to have that our distribution is normalised in probabilistic weight to one.  Otherwise we run into trouble.  The appropriate tools to tackle these issues will be developed in some greater detail and along slightly different lines in chapter 6.
\end{rmk}













\section{K\"ahler geometry and Convexity Theory}

\subsection{The notion of convexity in determining stability}


There is a well known theorem in minimal surface theory that gives conditions under which a surface is in fact area minimising.

\begin{thm} (folk lore).  If one has a minimal graph over a compact, convex domain, then it is in fact a length minimising graph, where by length I mean length in the general n-dimensional sense.
\end{thm}

This leads one to ask the natural questions:

If one has a statistical structure in a suitable sense, what are the conditions on a substructure for it to be a volume minimising solution?

Is it possible to have a weaker condition than convexity on a domain for it to be necessarily minimising?

It turns out that the theory of calibrations is uniquely suited to answering the second of these two questions.  As for the first, that will have to wait until later.  I give here the main results and ideas about calibrations without proof.  The main reference here is the paper of Harvey and Lawson on the matter \cite{[HL]}, though I mainly made use of a review article by Ivanov \cite{[Iv]}.

\begin{dfn} Let $\phi$ be a differential form of degree $k$ on $R^{n}$.  Suppose that $d\phi = 0$ and $\phi(\zeta) \leq vol(\zeta)$ for any oriented $k$-plane $\zeta$, where $vol$ is the volume form in $R^{n}$.  Then we will say that $\phi$ is a $n$-Euclidean calibration form.  \end{dfn}

The utility of such forms is that they can be used to prove the existence of stable globally minimising submanifolds:

\begin{thm}  If $X^{k}$ is a $k$-surface, and $\phi$ a $n$-Euclidean calibration form such that $\phi_{X}$ (the restriction of $\phi$ to $X$) is the volume form on $X$, then $X$ is globally minimal in $R^{n}$. \end{thm}

Of course, the notion can be extended to any smooth Riemannian manifold $M^{n}$.  In this case, we just say that $\phi$ is a calibration form.  The theorem goes through for this generalisation as well. 


\subsection{K\"ahler manifolds}

I shall now proceed to define carefully the idea of a K\"ahler manifold and develop the theory of these objects to the point where I am able to justify the theorem above.  In doing this I will follow  the treatment given by Lawson in [L2] very closely.

\begin{dfn}
A \emph{$n$-complex manifold} is a manifold $M$ that is locally diffeomorphic to $\mathcal{C}^{n}$, and which has biholomorphic transition functions (so that the differentials of these functions are everywhere complex linear).  In particular, this implies that there is a map $J_{p} : T_{p}M \rightarrow T_{p}M$ such that

\begin{equation}
\label{compstructure}
J_{p}^{2} = -1
\end{equation}

So if we have local coordinates $(z^{1},...,z^{m}) = (x^{1} + iy^{1},...,x^{m} + iy^{m})$ where the $z^{i}$ are complex and the $x^{i}, y^{i}$ are real then $J(\partial/\partial x^{i}) = \partial / \partial y^{i}$ and $J(\partial/\partial y^{i}) = - \partial/\partial x^{i}$.
\end{dfn}

We also have that

\begin{equation}
\label{torsionKahler}
T_{X,Y} = [JX,JY] - J[JX,Y] - J[X,JY] - [X,Y] = 0
\end{equation}

as can be easily checked.  There is a deep theorem of Newlander and Nirenberg that states that any $C^{\infty}$ manifold that admits a tensor field $J$ of type $(1,1)$ satisfying the above equations can be made into a complex manifold with an appropriate atlas of charts.

We would of course like to address the problem of what a natural choice of metric on a complex manifold should be.  When all is said and done, we require the following restrictions on the choice of metric:

\begin{equation}
\label{restrictionA}
g(JX,JY) = g(X,Y)
\end{equation}
for all $p \in M$ and all $X,Y \in T_{p}M$.  In other words we want $J$ to be an isometry, so the metric $g$ must be Hermitian.

\begin{equation}
\label{restrictionB}
(\nabla_{X}J)(Y) = \nabla_{X}(JY) - J(\nabla_{X}Y) = 0
\end{equation}
for all tangent vector fields $X$ and $Y$ on $M$, where $\nabla$ is the riemannian connection.  This is equivalent to requiring that $J$ be globally parallel to the riemannian connection.

\begin{dfn}
If a complex manifold $M$ satisfies these two equations, we say that $M$ is a K\"ahler manifold.
\end{dfn}

It is useful to extend a hermitian metric $g$ on a complex manifold $M$ to a complex valued, "sesquilinear" form $h$ defined in the following way:

$h(X,Y) = g(X,Y) + i\omega(X,Y)$

where

$\omega(X,Y) = g(X,JY)$

for all $X,Y \in T_{p}M$.

The condition of being K\"ahlerian imposes a strong restriction on $\omega$:

\subsection{Relation to minimal surface theory}



In minimal surface theory there is a well known result.

\begin{thm} All minimal graphs $f: \mathcal{R}^{2n} \rightarrow \mathcal{R}^{k}$ over domains of even dimension can be lifted to minimal graphs $\hat{f}: \mathcal{C}^{n} \rightarrow \mathcal{C}^{k}$. \end{thm}

More generally, we have the following result:

\begin{thm} If the ambient space $M$ is K\"ahler, then all complex submanifolds U with (possibly empty) boundary B are solutions of the plateau problem with boundary B. \end{thm}

This is reason enough to be relatively excited about using K\"ahler geometry.  Evidently we would like to use this somehow in relation to ideas of generalised length.

However, note that this theorem is limited in that it does not deal with the plateau problem for submanifolds of a K\"ahler manifolds of odd real dimension.  This concern cannot actually be resolved, and is in fact where the hope that somehow K\"ahler manifolds are natural objects to look at hits a stumbling block.   The truth of the matter is that one is adding a certain amount of structure and hence getting strong results for one's efforts; but the assumptions on structure cannot be extended to completely general spaces.

I will now proceed to prove the second theorem.

\begin{lem} If $M$ is a complex manifold with a hermitian metric $g$, then $g$ is Kahlerian iff $d\omega = 0$.

\begin{proof} Fix a point $p \in M$.  Let $X_{1},X_{2},X_{3}$ be tangent vectors to $p$ and extend them to local fields in $M$ such that $(\nabla_{X_{i}}X_{j})_{p} = 0$ for $i,j = 1,2,3$ (this can be achieved by parallel transporting these vectors along the geodesics that they generate).  So then $[X_{i},X_{j}]_{p} = 0$ and hence $d\omega_{p}(X_{1},X_{2},X_{3}) = X_{1}\omega(X_{2},X_{3}) - X_{2}\omega(X_{1},X_{3}) + X_{3}\omega(X_{1},X_{2})$.  Then since $\omega(X,Y) = g(X,JY)$ we have that

\begin{center}
$d\omega_{p}(X_{1},X_{2},X_{3}) = g(X_{2},(\nabla_{X_{1}}J)X_{3}) - g(X_{1},(\nabla_{X_{2}}J)X_{3}) + g(X_{1},(\nabla_{X_{3}}J)X_{2})$
\end{center}

Then if $g$ is K\"ahlerian, $(\nabla_{X}J)Y = 0$ for all $X,Y \in T_{p}M$ so this expression is zero.

To prove the converse, observe first that $(\nabla_{X}J)J = -J(\nabla_{X}J)$ as $J^{2} = -1$, and $\nabla_{JX} = J\nabla_{X}$ since the integrability tensor is $0$.  Then

\begin{center}
$d\omega_{p}(X_{1},JX_{2},JX_{3}) = g(JX_{2},-J(\nabla_{X_{1}}J)X_{3}) - g(X_{1},-(\nabla_{X_{2}}J)X_{3}) - g(X_{1},-(\nabla_{X_{3}}J)X_{2})$
\end{center}

So $d\omega(X_{1},X_{2},X_{3}) - d\omega(X_{1},JX_{2},JX_{3}) = 2g(X_{2},(\nabla_{X_{1}}J)X_{3})$, from which it easily follows that the condition that $d\omega = 0$ implies the K\"ahlerian condition $(\nabla_{X}J)Y = 0$ for all $X,Y$.
\end{proof}
\end{lem}

We now proceed to define another natural notion.

\begin{dfn}
Suppose $M$ is a complex manifold.  Then by a complex submanifold we mean an immersion $\theta : K \rightarrow M$ of a complex manifold $K$ such that the representations of $\theta$ in complex coordinates are holomorphic.
\end{dfn}

\begin{lem} Every complex submanifold of a K\"ahler manifold is K\"ahlerian in the induced metric.

\begin{proof} We need only prove things locally, so consider a small embedded complex submanifold $K \subset M$.  Then if $X,Y \in TK$, we have

\begin{center}
$\nabla_{X}(JY) = (\overline{\nabla}_{X}JY)^{T} = (J\overline{\nabla}_{X}Y)^{T} = J(\overline{\nabla}_{X}Y)^{T} = J\nabla_{X}Y$
\end{center}

So $K$ is K\"ahlerian.
\end{proof}
\end{lem}

\begin{thm}  Every complex submanifold of a K\"ahler manifold is minimal. \end{thm}

\begin{lem} (Wirtinger's Inequality).  Suppose $M$ is a K\"ahler manifold and let $K \subset M$ be any 2m dimensional oriented real submanifold.  Suppose $dV_{p}$ is the volume form of the induced metric on $K$ at $p$.  Then the restriction of $\omega^{m} = \omega \wedge ... \wedge \omega$, the $m^{th}$ power of the K\"ahler form of $M$ to $T_{p}K$ satisfies $\omega^{m}/m! \leq dV_{p}$, with equality iff $T_{p}K$ is a complex subspace of $T_{p}M$. \end{lem}

\begin{proof} (of Theorem).  If $M$ is any K\"ahler manifold and $\theta : K \rightarrow M$ any complex submanifold with $K$ compact possibly with non-empty boundary with real dimension $2m$.  Then the volume of $K$ in the induced metric is less than or equal to the volume of any other $2m$-dimensional submanifold homologous to $K$ in $M$.  In other words, $K$ is the solution of a plateau problem in $M$. \end{proof}


\subsection{Further thoughts}

One natural question to ask is:

Can we go the other way, ie, in certain cases, can we find for any solution of a plateau problem a complex submanifold corresponding to it?  Also, can the solutions of plateau problems in odd real dimensions in K\"ahler manifolds always be lifted to corresponding solutions of higher dimensional plateau problems in complex coordinates? (Obviously lifts may not be unique.  The key question here is whether such a lift will always exist.)

One would hope that the answer to these question would be yes, as this would be tremendously compelling support for the use of K\"ahlerian geometry to describe physical situations.  But in general the answer is certainly no.  For instance, simply take the boundary of a Plateau problem in such a way that it cannot possibly be matched to a complex submanifold of the ambient K\"ahler space.

However, it may be possible for very special K\"ahler manifolds to find a correspondence.  For instance, considering K\"ahler manifolds of complex dimension two with index one.  Then the sort of result we would like to prove would be:

\begin{conj} Let M be a K\"ahler manifold of complex dimension two.  Then all submanifolds of index one (ie causally convex submanifolds) must necessarily be complex submanifolds.  Hence all causally convex solutions to the Plateau problem in real dimension 2, 3 or 4 will be complex submanifolds. \end{conj}

The reason we might expect this to be true is because there actually is a correspondence between minimal graphs and complex subdomains of the complex plane, so it seems reasonable that we might be able to get the same result with a locally $\mathcal{C}^{2}$ manifold if we impose some sort of causality requirement.

In particular, we might expect a subset of solutions to certain physical systems of even dimension to be K\"ahler, since they might be realisable as solutions to the square roots of the relevant differential operator.

\chapter{A Survey of Geometric Measure Theory}

Geometric Measure Theory is an extraordinarily rich subject, initiated by Federer (see his famous unreadable treatise, \cite{[F]}), and still very much being developed by such people as Leon Simon and Frank Morgan.  My main references here are Simon's book \cite{[Si]}, Morgan's introductory text \cite{[M]}, and notes that I took from a series of lectures given by Marty Ross in the first half of 2006.

\section{Introduction and Motivation}

Geometric measure theory is essentially an extension of differential geometry designed to deal with possible convergence problems for variational problems.  For instance, one might be interested in finding a family of surfaces converging to say a minimal surface defined by some boundary in a higher dimensional domain.  This is known as the \emph{Plateau Problem}.  The essential nature of the game here is:

(i) Take some sort of geometric variational problem (like the Plateau problem)

(ii) Convergence may not be well defined for the usual class of "nice" manifolds.  So in order to get nice convergence, extend solutions to the variational problem to a class containing the "nice" manifolds but also containing uglier beasts.  (To gain something you must first lose something).  The uglier beasts would be rectifiable sets, rather than differential manifolds, for instance (the nature and construction of these will be explained shortly).

(iii) Identify conditions on the problem (ambient space, boundary conditions) that are sufficient to imply regularity of the solution to the variational problem, i.e. prove that the solution is not an ugly beast, but is one of the nice guys.  Results that assist with this process are things like the Allard regularity theorem and its cousins.

\subsection{The Plateau Problem as the Prototypical Example}

The Plateau Problem is actually the problem that historically the technology of geometric measure theory was developed for.  The statement of the problem, in its simplest, ungeneralised form, is quite simple:

\emph{Given a closed boundary curve in euclidean three space, what is the surface of minimal area spanning this curve?}

The main advantage of considering this problem is that it can be visualised very easily- it is clear, given some particular loop of curve, what will be bad candidate surfaces, for instance, and what will be good ones.  It is also clear that we should be able to "improve" upon an initial approximation to the best surface by deforming to an ideal, "minimal" surface smoothly.  There are problems with this approach, however- which is the whole point of geometric measure theory- but let us first assume that we can do this.

A good first step towards determining whether a surface is minimal is to find the necessary criterion for it to be stationary in area- for it is obvious that if one were to perturb such a surface very slightly in any way, such that we get a family of surfaces, the first derivative of the area of such a family evaluated at the ideal surface should be zero.  This approach leads one to what is known as the \emph{minimal surface equation}.  I should now proceed to demonstrate how the minimal surface equation is derived according to standard variational techniques, for the case in which a surface can be realised as a graph.


If the surface can be realised as a graph, then there is some $f : D \subset R^{n} \rightarrow R$ such that the area $A(f)$ is given by the expression

\begin{center} $A(f) = \int_{D}\sqrt{1 + (\nabla f)^{2}}dA$ \end{center}

Let $g_{t} = f + t\eta$ be an arbitary variation of $f$, i.e. such that $\eta\vert_{\partial D} = 0$.  Then if $f$ is a minimal graph we must have that

\begin{center} $\frac{dA(g_{t})}{dt}\vert_{t = 0} = 0$ \end{center}

(This is a necessary condition for a minimiser).

So since $A(t) = \int_{D}\sqrt{1 + (\nabla f + t \nabla \eta)^{2}}dA$,

\begin{center} $A'(t) = \frac{d}{dt}\int_{D}\sqrt{1 + (\nabla f + t \nabla \eta)^{2}}dA$ \end{center}

By a result from measure theory, we may exchange the integral sign and the derivative and get

\begin{align} A'(t) &= \int_{D}\frac{\partial}{\partial t}\sqrt{1 + (\nabla f + t\nabla \eta)^{2}}dA \nonumber \\
&= \int_{D}\frac{2(\nabla f + t\nabla \eta) \cdot \nabla \eta \frac{1}{2}}{\sqrt{1 + (\nabla f + t \nabla \eta)^{2}}}dA \end{align}

Then for any $\eta$,

\begin{center} $0 = A'(0) = \int_{D}\frac{\nabla f \cdot \nabla \eta}{\sqrt{1 + (\nabla f)^{2}}}dA$ \end{center}

Integrating by parts we get that

\begin{center} $[\eta \frac{\nabla f \cdot n}{\sqrt{1 + (\nabla f)^{2}}}]\vert_{\partial D} - \int_{D}\nabla \cdot \frac{\nabla f}{\sqrt{1 + (\nabla f)^{2}}}\eta dA = 0$ \end{center}

The first term vanishes since $\eta \vert_{\partial D} = 0$.  Since $\eta$ is otherwise arbitrary, we then conclude by the fundamental lemma of the calculus of variations that

\begin{center} $\nabla \cdot \frac{\nabla f}{\sqrt{1 + (\nabla f)^{2}}} = 0$ \end{center}

For a proof of the fundamental lemma of the calculus of variations, suppose $h : [x_{0},x_{1}] \rightarrow R$ is continuous and suppose $\int_{x_{0}}^{x_{1}} h \eta = 0$ for every $C^{1}$ $\eta : [x_{0},x_{1}] \rightarrow R$ with $\eta(x_{0}) = 0 = \eta(x_{1})$.  Then I claim $h = 0$ (if we merely assume $\int h \eta = 0$ for every measurable $\eta$, we may only conclude $h = 0 a.e.$).

The proof is by contradiction.  For suppose $h(a) \neq 0$.  Without loss of generality, suppose in particular that $h(a) > 0$.  But then by continuity it is possible to find an $\eta$ such that $\int_{x_{0}}^{x_{1}} h \eta > 0$, which is impossible.

\subsection{Possible problems with the standard variational approach}

As mentioned before, there are a couple of technical issues that might compromise the above naive approach.  I will describe them here.

The main problem is the issue of convergence.  For it is possible to find a family of surfaces that gets arbitrarily close to the minimal area, but which becomes pathological in the limit- for instance, one may get space filling curves- very horrible indeed!


Another problem that may occur might be for instance the possibility that in perturbing an initial smooth surface to try to find an ideal one, one may produce self intersections in a new surface- and hence it will no longer be a manifold.


\subsection{Geometric Measure Theory comes to the rescue}

Geometric measure theory deals with these issues by slightly weakening the family of things that one is working with- to surfaces that are smooth almost everywhere, and that have integer multiplicity.  In particular, one introduces the concept of \emph{rectifiability}.  It turns out that the space filling component of the pathological family of surfaces previously described is \emph{unrectifiable}; whereas the "good bit" is rectifiable.  There is a structure theorem about the objects one works with that in fact shows that there is a unique decomposition into these two separate pieces.  Finally, if one has "nice" conditions, such as a smooth embedded boundary curve and a smooth ambient space, we may conclude, using a result known as the Allard Regularity theorem, that the rectifiable bit is actually the smooth surface we are looking for.  Moreover, we are able to conclude that it satisfies the minimal surface equation, and so our "naive" analysis is justified.

I should remark that the Plateau Problem is not the only problem that is amenable to these methods.  Any variational problem can be treated in an analogous manner, and benefits to an equal degree from the tools of this theory.  It is merely that the Plateau Problem is easy to understand in a visual manner, and so makes a very good example.  Since we may in general be considering much broader variational problems, it makes sense to extend from merely considering integer multiplicities and consider \emph{any} value for the multiplicity at any point on the surface; this is where the idea of \emph{density} comes into its own.

\section{Sobolev Spaces}

Before moving on to the particulars of geometric measure theory, it is necessary first to understand some concepts and definitions from analysis.  In particular it is necessary to understand the so called Sobolev spaces $W^{k,p}(\Omega)$ and $W^{k,p}_{loc}(\Omega)$.

I shall follow \cite{[GT]} in my treatment of the motivation for these objects.

\subsection{Preliminary Results}

Definition. A Hilbert space is a set $\mathcal{H}$ together with a bilinear operation (.,.): $\mathcal{H} \times \mathcal{H} \rightarrow \mathcal{C}$, satisfying the following relations:

(i) $(x,y) = (y,x)$ for all $x,y \in \mathcal{H}$,

(ii) $(ax + by,z) = a(x,z) + b(y,z)$ for all $a,b \in \mathcal{R}$,$x,y,z \in \mathcal{H}$, and

(iii) $(x,x) \geq 0$ for all $x \in \mathcal{H}$, with equality iff $x = 0$.

We call this operation an \emph{inner product} on $\mathcal{H}$.  For $\mathcal{H}$ to be a Hilbert space we require that it must be \emph{complete} with respect to this inner product.

Let $\mathcal{H}$ be a Hilbert space, and $x,y \in \mathcal{H}$. Then we have the following inequalities:

The Schwarz Inequality.

\begin{equation}
|(x,y)| \leq \norm{x}\norm{y}
\label{Schwarz}
\end{equation}

The Triangle inequality.

\begin{equation}
\norm{x + y} \leq \norm{x} + \norm{y}
\end{equation}

Parallelogram law.

\begin{equation}
\norm{x + y}^{2} + \norm{x - y}^{2} = 2(\norm{x}^{2} + \norm{y}^{2})
\end{equation}

We also have:

\begin{thm} (Pythagoras' Theorem). If $(x,y) = 0$ (i.e. $x \perp y$), then $\norm{x + y}^{2} = \norm{x} + \norm{y}$.

\begin{proof} Very easy. \end{proof} \end{thm}

Now I shall prove the Schwarz inequality.

Let $y \neq 0$ first.  Then, choose $\alpha \in \mathcal{C}$ such that $x - \alpha y \perp y$.  So $(x - \alpha y, y) = 0$.  Hence $(x,y) - \alpha \norm{y}^{2} = 0$, or $\alpha = \frac{(x,y)}{\norm{y}^{2}}$.  By Pythagoras' theorem for $x - \alpha y$ and $y$,

\begin{align}
\norm{x}^{2} &= \norm{(x - \alpha y) + \alpha y}^{2} \nonumber \\
&= \norm{x - \alpha y}^{2} + \norm{\alpha y}^{2} \nonumber \\
&\geq |\alpha|^{2}\norm{y}^{2} = \frac{|(x,y)|^{2}}{\norm{y}^{4}}\norm{y}^{2} \nonumber \\
&= \frac{|(x,y)|^{2}}{\norm{y}^{2}}
\end{align}

From which the Schwarz inequality can easily be extracted.  The triangle inequality follows easily from the Schwarz inequality.

The following theorem is a very important one, and gets used throughout analysis.  We shall need it in what is to follow.

\begin{thm} (The Riesz Representation Theorem). For every bounded linear functional $F$ on a Hilbert space $\mathcal{H}$, there is a uniquely determined element $f \in \mathcal{H}$ such that $F(x) = (x,f)$ for all $x \in \mathcal{H}$ and $\norm{F} = \norm{f}$.

\begin{proof} Let $\mathcal{N} = {x | F(x) = 0}$ be the null space of $F$.  If $\mathcal{N} = \mathcal{H}$, the result follows trivially by taking $f = 0$.  Otherwise, since $\mathcal{N}$ is a closed subspace of $\mathcal{H}$, there exists an element $z \neq 0$ in $\mathcal{H}$ such that $(x,z) = 0$ for all $x \in \mathcal{N}$.  Then $F(z) \neq 0$ and moreover for any $x \in \mathcal{H}$,

\begin{center}
$F(x - \frac{F(x)}{F(z)}z) = F(x) - \frac{F(x)}{F(z)}F(z) = 0$
\end{center}

so $x - \frac{F(x)}{F(z)}z \in \mathcal{N}$.  Hence by our choice of $z$

\begin{center}
$(x - \frac{F(x)}{F(z)}z,z) = 0$
\end{center}

or in other words

\begin{center}
$(x,z) = \frac{F(x)}{F(z)}\norm{z}^{2}$
\end{center}

Hence $F(x) = (f,x)$ where $f = zF(z)/\norm{z}^{2}$, showing existence.  To prove uniqueness suppose $F(x) = (g,x)$.  Then in particular $0 = F(x) - F(x) = (f - g,x)$ for every $x \in \mathcal{H}$.  In particular $(f - g,f - g) = \norm{f - g} = 0$.  But this is true iff $f - g = 0$ by inner product axiom (iii). So $f = g$.

To see that $\norm{F} = \norm{f}$, we have, by the Schwarz inequality,

\begin{center}
$\norm{F} = sup_{x \neq 0}\frac{|(x,f)|}{\norm{x}} \leq sup_{x \neq 0}\frac{\norm{x}\norm{f}}{\norm{x}} = \norm{f}$,
\end{center}

and also that

\begin{center}
$\norm{f}^{2} = (f,f) = F(f) \leq \norm{F}\norm{f}$
\end{center}

proving equality of the norms. \end{proof} \end{thm}

The $L^{p}(\Omega)$ and $L^{p}_{loc}(\Omega)$ spaces consist of $p$-integrable functions over $\Omega$ and locally $p$-integrable functions over $\Omega$ respectively.

For $\psi$ to be $p$-integrable means that $\int_{\Omega}\psi^{p}$ is well defined and is finite.

For $\psi$ to be locally $p$-integrate means that for any $K \subset \subset \Omega$ then $\int_{K}\psi^{p}$ is well defined and is finite.

The $L^{p}(\Omega)$ norm is defined on $p$-integrable functions on $\Omega$ by

\begin{equation}
\norm{u}_{p;\Omega} = (\int_{\Omega}|u|^{p})^{1/p}
\end{equation}

A multi index $\alpha$ is a string $(\alpha_{1},...,\alpha_{n})$, and is used as shorthand for multi-variable differentiation, e.g. $D^{\alpha}f = (\frac{\partial^{\alpha_{1}} f}{\partial^{\alpha_{1}}x_{1}})...(\frac{\partial^{\alpha_{n}}f}{\partial^{\alpha_{n}}x_{n}})$.

\subsection{Motivation}

Let $\Omega$ be a closed domain, and $\phi \in C^{1}_{0}(\Omega)$ (singly differentiable functions on $\Omega$ with compact support).  By the divergence theorem a $C^{2}(\Omega)$ solution of $\Delta u = f$ satisfies

\begin{equation}
\int_{\Omega}Du.D\phi dx = -\int_{\Omega}f\phi dx
\end{equation}

since the above is equivalent to $\int_{\Omega}D_{i}(\phi D_{i}u)dx = 0$, which, by the divergence theorem, is equivalent to the statement $\int_{\partial \Omega}\phi Du.v ds = 0$ where $v$ is a unit normal to $\partial \Omega$.  Since $\Omega$ is closed, this statement makes sense.

Now the bilinear form

\begin{equation}
(u,\phi) = \int_{\Omega}Du.D\phi dx
\label{SoboIP}
\end{equation}

is an inner product on $C^{1}_{0}(\Omega)$.  Hence we can complete $C^{1}_{0}(\Omega)$ with respect to the metric induced by this equation to get a Hilbert space, which we notate as $W^{1,2}_{0}(\Omega)$.  This is an example of a so called Sobolev space.

The linear functional $F$ defined by $F(\phi) = -\int_{\Omega}f\phi dx$ may be extended, through appropriate choice of $f$, to a bounded linear functional on $W^{1,2}_{0}(\Omega)$.  Then by the Riesz representation theorem there exists an element $u \in W^{1,2}_{0}(\Omega)$ such that $(u,\phi) = F(\phi)$ for all $\phi \in C^{1}_{0}(\Omega)$.  Hence a \emph{generalised solution} to the Dirichlet problem, $\Delta u = f$, $u = 0$ on $\partial \Omega$, has been found.  So we have reduced the classical problem to a question of whether this solution we have found in $W^{1,2}_{0}(\Omega)$ is in fact in $C^{1}_{0}(\Omega)$, in other words, a regularity problem.

\subsection{The Sobolev Spaces \texorpdfstring{$W^{k,p}(\Omega)$, $W^{k,p}_{0}(\Omega)$ and $W^{k,p}_{loc}(\Omega)$}{}}

Let $u \in L^{1}_{loc}(\Omega)$.  Given a multi-index $\alpha$, $v \in L^{1}_{loc}(\Omega)$ is called the $\alpha^{th}$ \emph{weak derivative} of $u$ if

\begin{equation}
\int_{\Omega}\phi u dx = (-1)^{|\alpha|}\int_{\Omega}uD^{\alpha}\phi dx
\end{equation}

for all $\phi \in C_{0}^{\infty}(\Omega)$.  We write $v = D^{\alpha}u$, and observe that this object is defined up to sets of Lebesgue measure zero.

We then define the Sobolev space

\begin{equation}
W^{k,p}(\Omega) = L^{p}(\Omega) \cap \{u : D^{\alpha}u \in L^{p}(\Omega), |\alpha| \leq k \}
\end{equation}

We equip $W^{k,p}(\Omega)$ with the norm

\begin{equation}
\norm{u}_{k,p;\Omega} = (\int_{\Omega}\sum_{|\alpha| \leq k}|D^{\alpha}u|^{p}dx)^{1/p}
\end{equation}

which is easily seen to be equivalent to

\begin{equation}
\sum_{|\alpha| \leq k}\norm{D^{\alpha}u}_{p;\Omega}
\end{equation}

We define the local Sobolev spaces in the obvious manner:

\begin{equation}
W^{k,p}_{loc}(\Omega) = L^{p}_{loc}(\Omega) \cap \{u : D^{\alpha}u \in L^{p}(\Omega), |\alpha| \leq k\}
\end{equation}

Finally, we define the spaces $W^{k,p}_{0}(\Omega) = W^{k,p}(\Omega) \cap L^{p}_{0}(\Omega)$ where $L^{p}_{0}$ is the space of $p$-integrable functions on $\Omega$ with compact support.  In other words, $W^{k,p}_{0}$ is the set of $(k,p)$-Sobolev functions with compact support.

\subsection{Approximation of Sobolev functions by smooth functions}

\subsection{The Sobolev Inequality}

I follow \cite{[GT]} in this section.

We have the following results regarding comparability of different norms.  The first result relates the $L^{p}$ norms of a Sobolev function and its gradient.  The second relates the $L^{p^{*}}$ norm of a Sobolev function with its $(k,p)$-Sobolev norm.  Here $p^{*} = np/(n-kp)$.

\begin{thm} Let $\Omega \subset \mathcal{R}^{n}$, $n > 1$ be an open domain.  There is a constant $C = C(n,p)$ such that if $n > p$, $p \geq 1$, and $u \in W^{1,p}_{0}(\Omega)$ then

\begin{center}
$\norm{u}_{np/(n-p);\Omega} \leq C\norm{Du}_{p;\Omega}$
\end{center}

Furthermore, if $p > n$ and $\Omega$ is bounded, then $u \in C(\overline{\Omega})$ and

\begin{center}
$sup_{\Omega}|u| \leq C|\Omega|^{1/n - 1/p}\norm{Du}_{p;\Omega}$
\end{center}
\end{thm}

\begin{thm} Let $\Omega \subset \mathcal{R}^{n}$ be an open set.  There is a constant $C = C(n,k,p)$ such that if $kp < n, p \geq 1$, and $u \in W_{0}^{k,p}(\Omega)$, then

\begin{center}
$\norm{u}_{p^{*};\Omega} \leq C\norm{u}_{k,p;\Omega}$.
\end{center}

If $kp > n$, then $u \in C(\overline{\Omega}$ and

\begin{center}
$sup_{\Omega}|u| \leq C|K|^{1/p'}\{ \sum_{|\alpha| = 0}^{k-1}(diam K)^{|\alpha|}\frac{1}{\alpha!}\norm{D^{\alpha}u}_{p;K}$\\
 $+ (diam(K))^{k}\frac{(k - 1)!}(k - \frac{n}{p})^{-1}\norm{D^{k}u}_{p;K} \}$
\end{center}

where $K =$ spt $u$ and $C = C(k,p,n)$.
\end{thm}

\section{Further preliminaries from analysis}

\subsection{Lebesgue Measure}

$m$-dimensional Lebesgue Measure $L^{m}$ is defined on $R^{m}$ for sets $A$ as

\begin{center} $L^{m}(A) = inf \{ \Sigma_{j = 1}^{\infty}v(P_{j}) : A \subset \cup_{j = 1}^{\infty} \}$ \end{center}

It is what is known as an "outer measure", that is,

$L^{m} : P(R^{m}) \rightarrow [0,\infty]$ where $P(R^{m})$ is the power set of $R^{m}$, and satisfies the following axioms:

\begin{itemize}
\item[(i)] $L^{m}(\phi) = 0$
\item[(ii)] $A \subset B$ implies $L^{m}(A) \leq L^{m}(B)$ 
\item[(iii)] $L^{m}(\cup_{j = 1}^{\infty}A_{j}) \leq \sum_{j = 1}^{\infty}L^{m}(A_{j})$
\end{itemize}

and it also satisfies the property that $L^{m}($box$) = $Vol$($box$)$.

Lebesgue measure is a very important measure, because we have that all Borel sets of $R^{m}$ are $L^{m}$ measurable.  This follows from Caratheodory's Criterion, which states that if one has a measure $\mu$ on a metric space $X$, then if $\mu(A \cup B) = \mu(A) + \mu(B)$ whenever $d(A,B) > 0$ then $\mu$ is a Borel measure, ie all Borel sets in $X$ are $\mu$-measurable.

I remind the reader that a Borel set is an element of the Borel-algebra corresponding to a particular topology, which is the intersection of all $\sigma$-algebras containing the closed sets in that topology.  A $\sigma$-algebra is a set of sets, which is closed under complements, intersections, unions, and which contains the empty set and the whole space.  This is related to the concept of measurability, and hence to integration, since the set of all $\mu$-measurable sets in a metric space $X$ is a $\sigma$-algebra (a $\mu$-measurable set $A \subset X$ satisfies $\mu(B) = \mu(B \cap A) + \mu(B - A)$ for any $B$).

\begin{proof} (of the Caratheodory Criterion). See J. Koliha's notes \cite{[Ko]}. \end{proof}

\begin{rmk}  Lebesgue measure is Borel regular, that is, for any $A$, $L^{m}(A) = L^{m}(B)$ some some $B \supset A$ Borel.  To see this, let $B_{j} = \cup_{k}P_{jk}$ and $L(B_{j}) \leq L(A) + \frac{1}{j}$.  Set $B = \cap_{j}B_{j}$. \end{rmk}

\begin{rmk} $L^{m}$ is Radon, that is, any $A$ can be approximated by closed sets on the inside, and open sets on the outside.  Equivalently, this means that $L^{m}$ is Borel regular \emph{and} $L^{m}(A) < \infty$ for $A$ bounded. \end{rmk}

\subsection{Hausdorff Measure}

$n$-dimensional Hausdorff measure, $H^{n}$, on $R^{m}$, $n \geq 0$ is also a map from the power set of $R^{m}$ to $[0,\infty]$, and also satisfies the axioms (i), (ii) and (iii) of an outer measure.  Furthermore, $H^{n}(M) = n$-volume of $M$ by any classical method.

We define

\begin{center} $H_{\delta}^{n}(A) = inf \{ \sum_{j = 1}^{\infty}\omega_{n}(\frac{ diam \bar{B}_{j}}{2})^{n} : diam B_{j} \leq \delta, A \subset \cup_{j = 1}^{\infty}B_{j} \}$, \end{center}

where $\omega_{n}$ is the $n$-volume of the $n$-ball.  $n$-dimensional Hausdorff measure is then defined as

\begin{center} $H^{n}(A) = lim_{\delta \rightarrow 0^{+}}H^{n}_{\delta}(A)$ \end{center}

Note that $H^{n}_{\delta}$ is an (outer) measure, and, furthermore, that $H^{n}$ is a Borel regular measure.  To see the latter, let $A$, $B$ be two sets with $d = d(A,B) > 0$.  Take $\delta < d$.  Then $H^{n}_{\delta}(A \cup B) = H_{\delta}^{n}(A) + H_{\delta}^{n}(B)$.  So $H^{n}$ is Borel.  To see that it is in fact Borel regular, take $\delta_{j} = \frac{1}{j}$ and a family of covers $\{ T_{jk}\}$ such that $H^{n}_{\frac{1}{j}}(\cup_{k} \bar{T}_{jk}) = H^{n}(A) + \frac{1}{j}$.

However, $H^{n}$ is \emph{not} Radon if $n < m$.  For example, $H^{1}($unit $2$-disc$) = \infty$.

Other interesting facts include that

\begin{itemize}
\item[(i)] $H^{n}$ is an isometry, that is, $H^{n}(\rho(A)) = H^{n}(A)$ if $\rho : R^{m} \rightarrow R^{m}$ is an isometry.
\item[(ii)] $H^{0}$ is the counting measure ($\omega_{0} = 1$).
\end{itemize}

\emph{Note}.  $H^{n}$ makes sense in any metric space $X$, and for any $n \geq 0$, not necessarily a natural number.

\begin{rmk}  Hausdorff measure may otherwise seem quite arbitrary, but the fact that it is invariant under rigid motions of $R^{m}$ means that oddly rotated sets can be measured quite easily using it, whereas they would be frightful to compute using Lebesgue measure.  In fact, it turns out that we may compute the Lebesgue measure of horridly rotated sets using Hausdorff measure, due to the following extreme useful and deep theorem:
\end{rmk}

\begin{thm} $H^{n} = L^{n}$ on $R^{n}$.

\begin{proof} See Morgan. \end{proof} \end{thm}

Finally, it is sometimes useful to have some idea of the dimension of a set, particularly if it behaves in a slightly pathological way.  Hausdorff measure proves to be quite useful in understanding at least this level of behaviour of quite unruly beasts indeed.

\begin{dfn} (Hausdorff Dimension).  The Hausdorff dimension of a set $A$ is defined to be the supremum of numbers $s$ such that $H^{s}(A) = \infty$, or equivalently, the infimum of numbers $s$ such that $H^{s}(A) = 0$.
\end{dfn}

\subsection{Densities}

Let $A \subset R^{n}$, $a \in R^{n}$, $m \leq n$. We define the $m$-dimensional density $\Theta^{m}(A, a)$ of $A$ at $a$ by

\begin{equation}
\label{density}
\Theta^{m}(A, a) = lim_{r \rightarrow 0}\frac{\mathcal{H}^{m}(A \cap B^{n}(a, r))}{\alpha_{m}r^{m}}
\end{equation}

where $\alpha_{m}$ is the size of the ball of radius one in $R^{m}$.

Define $\Theta^{m}(\mu, a)$ where $\mu$ is a measure of $R^{n}$ by

\begin{equation}
\label{gendensity}
\Theta^{m}(\mu, a) = lim_{r \rightarrow 0}\frac{\mu(B^{n}(a, r))}{\alpha_{m}r^{m}}
\end{equation}

\subsection{Lipschitz and BV Functions}

\begin{dfn} A Lipschitz function, $f : R^{m} \rightarrow R^{n}$ is a function satifying the criterion:

\begin{center}
$|f(x) - f(y)| \leq C|x - y|$
\end{center}

for some constant $C$.  The minimal such $C$ is referred to as Lip $f$.
\end{dfn}

We have the following very important result about Lipschitz functions:

\begin{thm} (Rademacher's Theorem). A Lipschitz function $f : R^{m} \rightarrow R^{n}$ is differentiable almost everywhere. \end{thm}

Later we will be interested in looking at sets that can be realised locally as graphs of Lipschitz functions.  So from this theorem it follows that such objects are smooth except on a set of measure zero, ie they are smooth manifolds except on a set of measure zero.

\begin{proof} A proof of Rademacher's theorem can be found in Morgan's book \cite{[M]}. \end{proof}


There are other results:

\begin{thm} (Whitney's Theorem). Let $f : R^{m} \rightarrow R^{n}$ be Lipschitz.  Then for any $\epsilon > 0$  there is a $C^{1}$ $g : R^{m} \rightarrow R^{n}$ such that $L^{m}(\{ x : f \neq g \}) < \epsilon $, where $L^{m}$ is the $m$ dimensional Lebesgue measure. \end{thm}

\begin{thm} Lipschitz functions are weakly differentiable, that is, for any Lipschitz $f$ there exists a $g$ which is the weak derivative of $f$, ie., such that

\begin{center} $\int \eta D_{i}g = - \int f D_{i}\eta$ \end{center}

for any arbitrary $\eta$. \end{thm}





\begin{dfn} BV functions, otherwise known as functions of bounded variation, are Lipschitz functions whose weak derivative is a finite Radon measure.  That is, for any open set $\Omega \subset R^{n}$, we say $u \in L^{1}(\Omega)$ has bounded variation if $\exists$ a finite Radon measure $Du$ such that

\begin{center} $\int_{\Omega}u$ $div \phi = - \int_{\Omega}\phi \cdot Du$, \end{center}

for all $\phi \in C^{1}_{c}(\Omega,R^{n})$. We then write $u \in BV(\Omega)$.
\end{dfn}

Alternatively, we have the following equivalent statement:

\begin{dfn} A Lipschitz $u$ is said to be in $BV_{loc}(\Omega)$ if for each $W \subset \subset \Omega$ there is a constant $c(W) < \infty$ such that

\begin{center} $\int_{W}u$ $div \phi \leq c(W)sup\norm{\phi}$ \end{center}

for all vector functions $\phi = (\phi_{1},...,\phi_{n})$, $\phi_{i} \in C^{\infty}_{c}(W)$.
\end{dfn}

\begin{prop}  If $u \in BV_{loc}(\Omega)$, then for any $U \subset \subset \Omega$, there is a measure $\mu$ such that

\begin{center} $\int_{U}u$ $div \phi = \int_{U}\phi \cdot \mu$ \end{center}

for any $\phi \in C^{\infty}_{c}(U,R^{n})$.

\begin{proof} (See Simon \cite{[Si]} p.35) \end{proof} \end{prop}

Yet another way of expressing the same thing:

Define the variation of $u$ in $\Omega$ as $V(u,\Omega) = sup \{ \int_{\Omega}u$ $div \phi : \phi \in C^{1}_{c}(\Omega,R^{n}), \norm{\phi}_{L^{\infty}(\Omega)} \leq 1 \}$.  Then $BV(\Omega) = \{ u \in L^{1}(\Omega) : V(u,\Omega)$ is finite$\}$.

These things are important mainly because of the following compactness theorem, which is, in some sense, the fundamental theorem of geometric measure theory:

\begin{thm} (Compactness Theorem for BV functions).  If $\{u_{k}\}$ is a sequence of $BV_{loc}(U)$ functions satisfying

\begin{equation}
\label{BV compactness}
sup_{k \geq 1}(\norm{u_{k}}_{L^{1}(W)} + \int_{W}\norm{Du_{k}}) < \infty \end{equation}

for each $W \subset \subset U$, then there is a subsequence $\{u_{k'}\} \subset \{u_{k}\}$ and a $BV_{loc}$ function $u$ such that $u_{k'} \rightarrow u$ in $L^{1}_{loc}(U)$ and

\begin{equation}
\label{BV convergence}
\int_{W}\norm{Du} \leq lim inf \int_{W}\norm{Du_{k'}}
\end{equation}

for all $W \subset \subset U$.
\end{thm}

\begin{proof} I shall prove this result later. \end{proof}

\emph{Remarks}.
\begin{itemize}
\item[(i)] The finiteness condition above is necessary in the sense that in order to get a compactness result of this nature, we need the sequence of functions and their derivatives to be bounded in every bounded domain.  If we wanted to establish $L^{2}_{loc}(U)$ convergence, we would need all second derivatives to be bounded too.
\item[(ii)] The inequality following the finiteness condition is essentially a kind of Fatou's lemma for first order derivatives.
\end{itemize}

Having got this result, the game we now play is the following- we want to construct a family of $BV_{loc}$ functions to locally represent our surface, for example (in the case we are considering the Plateau problem).  It then follows from this theorem that this family will have a convergent subsequence.  So we want to do calculus of variations with $BV_{loc}$ functions.  The difficulties with this will lie of course in the fact that we may have corners and other pathological behaviours on sets of measure zero.

\subsection{Jacobians and the Area Formula}

In dealing with maps from one rectifiable set to another, it becomes a matter of interest as to how to compute the Jacobian of such a map, as this knowledge proves integral to establishing the area and coarea formulae (to follow).  If $f : R^{m} \rightarrow R^{n}$ is differentiable at a point $p$, the $k$ dimensional Jacobian of $f$ at $a$, $J_{k}f(a)$, is defined to be the maximal $k$-dimensional volume of the image under $Df(a)$ of a unit $k$-cube. 

If rank$(Df(a)) < k$ then $J_{k}f(a) = 0$.  Otherwise, if rank $(Df(a)) \geq k$, then $J_{k}f(a)^{2}$ is merely the sum of the squares of the determinants of the $k x k$ submatrices of $Df(a)$.  If $k = m$ or $n$, $J_{k}f(a)^{2} = det(Df^{T}(a)Df(a))$.

The definition of $J_{k}f$ has been arranged so as to make the following theorem easy to state.

\begin{thm} (The Area Formula).  Suppose $f : R^{m} \rightarrow R^{n}$ is Lipschitz for $m \leq n$.  Then

\begin{center}
$\int_{A}J_{m}f(x)d\mathcal{L}^{m}x = \int_{R^{n}}N(f\vert A,y)d\mathcal{H}^{m}y$
\end{center}

and

\begin{center}
$\int_{R^{m}}u(x)J_{m}f(x)d\mathcal{L}^{m}x = \int_{R^{n}}\sum_{x \in f^{-1}(y)}u(x)d\mathcal{H}^{m}y$
\end{center}

where $N(f\vert A, y)$ is the cardinality of the set $\{x \in A : f(x) = y \}$ and $u$ is any $\mathcal{H}^{m}$ integrable function.

\begin{proof} I will follow \cite{[M]}.  Suppose rank $Df = m$.  Let $\{s_{i}\}$ be a countable dense set of affine maps of $R^{m}$ onto $m$-dimensional planes in $R^{n}$ (recall an affine map is a linear transformation followed by a translation).  Suppose $E$ is a piece of $A$ such that for each $a \in E$ the affine functions $f(a) + Df(a)(x-a)$ and $s_{i}(x)$ are approximately equal.

Then det$s_{i}$ is approximately $J_{m}f$ on $E$, and $f$ is injective on $E$ since all affine maps are injective.

Since $f$ is differentiable, we can locally approximate $f$ by $f(a_{j}) + Df(a_{j})(x-a_{j})$ in some small neighbourhood of $a_{j}$ which we can call $E_{j}$ for a dense subset $\{a_{j}\}$ of $A$.  Since the set $\{s_{i}\}$ is dense there is some element of this set, call it $s_{j}$ after possible relabelling, such that $s_{j}$ is arbitrarily close to the local approximation to $f$ in $E_{j}$.  Then evidently these neighbourhoods cover $A$.

So for each piece $E_{i}$,

\begin{align}
\mathcal{H}^{m}(f(E_{i})) &\approx \mathcal{H}^{m}(s_{i}(E_{i}) \nonumber \\
&= \mathcal{L}^{m}(s_{i}(E_{i})) \nonumber \\
&= \int_{E_{i}}det(s_{i})d\mathcal{L}^{m} \nonumber \\
&\approx \int_{E_{i}}J_{m}f d\mathcal{L}^{m}
\end{align}

Summing over the sets $E_{i}$ gives

$\int_{R^{n}}(\text{number of sets }  E_{i}  \text{ intersecting } f^{-1}\{y\})\mathcal{H}^{m}y \approx \int_{A}J_{m}fd\mathcal{L}^{m}$.

which in the limit for a sequence of covers $E_{i,j}$ such that as $j \rightarrow \infty$, $diam(E_{i,j}) \rightarrow 0 \forall i$ uniformly, we get

$\int_{R^{n}}N(f|A,y)d\mathcal{H}^{m}y = \int_{A}J_{m}fd\mathcal{L}^{m}$, as required.

What if rank$(Df) < m$?  Then $\int_{A}J_{m}f$ is zero.  Evidently we would like to show that $\int_{R^{n}}N(f|A,y)d\mathcal{H}^{m}y$ is also zero.

Consider the function $g : R^{m} \rightarrow R^{n+m}, x \mapsto (f(x), \epsilon x)$.

Then $J_{m}g \leq \epsilon(Lip f + \epsilon)^{m-1}$.  This follows from the fact that in diagonal form $Df$ has maximal entries $Lip f,...,Lip f$ along the diagonal repeated rank$(Df) = p < m$ times followed by a string of zeroes.  So $Dg$ will have maximal entries $Lip f + \epsilon,...,Lip f + \epsilon,\epsilon$ along the diagonal since rank$(Dg) = m$.  Then $J_{m}g$ will be, maximally, the product of these entries, that is, $\epsilon(Lip f + \epsilon)^{m-1}$, as required.

Since $Dg$ has rank $m$,

$\mathcal{H}^{m}(f(A)) \leq \mathcal{H}^{m}(g(A)) = \int_{A}J_{m}g \leq \epsilon(Lip f + \epsilon)^{m-1}\mathcal{L}^{m}(A)$

Hence we have that $\mathcal{H}^{m}(f(A)) = 0$, since $\epsilon = 0$ corresponds to the map $f$.  From this it quickly follows that $\int_{R^{n}}N(f|A,y)d\mathcal{H}^{m}y$ is zero, and we are done.

The second part of the theorem follows by approximating $u$ by simple functions.
\end{proof} \end{thm}

\subsection{The Co-Area Formula}

The following result is very closely related to that of the previous section.  It essentially covers the other case, where we are dealing with maps from higher dimensional spaces into lower dimensional ones, rather than vice versa.

First I remind readers of a well known result from analysis, Fubini's Theorem.

\begin{thm} (Fubini's Theorem, Special Case).  Suppose $\mu$ is Hausdorff measure on $R^{n}$ and $\nu$ is Hausdorff measure on $R^{p}$.  Then if $f \in L^{1}(\mu \otimes \nu)$,

\begin{center}
$\int_{R^{n} \times R^{p}}f(s,t)d(\mu \otimes \nu) = \int_{R^{n}}(\int_{R^{p}}f(s,t)d\nu(t))d\mu(s) = \int_{R^{p}}(\int_{R^{n}}f(s,t)d\mu(s))d\nu(t)$.
\end{center}

\begin{proof}  Refer to [Ko], section $16.2$, for a proof of this result. \end{proof} \end{thm}

\begin{thm} (The Coarea Formula).  Suppose $f : R^{m} \rightarrow R^{n}$ is Lipschitz, with $m > n$.  Then if $A$ is a $\mathcal{L}^{m}$ measurable set, then

$\int_{A}J_{n}f(x)d\mathcal{L}^{m}x = \int_{R^{n}}\mathcal{H}^{m-n}(A \cap f^{-1}(y))d\mathcal{L}^{n}y$.

\begin{proof} Again following \cite{[M]}. If $f$ is orthogonal projection, then $J_{n}f = 1$ and the coarea formula reduces to the special case of Fubini's theorem above.

More generally, suppose $J_{n}f \neq 0$.  We may subdivide $A$ as in the proof of the area formula, and hence may assume $f$ is linear.  Then $f = L \circ P$, where $P$ is projection onto the orthogonal complement $V$ of $ker f$ and $L$ is some nonsingular linear map from $V$ to $R^{n}$.  Then

\begin{align}
\int_{A}J_{n}fd\mathcal{L}^{m} &= |det(L)|\mathcal{H}^{m}(A) \nonumber \\
 \intertext{since the map $f$ is linear and so the stretch factor from $A$ to $f(A)$ will be given by $|det(L)|$.}
&= |det(L)|\int_{P(A)}\mathcal{H}^{m-n}(P^{-1}\{y\})d\mathcal{L}^{n}y \nonumber \\
\intertext{(by definition of what it means to be a projection),}
&= \int_{L \circ P(A)}\mathcal{H}^{m-n}((L\circ P)^{-1}\{y\})d\mathcal{L}^{n}y
\intertext{since $V$ is $n$-dimensional and $L$ is a nonsingular map.} \nonumber
\end{align}

But this is precisely the result we wanted.  The general case for $f$ nonlinear comes through easily as a consequence. \end{proof} \end{thm}

\section{Key Constructions}

\subsection{Tangent Cones}

Define the $tangent$ $cone$ of $E$ at $a$ consisting of the $tangent$ $vectors$ of $E$ at $a$:

\begin{equation}
\label{TangentCone}
Tan(E, a) = \{r \in R: r \geq 0\}[\bigcap_{\epsilon > 0}Closure\{ \frac{x - a}{|x - a|} : x \in E, 0 < |x - a| < \epsilon \}]
\end{equation}

Intuitively speaking, the tangent cone to $a \in E$ is essentially what the rectifiable set $E$ looks like near $a$ to a first order (linear) approximation.  If $E$ is an $n$-dimensional manifold, the first order approximation will always be an $n$-dimensional plane.  But in rectifiable sets the tangent cone may very well look like a cone, or any number of more exotic structures.  Draw an $(n+1)$-dimensional disc surrounding $a$ by viewing $E$ near $a$ as embedded in $R^{n+1}$ for some sensible coordinate system.  Then the tangent cone at $a$ can be thought of as a subset of this $(n+1)$-disc, and in fact is uniquely determined by its intersection with the boundary of the disc, or the $n$-sphere.

\subsection{Rectifiable Sets and a Structure Theorem}

Recall that the natural measure for $m$-dimensional sets on $R^{n}$, is the Hausdorff measure, $\mathcal{H}^{m}$.

Borel subsets $B$ of $R^{n}$ are called $(\mathcal{H}^{m}, m)$ rectifiable if $B$ is a countable union of Lipschitz images of bounded subsets of $R^{m}$, with $\mathcal{H}(B) < \infty$.  I shall make this more precise later.  It is possible to associate to any rectifiable set $B$ a canonical tangent plane $T_{b}B$ for each $b \in B$ (also to be described in more detail later).  (For our treatment, sets will always be taken to have integer dimension, even though geometric measure theory allows the study of more pathological sets.) 

\begin{thm} (Structure Theorem). Let $E$ be an arbitrary subset of $R^{n}$ with $\mathcal{H}^{m}(E) < \infty$. Then $E$ can be decomposed as the union of two disjoint sets $E = A \cup B$ with $A$ $(\mathcal{H}^{m}, m)$ rectifiable and $B$ purely unrectifiable. \end{thm}

We can interpret this theorem as stating that any set $E$ in $R^{n}$ can be split uniquely into a curvelike bit (a bit where calculus makes sense) and an uncurvelike bit (where things get pathological).  This follows from the closely related theorem:

\begin{thm} (Structure Theorem (2)).  $n$-dimensional rectifiable sets admit a tangent cone $\mathcal{H}^{n}$-almost everywhere, and $n$-dimensional unrectifiable sets admit a tangent cone $\mathcal{H}^{n}$-almost nowhere.  Hence any arbitrary subset $E$ of $R^{n}$ splits uniquely into two disjoint sets $\hat{A} \cup \hat{B}$ with all points in $A$ admitting a tangent cone and all points in $B$ not admitting a tangent cone.  \end{thm}

\begin{rmk}  Note that $\hat{A}$ and $\hat{B}$ are more or less the same as $A$ and $B$, respectively, up to sets of $\mathcal{H}^{n}$ measure zero. \end{rmk}

\subsection{Currents}

The main building blocks of GMT are the rectifiable currents, $m$-dimensional oriented surfaces.  They are oriented rectifiable sets with integer multiplicities, finite area, and compact support.  One presumes an orientation can be defined by associating a basis of $m$ formlike objects to each tangent space in $B$.  Then, if it is possible to choose a nowhere vanishing $m$-formlike object across all of $B$ we could say that it is orientable.

Multiplicity is related to how one interprets integration over such a set $B$.  The way integration works is to use the standard measure $\mu$ for each point, then modify the measure pointwise by the map $\mu(p) \mapsto multiplicity(p)\mu(p)$.  To have integer multiplicity means that pointwise $multiplicity(p) \in \mathcal{Z}$.  An example of this is illustrated by taking the real plane $\mathcal{R}^{2}$ and folding it in half.  Then the fold line has multiplicity one and all other points in the resulting half plane have multiplicity two.

Anyway, the utility of rectifiable currents lies in a result from general measure theory that states that one can integrate a smooth differential form $\phi$ over such a set $S$, and hence view $S$ as a current, or a linear functional on differential forms,

\begin{center}
$\phi \mapsto \int_{S}\phi$
\end{center}

One also has a boundary operator from $m$ to $m-1$-dimensional currents

\begin{center}
$(\partial S)(\phi) = S(d\phi)$
\end{center}

though $\partial S$ will not necessarily be rectifiable even if $S$ is.

One might wonder if rectifiable currents which satisfy the problem of least area possess any geometric significance.  The Allard Regularity theorem tells us that for $m \leq 6$, an $m$-dimensional area-minimising rectifiable current in $R^{m+1}$ is a smooth embedded manifold.  In fact this is a best possible estimate, since there are examples in the literature of area minimising rectifiable currents which admit codimension seven singularities.  This is related to the theorem of Frankel and Lawson, and my interest in extending it.

\begin{rmk} For a more careful definition of these objects, consider any  oriented $m$-dimensional rectifiable set $S$.  Let $\bar{S}(x)$ denote the element of the tangent cone at $x$.  Then for any differential $m$-form $\phi$, define

\begin{center} $S(\phi) = \int_{S} \mu(x) \bar{S}(x) \cdot \phi d\mathcal{H}^{m}$ \end{center}

where $\mu(x)$ is an associated multiplicity to the point $x$.  $S$ can then be naturally thought of as a rectifiable current (provided $S$ also has compact support). \end{rmk}

\subsection{The Rectifiability Theorem}


\begin{dfn} A current $T$ is rectifiable if $T = \tau(M,\theta,\chi)$ where $M$ is countably $n$-rectifiable, $H^{n}$-measurable, $\theta$ is a positive locally $H^{n}$-integrable function on $M$, and $\chi(x)$ orients the tangent cone $T_{x}M$ of $M$ for $H^{n}$-a.e. $x \in M$. \end{dfn}

This theorem essentially gives a condition for a current to be a rectifiable current.

\begin{thm} (32.1 \cite{[Si]}). Suppose $T \in D_{n}(U)$ is such that $M_{W}(T) < \infty$ and $M_{W}(\partial T) < \infty$ for all $W \subset \subset U$, and $\theta^{*n}(\mu_{T},x) > 0$ for $\mu_{T}$-a.e. $x \in U$. Then $T$ is rectifiable. \end{thm}

\section{The Compactness Theorem for Integral Currents}


This theorem is a core result of geometric measure theory; it allows us to sensibly define convergence in certain special classes of geometric measure theoretic objects.

\begin{thm} (27.3 \cite{[Si]}).  If $\{T_{j}\} \subset D_{n}(U)$ is a sequence of integer multiplicity currents with

\begin{center} $sup_{j \geq 1}(M_{W}(T_{j}) + M_{W}(\partial T_{j})) < \infty$ $\forall W \subset \subset U$, \end{center}

then there is an integer multiplicity $T \in D_{n}(U)$ and a subsequence $\{T_{j'}\}$ such that $T_{j'} \rightarrow T$ in $U$. \end{thm}

The aim of this section will be to prove a limited version of this theorem.  The main reason I will not state a proof of the theorem in its full generality is essentially because the technical nature of it would disguise the main ideas underlying the establishment of this result.  Instead, I will in fact focus on providing a proof of the compactness theorem for BV functions stated earlier.  This has several advantages:

\begin{itemize}
\item[(i)] The compactness theorem for BV functions is in fact a special case (the codimension one case) of the above theorem, since a codimension one current can be viewed locally as a graph of a $BV_{loc}$ function,

\item[(ii)] Although this theorem is less general, it is considerably easier to prove and is a good model of the general case,

and

\item[(iii)] The theorem is readily extendible to the codimension $k$ case.
\end{itemize}

I will follow this proof with a focus on proving its main implication for the Plateau problem (and more general variational problems), namely that weak minimising solutions will exist to it.  

\begin{rmk} The above compactness theorem can be extended to a sequence of geometric measure theoretic objects, like currents (but not) of varying (not necessarily integer), but bounded multiplicities.  These beasts are known as \emph{varifolds}.  This allows one to apply this technology to such problems as variation of the scalar curvature on a surface, etc.  The proof of this result will be given in a later section. \end{rmk}

\subsection{Two important theorems from analysis}

\emph{Definition}.  A set of functions $F$ mapping from a space $X$ to a space $Y$ is said to be equicontinuous if for every $f \in F$, every $\epsilon > 0$, and every $x_{0} \in X$ there exists a $\delta$ such that if $\norm{x - x_{0}} < \delta$ then $\norm{f(x) - f(x_{0})} < \epsilon$.

Equicontinuity is stronger than uniform continuity, which in turn is stronger than continuity.

\emph{Definition}. A set of functions $F \subset C(X,Y)$ is called pointwise relatively compact if for all $x \in X$, the set $\{f(x) : f \in F \}$ is compact in $Y$.

I may now state the result:

\begin{thm} (Arzela-Ascoli). Suppose $S$ is a compact metric space.  Then a non-empty set $A \subset C(S)$ is totally bounded iff $A$ is bounded relative to the norm in $C(S)$ and equicontinuous.

\begin{proof} (lifted from J. Koliha's notes on analysis \cite{[Ko]}, Theorem 5.5.2).  Suppose $A$ is bounded and equicontinuous.  Then, for any fixed $\epsilon > 0$ there is a ball $B(a,\delta(a))$ for each $a \in S$ such that

\begin{center} $f \in A$ and $x \in B(a,\delta(a))$ implies $\Norm{f(x) - f(a)} < \frac{1}{3}\epsilon$ \end{center}

Since $S$ is compact, by the Heine Borel lemma there exist balls $B(x_{k},\delta(x_{k}))$ for $k = 1,..., N$ some finite $N$ that cover $S$.

Now define a mapping $P : A \rightarrow R^{N}$ setting $P(f) = (f(x_{1}),...,f(x_{N}))$.  I claim that $P(A)$ is bounded in $R^{N}$.  To see this, observe that since $A$ is bounded by hypothesis that there is a $c > 0$ such that

\begin{center} $\norm{f} \leq c$ for all $f \in A$ \end{center}

Hence, for each $f \in A$,

\begin{center} $\norm{P(f)}_{\infty} = max_{k}\Norm{f(x_{k})} \leq \norm{f} \leq c$ \end{center}

But bounded sets in $R^{N}$ are totally bounded, since $R^{N}$ is a complete metric space (once again we are using the Heine-Borel lemma).  So let $P(f_{1}),...,P(f_{p})$ be a $\frac{1}{3}\epsilon$-net for $P(A)$.  Then I claim that $f_{1},...,f_{p}$ is an $\epsilon$-net for $A$.

Suppose $f \in A$.  Then certainly $P(f) \in P(A)$, and there is a $f_{j}$ such that $\norm{P(f) - P(f_{j})}_{\infty} < \frac{1}{3}\epsilon$.  Recall also that for any $x$ there is $x_{k}$ with  $x \in B(x_{k},\delta_{k})$.  So

\begin{align}
\Norm{f(x) - f_{j}(x)} &\leq \Norm{f(x) - f(x_{k})} + \Norm{f(x_{k}) - f_{j}(x_{k})} + \Norm{f_{j}(x_{k}) - f_{j}(x)} \nonumber \\
&\leq \frac{1}{3}\epsilon + \norm{P(f) - P(f_{j})}_{\infty} + \frac{1}{3}\epsilon \nonumber \\
&< \frac{1}{3}\epsilon + \frac{1}{3}\epsilon + \frac{1}{3}\epsilon = \epsilon.
\end{align}

This is true for any $x \in S$, so $\norm{f - f_{j}} < \epsilon$ i.e. $A$ is totally bounded.

For the converse, suppose $A$ is totally bounded.  Fix $\epsilon > 0$.  Then there is an $\epsilon$-net $f_{1},...,f_{N}$ such that for any $f \in A$ there is an $f_{i}$ such that $\norm{f - f_{i}} < \epsilon$.  Since $S$ is compact, the functions $\Norm{f_{i}}$ are bounded from above by constants $C_{i}$.  Let $C = max_{i}C_{i}$.  Then

\begin{center} $\norm{f}_{\infty} \leq \norm{f - f_{i}}_{\infty} + \norm{f_{i}}_{\infty} < \epsilon + C < 2C$ \end{center}

for small enough $\epsilon$.  So we conclude that $A$ is bounded.

It remains to show that $A$ is equicontinous, that is, that for all $f \in A$ and all $x,y \in S$ there is a $\delta$ such that if $\Norm{x - y} < \delta$ then $\Norm{f(x) - f(y)} < \epsilon$.

Observe now that since $S$ is compact we may choose a $\delta$ small enough that this condition does hold for a finite number of functions, eg our $\epsilon$-net $f_{1},...,f_{N}$.  So then

\begin{center} $\Norm{f_{k}(x) - f_{k}(y)} < \epsilon$ for all $x,y \in S$ such that $\Norm{x - y} < \delta$ \end{center}

Finally

\begin{align}
\Norm{f(x) - f(y)} &\leq \Norm{f(x) - f_{k}(x)} + \Norm{f_{k}(x) - f_{k}(y)} + \Norm{f_{k}(y) - f(y)} \nonumber \\
&\leq \epsilon + \epsilon + \epsilon = 3\epsilon
\end{align}

so $A$ is equicontinuous as required.  This completes the proof.
\end{proof} \end{thm}

The following corollary is of crucial importance in our development, and is often also called the Arzela-Ascoli theorem:

\begin{cor} Suppose $S$ is a compact metric space.  Then a nonempty set $A \subset C(S)$ is compact iff $A$ is bounded, closed and equicontinuous. \end{cor}

This corollary is important in any application where one wants to show existence of a convergent subsequence of a set of functions.  Hence it has obvious applications in geometric measure theory, and, more generally, in existence theory for PDE.

A more general version of the theorem also holds:

\begin{thm} (Ascoli-Arzela). Let $X$ and $Y$ be metric spaces, $X$ compact.  Then a subset $F$ of $C(X,Y)$ is compact iff it is equicontinuous, pointwise relatively compact and closed. \end{thm}

\begin{rmk}  Note that the theorem I proved earlier was essentially for the case $Y = R^{n}$, but extending the proof to the general case is not too hard. \end{rmk}

I now remind the reader of a theorem I mentioned earlier when I was talking about Sobolev spaces:

\begin{thm} (The Riesz Representation Theorem). For every bounded linear functional $F$ on a Hilbert space $\mathcal{H}$, there is a uniquely determined element $f \in \mathcal{H}$ such that $F(x) = (x,f)$ for all $x \in \mathcal{H}$ and $\norm{F} = \norm{f}$. \end{thm}

This result allows one, in certain circumstances, to show certain notions are well defined.  It will have particular power in the next section, in which I shall make certain comments on Radon measures.

\subsection{Remarks on Radon Measures}

Recall that a measure $\mu$ on a space $X$ is Radon if it is Borel regular and if it is finite on compact subsets.  The importance of Radon measures comes mainly from the following fact:

\begin{thm}  There is a one to one correspondence between $\mu$-measurable functions $\nu : X \rightarrow H$ satisfying $\norm{\nu} = 1$ $\mu$ a.e. and linear functionals $L : K(X,H)$ where $H$ is any Hilbert space and $K(X,H)$ is the space of continuous functions from $X$ to $H$, provided that the functionals $L$ satisfy the finiteness condition

\begin{center} $sup \{ L(f) : f \in K(X,H), \norm{f} \leq 1, spt (f) \subset P\} < \infty$, for each compact $P \subset X$. \end{center} \end{thm}

\begin{proof} The correspondence is essentially achieved via use of the Riesz representation theorem.  Let $<.,.>$ be the inner product on $H$.  Then, for any $L$, there is a $\nu$, and vice versa, such that

\begin{center} $L(f) = \int_{X}<f,\nu>d\mu$. \end{center}

\end{proof}

\begin{cor} As a consequence of the above, we can identify the Radon measures on $X$ with the non-negative linear functionals on $K(X,R)$.  By abuse of notation, define from now on $\mu : K(X,R) \rightarrow R$ as $\mu(f) = \int_{X}fd\mu$. \end{cor}

The following theorem is important for the development to follow.

\begin{thm} Let $\{ \mu_{k} \}$ be a sequence of Radon measures on a space $X$.  Suppose $sup_{k}\mu_{k}(U) < \infty$ for every set $U \subset X$ with compact closure.  Then we may conclude that there is a subsequence $\{ \mu_{k'} \}$ which converges to a Radon measure $\mu$ on $X$ in the sense that

\begin{center} $lim \mu_{k'}(f) = \mu(f)$ for each $f \in K(X,R)$; \end{center}

in other words, the associated functionals converge.
\end{thm}

\subsection{Mollification}

\begin{dfn}  A symmetric mollifier is a function $\phi \in C^{\infty}_{c}(R^{n})$, $\phi \geq 0$, with support$(\phi) \subset B_{1}(0)$, such that $\int_{R^{n}}\phi = 1$ and $\phi(x) = \phi(-x)$. \end{dfn}

Now let $u$ be a $L^{1}_{loc}$ function.  We define $\phi_{\sigma}(x) = \sigma^{-n}\phi(\frac{x}{\sigma})$, where $\phi$ is a symmetric mollifier.  Then $u^{\sigma} = \phi_{\sigma} \circ u$ are the corresponding mollified functions.

The key idea of mollifying a function is to construct a sequence of nicer functions converging to it (in particular, smooth functions converging to something which is not guaranteed to be smooth), and conclude properties of the limit from results for the mollified entities.  This has powerful applications, because it is possible to do analysis on mollified functions (differentiation, integration, etc) but not necessarily on the more pathological entities we are interested in.

Observe that we may equivalently define the mollification of $u$ as

\begin{center} $u_{h}(x) = h^{-n}\int_{U}\phi(\frac{x - y}{h})u(y)dy$ \end{center}

for an appropriate symmetric mollifier $\phi$.  Recall also H\"older's inequality:

\begin{center} $\int_{U}uvdx \leq \norm{u}_{p}\norm{v}_{q}$ \end{center}

where $\frac{1}{p} + \frac{1}{q} = 1$.

An appropriate symmetric mollifier $\phi$ might well be $\phi(x) = cexp(\frac{1}{\Norm{x}^{2} - 1})$ for $\Norm{x} \leq 1$ and zero otherwise, for a choice of a constant $c$ such that $\int \phi dx = 1$.

The following result is mentioned in Leon Simon's book, and also in Gilbarg and Trudinger's text.  It is a standard result about mollification.  First, however, we will need an auxiliary lemma, which is essentially lemma 7.1 in \cite{[GT]}:

\begin{lem} (Auxiliary Lemma). Suppose $u \in C^{0}(U)$.  Then $u_{h}$ converges to $u$ uniformly on any domain $\Omega \subset \subset U$. \end{lem}

\begin{proof} Since we can define a $\phi$ such that

\begin{align} u_{h}(x) &= h^{-n}\int_{\Norm{x - y} \leq h}\phi(\frac{x - y}{h})u(y)dy \nonumber \\
&= \int_{\Norm{z} \leq 1}\phi(z)u(x - hz)dz
\end{align}

it follows that if $\Omega \subset \subset U$ and $2h < dist(\Omega, \partial U)$ then

\begin{align} sup_{\Omega}\Norm{u - u_{h}} &\leq sup_{x \in \Omega}\int_{\Norm{z} \leq 1}\phi(z)\Norm{u(x) - u(x - hz)}dz \nonumber \\
&\leq sup_{x \in \Omega}sup_{\Norm{z} \leq 1}\Norm{u(x) - u(x - hz)} \end{align}

But since $u$ is uniformly continuous over the set

\begin{center} $B_{h}(\Omega) = \{ x : dist(x, \Omega) < h \}$ \end{center}

as $B_{h}(\Omega)$ is compact and $u$ is continuous, it follows from the above inequality that $u_{h}$ tends to $u$ uniformly as $h$ tends to zero.
\end{proof}

\begin{lem}  Suppose  $u \in BV_{loc}(U)$.  Then the mollification of $u$, $u^{\sigma}$, converges in the limit as $\sigma \rightarrow 0$ to $u$ in $L^{1}_{loc}(U)$.  Also $\norm{Du^{\sigma}} \rightarrow \norm{Du}$ in the sense of Radon measures in $U$, as mentioned in the previous theorem. \end{lem}

\begin{proof} The following proof of the first part is lifted from Gilbarg and Tr\"udinger \cite{[GT]}, lemma 7.2.

Now, by H\"older's inequality, we have that

\begin{align} \Norm{u_{h}(x)}^{p} &= \Norm{\int_{\Norm{z} \leq 1}\phi(z)u(x - hz)dz}^{p} \nonumber \\
&\leq \int_{\Norm{z} \leq 1}\Norm{\phi(z)^{p-1}}\phi(z)\Norm{u(x - hz)}^{p}dz \nonumber \\
&\leq \int_{\Norm{z} \leq 1}\phi(z)^{p-1}dz \int_{\Norm{z} \leq 1}\phi(z)\Norm{u(x - hz)}^{p}dz \nonumber \\
&\leq \int_{\Norm{z} \leq 1}\phi(z)\Norm{u(x - hz)}^{p}dz \end{align}

In particular, if $\Omega \subset \subset U$ and $2h < dist(\Omega, \partial U)$ then

\begin{align} \int_{\Omega}\Norm{u_{h}}^{p}dx &\leq \int_{\Omega}\int_{\Norm{z} \leq 1}\phi(z)\Norm{u(x - hz)}^{p}dzdx \nonumber \\
&= \int_{\Norm{z} \leq 1}\phi(z)dz \int_{\Omega}\Norm{u(x - hz)}^{p}dx \nonumber \\
&\leq \int_{B_{h}(\Omega)}\Norm{u(x)}^{p}dx \end{align}

where $B_{h}(\Omega) = \{ x \in U : dist(x,\Omega) < h \}$.  The last inequality follows from H\"older's inequality and since $x - hz$ is in $B_{h}(\Omega)$ for all $h < \frac{1}{2}dist(\Omega, \partial U)$ and $\Norm{z} \leq 1$, provided $x \in \Omega$.  Observe certainly $\Omega \subset B_{h}(\Omega)$.

Then it now follows that

\begin{center} $\norm{u_{h}}_{L^{p}(\Omega)} \leq \norm{u}_{L^{p}(B_{h}(\Omega))}$ \end{center}

Now, given $\epsilon > 0$, choose a $C^{0}(U)$ function $w$ satisfying

\begin{center} $\norm{u - w}_{L^{p}(B_{h}(\Omega))} \leq \epsilon$ \end{center}

(we may do this because $C^{0}(U)$ is dense in $C^{p}(U)$ for all $p > 0$). I now invoke our auxiliary lemma to find a $h$ such that

\begin{center} $\norm{w - w_{h}}_{L^{p}(\Omega)} \leq \epsilon$ \end{center}

Then applying our previous estimate to the difference $u - w$, we finally obtain that

\begin{align} \norm{u - u_{h}}_{L^{p}(\Omega)} &\leq \norm{u - w}_{L^{p}(\Omega)} + \norm{w - w_{h}}_{L^{p}(\Omega)} + \norm{w_{h} - u_{h}}_{L^{p}(\Omega)} \nonumber \\
&\leq 2\epsilon + \norm{u - w}_{L^{p}(B_{h}(\Omega))} \leq 3\epsilon \end{align}

for a small enough $h$.  It follows that $u_{h}$ converges to $u$ in $L^{p}_{loc}(U)$, and, in particular, that $u_{h}$ converges to $u$ in $L^{1}_{loc}(U)$.

I follow Simon's book \cite{[Si]}, Lemma 6.2, for the proof of the second part. We wish to show

\begin{center} $lim_{\sigma \rightarrow 0}\int f\Norm{Du^{\sigma}} = \int f\Norm{Du}$ \end{center}

I claim it is easy to demonstrate that

\begin{center} $\int f\norm{Du} \leq lim inf_{\sigma \rightarrow 0} \int f\norm{Du^{\sigma}}$ \end{center}

since by definition

\begin{center} $\int_{W}\norm{Du} = - sup_{g}\int_{W} u div(g)$ \end{center}

where $g$ is a bump function (ie no larger in modulus than one, smooth, and contained in the support of $W$).

But then $\int_{W}f\norm{Du} = - sup_{g}\int_{W} u div(fg) \leq - lim inf _{\sigma \rightarrow 0}\int_{W} u div(fg_{\sigma}) =: lim inf_{\sigma \rightarrow 0}\int_{W}f\norm{Du^{\sigma}}$.

So it remains to show that

\begin{center} $limsup_{\sigma \rightarrow 0} \int f \norm{Du^{\sigma}} \leq \int f \norm{Du}$ \end{center}

Observe first of all that

\begin{center} $\int f \norm{Du^{\sigma}} = sup_{g}\int fg \cdot grad (u^{\sigma})$ \end{center}

with $g$ a bump function.  Also note that if $g$ is fixed, and for $\sigma < dist(spt(f), \partial U)$, we have

\begin{align}
\int_{U}fg \cdot grad(u^{\sigma}) &= - \int u^{\sigma}div(fg) \nonumber \\
&= - \int \phi_{\sigma} \star (u)div(fg) \nonumber \\
&= -\int u(\phi_{\sigma} \star div(fg)) \nonumber \\
&= - \int u div(\phi_{\sigma} \star fg)
\end{align}

But by definition of $\norm{Du}$, we have that the above is nothing other than $\leq \int_{W_{\sigma}}(f + \epsilon(\sigma))\norm{Du}$.  Here $\epsilon(\sigma) \rightarrow 0$ for $\sigma \rightarrow 0$, and $W = spt(f)$, $W_{\sigma} = \{ x \in U : dist(x,W) < \sigma \}$.  This is because

\begin{center} $\norm{\phi_{\sigma} \star fg} = f\norm{(\phi_{\sigma} \star g_{1},...,\phi_{\sigma} \star g_{n})} \leq \phi_{\sigma} \star f$ \end{center}

But it is clear that $\phi_{\sigma} \star f \rightarrow f$ uniformly in $W_{\hat{\sigma}}$ as $\sigma \rightarrow 0$, where $\hat{\sigma} < dist(W,\partial U)$.  So this proves the second part.
\end{proof}

\subsection{Proof of The Theorem, and Existence of (Weak) Solutions to the Plateau Problem}

\begin{proof} (of the Compactness Theorem for BV functions).  By the aforementioned lemma, to show $u_{k'} \rightarrow u$ in $L^{1}_{loc}(U)$ for some subsequence $\{u_{k'}\}$, it is sufficient to show that the sets

\begin{center} $\{u \in C^{\infty}(U) : \int_{W}(\norm{u} + \norm{Du})dL^{n} \leq c(W) \}$ \end{center}

where $W \subset \subset U$ are precompact in $L^{1}_{loc}(U)$ (remember, a set is precompact, or totally bounded, if for any number $E > 0$ there is a covering of that set by finitely many sets of maximal diameter $E$).  This is because certainly if $u$ belongs to the above class, its associated Radon measure $Du$ is bounded locally in $L^{\infty}_{loc}(U)$.  But if these sets are totally bounded then it follows that $Du$ is bounded also in $L^{1}_{loc}(U)$ by finiteness and compactness.

Then the result follows from Gilbarg and Tr\"udinger's book \cite{[GT]}, Theorem 7.22, since we can then conclude these sets are totally bounded as required. \end{proof}

\emph{Existence of Weak Solutions to the Plateau Problem}. The above theorem guarantees that given a sequence of surfaces realised as graphs of BV functions there will be a convergent subsequence.  But certainly if we perform calculus of variations then we will naturally get a sequence of smooth surfaces, with the limit satisfying the minimal surface equation.  But a sequence can only have one limit; hence this limit must be the solution given to us via the BV function compactness theorem.  However we are not guaranteed that it will be smooth; we only know in general that it will be the graph of some BV function.  Showing that what we have is in some sense smooth will be the focus of the next section.

There is in fact a generalisation of all this to varifolds, which are essentially currents with a measure $\theta$, sort of like metric-measure spaces in the sense of Gromov, but more pathological.  This generalisation, the compactness theorem for such objects, is known as the Allard regularity theorem.  However I will not mention this here, since the statement and proof of it are very technical and little instructive value would be gained from a study.  More on varifolds in the next section.

\section{Varifolds}

\subsection{Introduction}

Note that I have so far only treated the compactness and regularity theorems for the case of objects with \emph{integral} density.  In more general variational problems, where we might have some sort of natural Lagrangian that varies from point to point in our manifold, for instance, it is necessary to consider different, non-constant densities or equivalently different measures.  In this section I mention a few results that equip us with the means to deal with these additional complexities.  The most important theorems are the approximation and deformation theorems, which lead to the more general compactness theorem for rectifiable varifolds.  The constancy theorem is of considerable theoretical utility, and was used by Jon Pitts extensively in his thesis.  The boundary regularity theorem as stated here really only applies to integral varifolds, but it is also true more generally.

Varifolds are essentially a generalisation of the concept of currents.  The primary reason we are interested in these things is in order to extend our existence and regularity results to general variational problems; what we have developed so far is limited in scope to geometric variational problems on sets of integer multiplicity, which limits us to examining things like the Plateau Problem.  In fact, in later parts of this thesis, I will examine the problem of varying the \emph{Fisher Information} of a space, which is definitely not integer valued!

For a more precise definition, let $\Omega$ be a subset of $R^{n}$.  Then a general $m$-dimensional varifold on $\Omega$ is essentially a Radon measure on $\Omega \times G(n,m)$, where $G(n,m)$ is the set of all $m$-dimensional subspaces of $n$-dimensional Euclidean space.  However, for our purposes, we will mainly be interested in rectifiable varifolds.  Such objects may essentially be represented as a graph with multiplicity, i.e., as a pair $(V,f)$ where the mass of a set $U \subset V$ is defined as

\begin{center} $m(U) = \int_{U}fdL^{m}$ \end{center}

Another way of writing this is to define a new measure (which is not necessarily positive definite) such that $d\mu = fdL^{m}$ and then

\begin{center} $m(U) = \int_{U}d\mu$ \end{center}

In other words, varifolds lend themselves naturally to the investigation of problems other than that of minimising or extremising area, or are essentially a generalisation from considering the mass of a set induced via the Lebesgue measure to that induced by more general measures.

\subsection{The Constancy Theorem}

This result provides us with some control on the density of a varifold (or current) through a certain region.  It is hence quite useful to prove regularity results (remember, to show regularity usually amounts to proving $\norm{V}B(m,a) \leq \omega_{n}a^{n}(1 + \epsilon)$ for small $\epsilon$).  In practice this theorem is usually used to show that the hypotheses of the Allard Regularity Theorem are satisfied.

\begin{thm} (26.27, \cite{[Si]}). If $U$ is open in $R^{n}$, if $U$ is connected, if $T \in D_{n}(U)$ and if $\partial T = 0$, then $\exists$ $c$ a constant such that $T = c[[U]]$. \end{thm}

What this is essentially saying is that the density of a rectifiable current $T$ is constant provided its boundary $\partial T$ is zero.  Recall that $\partial T(\phi)$ is defined to be $T(d\phi)$.

\subsection{The Approximation and Deformation Theorems}

(The relevant references here are 30.2 of Simon's Book \cite{[Si]}, 7.1 of Morgan's Book \cite{[M]}, and 4.2.20 of Federer's treatise, \cite{[F]}.)

Following from the Deformation Theorem, the Approximation Theorem is a key result in the theory of currents.  It essentially states that any integral current $T$ can be approximated by a slight deformation of a polyhedral chain $P$.  Recall that a polyhedral chain is essentially a collection of simplices.

\begin{thm} Given an integral current $T \in D_{m}(R^{n})$ and $\epsilon > 0$, there exists a $m$-dimensional polyhedral chain $P$ in $R^{n}$, with support within a distance $\epsilon$ of the support of $T$, and a $C^{1}$ diffeomorphism $f$ of $R^{n}$ such that

\begin{center} $f_{\#}T = P + E$ \end{center}

with $M(E) \leq \epsilon$, $M(\partial E) \leq \epsilon$, $Lip(f) \leq 1 + \epsilon$, $Lip(f^{-1}) \leq 1 + \epsilon$, $\norm{f(x) - x} \leq \epsilon$, and $f =$ identity whenever $dist(x, spt(T)) \geq \epsilon$. \end{thm}

\begin{rmk} Note we have an error term, $E$, which this theorem states we have precise control over (the theorem says we can make the mass of the error as small as we like by choosing an appropriate $f$ and $P$).  This error plays a similar role in related regularity results that the tilt excess plays in the proof of the Allard Regularity Theorem. \end{rmk}

\subsection{The Boundary Rectifiability Theorem}

This result essentially speaks for itself.

\begin{thm} Suppose $T$ is an integer multiplicity current in $D_{n}(U)$ with $M_{W}(\partial T) < \infty$ for all $W \subset \subset U$. Then $\partial T$ is an integer multiplicity current. \end{thm}

\section{A couple of applications}

I have already mentioned the more obvious application of the theory to the Plateau problem.  However I should remark that it is possible to show that certain closed curves in three space do not generate nice smooth surfaces, but rather smooth structures with a stratification.  One can actually see this with soap films, for instance.

But for the purposes of this section I will avoid further discussion of this and focus instead on a result due to Frankel and Lawson, and an existence result due to Jon Pitts.  There are of course the other obvious applications to establishing regularity of solutions to other natural variational problems, and determining when singularities are formed similar to soap films given particular data.  From the example above, such singularities may well exist, even if the integrand and geometry are smooth.

\subsection{The Frankel-Lawson Result}

My original motivation for exploring the jungle of GMT was to understand the first and second order behaviour of variations about singularities in varifolds- this was towards the aim of extending a well known result of Lawson and Frankel, which I shall now proceed to reveal and prove, following \cite{[F1]} and \cite{[F2]} most closely.

\begin{thm} (Main Theorem).  Let $M_{n + 1}$ be a complete, connected manifold with positive Ricci curvature.  Let $V_{n}$ and $W_{n}$ be immersed minimal hypersurfaces of $M_{n + 1}$, each immersed as a closed subset, and let $V_{n}$ be compact.  Then $V_{n}$ and $W_{n}$ must intersect. \end{thm}

\begin{proof}  Suppose $V_{n}$ and $W_{n}$ are embedded submanifolds, and that they do not intersect.  Suppose $\gamma$ is a geodesic from $p_{0} \in V$ to $q_{0} \in W$ that realises the minimum distance $l$ from $V$ to $W$.  Clearly it must strike $V$ and $W$ orthogonally, otherwise it can easily be shortened.  Any unit vector $X'$ tangent to $V$ at $p_{0}$ parallel translates along $\gamma$ to a vector field $X$ along $\gamma$ tangent to $W$ at $q_{0}$.  This vector field gives rise to a variation of $\gamma$ allowing end points to vary.

\begin{lem} (Key Lemma). The first variation of arc length is 0, and the second variation of arc length is

\begin{center}
$L''_{X}(0) = B_{W}(X) - B_{V}(X) - \int_{0}^{l} K( X, T )ds$
\end{center}

where $T = \frac{\partial}{\partial t}$ and $X = \frac{\partial}{\partial \alpha}$ are coordinate vector fields defined on a ribbon $0 \leq t \leq l$, $-\epsilon \leq \alpha \leq \epsilon$.  At $\alpha = 0$, $T$ is the unit tangent field to $\gamma$, and the curve $\alpha = 0$ corresponds to the geodesic $\gamma$ running from $p_{0}$ to $q_{0}$.  The curve $t = 0$ is a curve in $V$; the curve $t = l$ is a curve in $W$.  $B_{V}(X)$ (similarly $B_{W}(X)$) represents the second fundamental form of $V$ (respectively $W$) at $p_{0}$ (respectively $q_{0}$) evaluated on $X(p_{0})$ (resp. $X(q_{0})$).
\end{lem}

\begin{proof} (of Key Lemma).

\begin{center}
$L'(\alpha) = \int_{0}^{l}\frac{\partial}{\partial\alpha}g(T,T)^{1/2}dt = \int_{0}^{l}\nabla_{X}g(T,T)^{1/2}dt$
\end{center}

hence

\begin{equation}
\label{firstvar}
L'(\alpha) = \int_{0}^{l}\frac{g(\nabla_{X}T,T)}{g(T,T)^{1/2}}dt
\end{equation}

Since $g(T,T) = 1$ along $\alpha = 0$ we get

\begin{center}
$L'(0) = \int_{0}^{l}g(\nabla_{X}T,T)dt = \int_{0}^{l}g(\nabla_{T}X,T)dt = 0$
\end{center}

since $\nabla_{T}X = 0$ as $X$ is parallel displaced along $\gamma$.  I.e. the first variation of arc length is 0.

Continuing from our first variational equation we get

\begin{center}
$L''(\alpha) = \int_{0}^{l}\nabla_{X}(\frac{g(\nabla_{T}X,T)}{g(T,T)^{1/2}})dt$
\end{center}

which becomes

\begin{center}
$L''(0) = \int_{0}^{l}\nabla_{X}g(\nabla_{T}X,T)dt - \int_{0}^{l}g(\nabla_{T}X,T)^{2}dt$
\end{center}

The second term vanishes since $\nabla_{T}X = 0$ ($X$ is parallel along $\gamma$).  So we are left with

\begin{center}
$L''(0) = \int_{0}^{l}g(\nabla_{T}\nabla_{X}X,T)dt + \int_{0}^{l}g(R(X,T)X,T)dt$
\end{center}

using the relation $\nabla_{X}\nabla_{T} - \nabla_{T}\nabla_{X} = R(X,T)$ for orthogonal vector fields $X$, $T$.

The first term can be rewritten:

\begin{center}
$g(\nabla_{T}\nabla_{X}X,T) = \nabla_{T}g(\nabla_{X}X,T) - g(\nabla_{X}X,\nabla_{T}T) = \frac{\partial}{\partial t}g(\nabla_{X}X,T)$
\end{center}

Putting this all together, together with the result that

\begin{equation}
K(T,X) = -g(R(X,T)X,T)
\end{equation}

we have that

\begin{center}
$L''_{X}(0) = \frac{d^{2}L}{d\alpha^{2}}(0) = g(\nabla_{X}X,T)_{q_{0}} - g(\nabla_{X}X,T)_{p_{0}} - \int_{0}^{l}K(T,X)dt$
\end{center}

It is clear that the first two terms are the second fundamental forms with respect to W and V respectively evaluated at $q_{0}$ and $p_{0}$.  (Recall that the second fundamental form of an embedding say $f : M \rightarrow \overline{M}$ is $B(X,Y) = \overline{\nabla}_{\overline{X}}\overline{Y} - \nabla_{X}Y$ where $\overline{X}, \overline{Y}$ are extensions of vector fields $X$ and $Y$ locally defined on $M$ to $\overline{M}$ and $\overline{\nabla}$ is the connection on $\overline{M}$.  $\nabla$ is defined by

\begin{center}
$\nabla_{X}Y = (\overline{\nabla}_{\overline{X}}\overline{Y})^{T}$
\end{center}

Then we have that for a codimension one submanifold, say $V$,  that $g(\overline{\nabla}_{\overline{X}}\overline{X},T)|_{V} = g(\overline{\nabla}_{\overline{X}}\overline{X} - \nabla_{X}X,T)|_{V} = g(B(X,X),T)|_{V}$ which is $B_{V}(X)$ by definition.
\end{proof}

Doing this for $n$ orthonormal vectors $X'_{1}, ... , X'_{n}$ spanning $T_{p_{0}}V$ and summing we get

\begin{center}
$\sum_{\alpha = 1}^{n}L''_{X_{\alpha}}(0) = \sum_{\alpha = 1}^{n}B_{W}(X_{\alpha}) - \sum_{\alpha = 1}^{n}B_{V}(X_{\alpha}) - \int_{0}^{l} Ric(T)ds$
\end{center}

where $Ric(T) = \sum_{\alpha = 1}^{n}K(X_{\alpha} , T)$.  Now it is well known that minimal hypersurfaces are characterised by having mean curvature zero, hence $\sum_{\alpha = 1}^{n}B_{W}(X_{\alpha}) = 0 = \sum_{\alpha = 1}^{n}B_{V}(X_{\alpha})$.  Since $Ric(T) > 0$ by our hypothesis, we conclude that for some $\alpha$, $L''_{X_{\alpha}} < 0$, contrary to our assumption that $\gamma$ was of minimal length.  \end{proof}

It is one of my current projects to try to extend this result to the case of minimal submanifolds which admit singularities.  There are examples of such things in the literature.  We would like to study the situation where one has minimal hypercones in an ambient space of positive Ricci curvature.  Three cases arise:

\begin{itemize}
\item[(i)] Where the point of closest approach is between two smooth points,
\item[(ii)] Where the point of closest approach is between a smooth point and a singular point,

and perhaps the most interesting case,

\item[(iii)] Where the point of closest approach is between two singular points.
\end{itemize}

This amounts to trying to prove the following.

\begin{conj}  Let $M_{n + 1}$ be a complete, connected manifold with positive Ricci curvature.  Let $V_{n}$ and $W_{n}$ be immersed minimal currents of $M_{n + 1}$ with multiplicity one, each immersed as a closed subset, and let both $V_{n}$ and $W_{n}$ be compact.  Then $V_{n}$ and $W_{n}$ must intersect. \end{conj}

\begin{rmk} Note that this is not as general as the Frankel-Lawson result, since we require \emph{both} $V$ and $W$ to be compact. \end{rmk}

\begin{thm}  The conjecture is true for $V_{n}$ an immersed minimal submanifold and $W_{n}$ an immersed minimal current of $M_{n+1}$. \end{thm}

\begin{proof}  Suppose $V_{n}$ and $W_{n}$ are embedded submanifolds, and that they do not intersect.  Suppose $\gamma$ is a geodesic from $p_{0} \in V$ to $q_{0} \in W$ that realises the minimum distance $l$ from $V$ to $W$.  The cases where both of $p_{0}$, $q_{0}$ are nonsingular has already been treated.  So suppose $q_{0}$ is singular.  Now clearly $\gamma$ must strike $V$ and $W$ orthogonally, otherwise it can easily be shortened.  

\begin{lem} (Key Lemma). The first variation of arc length is $0$.  The second variation of arc length is negative for some choice of piecewise smooth vector field.

\begin{proof}

We first establish that the tangent cone at the singular point $q_{0}$ is a tangent plane.  This enables us to sensibly talk about vector fields about $q_{0}$.  Let $\sigma$ be $\gamma$ parametrised backwards from $q_{0}$ to $p_{0}$.  Let $x_{n}$ be a sequence of points converging to $q_{0}$ in $W$.  Let $\alpha_{n}$ be the corresponding sequence of geodesic curves from $q_{0}$ to $x_{n}$.  Let $\tau_{n}$ be the corresponding sequence of geodesics curves from $x_{n}$ to $p_{0}$.  We proceed in two steps:

(i) We first prove that if the angle between $\alpha_{n}$ and $\sigma$ is acute, that the length of $\tau_{n}$, $L(\tau_{n})$, is less than $L(\sigma)$ for sufficiently large $n$, hence establishing that such is an impossibility if $q_{0}$ is to be the point of closest approach.

(ii) Next, we prove using the Hopf maximum principle that the angle between $\alpha_{n}$ and $\sigma$ cannot be obtuse either.

Step 1.  We first observe the following fact: For neighbourhoods sufficiently small in any manifold, the metric can be written as

\begin{center}
$g(\partial_{i},\partial_{j})_{p} = \delta_{ij} + \theta(\epsilon^{2})$
\end{center}

for points $p$ within a distance $\epsilon$ from the center of geodesic coordinates with the choice $\nabla_{\partial_{i}}\partial_{j} = 0$ at the origin of the coordinates.

So take $N$ sufficiently large s.t. for all $n \geq N$, $L(\alpha_{n}) < \epsilon$.  Take a series of $\epsilon$-balls that cover the geodesic triangle $\tau_{n}$, $\alpha_{n}$, $\sigma$.  Then there is an isometry up to order $\epsilon^{2}$ that maps this geodesic triangle onto a triangle in Euclidean 2-space.

Hence from now on we will assume we are working in Euclidean 2-space, and by abuse of notation will refer to the images of $\tau_{n}$, $\alpha_{n}$ and $\sigma$ by the same names.  To control error, we will take $N_{2} > N$ s.t. $L(\alpha_{n}) < \epsilon^{2}$ for all $n \geq N_{2}$.

Now, the angle between $-\sigma$ and $-\tau_{n}$ will become very small for $n$ large, and, by the first order taylor series expansion for the tangent of this angle, we compute it to be roughly $\frac{L(\alpha_{n})}{L(\sigma)}$ to order $\frac{L(\alpha_{n})^{2}}{L(\sigma)^{2}}$.

Proceed in steps of length $\epsilon$ along both $-\tau_{n}$ and $-\sigma$ until we arrive in the last $\epsilon$-ball about $q_{0}$.  Connect the finishing points A and D along $-\tau_{n}$ and $-\sigma$ respectively by a curve $\lambda$.  This curve $\lambda$ will divide the triangle into an isoceles triangle and a rectangle.  Through an easy computation, the identical interior angles within the rectangle where $\lambda$ meets $\sigma$ and $\tau_{n}$ are seen to be

\begin{center}
$\phi_{n} = \pi/2 + \frac{L(\alpha_{n})}{2L(\sigma)} + \theta(\frac{L(\alpha_{n})^{2}}{L(\sigma)^{2}})$
\end{center}

Choose $N_{3} \geq N_{2}$ s.t. for all $n \geq N$, $L(\alpha_{n}) < \frac{1}{L(\sigma)}$.

Then $\phi_{n} \leq \pi/2 + L(\alpha_{n})^{2}/2 + \theta(L(\alpha_{n})^{4})$, or $\phi_{n} \leq \pi/2 + \theta(\epsilon^{4})$.

Then clearly if one projects down onto the image of $\sigma$, one sees that the projection of the curve from A to $x_{n}$ is of the same length to order $\epsilon^{4}$ and the projection of the curve from $x_{n}$ to $q_{0}$ is nonzero because the angle between $\alpha_{n}$ and $\sigma$ is acute.  So, if P is the projection,

Length(curve from D to $q_{0}$) = Length(curve from A to $x_{n}$) + Length(P(curve from $x_{n}$ to $q_{0}$)) $+ \theta(\epsilon^{4})$

Since all other terms are of order $\epsilon^{2}$ the error term is negligible and it is easily seen that Length(curve from D to $q_{0}$) $>$ Length(curve from A to $x_{n}$), from which it easily follows that $L(\tau_{n}) < L(\sigma)$.

Step 2. Suppose now that there is a vector $v$ in the tangent cone at $q_{0}$ that meets $\sigma$ at an obtuse angle.  By Step 1, all of the tangent cone must be at least at right angles to $\sigma$.  Furthermore, by the Hopf maximum principle, all of the tangent cone must meet $\sigma$ at an obtuse angle if part of it does.  But then this is a contradiction to the minimality of $W$.

So the tangent cone is a plane.

What about the second variation? Consider the class of all sequences of smooth points $(x_{n},y_{n})$ such that $x_{n}$ converges to $p_{0}$ and $y_{n}$ converges to $q_{0}$.  From before, we know that there is a variation such that the distance between $x_{n}$ and $y_{n}$ can be reduced, no matter how close they are to $p_{0}$, $q_{0}$ respectively.

Since $V$, $W$ are compact the variational vector fields $K_{n}$ inducing a negative second variation corresponding to $(x_{n},y_{n})$ converge to some vector field $K_{0}$ corresponding to $(p_{0},q_{0})$ that induces a variation in length less than or equal to $0$.  We can talk about vector fields at $q_{0}$ since the tangent cone at $q_{0}$ was proven to be a tangent plane before.  If the variation is less than $0$ we are done.  If it is equal to $0$ then I claim, if all other variations are greater than or equal to $0$, this contradicts the minimality of $W$ (this can be seen by once again invoking the Hopf maximum principle).  Hence either all second order variations are $0$ or there exists a variation less than $0$, and we are done.

So assume all second order variations are $0$ about $q_{0}$.  But this is impossible because the ambient space is not Ricci flat. \end{proof} \end{lem}

The theorem now follows easily in an analogous manner to the smooth case. \end{proof}

\begin{rmk} The statement that variational vector fields inducing negative second variation in length on compact spaces corresponding to a convergent sequence of points themselves converge to another variational vector field inducing nonpositive second variation is not at all obvious, and requires some degree of proof. \end{rmk}

\emph{Note}: It may be possible to extend the argument of Frankel and Lawson to minimal submanifolds that arise as limits of smooth Riemannian submanifolds, following the approach of \cite{[Fu]}.

\subsection{Jon Pitts' Construction}

I discuss here the nature of the construction Jon Pitts uses to establish existence of minimal submanifolds of compact spaces under certain dimensional restrictions.  Many thanks to Marty Ross for his considerable assistance in helping me dissect the relevant manuscript \cite{[Pi]} to understand the main thrust of the arguments.  Jon Pitts' manuscript actually provides a rather beautiful demonstration of how many of the techniques and tools of geometric measure theory are used in practice.

What he proves in his treatise is the following theorem:

\begin{thm} Let $M^{n}$ be a compact $n$-manifold, with $n \leq 6$.  Then for each $2 \leq k \leq 5$ there is at least one minimal submanifold $T_{k}$ of $M$ of dimension $k$. \end{thm}

The particular example to bear in mind here is the $2$-sphere.  I construct a minimal submanifold in the following way.  Trace out a family, any family, of codimension one submanifolds of the $2$-sphere (circles) from pole to pole, parametrised by numbers in the interval $[0,1]$.  So we have a mapping $\gamma : [0,1] \rightarrow S^{2}$.  Define $L(\gamma) = max\{L(\gamma(t)) : t \in [0,1] \}$.  Then we hope that we can find a $\tau$ that realises $\Lambda(M) = \inf_{\gamma}\{L(\gamma)\}$.  For this $\tau$, define $V = \gamma(t)$ such that $L(\gamma(t)) = max\{L(\gamma(T)): T \in [0,1]\}$.  Then $V$ will be minimal.  In particular, $V$ will be an equatorial circle of $S^{2}$, which is of course minimal, even though it will not be stable.

So we might reasonably expect this procedure to allow us to construct unstable minimal submanifolds of more general spaces.  There are of course problems that might arise in this program; for instance, there might not be a $\tau$ that realises $\Lambda(M)$.  This necessitates the use of varifolds and the whole machinery of geometric measure theory. 

So, following this idea, first Jon Pitts proves that one can construct, via this min-max procedure, a guy that is in some sense a weak solution to our minimisation problem.  He constructs a varifold $V$ that possesses a set of properties that he calls \emph{almost minimising}.  One of the properties of this varifold $V$ is that there is a sequence $\{V_{i}\}$ converging to $V$.

One of his major results is his decomposition theorem, which essentially states that each of the $V_{i}$ splits locally into sheets $V_{il}$ with density bounds; in particular for applications, the density of these sheets will be tightly bound about an integer multiplicity.  It is in the construction of the tools to prove this theorem, in particular curvature bounds, that the dimensional restriction for our ambient space $M$ becomes necessary to make.

The compactness theorem for varifolds then states that a subsequence of the $V_{il}$ converges to varifolds $S_{i}$.  The density bounds are preserved for the $S_{i}$ via certain other results.  Since one has bounds on the density, one can use the Allard regularity theorem to prove that the $S_{i}$ are regular.

Pitts then uses a careful argument to prove that $V = \sum_{i}S_{i}$, and it follows that $V$ is regular since each of the $S_{i}$ are disconnected (or equal) and are regular.  Furthermore $V$ is minimal.

\chapter{Existence Theory for PDEs}

\section{Elliptic and Parabolic PDEs}

As I demonstrated earlier, geometric measure theory is sufficient to establish existence and regularity results to solutions of variational problems over smooth domains.  By such a problem I mean a situation in which one defines a certain geometric invariant over a space, performs calculus of variations, and then asks whether there are solutions to the resulting PDE.  However one might ask whether solutions exist to general classes of PDE boundary value problems, where one does not necessarily have the luxury of some action principle to rely on.  This is the game of existence theory.

As far as I know, the very interesting question of whether all PDEs defined over domains that admit regular (smooth) solutions can be realised as the stationarity condition for some geometric invariant over the same domain, is still very much an open one.  Even if the answer to this question was in the affirmative, however, it does not tell us which PDEs can be realised as such variational problems.  So existence theory will always be useful.

Understanding existence theory for elliptic and parabolic PDEs is very important, particularly in application to the Ricci flow.  Many estimates crucial to understanding limiting behaviour of solutions and regularity for geometric flows are a consequence of understanding the analogous results for parabolic PDEs.  Hence it is important, before discussing these flows (see ahead, the section "Geometric Evolution Equations"), to discuss the general state of knowledge on this more fundamental area of mathematics.  Many of the results for parabolic PDE can be quickly developed from the analogous results for elliptic PDE, so I shall talk about these jointly.


\subsection{Introduction}


An elliptic equation is an equation of the following general form:

\begin{center} $L(a,b,c)\phi := (a_{ij}D_{i}D_{j} + b_{i}D_{i} + c)\phi = 0$ \end{center}

where $a_{ij}\zeta_{i}\zeta_{j} \geq 0$ for every nonzero vector $\zeta$.  This is known as the ellipticity condition.

Sometimes a stronger condition is required, the notion of \emph{uniform ellipticity}:

\begin{center} $a_{ij}\zeta_{i}\zeta_{j} \geq C\norm{\zeta}^{2}$  \end{center}

for some positive $C$.

A parabolic equation is a slight departure:

\begin{center} $(L(a,b,c) - \partial_{t})\phi = 0$ \end{center}
 
where $L$ is an elliptic operator as before, and now $\phi = \phi(x,t)$.  

Equivalently we can consider the class of equations

\begin{center} $(L(a,g) + c)\phi := (g_{il}D_{l}(a_{ij}g_{jk}D_{k}) + c)\phi = 0$ \end{center}

where $a_{ij}$ is a nondegenerate matrix, and $b_{k} = g_{il}D_{l}(a_{ij}g_{jk})$ is an arbitrary function depending on the choice of $g$.  Usually also $g$ is a symmetric positive definite bilinear form, ie, a Riemannian metric.  Now, if $a$ is positive as before, we have an elliptic equation; if it admits one or more negative eigenvalues, then it is hyperbolic.  In the special case that it admits one negative eigenvalue, and for the corresponding eigenvariable $t$ we have $g\partial_{t}(ag)\partial_{t}\delta = -gag\partial^{2}_{t}\delta$, then the above class degenerates to parabolic type.

Such equations arise naturally in problems via separation of variables, where $c$ may incorporate some eigenvalue multiplied by a characteristic function depending on the nature of the problem.  Spherical harmonics are a good example of this.


\subsection{The Maximum Principle}

Roughly what the maximum principle states, for an elliptic PDE, is if $u$ satisfies the inequality $\Delta u + a^{i}D_{i}u + cu \geq 0$ and $u$ is bounded on $\partial U$, where $U \subset M$ is a codimension $0$ subset, say by $D$ then $u$ remains bounded by $D$ for all $x \in U$.  In particular, if $u \leq 0$ on $\partial U$, then $u \leq 0$ on all of $U$.

For a parabolic PDE, if $u$ is a solution of the equation $\frac{\partial}{\partial t}u \leq \Delta u + a^{i}D_{i}u + cu$, then if $u$ is bounded initially, say by $D$, then the maximum principle tells us that $u$ is bounded at time $t$ by $De^{ct}$.

I shall now proceed to make the above statements more precise, with proof.

\begin{dfn} We will say that $L(a,g)$ is uniformly elliptic if $0 < \lambda \norm{\zeta}^{2} \leq a_{ij}\zeta_{i}\zeta_{j} \leq \Lambda \norm{\zeta}^{2}$ for some constants $\lambda$, $\Lambda$. \end{dfn}

What I said above for elliptic PDEs is essentially a consequence of the following result:

\begin{thm} ((Theorem 3.5, \cite{[GT]}), \cite{[Ho]}, "The Strong Maximum Principle").  Suppose $L = L(a,g)$ is uniformly elliptic, $c = 0$, and that $L\phi \geq 0$ (or $\leq 0$) in a domain $U$.  Then if $\phi$ achieves its maximum (minimum) in the interior of $U$ it must be constant.  Furthermore, if $c \leq 0$ and $c/\lambda$ is bounded, $\phi$ cannot achieve a nonnegative maximum ( non positive minimum) in the interior of $U$ unless it is constant.

\begin{proof} (by contradiction).  Suppose $\phi$ is non constant and achieves its maximum $M \geq 0$ in the interior of $U$.  Then the set $\bar{U}$ on which $\phi < M$ is certainly a subset of $U$ and also $\partial \bar{U} \cap U$ is nonempty.  Choose now a point $x$ in $\bar{U}$ closer to the boundary of $\bar{U}$ than the boundary of $U$; ie closer to the maximum point than to the boundary of the original domain.  Construct the largest metric ball $B$ in $\bar{U}$ with $x$ as centre.  Then $\phi(y) = M$ for some $y$ on $\partial B$, while also $\phi < M$ in $B$.

To conclude the proof, we will need 

\begin{lem} (Lemma 3.4, \cite{[GT]}). If $L$ is uniformly elliptic, $c = 0$ and $L\phi \geq 0$ in $B$, then if $y \in \partial B$ and

\begin{itemize} \item[(i)] $\phi$ is continuous at $y$,
\item[(ii)] $\phi(y) > \phi(x)$ for all $x \in B$,
\item[(iii)] There exists a sphere within $B$ such that the intersection of that sphere with $\partial B$ contains $y$,
\end{itemize}

then the outer normal derivative satisfies

\begin{center} $\frac{\partial \phi}{\partial v}(y) > 0$ \end{center}

If $c \leq 0$ and $c/\lambda$ is bounded we may conclude the same provided $\phi(y) \geq 0$.  If $u(y) = 0$ there is no restriction on the sign of $c$.

\begin{proof} (of lemma).  The proof of this lemma as given in Gilbarg and Trudinger is not personally to my taste; it is primarily an analyst's proof, rigorous, but unnecessarily technical and not terribly instructive.  Essentially one defines an auxiliary function for comparison purposes and uses that to facilitate the argument.  For a more intuitive or geometric idea of why we expect the above to be true, at least in the case that $c = 0$, consider a geodesic path from the centre $x$ of the sphere in $B$ to the point $y$ on $\partial B$.  Furthermore we know that $L\phi$ is nonnegative on this geodesic, and $\phi$ is increasing along it at $y$.  However, the second piece of information alone tells us nothing; the derivative could easily be zero at $y$.  

Consider now the geodesic as the real line and $\phi$ as a polynomial.  Since $L\phi$ can be thought of as the "deformed acceleration" since it is elliptic, we must have that at $y$, $\phi(x)$ be locally of the form $x^{2k}$, $k \in N$ after an appropriate choice of chart since the acceleration is nonnegative.  By the second piece of data, we must be in the right branch of $x^{2k}$.  Since $y$ is not at zero, the derivative is nonzero and positive.

A similar argument, albeit with charts more carefully chosen, also goes through for a nonzero function $c$.
\end{proof}
\end{lem}

It then follows easily from the lemma and from before that $D\phi(y) \neq 0$, but this is impossible since $y$ was supposed to be a maximum.
\end{proof}
\end{thm}

The proof of the maximum principle for parabolic equations is proved along similar lines, though it requires some more care.  The main difference is of course the time dependence of the bound, or the "diffusion" of $u$ over time.

Primarily, the main difference is that the fundamental or Green's function solution to the equation $Lf - \partial_{t}f = \delta(x)\delta(t)$ satisfies the inequality

\begin{center} $f \leq e^{ct}g$ \end{center}

where $g$ is the solution to the equation $Lg = \delta(x)$.

It is then an easy consequence of the maximum principle above and the theory of Green's functions that if $u \leq D$ at $t = 0$, that $u \leq De^{ct}$ for arbitrary $t$.

Finding a proof of the above inequality is not an entirely trivial exercise.  Nonetheless I will make an attempt to at least sketch one.  We certainly know that

\begin{center} $Lf - \delta(t)Lg - \partial_{t}f = 0$ \end{center}

since $Lg = \delta(x)$.  Consequently if we factor $L = \bar{L} + c$ and integrate by parts twice, we get $\bar{L}(f - \delta(t)g) + g_{tt}\delta(t) - \partial_{t}f + c(f - \delta(t)g) = 0$.  But $g$ is time independent; hence the second term vanishes and we have

\begin{center} $\bar{L}h - \partial_{t}f + ch = 0$ \end{center}

where $h = f - \delta(t)g$.  But since $\bar{L}$ is a positive operator we can write

\begin{center} $\partial_{t}f \geq c(f - \delta(t)g)$ \end{center}

and consequently $\partial_{t}(\frac{f}{g}) \geq c (\frac{f}{g} - \delta(t))$, from which we conclude $ln\frac{f}{g} \leq c(t - H(t))$ and finally $f \leq ge^{ct - H(t)} \leq ge^{ct}$ which concludes the proof of the inequality and hence of the maximum principle for parabolic PDE.

\subsection{Differential Harnack Inequalities}

I prove the Harnack inequality for the heat equation, due to Li and Yau:

\begin{thm} If $\frac{\partial u}{\partial t} = \Delta u$ on $R^{n}$, $u > 0$, then $\Delta u - \frac{\norm{\nabla u}^{2}}{u} + \frac{nu}{2t} \geq 0$.

\begin{proof} Let $v = log$ $u$.  Then $\frac{\partial v}{\partial t} = \frac{1}{u}\Delta u$ and

\begin{center}
$\Delta v = \nabla_{k}(\frac{1}{u}\nabla_{k} u) = \frac{1}{u}\Delta u - \frac{1}{u^{2}}\norm{\nabla u}^{2}$
\end{center}

Hence

\begin{center}
$\frac{\partial v}{\partial t} = \Delta v + \norm{\nabla v}^{2}$
\end{center}

I show $\Delta v \geq -\frac{n}{2t}$.

\begin{align}
\frac{\partial}{\partial t}\nabla_{k}v &= \nabla_{k}(\Delta v + \norm{\nabla v}^{2}) \nonumber \\
&= \Delta \nabla_{k}v + 2\nabla_{l}v \nabla_{k}\nabla_{l}v
\end{align}

\begin{align}
\frac{\partial}{\partial t}\Delta v &= \Delta \Delta v + 2 \norm{\nabla^{2}v}^{2} + 2 \nabla_{l}v\nabla_{l}\Delta v \nonumber \\
&\geq \Delta \Delta v + \frac{2}{n}(\Delta v)^{2} + 2\nabla v \cdot \nabla \Delta v
\end{align}

At the minimum point of $\Delta v$, $\Delta \Delta v \geq 0$ and $\nabla \Delta v = 0$.  So,

\begin{center}
$\frac{\partial}{\partial t} \Delta v \geq \frac{2}{n}(\nabla v)^{2}$
\end{center}

Hence by the maximum principle

\begin{center}
$\Delta v \geq - \frac{n}{2t}$
\end{center}

Or, remembering $v = ln(u)$,

\begin{center}
$\frac{\partial}{\partial t}(log(u)) - \norm{\nabla log(u)}^{2} = \frac{\Delta u}{u} - \norm{\frac{\nabla u}{u}}^{2} \geq -\frac{n}{2t}$
\end{center}

This proves the Harnack inequality.
\end{proof}
\end{thm}


\section{Hyperbolic PDEs}

In this section I review the work of Shatah and Struwe \cite{[SS]}, who cover basic existence and regularity theory for hyperbolic wave equations.  This is of interest in application to the constitutive equation for a sharp Physical manifold (see the section on Information Theory and Physical Manifolds), where one has an equation of the form

\begin{center} $R_{ij} = 0$ \end{center} 

Following this I make some attempt at assembling an existence theory for more general hyperbolic equations, trying to extend Shatah and Struwe's techniques.  The main motivation is to apply this to slightly blurred Physical manifolds, where one has two coupled PDEs

\begin{center} $(1 + \epsilon \Delta)R_{ij} = 0$ \end{center}

and

\begin{center} $(\Delta + \epsilon \Delta^{2})R_{ij}  = 0$ \end{center}

Finally I attempt to establish an existence theory for mixed parabolic-hyperbolic PDEs.  This is related to steepest descent of the physical information for the previous two examples.

Unlike for parabolic and elliptic equations, we no longer have the luxury of having a maximum principle to draw upon, so we have to find other tricks.  In fact, it will turn out that existence theory for hyperbolic equations tends to rest on finding and exploiting "conserved" quantities.  

\subsection{Existence theory for hyperbolic equations}

In their book \cite{[SS]}, Shatah and Struwe develop an existence and regularity theory for equations of the form

\begin{equation}
\label{Shatah}
\Delta u = h(x,t,u,Du)
\end{equation}

Note that this is of interest in understanding existence and regularity of equations such as

\begin{center} $Ric(g) = 0$ \end{center}

since $Ric(g) \sim \Delta g + \{$correction terms depending on $g$ and first order derivatives of $g\}$.  This gives a tensor equation of the form

\begin{center} $\Delta g_{ij} = h(x,t,g_{ij},Dg_{ij})$ \end{center}

which is of the same form as the first equation.  I shall now proceed to sketch the key ideas in this theory.

\begin{rmk}  Due to the equivalence of parabolic PDE to hyperbolic PDE as I sketched earlier, we may make the simplifying assumption of considering only equations of the type

\begin{center} $\Delta g_{ij} = h(x,g_{ij},Dg_{ij})$ \end{center}

by absorbing time into our coordinate system, ie $x \mapsto \bar{x}, x = (\bar{x},t)$.  This will simplify the description of Shatah-Struwe theory considerably.
\end{rmk}

Given this simplifying assumption, it turns out that all we need can be obtained from the seventh page of \cite{[SS]}.  Let $u : M \rightarrow R$ , $L = L(u,p) : TM \rightarrow R$ be smooth functions.  Write $L(u,Du)$ as the Lagrangian density of $u$, and consider the action

\begin{center} $A(u;Q) = \int_{Q}L(u,Du)dm$ \end{center}

where $Q \subset M$ is an arbitrary subset.

Take a variation of $u$, $u_{\epsilon} = u + \epsilon \phi$, where $\phi$ is a function with compact support on $M$.  Then if we evaluate the first variation of the above action with such notation, we obtain after following the standard sort of procedure that

\begin{center} $\frac{d}{d\epsilon}A(u_{\epsilon};Q)\vert_{\epsilon = 0} = \int_{Q}(L_{u}(u,Du) - \partial_{i}L_{p_{i}}(u,Du))\phi dz$ \end{center}

Then $u$ is stationary with respect to $A$ iff $L_{u}(u,Du) = \partial_{i}L_{p_{i}}(u,Du)$.

Now, consider the Lagrangian $L(u,Du) = \frac{1}{2}\norm{\nabla u}^{2} + F(u,Du)$.  Then, after some intermediate working, one obtains that the above equality holds only if

\begin{center} $\Delta u = \frac{d}{d\epsilon}F(u_{\epsilon},Du_{\epsilon})\vert_{\epsilon = 0}$ \end{center}

Since $F$ was arbitrary, the right hand side of this expression is arbitrary.  Hence we can reverse engineer an action for the original Shatah-Struwe equation.  Via geometric measure theory, since we have constructed an action, the solution will exist and it will be smooth.

\subsection{Existence theory for "elastic" wave equations}

Before I finish I will briefly mention how one may establish an existence and regularity theory for equations of the type

\begin{center} $\Delta^{2} u = h(x,t,u,Du,D^{2}u,D^{3}u)$ \end{center}

The idea here is to consider the same sort of thing as for the above, but instead to look at the action

\begin{center} $\norm{\Delta u}^{2} - F(u,Du,D^{2}u,D^{3}u)$ \end{center}

then by the same general nonsense as before, we find that the first variation is zero iff

\begin{center} $\Delta^{2}u = \frac{d}{d\epsilon}F(u_{\epsilon},Du_{\epsilon},D^{2}u_{\epsilon},D^{3}u_{\epsilon})\vert_{\epsilon = 0}$ \end{center}

Then, since the right hand side is arbitrary, and via the argument once more from geometric measure theory that a smooth action implies smooth solutions, we are done.







\chapter{Geometric Evolution Equations}

The following notes are a typeset version, with some minor additions, of a most excellent course given by Ben Andrews of the ANU at the University of Queensland in the winter of 2006.  For further information on this topic, in particular the specialisation to examination of the Ricci Flow and its application to the classification of 3-manifolds, the interested reader is invited to investigate the sources Cao-Zhu \cite{[C]} and Morgan-Tian \cite{[MT]}.  John Morgan and Gang Tian have also quite recently written \cite{[MT2]}, which is a followup article to \cite{[MT]}.  Some existence theory of PDE may be required to understand some of the material in these references; hence the reader is invited to also keep a copy of \cite{[GT]} handy, since many of the methods introduced therein are applicable to weakly parabolic PDEs, of which the Ricci Flow is an example.

Even though the information here is not directly relevant to the rest of this work it has some bearing on classification of physical manifolds under a certain choice of signal function (see later section).  Note also that the fact that geometric evolution equations of heat type improve the Fisher information in and of itself is of considerable interest in finding extrema of this quantity (again, see the section on Fisher information later on).  Geometric evolution equations are also extremely useful in proving results in comparison geometry; for instance, one such result tells us that a manifold admitting a metric of positive Ricci curvature admits a metric of constant sectional curvature.  However, the primary motivation in preparing this section was mainly for my self-reference, and consolidation of my understanding of this most beautiful subject.

\section{The Heat Flow Equation as the Prototypical Example}

The heat equation for a function

\begin{center}
$u : \Omega \times [0,T) \rightarrow \mathcal{R}$,
\end{center}

where $\Omega$ is a smoothly bounded open domain in $R^{n}$ is the parabolic PDE

\begin{equation}
\label{heatequation}
\frac{\partial u}{\partial t}(p,t) = \Delta u(p,t)
\end{equation}

To solve a heat flow equation, we also need

\begin{itemize}
\item[(i)] Boundary data, eg Dirichlet data ($u(p,t) = 0$, $p \in \partial \Omega$), or Neumann data ($D_{n}u(p,t) = 0$, $p \in \partial \Omega$), or Periodic data, and also
\item[(ii)] Initial data, ie $u(p,0) = u_{0}(p)$.
\end{itemize}

\subsection{Maximum Principle}.

If $u$ is a smooth solution of the heat equation on $\Omega \times [0,T)$, then if $u_{0}(p) \geq 0 \forall p \in \Omega$, then $u(p,t) \geq 0 \forall p \in \Omega$, and $\forall t \geq 0$.

\emph{Proof}.  (Restricted to the Dirichlet case, though it follows for other boundary conditions as well).  We show that $u_{\epsilon}(p,t) = u(p,t) + \epsilon(1 + t) \geq 0$ for all $\epsilon > 0$.

Fix $\epsilon > 0$.  Then I claim that $u_{\epsilon}(p,0) \geq \epsilon > 0 \forall p \in \Omega$.  For instance, in the Dirichlet case, $u_{\epsilon}(p,t) = \epsilon(1 + t) > 0 \forall p \in \partial \Omega$.

Suppose the Result is not true.  Then let $(p_{0},t_{0})$ be such that $u_{\epsilon}(p_{0},t_{0}) = 0 = inf\{u_{\epsilon}(p,t) | p \in \Omega, 0 \leq t \leq t_{0} \}$.

Then $\frac{\partial u_{\epsilon}}{\partial t}(p_{0},t_{0}) \leq 0$, since $u_{\epsilon}(p_{0},t) \geq u_{\epsilon}(p_{0},t_{0})$ for $0 \leq t \leq t_{0}$.

But since $(p_{0},t_{0})$ was defined to be the point realising the infimum of $u_{\epsilon}$, we have also that $\Delta u_{\epsilon}(p_{0},t_{0}) \geq 0$.

So then

$0 \geq \frac{\partial u_{\epsilon}}{\partial t}(p_{0},t_{0}) = \frac{\partial u}{\partial t}(p_{0},t_{0}) + \epsilon = \Delta u(p_{0},t_{0}) + \epsilon = \Delta u_{\epsilon}(p_{0},t_{0}) + \epsilon \geq \epsilon$,

a contradiction.

\subsection{Comparison Principle}.

If $u, v$ are two smooth solutions of the heat equation with $u(p,0) \geq v(p,0) \forall p \in \Omega$ then $u(p,t) \geq v(p,t)$.

\emph{Proof}.  Observe that $u - v$ also satisfies the heat equation.  The result then follows from Result (i).

\subsection{Averaging Property}.

$sup_{p \in \Omega}u(p,t)$ is monotone nonincreasing.

\begin{align}
\frac{d}{dt}\int_{\Omega}u^{2} &= 2\int_{\Omega}u\Delta u \nonumber \\
&= 2\int_{\partial \Omega}uD_{n}u - 2\int_{\Omega}\Vert Du \Vert^{2} \leq 0
\end{align}

"Energy"

\begin{align}
\frac{d}{dt}\int_{\Omega}\Vert Du \Vert ^{2} &= 2\int_{\Omega}D_{i}u D_{i}\Delta u \nonumber \\
&=2\int_{\Omega}D_{i}u \Delta D_{i}u \nonumber \\
&=2\int_{\partial \Omega}D_{n}u \Delta u - 2\int_{\Omega}(\Delta u)^{2} \leq 0,
\end{align}

if we have either Dirichlet or von Neumann boundary conditions, in which case the first term vanishes.

"Shannon Entropy"

\begin{align}
\frac{d}{dt}\int_{\Omega}u log (u) &= \int_{\Omega}(1 + log u)\Delta u \nonumber \\
&= - \int_{\Omega}\frac{\Vert \nabla u \Vert}{u} \leq 0
\end{align}

"Fisher Information"

Define the Fisher Information $I(u)$ as

\begin{equation}
I = \int_{\Omega}\frac{\norm{Du}^{2}}{u}
\end{equation}

\begin{align}
\frac{d}{dt}I &= \frac{d}{dt}\int_{\Omega}\frac{D_{i}u D_{i}u}{u} \nonumber \\
&= 2\int_{\Omega}\frac{D_{i}u D_{i}\Delta u}{u} - \int_{\Omega}\frac{D_{i}u D_{i}u}{u^{2}}\Delta u \nonumber \\
&= 2\int_{\partial \Omega}\frac{D_{n}u \Delta u}{u} - 2\int_{\Omega}D_{i}(\frac{D_{i}u}{u})\Delta u - \int_{\Omega}\frac{D_{i}u D_{i} u}{u^{2}}\Delta u \nonumber \\
&= 0 + \int_{\Omega}\frac{D_{i}u D_{i}\Delta u}{u} - \int_{\Omega}\frac{(\Delta u)^{2}}{u} \nonumber \\
&= \frac{1}{2}\frac{d}{dt}I - \int_{\Omega}\frac{(\Delta u)^{2}}{u}
\end{align}

Hence

\begin{center}
$\frac{d}{dt}I = -2\int_{\Omega}\frac{(\Delta u)^{2}}{u} \leq 0$,
\end{center}

if $u$ is a positive function on $\Omega$.

\subsection{Smoothing}.

eg. Suppose $\norm{u_{0}(p,0)} \leq C_{0}$
(Then $\norm{u_{0}(p,t)} \leq C_{0} \forall p, \forall t \geq 0$)

Then for each $k \in N$ there exists a constant $C_{k}$ such that

\begin{center}
$\norm{D^{k}u(p,t)}^{2} \leq \frac{C_{k}^{2}}{t^{k}}C_{0} \forall p \in \Omega, t > 0$
\end{center}

Note that

\begin{center}
$D^{k}u = \Sigma_{\{\alpha_{1},...,\alpha_{n} | \alpha_{1} + ... + \alpha_{n} = k\}}\frac{\partial^{k}u}{(\partial x^{1})^{\alpha_{1}}...(\partial x^{n})^{\alpha_{n}}}$
\end{center}

\emph{Proof}. (Periodic case)

We prove the result by induction.  Suppose we have

\begin{center}
$\norm{D^{k}u(p,t)}^{2} \leq \hat{C}_{k}^{2}$, $\forall p \in \Omega, t \in [0,T]$
\end{center}

We will bound $\norm{D^{k+1}u}^{2}$.

\begin{center}
$\frac{\partial}{\partial t}D^{\alpha}u = D^{\alpha}\Delta u = \Delta(D^{\alpha}u)$
\end{center}

Hence

\begin{align}
\frac{\partial}{\partial t}\norm{D^{k}u}^{2} &= \Sigma_{\alpha}\frac{\partial}{\partial t}\norm{D^{\alpha}u}^{2} \nonumber \\
&= 2\Sigma_{\alpha}D^{\alpha}u\Delta(D^{\alpha}u)
\end{align}

Now

\begin{align}
\Delta \norm{D^{\alpha}u}^{2} &= \Sigma_{i}D_{i}D_{i}\norm{D^{\alpha}u}^{2} \nonumber \\
&= \Sigma_{i}2D_{i}(D^{\alpha}u D^{\alpha}D_{i}u) \nonumber \\
&= 2D^{\alpha}u \Delta(D^{\alpha}u) + 2\Sigma_{i}\norm{D_{i}D^{\alpha}u}^{2}
\end{align}

So we deduce

\begin{center}
$\frac{\partial}{\partial t}\norm{D^{k}u}^{2} \leq \Delta \norm{D^{k}u}^{2} - \norm{D^{k+1}u}^{2}c(k,n)$
\end{center}

for some constant $c(k,n)$.

Also note similarly that

\begin{center}
$\frac{\partial}{\partial t}\norm{D^{k+1}u}^{2} \leq \Delta \norm{D^{k+1}u}^{2}$
\end{center}

Define

\begin{center}
$Q = \frac{\norm{D^{k+1}u}^{2}}{2\hat{C}_{k}^{2} - \norm{D^{k}u}^{2}}$
\end{center}

Note that the denominator will always be positive.

Furthermore, note that rearranging our result above we get that

\begin{equation}
\label{smoothing1}
\frac{\partial}{\partial t}(2 \hat{C}_{k}^{2} - \norm{D^{k}u}^{2}) \geq \Delta(2\hat(C)_{k}^{2} - \norm{D^{k}u}^{2}) + c(k,n)\norm{D^{k+1}u}^{2}
\end{equation}

We would like to compute a heat equation type inequality for $Q$.  For this we make use of the following lemma:

\emph{Lemma}.  If $f$ and $g$ satisfy $f,g > 0$ and

\begin{center}
$\frac{\partial f}{\partial t} \leq \Delta f + P$

$\frac{\partial g}{\partial t} \geq \Delta g + R$
\end{center}

then

\begin{center}
$\frac{\partial}{\partial t}(\frac{f}{g}) \leq \Delta(\frac{f}{g}) + 2\frac{\nabla g}{g}.\nabla(\frac{f}{g}) + \frac{1}{g}(P - \frac{f}{g}R)$
\end{center}

\emph{Proof of Lemma}. Merely a careful calculation.

By the lemma, we have that

\begin{equation}
\label{Qsmoothing}
\frac{\partial}{\partial t}Q \leq \Delta Q + V.DQ - Q\frac{c(k,n)\norm{D^{k+1}}^{2}}{2\hat{C}_{k}^{2} - \norm{D^{k}u}^{2}}
\end{equation}

where

$V = 2\frac{D(2\hat{C}_{k}^{2} - \norm{D^{k}u}^{2})}{2\hat{C}_{k}^{2} - \norm{D^{k}u}^{2}}$

I claim that $Q(p,t) \leq \frac{1}{c(k,n)t}$.  From this our result would automatically follow.

So observe that by the first inequality we have that

\begin{center}
$\frac{\partial}{\partial t}(tQ - \frac{1}{c}) = Q + t\frac{\partial Q}{\partial t}$

$\leq \Delta (tQ - \frac{1}{c}) + V.D(tQ - \frac{1}{c}) - cQ(tQ - \frac{1}{c})$
\end{center}

Now by the existence theory for parabolic PDE if one has an equation of the form

\begin{center}
$\frac{\partial}{\partial t}X \leq \Delta X - b^{j}D_{j}X + \mu X$
\end{center}

where $b^{j}$ and $\mu$ are smooth then we have a maximum principle that states that, provided $X(p,0) \leq 0$ for all $p \in \Omega$ then $X(p,t) \leq 0$, $\forall p \in \Omega, t \geq 0$.  (For a proof of this maximum principle, see the previous section on existence theory.)

Certainly $tQ - \frac{1}{c} < 0$ at $t = 0$, so the conditions of the maximum principle are satisfied and our claim follows.

Hence

\begin{center}
$\norm{D^{k+1}u}^{2} \leq \frac{1}{ct}(2\hat{C}_{k}^{2} - \norm{D^{k}u}^{2}) \sim \frac{\hat{C}_{k}^{2}}{ct}$
\end{center}

and we are done.

\section{The Curve Shortening Flow (CSF)}

\subsection{Introduction}

First a brief discussion about curves.

Consider $X_{0} : R \rightarrow R^{2}$ such that $X_{0}(u + n) = X_{0}(u)$, for $n \in Z$.  So we have an immersion of the circle into the plane.  We are interested in finding properties invariant under rigid motions and reparametrisation.

The standard or canonical parametrisation is in terms of arc length.

Define

\begin{center}
$\phi(s) = \int_{0}^{s}\norm{\frac{\partial X}{\partial u}}du$
\end{center}

The unit tangent vector $T$ is defined as

\begin{center}
$T = \frac{\partial X/\partial u}{\norm{\partial X/\partial u}}$
\end{center}

The unit normal vector $N$ is defined as

\begin{center}
$N = R_{-\frac{\pi}{2}}T$
\end{center}

where $R_{\theta}$ is rotation by $\theta$.

Since $\norm{T}^{2} = 1$, $2<T,\frac{\partial T}{\partial s}> = 0$.

Define

\begin{center}
$\frac{\partial T}{\partial s} = - \kappa N$
\end{center}

The curvature $\kappa$ determines $X_{0}$.

We have, similarly to $T$, that $\norm{N}^{2} = 1$ and consequently $2<N,\frac{\partial N}{\partial s}> = 0$.

Also the relation $0 = <N,T>$ implies that

\begin{center}
$0 = <\frac{\partial N}{\partial s},T> + <N,-\kappa N>$,
\end{center}

or

\begin{center}
$\frac{\partial N}{\partial s} = \kappa T$
\end{center}

Now I discuss the curve shortening flow.  Consider once more a smooth immersion $X_{0}$.  We wish to find $X : R/Z \times [0,T) \rightarrow R^{2}$ such that

\begin{center}
$\frac{\partial X}{\partial t}(u,t) = \frac{\partial^{2}X}{\partial s^{2}}(u,t)$
\end{center}

with the initial condition $X(u,0) = X_{0}(u)$.  This specifies our curve shortening flow.

\emph{Note}.  Since $\frac{\partial X}{\partial s} = T$ we have that $\frac{\partial^{2}X}{\partial s^{2}} = \frac{\partial T}{\partial s} = -\kappa N$, and hence that

\begin{center}
$\frac{\partial X}{\partial t} = -\kappa N$
\end{center}

\subsection{The Shrinking Circle}

As an example we consider the circle under curve shortening flow.  The solution is 

\begin{center}
$X(u,t) = (r_{0}^{2} - 2t)^{1/2}(cos u, sin u)$
\end{center}

\subsection{Geometric Invariance}

If $K$ is a rigid motion of $R^{2}$, $\phi : R \rightarrow R$ smooth, $\phi ' > 0$, then $\hat{X}(u,t) = K(X(\phi(u),t))$ is also a solution of curve shortening flow.

\subsection{Avoidance Principle}

Suppose $X, Y : R/Z \times [0,T) \rightarrow R^{2}$ are solutions of CSF with $X(u,0) \neq Y(v,0)$ for all $u,v$, then $X(u,t) \neq Y(v,t)$ for all $u,v \in R/Z$, all $0 \leq t < T$.  ie if we have two curves that start disjoint they remain disjoint under CSF.

\emph{Proof}.  Define $\rho : R/Z \times R/Z \times [0,T) \rightarrow R$ as

\begin{center}
$\rho(u,v,t) = \norm{X(u,t) - Y(v,t)}^{2}$
\end{center}

We have $\rho \geq \rho_{0} > 0$ at $t = 0$ by hypothesis.

For any $t \geq 0$, let $u,v \in R/Z$ be such that

\begin{center}
$\rho(u,v,t) = inf\{\rho(p,q,t) | p,q \in R/Z \}$
\end{center}

At this point we have

\begin{center}
$\frac{\partial \rho}{\partial u} = 0 = \frac{\partial \rho}{\partial v}$
\end{center}

and, by the nature of the infimum,

\begin{center}
$D^{2}\rho \geq 0$
\end{center}

($D^{2}\rho$ is the matrix of mixed spatial derivatives of 2nd order of $\rho$).

Choose parametrisations so that $u,v$ are arc-length parameters at time $t$.

Then

\begin{align}
&\frac{\partial \rho}{\partial u} = 2<X - Y,T_{X}> \\
\intertext{and}
&\frac{\partial \rho}{\partial v} = -2<X - Y,T_{Y}>
\end{align}

We also have

\begin{align}
&\frac{\partial^{2}\rho}{\partial u^{2}} = 2<T_{X},T_{X}> + 2<X - Y, -\kappa_{X}N_{X}> \\
&\frac{\partial^{2}\rho}{\partial v^{2}} = 2 - 2<X - Y,-\kappa_{Y}N_{Y}> \\
\intertext{and}
&\frac{\partial^{2}\rho}{\partial u \partial v} = -2<T_{Y},T_{X}>
\end{align}

If $\frac{\partial \rho}{\partial u} = 0 = \frac{\partial \rho}{\partial v}$ then we can choose signs so that $T_{X} = T_{Y} \perp X - Y$.

So

\begin{align}
\frac{\partial \rho}{\partial t} &= 2<X - Y, -\kappa_{X}N_{X} + \kappa_{Y}N_{Y}> \nonumber \\
&= \frac{\partial^{2}\rho}{\partial u^{2}} + \frac{\partial^{2}\rho}{\partial v^{2}} - 2 - 2 \nonumber \\
&= v.D^{2}\rho.v \geq 0
\end{align}

where $v = [1,1]$.

So the distance between $X$ and $Y$ is nondecreasing, which proves the avoidance principle.

\subsection{Maximum Principle}

Suppose $u : M \times [0,T] \rightarrow R$ is smooth, and $M$ is a compact boundary without boundary.  Suppose also that for all $(x,t) \in M \times [0,T)$, in particular for $u(x,t) = inf\{u(y,t) | y \in M\}$ that $\frac{\partial u}{\partial t}(x,t) \geq 0$.  Then $inf\{u(y,t) | y \in M\}$ is nondecreasing in $t$.

\emph{Proof}. Look at the infimum point of $u_{\epsilon} = u + \epsilon(1 + t) - inf u$, where $u_{\epsilon}$ first hits $0$.  The argument is similar to the one we made before.

\emph{Warning}. The following result is not true:

If $u : M \times [0,T)$ is smooth, $u(x,0) > 0$ and $\frac{\partial u}{\partial t} \geq 0$ whenever $u(x,t) = inf\{u(y,t) | y \in M\} = 0$, then $u(x,t) \geq 0$ for all $x$, $t \geq 0$.

\subsection{Preserving Embeddedness}

If $X_{0}$ is embedded, and $X : R/Z \times [0,T) \rightarrow R^{2}$ is a smooth solution of CSF, with $sup\norm{\kappa(u,t)} \leq C$, then $X_{t}(.) = X(.,t)$ is an embedding.  (Note in particular that $\kappa$ of a curve is bounded if the curve can be parametrised by arc length).

\emph{Proof}. We need

(i) $\norm{X_{u}} \neq 0$

and

(ii) $X(.,t)$ is one to one.

Proof of (i):

\begin{align}
\frac{\partial}{\partial t}\norm{X_{u}} &= \frac{1}{\norm{X_{u}}}<X_{u},\frac{\partial}{\partial t}\frac{\partial X}{\partial u}> \nonumber \\
&= <T, \frac{\partial}{\partial s}(-\kappa N)>\norm{X_{u}} \nonumber \\
&= -\kappa^{2}\norm{X_{u}}
\end{align}

Hence

\begin{align}
&\norm{X_{u}(u,t)} \leq e^{C^{2}t}\norm{X_{u}(u,0)} \\
\intertext{and}
&\norm{X_{u}(u,t)} \geq e^{-C^{2}t}\norm{X_{u}(u,0)}
\end{align}

proving (i).

Proof of (ii):

I first claim that if the arc length between two points $u$ and $v$ is $\delta_{0}$ that

Consider $\rho(u,v,t) = \norm{X(u,t) - X(v,t)}^{2}$ on $\{(u,v,t) |$ arc length from $u$ to $v$ at time $t$ $\geq \delta_{0}\}$.

The claim is then that for any two points $u$, $v$ on the curve at time zero a distance $\delta_{0}$ or greater apart that $\rho(u,v,t) \geq D\delta_{0}^{2}$.

Well, the maximum principle tells us that

\begin{center}
$inf_{(s,t)}\rho \geq min\{inf_{(\partial s,t)}\rho, inf_{(s,0)}\rho\}$
\end{center}

so the points remain separated if they are separated initially, proving (ii).

\subsection{Finite Time Singularity}

Any bounded embedded curve will reach a singularity after a finite amount of time under CSF.

\emph{Proof}.  Draw a circle enclosing the curve.  By the avoidance principle, under CSF neither curve will ever intersect.  But under CSF the enclosing circle shrinks to a point.  Hence so must the curve.

\subsection{CSF as a steepest descent flow for length}

It turns out more generally that all geometric evolution equations can be viewed as a steepest descent flow for some functional.  We shall see this later for the mean curvature flow, and, later still, for the Ricci flow.  But for now consider the following functional:

\begin{center}
$L(t) = \int_{0}^{1}\norm{\frac{\partial X}{\partial u}}du$
\end{center}

Consider an arbitrary variation

\begin{center}
$\frac{\partial X}{\partial t}(u,t) = V(u,t)$
\end{center}

Then

\begin{center}
$\frac{\partial}{\partial t}\norm{X_{u}} = <T,\frac{\partial}{\partial s}V>\norm{X_{u}}$
\end{center}

Hence

\begin{align}
\frac{d}{dt}L(t) &= \int_{0}^{1}<T,\frac{\partial V}{\partial s}>ds \nonumber \\
&= \int_{0}^{1}\frac{\partial}{\partial s}<T,V>ds - \int_{0}^{1}<V,-\kappa N>ds 
\end{align}

The first term vanishes since the variation vanishes on the boundary of the curve, or, if it is a closed curve, since the end points are the same.

Hence

\begin{center}
$\frac{d}{dt}L(t) = -\int_{0}^{1}<V,-\kappa N>ds \geq - (\int_{0}^{1}\norm{V}^{2}ds)^{1/2}(\int_{0}^{1}\kappa^{2}ds)^{1/2}$
\end{center}

by the Cauchy-Schwarz inequality.

So we conclude that $\frac{dL}{dt}$ is minimised along all variations with $\int \norm{V}^{2} = 1$ by $V \propto -\kappa N$, since then this inequality is realised.  But this means that

\begin{center}
$\frac{\partial X}{\partial t}(u,t) = V(u,t) = -\kappa(u,t)N(u,t)$,
\end{center}

so we recover CSF.

\subsection{Existence}

Let $(x,y) = X : R/Z \times [0,T) \rightarrow R^{2}$.

Then

\begin{center}
$\frac{\partial X}{\partial u} = (x_{u},y_{u})$
\end{center}

and

\begin{center}
$T = \frac{(x_{u},y_{u})}{\sqrt{x_{u}^{2} + y_{u}^{2}}}$
\end{center}

Now

\begin{align}
-\kappa N &= \frac{\partial T}{\partial s} = \frac{1}{\sqrt{x_{u}^{2} + y_{u}^{2}}}\frac{\partial}{\partial u}(\frac{(x_{u},y_{u})}{\sqrt{x_{u}^{2} + y_{u}^{2}}}) \nonumber \\
&= \frac{(x_{uu},y_{uu})}{\sqrt{x_{u}^{2} + y_{u}^{2}}} - \frac{(x_{u},y_{u})}{(x_{u}^{2} + y_{u}^{2})^{2}}(x_{u}x_{uu} + y_{u}y_{uu})
\end{align}

We can then write CSF in this choice of coordinates as

\[ \frac{\partial}{\partial t} \left( \begin{array}{c}
x \\
y \end{array} \right) = \frac{1}{(x_{u}^{2} + y_{u}^{2})}\left( \begin{array}{cc}
y_{u}^{2} & -x_{u}y_{u} \\
-x_{u}y_{u} & x_{u}^{2} \end{array} \right)
\left( \begin{array}{c}
x_{uu} \\
y_{uu} \end{array} \right) \]

\begin{dfn} A system $\frac{\partial}{\partial t}X^{\alpha} = A^{ij\alpha}_{\beta}\frac{\partial^{2}x^{\beta}}{\partial u^{i} \partial u^{j}}$ on $\Omega \times [0,T)$, $\Omega \subset R^{n}$, is called strongly parabolic if for each $\zeta \in R^{n}\ \{0\}$,

\begin{center}
$<A^{ij\alpha}_{\beta}\zeta_{i}\zeta_{j}\eta_{\alpha},\eta_{\beta}> \geq \delta_{0}\norm{\zeta}^{2}\norm{\eta}^{2}$
\end{center}
\end{dfn}

\begin{rmk} Note that in the case $n = 1$, 

\begin{center}
$\frac{\partial}{\partial t}X^{\alpha} = A^{\alpha}_{\beta}(X^{\beta})_{uu}$
\end{center}

this amounts to requiring that

\begin{center}
$<A^{\alpha}_{\beta}\eta_{\alpha},\eta_{\beta}> \geq \delta \norm{\eta}^{2}$
\end{center}
\end{rmk}

In the case of CSF, we have that

\[ A^{\alpha}_{\beta} = \frac{1}{x_{u}^{2} + y_{u}^{2}}
\left( \begin{array}{cc}
y_{u}^{2} & -x_{u}y_{u} \\
-x_{u}y_{u} & x_{u}^{2} \end{array} \right) \]

\[ = \frac{1}{x_{u}^{2} + y_{u}^{2}}
\left( \begin{array}{c}
y_{u} \\
-x_{u} \end{array} \right)
\left( \begin{array}{cc}
y_{u} & -x_{u} \end{array} \right) \]

which is NOT parabolic.

To fix this, we add a time dependent reparametrisation to CSF.  (It will turn out that the same problem arises with mean curvature flow and Ricci flow and can be fixed in an analogous manner.)

Define $\phi : R/Z \times [0,T) \rightarrow R/Z$ such that $\frac{\partial \phi(u,t)}{\partial t} = V(\phi(u,t),t)$.

Define $\hat{X}(u,t) = X(\phi(u,t),t)$.  By our remark before this will also be a solution of CSF. Now

\begin{align}
\frac{\partial}{\partial t}\hat{X}(u,t) &= \frac{\partial}{\partial t}X(\phi(u,t),t) + X_{u}V(\phi(u,t),t) \nonumber \\
&= -\kappa_{\hat{X}}N_{\hat{X}}(u,t) + X_{u}V(\phi(u,t),t)
\end{align}

Choose

\[ V = \frac{1}{(x_{u}^{2} + y_{u}^{2})}
\left( \begin{array}{cc}
x_{u} & y_{u} \end{array} \right)
\left( \begin{array}{c}
x_{uu} \\
y_{uu} \end{array} \right) \]

Then

\[ \frac{\partial}{\partial t}\left( \begin{array}{c}
\hat{x} \\
\hat{y} \end{array} \right) = \frac{\partial}{\partial t}\hat{X} = \frac{1}{\hat{x}_{u}^{2} + \hat{y}_{u}^{2}}\left( \begin{array}{c}
\hat{x}_{uu} \\
\hat{y}_{uu} \end{array} \right) \]

which is strongly parabolic.

\emph{Short Time Existence Result}. For any smooth $X_{0}$, there exists a unique smooth solution of CSF $X : R/Z \times [0,T) \rightarrow R^{2}$ with $X(u,0) = X_{0}(u)$, for some $T > 0$.

\emph{Proof}.  This follows quickly from the reparametrisation above, which means that existence (and uniqueness) follow from the analogous results for the standard heat equation.

\emph{Global Existence Theorem}. For any smooth immersion $X_{0} : R/Z \rightarrow R^{2}$ there exists a unique smooth solution $X : R/Z \times [0,T) \rightarrow R^{2}$ satisfying CSF with $X(u,0) = X_{0}(u)$ on a maximal time interval $T < \infty$, and $sup\{\norm{\kappa(u,t)}, u \in R/Z\} \rightarrow \infty$ as $t \rightarrow T$.

\subsection{Evolution of Curvature}

We deduce the equation for evolution of the curvature (which, by a remark earlier, completely determines the curve) for a solution of CSF.  This is analogous to equivalent results for the Ricci flow, where one can deduce for 3-manifolds the evolution equation for the full Riemann tensor (which completely determines the manifold) for a solution of the Ricci flow.

Starting from

\begin{center}
$\frac{\partial}{\partial t}X = -\kappa N$
\end{center}

we note that

\begin{align}
\frac{\partial}{\partial t}X_{u} &= \frac{\partial}{\partial u}(-\kappa N) = \norm{X_{u}}\frac{\partial}{\partial s}(-\kappa N) \nonumber \\
&= (-\frac{\partial \kappa}{\partial s}N - \kappa^{2}T)\norm{X_{u}}
\end{align}

Now

\begin{align}
\frac{\partial}{\partial t}\norm{X_{u}} &= \frac{\partial}{\partial t}\sqrt{<X_{u},X_{u}>} \nonumber \\
&= \frac{1}{\norm{X_{u}}}<X_{u},\frac{\partial}{\partial t}X_{u}> \nonumber \\
&= <T, \frac{\partial}{\partial t}X_{u}> \nonumber \\
&= -\kappa^{2}\norm{X_{u}}
\end{align}

immediately from our previous computation.

Also note that

\begin{center}
$\frac{\partial}{\partial t}T = \frac{\partial}{\partial t}(\frac{X_{u}}{\norm{X_{u}}}) = -\frac{\partial \kappa}{\partial s}N$
\end{center}

It then quickly follows from $<T,N> = 0$ that

\begin{center}
$\frac{\partial}{\partial t}N = \frac{\partial \kappa}{\partial s}T$
\end{center}

Now, for any function $f$, we have that

\begin{align}
\frac{\partial}{\partial t}\frac{\partial}{\partial s}f &= \frac{\partial}{\partial t}(\norm{X_{u}}^{-1}\partial_{u}f) \nonumber \\
&= \frac{1}{\norm{X_{u}}}\partial_{u}\partial_{t}f - \frac{1}{\norm{X_{u}}^{2}}(-\kappa^{2}\norm{X_{u}})\partial_{u}f \nonumber \\
&= \partial_{s}\partial_{t}f + \kappa^{2}\partial_{s}f
\end{align}

In other words

\begin{center}
$[ \partial_{t}, \partial_{s} ] = \kappa^{2}\partial_{s}$
\end{center}

Then

\begin{align}
\frac{\partial}{\partial t}(\frac{\partial}{\partial s}T) &= \frac{\partial}{\partial s}(\frac{\partial}{\partial t}T) + \kappa^{2}\frac{\partial}{\partial s}T \nonumber \\
&= \partial_{s}(-\kappa_{s}N) - \kappa^{3}N \nonumber \\
&= -\kappa_{ss}N - \kappa_{s}\kappa T - \kappa^{3}N
\end{align}

But note that also

\begin{center}
$\frac{\partial}{\partial t}(\frac{\partial}{\partial s}T) = \frac{\partial}{\partial t}(-\kappa N) = -\kappa_{t}N - \kappa \kappa_{s}T$
\end{center}

So the terms $\kappa_{s}\kappa T$ cancel and we are left with the evolution equation for the curvature for a solution of CSF:

\begin{equation}
\label{curvatureCSF}
\frac{\partial}{\partial t}\kappa = \kappa_{ss} + \kappa^{3}
\end{equation}

\subsection{Bounds on Curvature for CSF}

In the subject of geometric evolution equations, it is often important to be able to control geometric quantities, particularly if we want to understand the nature of limiting solutions.  For the CSF, we have the nice result that if we have a global bound on the curvature $\kappa$, we get bounds on

\begin{center}
$\norm{\frac{\partial^{k}}{\partial s^{k}}\kappa}^{2}$
\end{center}

for all $k$.  I will prove this only for $k = 1$, but the core ideas remain the same.  The key is to use the evolution equation for the curvature we have just derived.

So suppose $\kappa^{2} < \frac{C}{2}$. I claim that we get a bound on $\kappa_{s}$.  

\emph{Proof}.  First note that

\begin{align}
\frac{\partial}{\partial t}\kappa_{s} &= (\kappa_{ss} + \kappa^{3})_{s} + \kappa^{2}\kappa_{s} \nonumber \\
&= \kappa_{sss} + 4\kappa^{2}\kappa_{s}
\end{align}

and hence that

\begin{center}
$\frac{\partial}{\partial t}\kappa_{s}^{2} = (\kappa_{s})^{2}_{ss} - 2\kappa_{ss}^{2} + 8\kappa^{2}\kappa_{s}^{2}$
\end{center}

Now if we denote $g = \frac{\kappa_{s}^{2}}{C^{2} - \kappa^{2}}$, we get that

\begin{align}
\frac{\partial}{\partial t}g &\leq g_{ss} + ag_{s} + 8\kappa^{2}g - 2g^{2} + \frac{2\kappa^{4}\kappa_{s}^{2}}{(C^{2} - \kappa^{2})^{2}} \text{ for some $a$}, \nonumber \\
&\leq g_{ss} + ag_{s} + Cg - 2g^{2} \text{ for some new constant $C$.}
\end{align}

Then, as before, since initially $g$ is a positive bounded function we get, via the maximum principle, that

$g(u,t) \leq D$ for all $u \in \Omega$, $t \geq 0$, for some positive number $D$ or, in other words, that

\begin{center}
$\kappa_{s}^{2} \leq C^{2}D$
\end{center}

which is what we wanted to show.

Also, we want to show that $X$ remains smooth.  The way to do this is to control derivatives of $X$.  Now

\begin{center}
$\frac{\partial}{\partial t}\norm{X_{u}} = -\kappa^{2}\norm{X_{u}}$
\end{center}

this implies that

\begin{center}
$e^{-C^{2}t}\norm{X_{u}(u,0)} \leq \norm{X_{u}(u,t)} \leq e^{C^{2}t}\norm{X_{u}(u,0)}$
\end{center}

ie $X_{u}$ is bounded.

For higher derivatives, note that

\begin{align}
X_{uu} &= <X_{uu},N>N + <X_{uu},T>T \nonumber \\
&= (\frac{\partial}{\partial u}<X_{u},N> - <X_{u},\kappa X_{u}>)N + <X_{uu},T>T \nonumber \\
&= -\kappa \norm{X_{u}}^{2}N + < X_{uu},T>T
\end{align}

So we need to control $<X_{uu},T>$.

More generally, we will find that we will need to control for $k^{th}$ order terms a term of the form

\begin{center}
$\frac{\partial}{\partial t}<\frac{\partial^{k}X}{\partial u^{k}},T> = -\kappa^{2}<\frac{\partial^{k}X}{\partial u^{k}},T> +$ (lower order in $<\frac{\partial^{i}X}{\partial u^{i}}>)(\kappa,\kappa_{s},...)$
\end{center}

So, once again by the maximum principle and induction, $\norm{<\frac{\partial^{k}X}{\partial u^{k}},T>}$ remains bounded if it is bounded initially.  It turns out that this is all we need to conclude:

\begin{thm} If $X$ exists on $R/Z \times [0,T)$, then $\norm{\frac{\partial^{k}X}{\partial u^{k}}}$ is bounded on $R/Z \times [0,T)$ and $\norm{\frac{\partial^{k}}{\partial u^{k}}\frac{\partial^{l}}{\partial t^{l}}X}$ is bounded. \end{thm}

The Arzela Ascoli Theorem now tells us that $X_{t}(.) = X(.,t) \rightarrow X_{T}(.)$, i.e. the flow converges smoothly to a sensible limit at $t = T$ provided we have bounds on the curvature.  We may then apply short time existence starting from $X_{T}$ to extend the solution further.

To digress briefly, recall the statement of the Arzela-Ascoli theorem:

\begin{thm} (Ascoli-Arzela). Let X be a compact metric space, Y a metric space. Then a subset F of C(X,Y) is compact in the compact-open topology if and only if it is equicontinuous, pointwise relatively compact and closed. \end{thm}

Note that this proves the Global Existence Theorem which I stated earlier, since this process can only stop when the curvature becomes unbounded.

\subsection{An Isoperimetric Estimate}

Consider a closed curve $X_{t}$ with no self intersections in $R^{2}$.  Let $p,q$ be two points on the curve, $d$ be the distance separating them in the plane, $l$ be the distance separating them along the curve, and $L$ the total length of the curve.  In particular, note that if the curve is precisely a circle of radius $r$, and $p$ and $q$ are separated by an angle $\theta$, that $L = 2\pi r$, $l = r\theta$, $\frac{d}{2} = rsin(\frac{\theta}{2})$ and $\theta = \frac{2\pi l}{L}$.  Then we can see that

\begin{center}
$\pi = \frac{L}{d}sin(\frac{\pi l}{2})$
\end{center}

So define $Z(p,q,t) = \frac{L}{d}sin(\frac{\pi l}{2})$ for \emph{any} curve $X_{t}$.

Then $sup_{(p,q)}Z \geq \pi$, with equality realised iff $X_{t}$ is a circle.

\emph{Claim}.  $sup_{(p,q)}Z(p,q,t)$ is nonincreasing under CSF.

(In fact, one can show that $sup_{(p,q)}Z(p,q,t)$ is decreasing under CSF if it is initially $> \pi$.  This of course shows that for curves with no initial self intersections the limiting solution of renormalised CSF is actually a circle.)

\section{Some Geometric Background on Hypersurfaces}

\subsection{Introduction}

A hypersurface $M^{n} \subset R^{n+1}$ is locally the level set of a function - i.e. $\forall x \in M$ there is a smooth $g : U \rightarrow R$, $\norm{Dg} \neq 0$, where $U \subset R^{n+1}$ is an open set st $x \in U$.  Then

\begin{center}
$M \cap U = g^{-1}(0)$
\end{center}

In the compact case, it is possible to choose a global chart, ie. a $g : R^{n+1} \rightarrow R$ such that $g^{-1}(0) = M$, and $x \in M$ implies that $\norm{Dg(x)} \neq 0$.

We have the Implicit Function Theorem which states that there exist local parametrisations, ie, given $x \in M$ there exists an open neighbourhood $U$ and an $X : V \subset R^{n} \rightarrow U$ which is smooth and an embedding.  In other words, $X$ is one to one and $DX(z) : R^{n} \rightarrow R^{n+1}$ is injective.

We can interpret the tangent space to a point $x$ in a hypersurface $M$ as

\begin{center}
$T_{x}M = \{ v \in R^{n+1} | D_{x}g(v) = 0\} = ker Dg(x)$
\end{center}

Or, if $X$ is a local parametrisation such that $x = X(p)$ then

\begin{center}
$T_{x}M = im(DX(p))$
\end{center}

In other words, we get a natural basis for $T_{x}M$ as

\begin{center}
$\{ \frac{\partial X}{\partial p^{1}},...,\frac{\partial X}{\partial p^{n}} \}$
\end{center}

The unit normal to a hypersurface is

\begin{center}
$N = \frac{\nabla g}{\norm{\nabla g}} \perp TM$
\end{center}

The metric induced by the embedding of the hypersurface in $R^{n+1}$ is

\begin{center}
$g_{ij} = \frac{\partial X}{\partial p^{i}} \cdot \frac{\partial X}{\partial p^{j}}$
\end{center}

\subsection{The Second Fundamental Form}

We define the second fundamental form for a hypersurface $M$ as

\begin{center}
$h_{ij} = -\frac{\partial^{2}X}{\partial p^{i}\partial p^{j}} \cdot N$
\end{center}

The second fundamental form plays an analogous role for a hypersurface that the curvature plays for a curve; it completely specifies how the surface $M^{n}$ embeds in $R^{n+1}$.

For example, take $M^{n} = S^{n} \subset R^{n+1}$.  Then $g = \norm{X}^{2} - 1$, and $N = X$.  This implies that $h_{ij} = g_{ij}$.

The Implicit Function Theorem implies that if $X$ and $Y$ are local parametrisations of $M$ near $x \in M$  then $Y = X \circ \phi$ where $\phi : A \subset R^{n} \rightarrow B \subset R^{n}$ is a diffeomorphism.  In particular,

\begin{center}
$\frac{\partial Y}{\partial q^{i}} = \frac{\partial}{\partial q^{i}}(X(\phi(q))) = \frac{\partial X}{\partial p^{j}}\frac{\partial \phi^{j}}{\partial q^{i}}$
\end{center}

Define

\begin{center}
$\Lambda_{i}^{j} = \frac{\partial \phi^{j}}{\partial q^{i}}$
\end{center}

as the Jacobian of $\phi$.

Then observe that

\begin{center}
$g_{ij}^{Y} = \Lambda_{i}^{k}\Lambda_{j}^{l}g_{kl}^{X}$
\end{center}

so the metric transforms like a tensor as we would expect.

More interestingly, observe that

\begin{align}
\frac{\partial^{2}Y}{\partial q^{i} \partial q^{j}} &= \frac{\partial}{\partial q^{i}}(\frac{\partial X}{\partial p^{k}}\frac{\partial \phi^{k}}{\partial q^{j}}) \nonumber \\
&= \frac{\partial^{2}X}{\partial p^{k}\partial p^{l}}\frac{\partial \phi^{k}}{\partial q^{j}}\frac{\partial \phi^{l}}{\partial q^{i}} + \frac{\partial X}{\partial p^{k}}\frac{\partial^{2}\phi^{k}}{\partial q^{j}\partial q^{i}}
\end{align}

Now by definition

\begin{align}
h_{ij}^{Y} &= -\frac{\partial^{2}Y}{\partial q^{i}\partial q^{j}} \cdot N \nonumber \\
&= -\frac{\partial^{2}X}{\partial p^{k}\partial p^{l}} \cdot N \Lambda_{i}^{k}\Lambda_{j}^{l} \nonumber \\
&= h_{kl}^{X}\Lambda_{i}^{k}\Lambda_{j}^{l}
\end{align}

proving that the second fundamental form also transforms like a tensor!

The second fundamental form has a natural interpretation.  Suppose $v \in T_{x}M$.  Then $h(v,v)$ is the curvature of the curve $M \cap \Pi$ in the plane $\Pi$, where $\Pi = span\{v, N\}$.

\subsection{Differentiation on a Hypersurface}

Let $V$ be a tangent vector field to our hypersurface $M$, $V = V^{i}\frac{\partial X}{\partial p^{i}}$.  Then, given a smooth function $f$ we can differentiate $f$ in direction $V$ to get

\begin{center}
$Vf = V^{i}\frac{\partial f}{\partial p^{i}}$
\end{center}

Suppose now $V,W$ are tangent vector fields to $M$.  We compute

\begin{align}
V(WX) &= D_{V}(W^{i}\frac{\partial X}{\partial p^{i}}) \nonumber \\
&= V^{j}\frac{\partial}{\partial p^{j}}(W^{i}\frac{\partial X}{\partial p^{i}}) \nonumber \\
&= V^{j}\frac{\partial W^{i}}{\partial p^{j}}\frac{\partial X}{\partial p^{i}} + V^{j}W^{i}\frac{\partial^{2}X}{\partial p^{j}\partial p^{i}} \nonumber \\
&= (V^{j}\frac{\partial W^{i}}{\partial p^{j}}\frac{\partial X}{\partial p^{i}} + V^{j}W^{i}(\frac{\partial^{2}X}{\partial p^{i}\partial p^{j}})^{tang}) + V^{j}W^{i}(\frac{\partial^{2}X}{\partial p^{i}\partial p^{j}} \cdot N)N
\end{align}

This motivates us to define

\begin{center}
$V(WX) = (\nabla_{V}W)(X) - h(V,W)N$
\end{center}

$\nabla$ is nothing other than the covariant derivative of $M$.

\subsection{The Structure Equation for Hypersurfaces}

Observe that trivially for any $f$ that

\begin{center}
$0 = \frac{\partial^{2}f}{\partial x^{i}\partial x^{j}} - \frac{\partial^{2}f}{\partial x^{j}\partial x^{i}}$
\end{center}

Take $f = X$.  Then

\begin{center}
$0 = (\nabla_{\partial_{i}}\partial_{j} - \nabla_{\partial_{j}}\partial_{i})X - (h_{ij} - h_{ji})N$
\end{center}

Differentiating one more time, we observe that

\begin{align}
0 &= \partial_{i}\partial_{j}(\partial_{k}X) - \partial_{j}\partial_{i}(\partial_{k}X) \nonumber \\
&= \partial_{i}((\nabla_{\partial_{j}}\partial_{k})X - h_{jk}N) - ( i \leftrightarrow j) \nonumber \\
&= (\nabla_{\partial_{i}}(\nabla_{\partial_{j}}\partial_{k}))X - h(\partial_{i},\nabla_{\partial_{j}}\partial_{k})N - (\nabla_{i}h_{jk} + h(\nabla_{\partial_{i}}\partial_{j},\partial_{k}) \nonumber \\
&+ h(\partial_{j}, \nabla_{\partial_{i}}\partial_{k}))N - h_{jk}h_{ip}g^{pq}\partial_{q}X - (i \leftrightarrow j) \nonumber \\
&= (\nabla_{\partial_{i}}(\nabla_{\partial_{j}}\partial_{k}) - \nabla_{\partial_{j}}(\nabla_{\partial_{i}}\partial_{k}) \nonumber \\
&- (h_{jk}h_{ip} - h_{ik}h_{jp})g^{pq}\partial_{q})X + (\nabla_{\partial_{j}}h_{ik} - \nabla_{\partial_{i}}h_{jk})N
\end{align}

(The notation $(i \leftrightarrow j)$ means to repeat the previous expression in the line of calculation but under a reversal of the indices $i$ and $j$).

From examining the tangential and normal components of this we deduce the

\emph{Codazzi Equation}.

\begin{equation}
\nabla_{\partial_{i}}h_{jk} = \nabla_{\partial_{j}}h_{ik}
\end{equation}

and the

\emph{Gauss Equation}.

\begin{equation}
[\nabla_{\partial_{i}},\nabla_{\partial_{j}}]\partial_{k} = (h_{jk}h_{ip} - h_{ik}h_{jp})g^{pq}\partial_{q}
\end{equation}

\subsection{Principal Curvatures}

The eigenvalues of $h$ with respect to $g$ are the principal curvatures.  So, if we choose a basis $e_{1},...,e_{n}$ for $T_{x}M$ such that

\begin{center}
$g(e_{i},e_{j}) = \delta_{ij}$
\end{center}

then

\begin{center}
$h(e_{i},e_{j}) = \kappa_{i}\delta_{ij}$
\end{center}

The Mean Curvature is defined to be

\begin{center}
$H = \sum_{i = 1}^{n}\kappa_{i} = g^{ij}h_{ij}$
\end{center}

The Gauss Curvature is defined to be

\begin{center}
$K = \Pi_{i = 1}^{n}\kappa_{i} = \frac{det(h)}{det(g)}$
\end{center}

\section{Mean Curvature Flow (MCF)}

\subsection{Introduction; MCF as a steepest descent flow of volume}

MCF arises from a steepest descent flow of the volume of a hypersurface.  In this way it is a natural generalisation of the curve shortening flow, which I showed previously arises as the steepest descent flow of the length of an immersed curve in $R^{2}$.

The $n$-dimensional volume of $M$ is given by

\begin{center}
$\vert M \vert = \int \sqrt{det(g)}$
\end{center}

Consider an arbitrary variation

\begin{center}
$X : M_{0} \times [0,T) \rightarrow R^{n+1}$
\end{center}

where $\frac{\partial X}{\partial t} = V = \omega X + fN$.

Recall $g_{ij} = \frac{\partial X}{\partial p^{i}} \cdot \frac{\partial X}{\partial p^{j}}$.

Now

\begin{align}
\frac{\partial}{\partial t}\frac{\partial X}{\partial p^{i}} &= \frac{\partial}{\partial p^{i}}(\omega X + fN) \nonumber \\
&= (\nabla_{\partial_{i}}\omega)X - h(\partial_{i},\omega)N + \frac{\partial f}{\partial p^{i}}N + fh_{i}^{p}\partial_{p}X
\end{align}

where $h_{i}^{p}$ is defined to be $h_{ik}g^{kp}$.

So

\begin{align}
\frac{\partial}{\partial t}g_{ij} &= \frac{\partial}{\partial t}\frac{\partial X}{\partial p^{i}} \cdot \frac{\partial X}{\partial p^{j}} + (i \leftrightarrow j) \nonumber \\
&= g(\nabla_{\partial_{i}}\omega,\partial_{j}) + fh_{i}^{p}g_{pj} + (i \leftrightarrow j) \nonumber \\
&= g(\nabla_{\partial_{i}}\omega,\partial_{j}) + g(\partial_{i},\nabla_{\partial_{j}}\omega) + 2fh_{ij}
\end{align}

Hence

\begin{align}
\frac{\partial}{\partial t}det(g) &= det(g)g^{ij}\frac{\partial}{\partial t}g_{ij} \nonumber \\
&= 2det(g)(\nabla_{\partial_{i}}\omega^{i}) + 2det(g)fH
\end{align}

So

\begin{equation}
\frac{\partial}{\partial t}\sqrt{det(g)} = [(\nabla_{\partial_{i}}\omega^{i}) + fH]\sqrt(det(g))
\end{equation}

So we may finally conclude that

\begin{align}
\frac{d}{dt}\vert M \vert &= \int_{M}(div(\omega) + fH)d\mu(g) \nonumber \\
&= \int_{M}fHd\mu(g) \text{ by the divergence theorem}
\end{align}

It is evident then that steepest descent flow is given by the variation

\begin{center}
$\frac{\partial X}{\partial t} = - kHN$
\end{center}

for some positive nonzero constant $k$.  So choose $k = 1$.  This is the equation of Mean Curvature Flow.

Alternatively, we can think of MCF as a natural heat equation in the sense that if we take the variation

\begin{center}
$\frac{\partial X}{\partial t} = \Delta X$
\end{center}

where $\Delta f = g^{ij}\nabla_{i}\nabla_{j}f$ is the Laplacian of $f$, then

\begin{center}
$\Delta X = g^{ij}(-h_{ij}N) = - HN$
\end{center}

which gives us MCF again.

More precisely, given a compact hypersurface $X_{0}$, we want to find $X : M_{0} \times [0,T) \rightarrow R^{n+1}$ such that

\begin{center}
$\frac{\partial X}{\partial t} = -HN$
\end{center}

and

\begin{center}
$X(z,0) = z$
\end{center}

for all $z \in M_{0}$.

The solution to this initial value problem if it exists is the MCF of $M_{0}$.

\subsection{Existence Results}

\emph{Short Time Existence}.

We can establish short time existence of the MCF in two different ways-

(i) Write $M_{t}$ as a graph over $M_{0}$, i.e. $X(z,t) = z + \rho(z,t)N(z,0)$.  It is clear that we will have short time existence since if we flow along a normal vector field the graph will have no self intersections for small times $t$.

or

(ii) Write

\begin{center}
$\frac{\partial X}{\partial t} = g^{ij}(\frac{\partial^{2}X}{\partial p^{i}\partial p^{j}} - (\nabla_{\partial_{i}}\partial_{j})X)$
\end{center}

(The bracketed term is the normal component of $\frac{\partial^{2}X}{\partial p^{i}\partial p^{j}}$)

Let $\dot{\nabla}$ be the covariant derivative at $t = 0$.

Then

\begin{center}
$\frac{\partial X}{\partial t} = g^{ij}(\dot{\nabla}_{i}\dot{\nabla}_{j}X + (\dot{\nabla}_{i}\partial_{j}  - \nabla_{i}\partial_{j})X)$
\end{center}

Note $V(z,t)X = (\dot{\nabla}_{i}\partial_{j}  - \nabla_{i}\partial_{j})X$ is defined independent of local parametrisation.  Hence we have short term existence.

\begin{thm} (Global Existence Theorem).

If $T$ is the maximal time of existence, then $sup_{M_{t}}\norm{h}^{2} \rightarrow \infty$ as $t \rightarrow T$.  Here $\norm{h}^{2} = g^{ij}g^{kl}h_{ik}h_{jl} = \sum_{i}\kappa_{i}^{2}$.
\end{thm}

\subsection{Induced Evolution of Various Quantities along MCF}

\emph{Evolution of Metric}.

\begin{align}
\frac{\partial}{\partial t}g_{ij} &= \frac{\partial}{\partial t}(\frac{\partial X}{\partial p^{i}} \cdot \frac{\partial X}{\partial p^{j}}) \nonumber \\
&= -2Hh_{ij}
\end{align}

\emph{Evolution of Normal}.

Define

\begin{center}
$\frac{\partial}{\partial t}N = A^{i}\partial_{i}X$
\end{center}

We compute the $A^{i}$.

Now certainly

\begin{center}
$0 = N \cdot T = N \cdot \frac{\partial X}{\partial p^{i}}$
\end{center}

So

\begin{center}
$0 = \frac{\partial N}{\partial t} \cdot \frac{\partial X}{\partial p^{i}} + N \cdot \frac{\partial}{\partial p^{i}}(-HN)$
\end{center}

The first term is nothing other than $A^{i}g_{ji}$.  The second term is clearly $-\frac{\partial H}{\partial p^{i}}$.  So we conclude

\begin{center}
$A^{j}g_{ji} = \frac{\partial H}{\partial p^{i}}$
\end{center}

From which we derive finally the evolution equation for the normal to the hypersurface:

\begin{align}
\frac{\partial N}{\partial t} &= \frac{\partial H}{\partial p^{i}}g^{ij}\partial_{j}X \nonumber \\
&= (\nabla H)X 
\end{align}

\emph{Evolution of $h$ and also the mean curvature $H$}.

Note first that

\begin{align}
\partial_{t}\partial_{i}N &= \partial_{t}(h_{i}^{k}\partial_{k}X) \nonumber \\
&= (\frac{\partial}{\partial t}h_{i}^{k})\partial_{k}X + h_{i}^{k}\partial_{k}(-HN) \nonumber \\
&= \frac{\partial}{\partial t}h_{i}^{k})\partial_{k}X - Hh_{k}^{p}\partial_{p}X
\end{align}

But also

\begin{align}
\partial_{t}\partial_{i}N &= \partial_{i}(\nabla_{k}Hg^{kl}\partial_{l}X) \nonumber \\
&= (\nabla_{i}\nabla_{k}Hg^{kl}\partial_{l}X) - h(\nabla H, \partial_{i})N
\end{align}

So if we equate tangential components of the above expressions, we get the relation

\begin{equation}
\frac{\partial}{\partial t}h_{i}^{j} = \nabla_{i}\nabla_{k}(H)g^{kj} + Hh_{i}^{k}h_{k}^{j}
\end{equation}

which is the evolution equation for $h$.

Consequently

\begin{center}
$\frac{\partial}{\partial t}H = \Delta H + H\norm{h}^{2}$,
\end{center}

the evolution equation for $H$.

\subsection{The Huisken Rescaling Result}

The proof of the following result is due to Huisken.  It is a typical example of how under a geometric flow the limiting behaviour can be very nice.  The way we understand limiting behaviour, furthermore, is done by rescaling our solution- typified by the final statement in the result which follows.

\emph{The Result}.  If $M_{0}$ is a compact convex hypersurface in $R^{n+1}$, $n \geq 2$, then the solution of MCF remains smooth and convex on a maximal interval $[0,T)$.  Furthermore,

\begin{center}
$X(p,t) \rightarrow x_{0} \in R^{n+1}$ as $t \rightarrow T$
\end{center}

and

\begin{center}
$\frac{X(p,t) - x_{0}}{(2n(T - t))^{1/2}} \rightarrow_{C^{\infty}} X_{T}$
\end{center}

where $X_{T}(M_{0}) = S^{n}$.

\begin{rmk}  Note that if $M_{0} = r_{0}z$, where $z \in S^{n}$, then $X(z,t) = r(t)z$ and $N(z,t) = z$.  So

\begin{center}
$h_{i}^{j}\partial_{j}X = \partial_{i}N = \partial_{i}z = \partial_{i}(\frac{X}{r}) = \frac{1}{r}\partial_{i}X$
\end{center}

Hence $h_{i}^{j} = \frac{1}{r}\delta_{i}^{j}$.  Now $\kappa_{1} = \cdots = \kappa_{n} = \frac{1}{r}$, so $H = \frac{n}{r}$.

Then MCF becomes $\frac{\partial r}{\partial t} = -\frac{n}{r}$, which implies that $r(t) = (r_{0}^{2} - 2nt)^{1/2}$.  So if we set $T = \frac{r_{0}^{2}}{2n}$ then $r(t) = (2n(T - t))^{1/2}$.  Hence it is clear to see that a natural rescaling of our solution should feature the quotient of some factor such as this.
\end{rmk}

\subsection{Simon's Identity and an application}

We derive an identity that will prove useful later on.  Note that  

\begin{align}
\nabla_{i}\nabla_{j}H &= \nabla_{i}\nabla_{j}(g^{kl}h_{kl}) \nonumber \\
&= g^{kl}\nabla_{i}\nabla_{j}h_{kl} \text{ (since $\nabla_{i}g_{jk} = 0$)} \nonumber \\
&= g^{kl}\nabla_{i}\nabla_{k}h_{jl} \text{ (by Codazzi)}
\end{align}

This is Simon's Identity.  For an example of its application, we compute

\begin{align}
0 &= (\nabla_{i}\nabla_{k} - \nabla_{k}\nabla_{i})h(\partial_{j},\partial_{l}) \nonumber \\
&= \nabla_{i}(\nabla_{k}h_{jl} + h(\nabla_{k}\partial_{j},\partial_{l}) + h(\partial_{j},\nabla_{k}\partial_{l})) - (i \leftrightarrow k) \nonumber \\
&= \nabla_{i}\nabla_{k}h_{jl} + \nabla h(\nabla_{k}\partial_{j},\partial_{l}) + h(\nabla \partial, \nabla \partial) \nonumber \\
&+ h(\nabla_{i}\nabla_{k}\partial_{j},\partial_{l}) + h(\partial_{j},\nabla_{i}\nabla_{k}\partial_{l}) - (i \leftrightarrow k)
\end{align}

Note that the second and third terms can be made to vanish if we choose coordinates so that $\nabla_{k}\partial_{j}(x) = 0$.

Continuing under this choice of coordinates, we get that

\begin{center}
$0 = [\nabla_{i},\nabla_{k}]h_{jl} + h([\nabla_{i},\nabla_{k}]\partial_{j},\partial_{l}) + h(\partial_{j},[\nabla_{i},\nabla_{k}]\partial_{l})$
\end{center}

Hence

\begin{center}
$[\nabla_{i},\nabla_{k}]h_{jl} = h_{ij}h_{k}^{p}h_{pl} - h_{kj}h_{j}^{p}h_{pl} + h_{il}h_{k}^{p}h_{jp} - h_{kl}h_{i}^{p}h_{jp}$
\end{center}

Which implies after using the Codazzi equation and Simon's Identity that

\begin{center}
$\nabla_{i}\nabla_{j}H = g^{kl}\nabla_{k}\nabla_{l}h_{ij} + h_{ij}\norm{h}^{2} - Hh_{i}^{p}h_{pj}$
\end{center}

From which we finally derive

\begin{align}
\frac{\partial}{\partial t}h_{i}^{j} &= (\nabla_{i}\nabla_{k}H + Hh_{i}^{p}h_{pk})g^{kj} \nonumber \\
&= \Delta h_{i}^{j} + h_{i}^{j}\norm{h}^{2}
\end{align}

Note that this is the same evolution equation for $h$ as we derived before, except that we did not make the assumption that $M$ was a hypersurface, ie this equation is more generally applicable.

\subsection{Monotonicity Formula for the MCF and application to MCF Solitons}

Once again we use the simpler case of MCF to motivate the more technically complex analogues in RF.  It is evidently important to understand the limiting behaviour of geometric flows, and it turns out that particular self-similar limits, known as \emph{soliton solutions}, are very important.  In fact, both for the MCF and the RF, it will turn out that any limit of rescaled flows about any point will be such a solution.  To prove this, we need a \emph{monotonicity formula} for the relevant flow.  In this section I shall introduce such a formula for the MCF, and demonstrate its application to proving that all limits of the MCF are self-similar.  Finally, I will finish with mentioning a partial classification result for these creatures.

So, let $M_{t}$ be a solution of MCF in $R^{n+1}$. Then we have a monotonicity formula for the MCF, due to Huisken:

\begin{center}
$\frac{d}{dt}(\int_{M_{t}}(t_{0} - t)^{-n/2}exp(-\frac{\norm{x - x_{0}}^{2}}{4(t_{0} - t)}d\mu(g_{t})) \leq 0$
\end{center}

for $t_{0} \geq t$.  

\emph{Remarks}.
\begin{itemize}
\item[(1)]  Note that the fundamental solution of the heat equation on $R^{n+1}$ is $\rho(x,t) = ct^{-\frac{n+1}{2}}exp(-\frac{\norm{x}^{2}}{4t})$.
\item[(2)] Furthermore, if $u$ satisfies $\frac{\partial u}{\partial t} = \Delta u$ in $R^{n}$ then
\end{itemize}

\begin{align}
u(x_{0},t_{0}) &= c_{n}\int(t_{0})^{-n/2}exp(-\frac{\norm{x - x_{0}}^{2}}{4t_{0}})u(x,0)dx^{n} \text{(as can be proved using Green's functions)} \nonumber \\
&= c_{n}\int (t_{0} - t)^{-n/2}exp(-\frac{\norm{x - x_{0}}^{2}}{4(t_{0} - t)})u(x,t)dx^{n} \text{for} 0 \leq t \leq t_{0}
\end{align}

Hence 

\begin{align}
\frac{d}{dt}\int (t_{0} - t)^{-n/2}exp(-\frac{\norm{x - x_{0}}^{2}}{4(t_{0} - t)})u(x,t)dx^{n} = 0
\end{align}

Now, consider solitons of MCF, ie $M_{t} = \lambda(t)M_{0}$.  So we have $X(p,t) = \lambda(t)X(\phi(p,t),0)$, which implies $-H \mu \vert_{(p,t)} = \frac{\lambda '}{\lambda}X \vert_{(p,t)} + \lambda \frac{\partial X}{\partial x^{j}} \vert_{(\phi(p,t),0)}\frac{\partial \phi^{j}}{\partial t}$.

Taking the inner product with $\mu$, we get

\begin{center}
$H = - \frac{\lambda '}{\lambda}<X,\mu>$
\end{center}

"The MCF Soliton Equation".

We may scale $\frac{-\lambda '}{\lambda} = \frac{1}{2(T - t)}$ (Then $H_{t} = \frac{1}{\lambda}H_{0}$), so $\frac{\partial \lambda}{\partial t} = -\frac{c}{\lambda}$.

Under normal variations

\begin{center}
$\frac{\partial X}{\partial t} = f\mu$
\end{center}

So $\frac{\partial}{\partial t}g_{ij} = 2fh_{ij}$, and hence $\frac{\partial}{\partial t}(d\mu) = fH(d\mu)$.

Therefore,

\begin{center}
$\frac{\partial}{\partial t}\norm{X}^{2} = 2<X, \frac{\partial X}{\partial t}> = 2f<X, \mu>$
\end{center}

so

\begin{center}
$\frac{\partial}{\partial t}(exp(-\frac{c}{2}\norm{X}^{2})d\mu) = f(H - c<x,\mu>)exp(-\frac{c}{2}\norm{X}^{2})d\mu$
\end{center}

Choose $c = \frac{1}{2(T - t)}$.

Then we have from the relation immediately above, that a MCF soliton is a critical point of $\int_{M_{t}}exp(-\frac{\norm{X_{t}}^{2}}{4(T - t)})d\mu$, where $M_{t} = \lambda(t)M_{0}$.  We immediately see the connection with the monotonicity formula, because, after reparametrisation, we see that a MCF soliton is a critical point of $\int_{M_{0}}(t_{0} - t)^{-n/2}exp( - \frac{\norm{X_{0}}^{2}}{4(T - t)})d\mu$.

We recall the monotonicity formula for the MCF, as stated above:

$\frac{d}{dt}\int_{M}(t_{0} - t)^{-n/2}exp(- \frac{\norm{x - x_{0}}^{2}}{4(t_{0} - t)})d\mu \leq 0$.

Now I claim that we acheive equality if and only if $M_{t}$ is a MCF soliton.  Certainly if $M_{t}$ is a MCF soliton, this is the case.  So it remains to prove the converse.  In particular, I will in the process prove the monotonicity formula.

(To prove that we can always choose an appropriate $c$, we compute

\begin{center}
$-\frac{1}{\lambda}H\mu \vert_{0} = -H\mu \vert_{M_{t}} = \lambda ' X_{0} = \frac{\lambda '}{c}H\mu \vert_{0} +$ something tangential
\end{center}

This implies that $\lambda \lambda ' \vert_{t} = \lambda \lambda '_{0} = - c$ and so $\lambda^{2}(t) = 1 - 2ct = 2c(T - t)$, from which we intuit $\frac{\lambda '}{\lambda} = \frac{1}{2}\frac{(\lambda^{2})'}{\lambda^{2}} = \frac{1}{2}log(2c(T - t))' = \frac{1}{2(T - t)}$ as required.)

To prove the converse, we adopt the notation $\rho : R^{n+1} \times [0,\infty) \rightarrow R$ for the map

\begin{center}
$\rho(x,t) = (t_{0} - t)^{-\frac{n+1}{2}}exp(- \frac{\norm{x - x_{0}}^{2}}{4(t_{0} - t)})$
\end{center}

Note that $\frac{\partial}{\partial t} = - \Delta^{R^{n+1}} \rho$.

We compute

\begin{center}
$\frac{d}{dt}\int (t_{0} - t)^{1/2}\rho(X(p,t),t)d\mu$

$= - \int \frac{1}{2}(t_{0} - t)^{-1/2}\rho d\mu - \int (t_{0} - t)^{1/2}\rho H^{2}d\mu + \int (t_{0} - t)^{1/2}( - \Delta^{R^{n+1}}\rho + D \rho(-H \mu))d\mu$
\end{center}

We can write (in orthonormal coordinates with $e_{n+1} = \mu(x)$, $e_{1}, \cdots, e_{n}$ a basis for $T_{X}M$)

\begin{center}
$\Delta^{R^{n+1}}\rho = D^{2}\rho(\mu,\mu) + \frac{\partial^{2}\rho}{(\partial x^{i})^{2}}$
\end{center}

Also,

\begin{align}
\nabla \nabla \rho(e_{i},e_{i}) &= \frac{d^{2}}{ds^{2}}\rho(\gamma(s)) \nonumber \\
&= \frac{d}{ds}(D\rho(\gamma ')) \nonumber \\
&= D\rho(\gamma '') + D^{2}\rho(\gamma ', \gamma ') 
\end{align}

where $\gamma$ is a geodesic in $M$, $\gamma(0) = X$ and $\gamma '(0) = e_{i}$

Now $\gamma '' = -\kappa N = -h(\gamma ', \gamma ')\mu$ so

\begin{center}
$\nabla_{i}\nabla_{i}\rho = D^{2}\rho(e_{i},e_{i}) - h(e_{i},e_{i})D_{\mu}\rho$
\end{center}

and hence

\begin{center}
$\Delta^{R^{n+1}}\rho = D^{2}\rho(\mu,\mu) + \Delta^{M}\rho + HD_{\mu}\rho$
\end{center}

So

\begin{align}
\frac{d}{dt}\int (t_{0} - t)^{1/2}\rho d\mu &= - \int (t_{0} - t)^{1/2}[\frac{\rho}{2(t_{0} - t)} + D_{\mu}D_{\mu}\rho] \nonumber \\
&- \int (t_{0} - t)^{1/2}\rho [H^{2} - \frac{2HD_{\mu}\rho}{\rho} + \norm{\frac{D_{\mu}\rho}{\rho}}^{2}] + \int (t_{0} - t)^{1/2}\frac{(D_{\mu}\rho)^{2}}{\rho} \nonumber \\
&= -\int (t_{0} - t)^{1/2}\rho \norm{H - \frac{D_{\mu}\rho}{\rho}}^{2}d\mu \nonumber \\
&= -\int (t_{0} - t)^{1/2}[D_{\mu}D_{\mu}\rho - \frac{(D_{\mu}\rho)^{2}}{\rho} + \frac{\rho}{2(t_{0} - t)}]
\end{align}

\emph{Claim}. (1) For $\rho$ as above,

\begin{center}
$D_{i}D_{j}\rho - \frac{D_{i}\rho D_{j}\rho}{\rho} + \frac{\rho}{2(t_{0} - t)}\delta_{ij} = 0$
\end{center}

(2) For any solution $u$ of $\frac{\partial u}{\partial t} = \Delta u$ on $R^{n}$ with $u \geq 0$,

\begin{center}
$D_{i}D_{j}u - \frac{D_{i}u D_{j}u}{u} + \frac{u\delta_{ij}}{2t} \geq 0$
\end{center}

(This occurs if and only if $D_{i}D_{j}log u + \frac{\delta_{ij}}{2t} \geq 0$).

\emph{Remark}. It follows from this claim that a (smooth) limit of rescaled flows about any point is a shrinking MCF soliton.

\emph{Proof of Claim}. Recall if $\mu$ is large,

\begin{center}
$\frac{\partial}{\partial t}\mu = \Delta \mu + \norm{\nabla \mu}^{2}$
\end{center}

Hence

\begin{center}
$\frac{\partial}{\partial t}D_{i}\mu = \Delta D_{i}\mu + 2D_{k}\mu D_{k}D_{i}\mu$
\end{center}

and

\begin{center}
$\frac{\partial}{\partial t}D_{i}D_{j}\mu = \Delta D_{i}D_{j}\mu + 2D_{k}\mu D_{k}(D_{i}D_{j}\mu) + 2D_{j}D_{k}\mu D_{k}D_{i}\mu$
\end{center}

One can then apply the maximum principle to show that $tD_{i}D_{j}\mu + \delta_{ij} \geq 0$, since this is certainly true when $t = 0$.  This proves the claim, and the monotonicity formula, since we have that

\begin{center}
$\frac{d}{dt}\int_{M_{t}}(t_{0} - t)^{1/2}\rho d\mu \leq 0$
\end{center}

by the claim.  Furthermore, equality is only achieved if $H = \frac{D_{\mu}\rho}{\rho} = D_{\mu}(-\frac{\norm{X - x_{0}}^{2}}{4(t_{0} - t)} + c(t)) = -\frac{ <X - x_{0},\mu>}{2(t_{0} - t)}$- but this is just the MCF soliton equation.  In other words, this implies that a smooth limit of rescaled flows about any point is a shrinking MCF soliton, as required.

We have the following theorem, due to Huisken, which goes some way towards 
classifying these beasts:

\emph{Theorem}. A complete, shrinking soliton of MCF with $H > 0$ is either
\begin{itemize}
\item[(i)] A shrinking sphere,
\item[(ii)] $S^{k} \times R^{n-k}$, 
\item[(iii)] or (A shrinking self similar solution of CSF) $\times R^{n-1}$.
\end{itemize}

\section{Proof of the Huisken rescaling result of MCF}

\subsection{Preliminaries}

We set out to prove the result that I stated before, namely that compactness and convexity are preserved under MCF.  It is clear that compactness is preserved, since we are following a steepest descent flow of the $n$-volume and the $n$-volume is initially finite by hypothesis.  To show that under MCF the surface flows to a point will require invoking a corresponding avoidance principle as for CSF, which is again very easy.

So the hard part is demonstrating that convexity is preserved if the hypersurface is convex initially. Let me define convexity for a hypersurface more precisely:

$M_{0}$ is convex if $h_{ij} \geq \epsilon g_{ij}$, for some $\epsilon > 0$, where $h_{ij}$ is the second fundamental form of $M_{0}$ and $g_{ij}$ is its metric.

So this reduces the problem of proving that convexity is preserved to demonstrating that under MCF,

\begin{center}
$Q_{ij} = h_{ij} - \epsilon g_{ij} \geq 0$
\end{center}

is preserved.  To show this I shall invoke a Maximum Principle for $Q_{ij}$.  I will state a Maximum Principle for vectors first, and then show how this can be extended to deal with 2-tensors.  The preservation of the inequality will follow as a consequence of this principle.

\subsection{The Maximum Principle for Vectors}

Suppose $\frac{\partial}{\partial}u^{\alpha} = \Delta u^{\alpha} + V^{\alpha}(u)$, $\alpha = 1, \cdots, k$ on $R^{n}$ (with periodic $u_{0}^{\alpha}$).

Let $K$ be a closed convex set in $R^{k}$ such that $u(z,0) \in K$ for all $z \in R^{n}$ and for points $y \in \partial K$, $V(y)$ "points into $K$".  More precisely, $y + sV(y) \in K$ for $s$ small and positive. 

Then we conclude that $U(z,t) \in K$ for all $z \in R^{n}$, all $t \geq 0$.

Alternatively, $V(u)$ "points into $K$" for any $u \in \partial K$ means the solution of the ODE $\frac{d}{dt}U^{\alpha} = V^{\alpha}(U)$ with $U^{\alpha}(0) = z$ has $U^{\alpha}(t) \in K$ for $t \geq 0$.

\emph{Proof}.

\emph{Claim}.  There exists a smooth convex function $f$ on $R^{n} - \partial K$ such that $f \rightarrow 0$ on $\partial K$, and $f \sim d(\cdot,\partial K)$.

\emph{Proof of Claim}.  Given $y \in R^{n} - \partial K$, choose $z$ to be the nearest point to $y$ in $\partial K$.  Then define $f$ to be the distance between them.  Note in particular that $\nabla d(y) = \nabla d(z)$.

\emph{Continuation of Proof}.  We compute

\begin{align}
\frac{\partial}{\partial t}f(u) = \frac{\partial f}{\partial u^{\alpha}}(\Delta u^{\alpha} + V^{\alpha}(u))
\end{align}

\begin{align}
\Delta f &= g^{kl}\nabla_{k}(\frac{\partial f}{\partial u^{\alpha}}\nabla_{l}u^{\alpha}) \nonumber \\
&= \frac{\partial f}{\partial u^{\alpha}}\Delta u^{\alpha} + g^{kl}\frac{\partial^{2}f}{\partial u^{\alpha}\partial u^{\beta}}\nabla_{k}u^{\alpha}\nabla_{l}u^{\beta}
\end{align}

Hence

\begin{align}
\frac{\partial}{\partial t}f &= \Delta f - g^{kl}D^{2}f(\nabla_{k}u,\nabla_{l}u) + Df \cdot V \text{ (Note that $f > 0$).} \nonumber \\
&\leq \Delta f + Df \cdot V \text{ (since $D^{2}f$ is positive as $f$ is convex)}
\end{align}

Now, suppose we look at $u \in R^{n} - \partial K$, such that $u_{0}$ is the closest point, and $w = \nabla f(u_{0})$ with $u = u_{0} + tw$.  Note that $w \cdot V(u_{0}) < 0$.  Then

\begin{center}
$Df \cdot V = \frac{\partial f}{\partial u^{\alpha}}V^{\alpha}$

$= \frac{\partial f}{\partial u^{\alpha}}\vert_{u_{0} + tw}(V^{\alpha}(u_{0})) + \frac{\partial f}{\partial u^{\alpha}}\vert_{u_{0} + tw}(V^{\alpha}(u) - V^{\alpha}(u_{0}))$
\end{center}

Now, the first term is negative by our initial remark that $w \cdot V(u_{0}) < 0$.  Since as $V$ is smooth and for $u$, $u_{0}$ sufficiently close, $\vert V^{\alpha}(u) - V^{\alpha}(u_{0}) \vert \leq C\norm{u - u_{0}} = Cd = Cf$ for some constant $C$.  Note also that $\frac{\partial f}{\partial u^{\alpha}}\vert_{u} = \nabla f(u)$ is also bounded since $K$ is compact and hence the function $\norm{\nabla f}$ will be bounded on the boundary of $K$.  So the second term is bounded, and so we can conclude

\begin{center}
$\frac{\partial f}{\partial t} \leq \Delta f + Cf$
\end{center}

for some constant $C > 0$.

This implies by the standard maximum principle that $f \leq  0$ for $t \geq 0$, since $f \leq 0$ initially.  (Since $u(z,0) \in K$, $f(u(z,0)) \leq 0$).  So $f(u(z,t)) \leq 0$, or $u(z,t) \in K$ for all $t \geq 0$.

\subsection{Extension of Technique to 2-tensors on Hypersurfaces}

\begin{thm} Let $T$ is a symmetric $2$-tensor on $M$ satisfying

\begin{center}
$\frac{\partial}{\partial t}T_{kl} = \Delta T_{kl} + \mathcal{Q}_{kl}(T)$
\end{center}

Suppose $K$ is a $O(n)$ invariant convex set (we need it to be $O(n)$ invariant otherwise under a change of basis or representation for $T$ the following would not be true) in $Sym(n) = \{A \in GL(n) | A^{t} = A\}$.

Suppose further that $A_{ij} = T(e_{i},e_{j}) \in K$ for any orthonormal basis $\{e_{1}, \cdots, e_{n}\}$ for $T_{x}M$, all $x \in M$, at $t = 0$.  In other words, "$T \in K$".

Finally, assume that $\mathcal{A}$ points into $K$ for $A \in \partial K$. Then $T \in K$ for $t \geq 0$. \end{thm}

\begin{proof} Given a local parametrisation of $M$, we get $\{\partial_{1}, \cdots, \partial_{n}\}$ is a smooth basis of $T_{x}M$.  Perform Gram-Schmidt to get $\{ E_{1}, \cdots, E_{n} \}$ orthonormal with respect to $g$.

Then write $A_{ij} = T(E_{i}, E_{j})$.

We have that

\begin{center}
$\frac{\partial}{\partial t}A_{ij} = \Delta A_{ij} + \mathcal{Q}(A)_{ij}$
\end{center}

Then the idea is to work with this tensor.  The remainder of the argument is completely analogous to the previous case, if only a little more complicated.
\end{proof}

\subsection{Application of 2-tensor Extension to Proof of Huisken}

Define

\begin{align}
K &= \{ A | A(e,e) \geq \epsilon \text{ for all $e$ st } \norm{e} = 1 \} \nonumber \\
&= \cap_{e | \norm{e} = 1}\{A | A(e,e) \geq \epsilon \} \end{align}

Then $K$ is the intersection of convex half spaces and hence is convex.

Consider now the ODE

\begin{center}
$V_{ij}(A) = A_{ij}\norm{A}^{2}$
\end{center}

Then if $A \in K$, $A + sV(A) = A(1 + s\norm{A}^{2}) \geq \epsilon(1 + \norm{A}^{2}) > \epsilon$, $s > 0$.  ie, $V$ points into $K$. So the criteria for the 2-tensor maximum principle are met, and we conclude that $A$ remains in $K$ under mean curvature flow.  In particular, if we define $A = h$ and since $M$ is compact, $g(v,v)$ realises its maximum on $M$ and hence $h(v,v) \geq max_{(x \in M, w \in T_{x}M)}(g(w,w))\epsilon = \hat{\epsilon}$.  The relevant ODE is $V_{ij}(h) = h_{ij}\norm{h}^{2}$, which is consistent with the above, and we are done.

\subsection{Preservation of bounded curvature}

As another application of the 2-tensor maximum principle, consider

\begin{center}
$K = \{ A \vert A(v,v) \leq CA(w,w)$ for all $v,w : \norm{v} = \norm{w} = 1 \}$
\end{center}

Then once again, $K$ is convex, a cone, and preserved by the above ODE.

So if $\kappa_{max}(x,0) \leq C\kappa_{min}(x,0)$ $\forall x \in M$, then $\kappa_{max}(x,t) \leq C\kappa_{min}(x,t)$ $\forall x \in M, t \geq 0$.

\emph{Remark}.  Note if $\kappa_{max}(x,t) \leq \kappa_{min}(x,t)$ $\forall x$, ie $h_{ij} = f(x)g_{ij}$, then $f$ is constant on connected components.

\emph{Proof of Remark}. By Codazzi, $\nabla_{k}h_{ij} = \nabla_{i}h_{kj}$.  This implies that $(\nabla_{i}f)g_{kj} = (\nabla_{k}f)g_{ij}$ since $\nabla g = 0$.  Choose $i = j \neq k$, and $g = \delta$.  Then we conclude $\nabla_{k}f = 0$, or $f$ is constant.

\section{Rescaling}

Rescaling is a recurring theme in the subject of geometric evolution equations.  It is often important to understand the limiting behaviour of geometric flows, in order to get a better understanding of the singularities which may develop, and hence being able to justify surgery in certain circumstances.  This becomes particularly important in the Ricci flow, where the idea is to break a three manifold into pieces by performing successive surgeries, until the manifold has been decomposed into "prime" components.  The understanding of how these "prime" components then collapse to points is key to their classification.

\subsection{Preliminaries}

Now, MCF has a scaling symmetry: if $X : M_{0}^{n} \times [0,T) \rightarrow R^{n+1}$ satisfies MCF, then so does

\begin{center}
$X_{(\lambda,\bar{x},\bar{t})}(p,t) = \lambda(X(p,\hat{t}/\lambda^{2}) - \bar{x})$
\end{center}

where $\hat{t} = t - \bar{t}$.

\begin{dfn} (Convergence of hypersurfaces).

A sequence of closed hypersurfaces $M_{k}$ converges to $M_{\infty}$ if there exist

\begin{center}
$g_{k} \in C^{\infty}(R^{n+1})$, $k \in N \cup \{\infty\}$,

$M_{k} = g_{k}^{-1}(0)$, ie the $M_{k}$ are level sets of the $g_{k}$,
\end{center}

such that

$\norm{\nabla g_{k}(x)} \geq 1$ for all $x \in M_{k}$ (all points $x$ are regular points, and hence the $M_{k}$ are well defined submanifolds of $R^{n+1}$),

and

$g_{k} \rightarrow g_{\infty}$ in $C^{\infty}(B_{R})$ for all $R$. \end{dfn}

\subsection{A compactness theorem}

Let $M_{k}$ be any sequence of hypersurfaces

\begin{center}
$\norm{\nabla^{(j)}h}^{2}(x) \leq C(R,j)$
\end{center}

for all $x \in M_{k} \cap B_{R}$, $j \geq 0$, 

and assume we have a bound below on the "tube radius"

\begin{center}
$r_{-}(M_{k},R) = sup\{r \geq 0 \vert Y : M_{k} \cap B_{R} \times (-r,r) \rightarrow R^{n+1}, Y(x,s) = x + sN(x)$ is injective $\}$.
\end{center}

(This rules out, for instance, a sequence of manifolds $M_{k}$ embedded in $R^{n+1}$ with a gap of distance $1/k$ apart- there is a more general theorem, the varifold compactness theorem, that allows such things to occur, but that is a story for another day, or, at least, another section (see section on Geometric measure theory).)

i.e. $r_{-}(M_{k},R) \geq \epsilon > 0$, some $\epsilon$.  So $\exists R > 0$ such that $M_{k} \cap B_{R} \neq \phi$ for all $k$.

Then we conclude that there is a convergent subsequence $\{M_{p_{k}}\}_{k \in N}$.

\subsection{Existence of smooth limit flows}

\begin{dfn}  A limit of rescalings $X_{(\lambda_{k},x_{k},t_{k})}$, where $(x_{k},t_{k}) \rightarrow (\bar{x},\bar{t})$ is called a limit flow. \end{dfn}

\emph{Claim}. For $M_{0}$ convex, there exists a smooth limit flow at $(\bar{x},T)$, some $\bar{x} \in R^{n+1}$.

\begin{proof} We use results from existence theory of parabolic PDE to conclude smoothness results on smaller and smaller balls.


For $\bar{t} < T$, choose $\lambda = \frac{1}{r_{-}}$.  For $0 \leq t \leq \frac{1}{8n}$, $M_{\lambda}(t)$ is between $S^{n}(C_{0})$ and $S^{n}(\frac{1}{2})$.

On $B_{1/4}$ we can write $M_{(\lambda,t)}$ = graph$(u(.,t))$ (on $B_{1/4} \times [0,\frac{1}{8n}]$).  Then we observe that $\norm{u(.,t)} \leq c_{0}$ and $\norm{Du(.,t)} \leq c_{1}$ some constants $c_{0},c_{1}$, and

\begin{center}
$\frac{\partial u}{\partial t} = (\delta^{ij} - \frac{D_{i}uD_{j}u}{1 + \norm{Du}^{2}})D_{i}D_{j}u$
\end{center}

Now the H\"older gradient estimate holds on $B_{1/8} \times [\frac{1}{16n},\frac{1}{8n}]$.  Therefore, by the interior Schauder estimates,

\begin{center}
$\norm{u}_{C^{(k,\alpha)}(B_{1/16} \times [\frac{3}{32n},\frac{1}{8n}])} \leq C_{k}(n,c_{0})$
\end{center}

$\norm{\nabla^{(j)}h}^{2}$ is computed from $D^{(j+2)}u$.  This implies

\begin{center}
$\norm{\nabla^{(j)}h}^{2} \leq C_{j}(n,c_{0})$
\end{center}

on $M_{(\lambda,t)}$, $\frac{3}{32n} \leq t \leq \frac{1}{8n}$.

This means that the tube radius is controlled, ie $M_{(\lambda,\bar{t})} \cap B_{k} \neq \phi$, and so by the compactness theorem, there is a limit flow at $(\bar{x},T)$. \end{proof}


\section{The 2 dimensional Ricci Flow (2DRF)}

Consider the PDE

\begin{center}
$\frac{\partial}{\partial t}g_{ij} = -2R_{ij}$
\end{center}

Given $M$ a compact manifold, with riemannian metric $g_{0}$, we want to find $g(x,t)$ a smooth inner product on $T_{x}M$ satisfying this equation with $g(x,0) = g_{0}(x)$.  We then say that $g(x,t)$ is a solution to the Ricci flow of $(M,g_{0})$.

\subsection{Introduction and Existence}

Note that for $n = 2$, $R_{11} = g^{kl}R_{1k1l} = \sum_{k}R_{1k1k} = R_{1212} = sect(T_{x}M) = K$.  Hence $R_{ij} = K g_{ij}$, and

\begin{center}
$\frac{\partial}{\partial t}g_{ij} = -2Kg_{ij}$
\end{center}

But note that $R = R_{11} + R_{22} = 2K$.

Hence $g_{ij}(x,t) = e^{2u(x,t)}g_{ij}(x,0)$ for some function $u$.

Compute

\begin{center}
$R(g(t)) = R(e^{2u}g_{0}) = e^{-2u}(R(g_{0}) - 2\Delta_{g_{0}}u)$
\end{center}

\begin{align}
\nabla^{g}_{\partial_{i}}\partial_{j} &= \frac{1}{2}g^{kl}(\partial_{i}g_{jl} + \partial_{j}g_{il} - \partial_{l}g_{ij})\partial_{l} \nonumber \\
&= \frac{1}{2}e^{-2u}(g_{0})^{kl}(\partial_{i}(e^{2u}(g_{0})_{jl}) \cdots )\partial_{l} \nonumber \\
&= \nabla^{g_{0}}_{\partial_{i}}\partial_{j} + \partial_{i}u\partial_{j} + \partial_{j}u\partial_{i} - (g_{0})_{ij}\partial_{p}u(g_{0})^{pq}\partial_{q} \end{align}

From this we may compute the curvature:

\begin{center}
$2e^{2u}\frac{\partial u}{\partial t}g_{0} = - e^{-2u}(R_{0} - 2\Delta_{0}u)e^{2u}g_{0}$
\end{center}

which implies that

\begin{center}
$\frac{\partial u}{\partial t} = e^{-2u}(\Delta_{g_{0}}u - K_{0})$
\end{center}

So short time existence is ok.

\subsection{Evolution of the Scalar Curvature}

We compute $\frac{\partial}{\partial t}R$, if $g_{ij} = -fg_{ij}$.

Now recall $\Gamma_{ij}^{k} = \frac{1}{2}g^{kl}(\partial_{i}g_{jl} + \partial_{j}g_{il} - \partial_{l}g_{ij})$ is the expression for the Christoffel symbols in local coordinates.  Observe that, although $\Gamma$ is not a tensor, that $\frac{\partial}{\partial t}\Gamma$ is.  For a proof of this, let $\nabla, \bar{\nabla}$ be covariant derivatives, and consider $T(U,V) = \nabla_{U}V - \bar{\nabla}_{U}V = \nabla_{U}(V^{j}\partial_{j}) - \bar{\nabla}_{U}(V^{j}\partial_{j}) = (UV^{j})\partial_{j} + V^{j}\nabla_{U}\partial_{j} - (UV^{j})\partial_{j} - V^{j}\bar{\nabla}_{U}\partial_{j} = U^{i}V^{j}T(\partial_{i},\partial_{j})$.  Since $\frac{\partial}{\partial t}\Gamma$ is merely a limit of expressions like this, each of which are tensors, so much the limit be a tensor.

Hence we can compute $\frac{\partial}{\partial t}\Gamma$ in any local coordinate system.  Choose local coordinates such that $\Gamma(x_{0},t_{0}) = 0$.  Then

\begin{align}
\frac{\partial}{\partial t}\Gamma_{ij}^{k}\vert_{(x_{0},t_{0})} &= \frac{1}{2}g^{kl}(\partial_{i}(-fg_{jl}) + \partial_{j}(-fg_{il}) - \partial_{l}(-fg_{ij})) \nonumber \\
&= \frac{1}{2}g^{kl}(\nabla_{i}(-fg_{jl}) + \nabla_{j}(-fg_{il}) + \nabla_{l}(fg_{ij})) \nonumber \\
&= -\frac{1}{2}\nabla_{i}(f) \delta_{j}^{k} - \frac{1}{2}\nabla_{j}(f) \delta_{i}^{k} + \frac{1}{2}g_{ij}\nabla_{l}(f) g^{lk}
\end{align}

Now, recall by definition

\begin{center}
$\nabla_{i}\partial_{j} = \Gamma_{ij}^{p}\partial_{p}$
\end{center}

so

\begin{center}
$\nabla_{k}(\nabla_{i}\partial_{j}) = (\partial_{k}\Gamma_{ij}^{p})\partial_{p} + \Gamma_{ij}^{p}\Gamma_{kp}^{q}\partial_{q}$
\end{center}

In particular, recall that

\begin{center}
$R_{ikj}^{l} = \partial_{k}\Gamma_{ij}^{l} - \partial_{i}\Gamma_{kj}^{l} + \Gamma \star \Gamma$
\end{center}

so

\begin{align}
\frac{\partial}{\partial t}R_{ikj}^{l} &= \nabla_{k}(-\frac{1}{2}\nabla_{i}(f) \delta_{j}^{k} - \frac{1}{2}\nabla_{j}(f) \delta_{i}^{k} + \frac{1}{2}g_{ij}\nabla_{l}(f) g^{lk}) - (i \leftrightarrow k) \nonumber \\
&= -\frac{1}{2}\nabla_{k}\nabla_{j}f \delta_{i}^{l} + \frac{1}{2}g_{ij}\nabla_{k}\nabla_{p}f g^{pl} + \frac{1}{2}\nabla_{i}\nabla_{j}f \delta_{k}^{l} - \frac{1}{2}g_{kj}\nabla_{i}\nabla_{p}f g^{pl}
\end{align}

Then, for the Ricci tensor, we have

\begin{align}
\frac{\partial}{\partial t}R_{ij} &= -\frac{1}{2}\nabla_{i}\nabla_{j}f + \frac{1}{2}g_{ij}\Delta f + \nabla_{i}\nabla_{j}f \frac{1}{2}n - \frac{1}{2}\nabla_{i}\nabla_{j}f \nonumber \\
&= \frac{1}{2}g_{ij}\Delta f
\end{align}

But

\begin{center}
$\frac{\partial}{\partial t}R_{ij} = \frac{\partial}{\partial t}(\frac{1}{2}Rg_{ij}) = \frac{1}{2}\frac{\partial R}{\partial t}g_{ij} + \frac{1}{2}R(-fg_{ij})$
\end{center}

So

\begin{center}
$\frac{\partial R}{\partial t} = \Delta f + Rf$
\end{center}

is the evolution equation for the scalar curvature.

Specialising to 2D Ricci flow, where $f = R$, we get

\begin{center}
$\frac{\partial R}{\partial t} = \Delta R + R^{2}$
\end{center}

Consequently, we may immediately conclude that [I don't understand this]

\begin{center}
$R_{min}(t) \geq \frac{R_{min}(0)}{1 - R_{min}(0)t} \geq -\frac{1}{t}$
\end{center}

\subsection{Normalisation}

We compute the change in area $A(t)$ of the surface $M(t)$ under the 2DRF

\begin{center}
$\frac{\partial}{\partial t}g_{ij} = -Rg_{ij}$
\end{center}

Now,

\begin{center}
Area $= A = \int_{M}\sqrt{det(g)}$
\end{center}

\begin{center}
$\frac{\partial}{\partial t}det(g) = det(g)tr(g^{-1}\frac{\partial g}{\partial t}) = det(g)(-2R)$
\end{center}

\begin{center}
$\frac{\partial}{\partial t}\sqrt{det(g)} = -R\sqrt{det(g)}$
\end{center}

Hence

\begin{center}
$\frac{d}{dt}A(g(t)) = - \int_{M}Rd\mu(g(t))$
\end{center}

Out of interest, recall the Gauss-Bonnet theorem: $\int_{M} R = 4\pi \chi(M)$.

But anyway,

\begin{center}
$\int_{M}Rd\mu(g) = \int e^{-2u}(R_{0} - 2\Delta_{0}u)e^{2u}d\mu(g_{0}) = \int R_{0}d\mu(g_{0})$
\end{center}

so the rate of change in area is constant in time.

Hence if $A(0)$ is finite there will exist a $T > 0$ such that $A(T) = 0$.  So it is of some interest to normalise our flow to fix the area.

So define $\hat{g}(x,t) = g(x,t)(\frac{A(g(0))}{A(g(t))})$.  Notice that this implies that Area$(\hat{g}(t)) = A(\hat{g}(0))$.

Now,

\begin{center}
$\frac{\partial}{\partial t}\hat{g} = -R(g)g\frac{A(0)}{A(t)} - \frac{gA(0)}{A(t)^{2}}(-\int R)$
\end{center}

Define $r = \frac{\int R}{A(0)}$ (a constant).

Then

\begin{center}
$\frac{\partial}{\partial t}\hat{g} = -\frac{A(0)}{A}(Rg - r\hat{g})$
\end{center}

Note that $\hat{R}\hat{g} = Rg$, and define a time variable $\tau$ such that $\frac{\partial}{\partial \tau} = \frac{A}{A(0)}\frac{\partial}{\partial t}$.  Then

\begin{center}
$\frac{\partial}{\partial \tau}\hat{g} = - (\hat{R} - r)\hat{g}$
\end{center}

This is the normalised 2D Ricci Flow equation.

Now, if we were to substitute $f = \hat{R} - r$ in the computation before, we get the evolution equation for the normalised curvature:

\begin{center}
$\frac{\partial}{\partial \tau}\hat{R} = \Delta \hat{R} + \hat{R}(\hat{R} - r)$
\end{center}

\emph{Special Case}.  Suppose $R(g_{0}) \leq -\delta < 0$.  Then observe that since

\begin{center}
$\frac{\partial}{\partial t}(\hat{R} - r) = \Delta(\hat{R} - r) + \hat{R}(\hat{R} - r)$,
\end{center}

by the maximum principle

\begin{itemize}
\item[(i)] $\hat{R} \leq -\delta$ is preserved,
\item[(ii)] $\frac{\partial}{\partial t}(\hat{R}_{max} - r) \leq -\delta(\hat{R}_{max} - r)$, which implies that $\hat{R}_{max} - r \leq Ce^{-\delta t}$,
\item[(iii)] $\norm{\hat{R} - r} \leq Ce^{-\delta t}$ for all $\delta \in (0, \norm{r})$,
\item[(iv)] $\frac{\partial}{\partial t}\hat{u} = - (\hat{R} - r)$ implies that $\norm{\hat{u}(x,t) - \hat{u}(x,0)} \leq Ce^{-\delta t}$, which implies that $\hat{u}(x,t) \rightarrow \hat{u}_{\infty}(x)$ in the $C^{\infty}$ norm, with $R(e^{2\hat{u}_{infty}}g_{0}) = r$.
\end{itemize}

\subsection{Soliton Solutions}

"Soliton" solutions are defined to be solutions which evolve without changing geometry, (up to scaling).  In particular, for such a solution, there exist diffeomorphisms $\phi_{t} : M \rightarrow M$, $\phi_{0} = id$ such that

\begin{center}
$\hat{g}_{(\phi(x,t),t)}(\frac{\partial \phi_{t}}{\partial x^{i}}(x,t), \frac{\partial \phi_{t}}{\partial x^{j}}(x,t)) = \hat{g}_{(x,0)}(\partial_{i},\partial_{j})$
\end{center}

In other words, $(\phi_{t})_{\star}\hat{g}_{t} = \hat{g}_{0}$.

But this implies that

\begin{center}
$\frac{\partial}{\partial t} = \mathcal{L}_{\frac{\partial \phi}{\partial t}}g = \nabla_{i}V_{j} + \nabla_{j}V_{i}$
\end{center}

for some vector field $V$, where $\mathcal{L}$ is the Lie derivative.


Hence

\begin{center}
$(R - r)g_{ij} + \nabla_{i}V_{j} + \nabla_{j}V_{i} = 0$
\end{center}

A gradient soliton is one such that $V$ can be realised as the gradient of a scalar function $f$.  But then $\nabla_{i}V_{j} = \nabla_{j}V_{i} = \nabla_{i}\nabla_{j}f$, so

\begin{center}
$(R - r)g_{ij} + 2\nabla_{i}\nabla_{j}f = 0$,
\end{center}

the Gradient Soliton equation.  It turns out that the classification of solutions of this equation is precisely what is required for the classification of prime three manifolds, ie, to establish the geometrisation conjecture of Thurston.  As a final comment, note that if we contract the above, we get a related equation from which we lose no information in the case of the 2D Ricci Flow:

\begin{center}
$(R - r) + \Delta f = 0$
\end{center}

It turns out that in $n$ dimensions, the Gradient Ricci Soliton equation is

\begin{center}
$-2R_{ij} = 2\nabla_{i}\nabla_{j}f - \lambda g_{ij}$
\end{center}

\subsection{Computing the "Variational Derivative" for the 2D Gradient Ricci Soliton equation}

This computation is relevant later on when we consider the Perelman functional for the $n$ dimensional Ricci flow.

Differentiate the 2D Gradient Ricci Soliton equation:

\begin{align}
0 &= \nabla_{k}R_{ij} + \nabla_{k}\nabla_{i}\nabla_{j}f - \nabla_{i}R_{kj} - \nabla_{i}\nabla_{k}\nabla_{j}f \nonumber \\
&= \nabla_{k}R_{ij} - \nabla_{i}R_{kj} + R_{kij}^{p}\nabla_{p}f
\end{align}

The trace of this expression is

\begin{center}
$0 = \nabla_{k}R - \nabla_{p}R_{k}^{p} - R_{k}^{p}\nabla_{p}f$
\end{center}

Recall the second Bianchi identity:

\begin{center}
$\nabla_{k}R_{ijpq} + \nabla_{i}R_{jkpq} + \nabla_{j}R_{kipq} = 0$
\end{center}

Taking the trace of this, we get

\begin{center}
$0 = \nabla_{k}R_{jq} - \nabla_{j}R_{kq} + \nabla_{i}R_{iqjk}$
\end{center}

Taking the trace once more, we get

\begin{center}
$0 = \nabla_{k}R - \nabla_{j}R_{k}^{j} - \nabla_{j}R_{k}^{j}$
\end{center}

ie $\nabla_{p}R_{k}^{p} = \frac{1}{2}\nabla_{k}R$.

If dim$(M) = n = 2$, we conclude further that

\begin{align}
0 &= \frac{1}{2}\nabla_{k}R - \frac{1}{2}R\delta_{k}^{p}\nabla_{p}f \nonumber \\
&= \frac{1}{2}(\nabla_{k}R - R\nabla_{k}f),
\end{align}

or

\begin{center}
$\nabla_{k}(Re^{-f}) = 0$
\end{center}

$Re^{-f}$ is quite reminiscent of the general form of the integrand in the Perelman functional (see later).

Now,

\begin{align}
\nabla_{k}(\norm{\nabla f}^{2} + R) &= 2\nabla_{q}f \nabla_{k}\nabla_{q}f + \nabla_{k}R \nonumber \\
&= (\nabla_{k}R - R\nabla_{k}f) + r\nabla_{k}f
\end{align}

since $\nabla_{k}\nabla_{q}f = -\frac{1}{2}(R - r)g_{kq} = -(R - r)\nabla_{k}f$.

Hence $\nabla_{k}(\norm{\nabla f}^{2} + R - rf) = 0$.

\subsection{\texorpdfstring{$C^{\infty}$}{C-infinity} convergence of the 2DRF}

I prove that normalised 2DRF on the maximal interval $[0,\infty)$ converges in the $C^{\infty}$ norm to a limiting solution for the metric $g_{\infty} = e^{2u_{\infty}}g_{0}$, for $g(0,.) = g_{0}(.)$, provided that $R_{max}$ is initially finite, ie the initial scalar curvature is bounded.  This is a model of proving smooth convergence of 3DRF solitons to a sensible limit.

So let $g(t)$ be any solution of normalised 2DRF.  At $t = 0$, let $f$ be a solution of $\Delta f + (R - r) = 0$.

For $t > 0$, define $f$ to be the solution of

\begin{center}
$\frac{\partial f}{\partial t} = \Delta f + rf$
\end{center}

\emph{Claim}.  If $f$ is the solution of this equation, then $(\Delta f + (R - r))(x,t) = 0$ for all $x \in M$, $t \geq 0$.

\emph{Proof of Claim}.

Well certainly if $f$ satisfies this equation, then

\begin{center}
$\frac{\partial}{\partial t}(\Delta f + (R - r)) = \frac{\partial}{\partial t}\Delta f + \Delta R + R(R - r)$
\end{center}

Now by definition,

\begin{center}
$\frac{\partial}{\partial t}\Delta f = \frac{\partial}{\partial t}(\frac{1}{\sqrt{det(g)}}\frac{\partial}{\partial x^{i}}(\sqrt{det(g)}g^{ij}\frac{\partial}{\partial x^{i}}f))$
\end{center}

$g = e^{2u}g_{0}$, so $\sqrt{det(g)} = e^{2u}\sqrt{det(g_{0})}$, and $g^{-1} = e^{-2u}g_{0}^{-1}$.

Hence

\begin{center}
$\Delta f = e^{-2u}\Delta_{g_{0}}f$
\end{center}

So

\begin{center}
$\frac{\partial}{\partial t}\Delta f = - (-(R - r))\Delta f + \Delta \frac{\partial f}{\partial t}$
\end{center}

Then

\begin{align}
\frac{\partial}{\partial t}(\Delta f + (R - r)) &= \Delta(\Delta f + rf) + \Delta R + R(R - r) + (R - r)\Delta f \nonumber \\
&= \Delta(\Delta f + (R - r)) + R(\Delta f + (R - r))
\end{align}

But then by the maximum principle, if $(\Delta f + (R - r))(x,0) = 0$, then $(\Delta f + (R - r))(x,t) = 0$ for all $(x,t)$, which proves the claim.

I now prove the convergence of normalised 2DRF to a sensible solution in the case that $r < 0$.  Essentially what we need to show is that $R$ and all its derivatives are bounded for all $x \in M$, $t \geq 0$.  Then, by the Ascoli-Arzela theorem, we know that as $t \rightarrow \infty$, there must be a subsequence of the $(M,g_{t})_{t \in R^{+}}$, say $(M,g_{t_{k}})_{k \in N}$, that converges to a smooth limit $(M,g_{\infty})$.  But then this limit must be unique because the space of smooth metrics with bounded curvature is compact and Hausdorff.

We aim to bound $R$, since then bounds on its derivatives are automatic, by the smoothing property of geometric evolution equations of heat type (see the original discussion of the properties of the heat equation).

Since

\begin{center}
$\frac{\partial}{\partial t}f = \Delta f + rf$
\end{center}

we see by the maximum principle for parabolic PDE that $\norm{f} \leq Ce^{rt}$

Now

\begin{align}
\frac{\partial}{\partial t}\nabla_{i}f &= \nabla_{i}(\Delta f + rf) \nonumber \\
&= g^{kl}\nabla_{i}\nabla_{k}\nabla_{l}f + g^{kl}R_{ikl}^{p}\nabla_{p}f + r\nabla_{i}f \nonumber \\
&= \delta \nabla_{i}f - \frac{1}{2}R\nabla_{i}f + r\nabla_{i}f 
\intertext{Hence}
\frac{\partial}{\partial t}\norm{\nabla f}^{2} &= \frac{\partial}{\partial t}(g^{ij}\nabla_{i}f\nabla_{j}f) \nonumber \\
&= 2\nabla_{i}f(\Delta \nabla_{i}f - \frac{1}{2}R\nabla_{i}f + r\nabla_{i}f) + (R - r)\norm{\nabla f}^{2}
\end{align}

(Note that if $\frac{\partial}{\partial t}g_{ij} = \mu g_{ij}$ then $\frac{\partial}{\partial t}g^{ij} = - \mu g^{ij}$.)

Hence

\begin{align}
\frac{\partial}{\partial t}\norm{\nabla f}^{2} &= \Delta \norm{\nabla f}^{2} - 2\norm{\nabla^{2}f}^{2} - R\norm{\nabla f}^{2} + 2r\norm{\nabla f}^{2} + R\norm{\nabla f}^{2} - r\norm{\nabla f}^{2} \nonumber \\
&= \Delta \norm{\nabla f}^{2} - 2\norm{\nabla^{2}f}^{2} + r\norm{\nabla f}^{2} \nonumber \\
&\leq \Delta \norm{\nabla f}^{2} + r\norm{\nabla f}^{2} \end{align}

from which we conclude, again by the maximum principle for parabolic PDE, that

\begin{align}
\norm{\nabla f}^{2}(x,t) \leq Ce^{rt}
\end{align}

for all $x \in M$.

We wish to control $R$, however; this will require looking to higher derivatives of $f$.  So

\begin{center}
$\frac{\partial}{\partial t}(\norm{\nabla f}^{2} + R) = \Delta(\norm{\nabla f}^{2} + R) - 2\norm{\nabla^{2}f}^{2} + r\norm{\nabla f}^{2} + R(R - r)$
\end{center}

Recall that $\Delta f + (R - r) = 0$ by hypothesis, so

\begin{center}
$\norm{\nabla^{2}f}^{2} \geq \frac{1}{2}(\Delta f)^{2} = \frac{1}{2}(R - r)^{2}$
\end{center}

Therefore

\begin{center}
$\frac{\partial}{\partial t}(\norm{\nabla f}^{2} + R) \leq \Delta(\norm{\nabla f}^{2} + R) - (R - r)^{2} + r(\norm{\nabla f}^{2} + R) + R(R - r) - rR + r(R - r)$
\end{center}

Hence

\begin{center}
$\frac{\partial}{\partial t}(\norm{\nabla f}^{2} + R - r) \leq \Delta(\norm{\nabla f}^{2} + R - r) + r(\norm{\nabla f}^{2} + R - r)$
\end{center}

So, once more by the maximum principle,

\begin{center}
$\norm{\nabla f}^{2} + R - r \leq Ce^{rt}$,
\end{center}

and in particular,

\begin{center}
$R - r \leq Ce^{rt}$
\end{center}

since by our previous result $\norm{\nabla f} \leq Ce^{rt}$, and by abuse of notation I am not relabelling constants.

But we already know that $R - r \geq -C e^{rt}$, since by the evolution equation for the scalar curvature,

\begin{align}
\frac{\partial}{\partial t} &= \Delta (R - r) + R(R - r) \nonumber \\
&= \Delta(R - r) + (R - r)^{2} + r(R - r) \nonumber \\
&\geq \Delta (R - r) + r(R - r)
\end{align}

so by the maximum principle, $R - r \geq -Ce^{rt}$.  We have now established that

\begin{align}
\norm{R - r} \leq Ce^{rt}
\end{align}

which is what we needed to show.

In the case that $r = 0$, we need to be a little bit more careful, but not overly so.  Our equation for $f$ reduces to

\begin{center}
$\frac{\partial}{\partial t}f = \Delta f$
\end{center}

so $\norm{f} \leq C$.

Like before,

\begin{center}
$\frac{\partial}{\partial t}(\norm{\nabla f}^{2} + R) \leq \Delta(\norm{\nabla f}^{2} + R)$
\end{center}

so $\norm{\nabla f}^{2} + R \leq C$

Now

\begin{align}
\frac{\partial}{\partial t}(2t\norm{\nabla f}^{2} + f^{2}) &\leq 2\norm{\nabla f}^{2} + \Delta (2t \norm{\nabla f}^{2}) + \Delta f^{2} - 2\norm{\nabla f}^{2} \nonumber \\
&= \Delta(2t\norm{\nabla f}^{2} + f^{2})
\end{align}

so $2t\norm{\nabla f}^{2} + f^{2} \leq C$, and hence $\norm{\nabla f}^{2} \leq Cmin\{1,\frac{1}{t}\}$.

Now

\begin{align}
\frac{\partial}{\partial t}(2 \norm{\nabla f}^{2} + R) &\leq \Delta (2 \norm{\nabla f}^{2} + R) - 2R^{2} + R^{2} \nonumber \\
&\leq \Delta(2\norm{\nabla f}^{2} + R) - R^{2}
\end{align}

I claim that $2\norm{\nabla f}^{2} + R \leq \frac{C}{t}$, and hence that, putting together the previous information we deduced, $R \leq \frac{C}{t}$.

Now, we already know that $R \geq -\frac{1}{t}$ since $\frac{\partial R}{\partial t} = \Delta R + R^{2}$, so we get that $R$ is controlled.  We get bounds on $\norm{\nabla^{k}R}$ as before.  So normalised 2DRF will converge to a sensible limiting solution $e^{2u_{\infty}}g_{0}$ for the metric.

In fact, we can conclude more than this.  Since $\frac{\partial}{\partial t}u = -\frac{1}{2}R$, and $\frac{\partial f}{\partial t} = \Delta f = -R$, we get that $\frac{\partial}{\partial t}(2u - f) = 0$.  But since $f$ is bounded, this means that $u$ must be bounded.  But if $u$ is bounded, we must get that $u \rightarrow u_{\infty}$ such that $R(e^{2u^{\infty}}g_{0}) = 0$.

Finally, in the case that $r > 0$, we can no longer use such basic arguments and need to use more advanced tools.  Proving convergence for the cases $r < 0$ and $r = 0$ are the analogies of the classification of the solutions to normalised 3DRF that Hamilton and his contemporaries made in the 80s and early 90s.  The case $r > 0$ with $R > 0$ initially requires the introduction of an "entropy" to prove convergence, much like Perelman did for the 3DRF.  In particular, one defines the quantity

\begin{align}
\mathcal{Z} = \int_{M}R log(R)
\end{align}

and shows that it decreases with the time parameter under normalised 2DRF.  In particular, you find that $\frac{d}{dt}\mathcal{Z} \leq 0$, and $\frac{d}{dt}\mathcal{Z} = 0$ if and only if $g$ is a gradient soliton.  Also, at least for 2DRF, you need to use a related Differential Harnack Inequality.   In particular, we have

\begin{thm} (Harnack Inequality for the 2DRF). For any solution of 2DRF with $R > 0$, we have the following inequality:

\begin{align}
\frac{\partial R}{\partial t} - \frac{\norm{\nabla R}^{2}}{R} + \frac{R}{t} \geq 0
\end{align} \end{thm}

\begin{proof} Recall the related Harnack Inequality for the heat equation:

If $\frac{\partial u}{\partial t} = \Delta u$ on $R^{n}$, $u > 0$ then $\Delta u - \frac{\norm{\nabla u}^{2}}{u} + \frac{nu}{2t} \geq 0$.

Now for a soliton:

\begin{center}
$0 = \nabla_{i}\nabla_{j}f + \frac{1}{2}(R - r)g_{ij}$
\end{center}

hence $\nabla_{k}R = R\nabla_{k}f$.  So $\nabla_{k}\nabla_{l}R = \nabla_{l}R\nabla_{k}f + R\nabla_{k}\nabla_{l}f = \frac{\nabla_{l}R\nabla_{k}R}{R} - \frac{1}{2}R(R - r)g_{kl}$.

Hence

\begin{center}
$\frac{\partial R}{\partial t} - \frac{\norm{\nabla R}^{2}}{R} = \Delta R - \frac{\norm{\nabla R}^{2}}{R} + R(R - r) = 0$
\end{center}

Compute:

\begin{center}
$\frac{\partial}{\partial t}(\frac{\Delta R}{R} - \frac{\norm{\nabla R}^{2}}{R^{2}} + R - r)$

$\geq \Delta(\frac{\Delta R}{R} - \frac{\norm{\nabla R}^{2}}{R^{2}} + R - r) + \frac{\nabla R}{R} \cdot \nabla(\frac{\Delta R}{R} - \frac{\norm{\nabla R}^{2}}{R^{2}} + R - r) + (\frac{\Delta R}{R} - \frac{\norm{\nabla R}^{2}}{R^{2}} + (R - r))^{2}$
\end{center}

So we conclude by the maximum principle that

\begin{center}
$\frac{\Delta R}{R} - \frac{\norm{\nabla R}^{2}}{R^{2}} + R - r \geq - \frac{1}{t}$
\end{center}

The Harnack inequality for the 2DRF now is an easy consequence. \end{proof}

To complete the argument of sensible convergence of the normalised 2DRF, we use this inequality and the entropy to bound $R_{t}$ above and below for all $t \geq 0$.  This gives bounds, as usual, on $\norm{\nabla^{k}R}$, and using Ascoli-Arzela we conclude that there is $C^{\infty}$ convergence of $u$ to some limit $u_{\infty}$, giving a metric $g_{\infty} = e^{2u_{\infty}}g_{0}$.  The stationarity of the entropy at this solution tells us that this will be a gradient Ricci soliton.

In particular, $R_{\infty}$ will be constant and positive, so we conclude that $M$ is diffeomorphic to $S^{2}$ or $RP^{2}$. 

Of course this argument (for $r > 0$), requires that $R$ be initially positive everywhere in $M$.  So we need a more general argument for when this is not the case, ie for when there are points in $M$ initially such that $R < 0$ at these points, in particular, that

\begin{center}
$R \geq -\frac{1}{t}$
\end{center}

This requires us to modify the above notion of entropy and the Harnack inequalities, but it is possible to make arguments analogous to the above to make this all work.

\section{The Ricci Flow}

We now consider the $n$-dimensional Ricci flow.

\subsection{Short time existence}

I establish short time existence for the $n$-dimensional Ricci flow

\begin{center}
$\frac{\partial}{\partial t}g_{ij} = -2R_{ij}$
\end{center}

Recall that

\begin{center}
$\frac{\partial}{\partial t}g_{ij} = A_{ij}^{pqrs}\partial_{r}\partial_{s}g_{pq} +$ lower order terms
\end{center}

We wish to show that this system is strongly parabolic, so that we can use existence results for such equations (just use existence theory for the heat equation, since all strongly parabolic equations are equivalent to the heat equation under an appropriate reparametrisation).  So, we wish to establish that

\begin{center}
$<A_{ij}^{pqrs}\xi_{r}\xi_{s}\eta_{pq},\eta_{ij}> \geq \epsilon \norm{\xi}^{2}\norm{\eta}^{2}$
\end{center}

for any $\xi$, $\eta$, for some $\epsilon > 0$.

Once again, remember that

\begin{center}
$\Gamma_{ij}^{k} = \frac{1}{2}g^{kl}(\partial_{i}g_{jl} + \partial_{l}g_{il} - \partial_{l}g_{ij})$
\end{center}

So

\begin{align} R_{ikj}^{l}\partial_{l} &= \nabla_{k}(\nabla_{i}\partial_{j}) - \nabla_{i}(\nabla_{k}\partial_{j}) \nonumber \\
&= \nabla_{k}(\Gamma_{ij}^{l}\partial_{l}) + \nabla_{i}(\Gamma_{kj}^{l}\partial_{l}) \nonumber \\
&= (\partial_{k}\Gamma_{ij}^{l} - \partial_{i}\Gamma_{kj}^{l})\partial_{l} + \text{lower order terms} \nonumber \\
&= \frac{1}{2}g^{lm}(\partial_{k}(\partial_{i}g_{jm} + \partial_{j}g_{im} - \partial_{m}g_{ij}) - \partial_{i}(\partial_{k}g_{jm} + \partial_{j}g_{km} - \partial_{m}g_{kj})) + \text{lower order terms} \nonumber \\
&= \frac{1}{2}g^{lm}(\partial_{k}\partial_{j}g_{im} - \partial_{k}\partial_{m}g_{ij} - \partial_{i}\partial_{j}g_{km} + \partial_{i}\partial_{m}g_{kj}) + \text{lower order terms}
\intertext{Hence}
R_{ij} &= -\frac{1}{2}g^{kl}\partial_{k}\partial_{j}g_{ij} + \frac{1}{2}\partial_{j}(g^{kl}\partial_{k}g_{il}) + \frac{1}{2}\partial_{i}(g^{kl}\partial_{k}g_{jl}) - \frac{1}{2}g^{kl}\partial_{i}\partial_{j}g_{kl} + \text{lower order terms} \nonumber \\
&= -\frac{1}{2}g^{kl}\partial_{k}\partial_{j}g_{ij} + \partial_{i}(\frac{1}{2}g^{kl}\partial_{k}g_{jl} - \frac{1}{4}g^{kl}\partial_{j}g_{kl}) + \partial_{j}(\frac{1}{2}g^{kl}\partial_{k}g_{il} - \frac{1}{4}g^{kl}\partial_{i}g_{kl}) + \text{lower order terms} \end{align}

Hence recalling that $R_{ij}$ is a tensor that our calculation is hence the same in all coordinate systems, we get that

\begin{center}
$-2R_{ij} = g^{kl}\partial_{k}\partial_{l}g_{ij} + \nabla_{i}X_{j} + \nabla_{j}X_{i} +$ lower order terms,
\end{center}

where $X_{i} = \frac{1}{2}\partial_{i}g_{kl} - g^{kl}\partial_{k}g_{il}$.

It is possible to reparametrise this expression such that $X = 0$.  Then $A_{ij}^{pqrs} = g^{rs}\delta_{ij}^{pq}$, and so

\begin{center}
$g^{ij}g^{rs}\delta_{ij}^{pq}\xi_{r}\xi_{s}\eta_{pq}\eta_{ij} = \norm{\xi}^{2}\norm{\eta}^{2}$
\end{center}

Hence this Ricci flow is \emph{weakly} parabolic, and short time existence follows.

\subsection{Evolution of the Curvature}

It is possible to determine how the Riemann curvature tensor evolves along a one parameter family of metrics on a fixed space $M$.

In particular, let us compute
\begin{align}
\frac{\partial}{\partial t}R_{ikj}^{l} &= \nabla_{k}(\partial_{t}\Gamma_{ij}^{l}) - \nabla_{i}(\partial_{t}\Gamma_{kj}^{l}) \nonumber \\
&= \nabla_{k}(\frac{1}{2}g^{lm}(\nabla_{i}\frac{\partial}{\partial t}g_{jl} + \nabla_{j}\frac{\partial}{\partial t}g_{il} - \nabla_{l}\frac{\partial}{\partial t}g_{ij})) - (i \leftrightarrow k) \nonumber \\
&= \nabla_{k}\nabla^{l}R_{ij} - \nabla_{k}\nabla_{i}R_{j}^{l} - \nabla_{k}\nabla_{j}R_{i}^{l} - \nabla_{i}\nabla^{l}R_{kj} + \nabla_{i}\nabla_{k}R_{j}^{l} + \nabla_{i}\nabla_{j}R_{k}^{l} \end{align}

(Define $\nabla_{i}\nabla_{k}R_{j}^{l} = R \star R$)

Recall the second Bianchi identity:

\begin{center}
$0 = \nabla_{i}R_{jklm} + \nabla_{j}R_{kilm} + \nabla_{k}R_{ijlm}$
\end{center}

From which follows the relation

\begin{center}
$0 = \nabla_{l}R_{ij} - \nabla_{i}R_{lj} + \nabla_{k}R_{lijk}$
\end{center}

So using this relation, we get that:
\begin{align}
\frac{\partial}{\partial t}R_{ikj}^{l} &= - \nabla_{k}\nabla_{p}R_{ljip} + \nabla_{i}\nabla_{p}R_{ljkp} + R \star R \nonumber \\
&= \nabla_{p}(\nabla_{i}R_{ljkp} - \nabla_{k}R_{ljip}) + R \star R \nonumber \\
&= \nabla_{p}(\nabla_{i}R_{kplj} + \nabla_{k}R_{pilj}) + R \star R \nonumber \\
&= -\nabla_{p}\nabla_{p}R_{iklj} + R \star R \text{(using the second Bianchi identity once more)} \nonumber \\
&= \Delta R_{ikjl} + R \star R \end{align}

\subsection{The Structural Tensor \texorpdfstring{$E$}{E}}

The structural tensor $E$ takes the same role for general $3$ manifolds undergoing Ricci Flow that the second fundamental form $A$ takes for hypersurfaces undergoing MCF.  Hence it has a great deal of importance.  In this section I will proceed to define this object and derive the evolution equation that it obeys.

Let $(M^{3},g)$ be orientable.  Let $\mu(u,v,w)$ be the signed volume of the parallelpiped generated by $u,v$, and $w$.  $\mu$ is of course the alternating tensor; in coordinates, $\mu_{ijk} = 1$ if $(i,j,k)$ is a positive permutation, $-1$ if $(i,j,k)$ is a negative permutation and $0$ otherwise.

Define $E^{ij} = \frac{1}{4}\mu^{iab}\mu^{jcd}R_{abcd}$, so $R_{abcd} = \mu_{abi}\mu_{cdj}E^{ij}$ (this follows from the identity $g^{ij}\mu_{iab}\mu_{jcd} = g_{ac}g_{bd} - g_{ad}g_{bc}$).

The eigenvalues of $E$ are called the \emph{principal sectional curvatures}.  In a frame which diagonalises $E$, we get that

\[ E  = \left( \begin{array}{ccc}
\lambda & 0 & 0 \\
0 & \mu & 0 \\
0 & 0 & \nu \end{array} \right) \]

with $R_{1212} = \nu$, $R_{1313} = \mu$, and $R_{2323} = \lambda$.

In particular, if $\norm{v} = 1$, then $E(v,v) = Sect(v^{\perp})$, which gives us some form of intuition for what $E$ is measuring.

I now proceed to determine the evolution equation for the components of the structural tensor.  Now, we know that

\begin{center}
$\frac{\partial}{\partial t}R_{ijkl} = \Delta R_{ijkl} + R \star R$
\end{center}

Furthermore, we have that

\begin{center}
$\frac{\partial}{\partial t}\mu_{ijk} = -R\mu_{ijk}$
\end{center}

and hence that

\begin{center}
$\frac{\partial}{\partial t}\mu^{ijk} = R\mu^{ijk}$
\end{center}

Hence

\begin{center}
$\frac{\partial}{\partial t}E_{i}^{j} = \Delta E_{i}^{j} + 2E_{i}^{k}E_{k}^{j} + \mu_{iab}\mu^{jcd}E_{c}^{a}E_{d}^{b}$
\end{center}

which is the evolution equation for the structural tensor.  This is very useful in the 3DRF.

\subsection{Convergence Results for Ricci Flow}

The following theorem, due to Hamilton in 1982, is a fairly significant result in the development of the Ricci flow:

\begin{thm}  If $(M^{3},g_{0})$, $M$ compact, has $Ricci(g_{0}) > 0$ then there exists a solution of Ricci flow starting from $g_{0}$ on a maximal time interval $[0,T)$.  Furthermore, $Vol(g(t)) \rightarrow 0$ as $t \rightarrow T$, and the normalised Ricci flow converges smoothly to a metric of constant positive sectional curvature, that is

\begin{center}
$\frac{(Vol(S^{3}))^{2/3}g(t)}{(Vol(g(t)))^{2/3}} \rightarrow g_{T}$
\end{center}

with Sect$(g_{T}) =$ constant $> 0$. \end{thm}

\begin{proof} I will prove this later. \end{proof}

In fact, we have a more general result due to B\"ohm and Wilking.  Before I mention it, however, I should introduce the curvature operator. The curvature operator $\bar{R}$ is a bilinear form on $\Lambda^{2}TM$, which is defined by the relation

\begin{center}
$\bar{R}(a^{ij}\partial_{i} \wedge \partial_{j}, b^{kl}\partial_{k} \wedge \partial_{l}) = a^{ij}b^{kl}R_{ijkl}$
\end{center}

In particular, $\bar{R}(\partial_{i} \wedge \partial_{j},\partial_{i} \wedge \partial_{j}) = Sect(\partial_{i} \wedge \partial_{j})$ if $\partial_{i}$ and $\partial_{j}$ are orthonormal.

\begin{rmk} If $n \geq 4$ there are elements of $\Lambda^{2}TM$ which are not of the form $u \wedge v$ for any $u, v \in TM$.  Furthermore, if $\bar{R} \geq 0$ we may conclude $Sect \geq 0$, \emph{but not vice versa} for the aforementioned reason. \end{rmk}

So now, the result:

\begin{thm} (Bohm and Wilking, 2006). If $(M^{n},g_{0})$ is compact, and such that the sum of the smallest two eigenvalues of the curvature operator $\bar{R}(x)$ are positive, for all $x \in M$, then $M^{n}$ converges under Ricci flow to a manifold of constant curvature (modulo scaling). \end{thm}

\begin{rmk} Note that having the smallest two eigenvalues positive is closely related to the notion of 2-convexity, which is the main notion used in the Huisken-Sinestrari result (which I will mention later on).  Essentially for a hypersurface to be 2-convex means that the sum of its smallest two principal curvatures is positive everywhere. \end{rmk}

\begin{proof} (B\"ohm and Wilking, sketch). In order to prove this, we would like an analogue of the notion of the structural tensor for $n$ dimensional manifolds.  So, given any basis $\{ \Pi_{\alpha} : \alpha = 1, \cdots, \frac{n(n-1)}{2}\}$ for $\Lambda^{2}T_{x}M$ we get (writing $\bar{R}_{\alpha \beta} = \bar{R}(\Pi_{\alpha},\Pi_{\beta})$)

\begin{center}
$\frac{\partial}{\partial t}\bar{R}_{\alpha}^{\beta} = \Delta \bar{R}_{\alpha}^{\beta} + 2(R_{\alpha}^{\gamma}R_{\gamma}^{\beta} + C_{\alpha \gamma \delta}C^{\beta \xi \tau}R_{\xi}^{\gamma}R_{\tau}^{\delta})$
\end{center}

The $C_{\alpha \beta \gamma}$ are structure constants taking the analogous role to the structural constants $E_{ij}$ as I defined before, and are defined by 

\begin{center}
$[\Pi_{\alpha},\Pi_{\beta}] = C_{\alpha \beta \gamma}G^{\gamma \delta}\Pi_{\delta}$,
\end{center}

where $G$ is the metric on $\Lambda^{2}TM$.

The problem is then, like with the Huisken rescaling result of MCF, to find $SO(n)$ invariant convex subsets $\Omega$ of $Sym(\Lambda^{2}TM)$ such that the flow lines of the ODE corresponding to

\begin{center}
$\frac{\partial}{\partial t}\bar{R} = \bar{R}^{2} + \bar{R} \star \bar{R}$
\end{center}

never leave $\Omega$. \end{proof}

\subsection{The Perelman Functional; connection with the Fisher Information}

It turns out that the Ricci Flow can be realised as the steepest descent flow of the following functional

\begin{center}
$F(g,f) = \int_{M}(R_{g} + \norm{\nabla f}^{2}_{g})e^{-f}d\mu(g)$
\end{center}

Note that

\begin{center}
$\frac{d}{dt}F = -\int (R_{ij} + \nabla_{i}\nabla_{j}f)\dot{g}_{ij}e^{-f}d\mu + \int (2\Delta f - \norm{\nabla f}^{2} + R)(\frac{1}{2}g^{ij}\dot{g}_{ij} - \dot{f})e^{-f}d\mu$
\end{center}

Fix a smooth measure $d\xi$ and given $g$, define $f(g)$ by $e^{-f}d\mu(g) = d\xi$.

Then the variation of $F_{\xi}(g) = F(g,f(g))$ is

\begin{center}
$\frac{d}{dt}F_{\xi} = -\int (R_{ij} + \nabla_{i}\nabla_{j}f)\dot{g}_{ij}d\xi$
\end{center}

ie steepest descent flow is

\begin{center}
$\frac{\partial}{\partial t}g_{ij} = -2(R_{ij} + \nabla_{i}\nabla_{j}f)$
\end{center}

which is nothing other than the Ricci flow with reparametrisation.

\emph{Note}: The following requires some reference to the section on Fisher Information and Physical Manifolds.

Now note that the Physical Information Functional for a sharp Riemannian manifold $M$, is

\begin{center}
$K = \int_{M}(R_{\hat{g}} - \norm{\bar{\psi}}^{2}_{\hat{g}})d\mu(\hat{g})$
\end{center}

Suppose $\bar{\psi}$ is an irrotational vector field, in other words,

\begin{center}
$\bar{\psi} = \nabla_{i}\hat{f}$
\end{center}

for some function $\hat{f}$.

Then

\begin{center}
$K = \int_{M}(R_{\hat{g}} - \norm{\nabla \hat{f}}^{2}_{\hat{g}})d\mu(\hat{g})$
\end{center}

Now consider a diffeomorphism of $M$ that sends $\nabla_{\hat{g}}$ to $e^{-f}\nabla_{g}$ some function $f$.  Then $R_{\hat{g}} = R_{g}e^{-2f}$, $\norm{\nabla \hat{f}}_{\hat{g}}^{2} = \norm{\nabla \hat{f}}^{2}_{g}e^{-2f}$, $d\mu(\hat(g)) = e^{f}d\mu(g)$, and finally

\begin{center}
$K = \int_{M}(R_{g} - \norm{\nabla \hat{f}}_{g}^{2})e^{-f}d\mu(g)$
\end{center}

So if we choose a natural choice of flow on $M$, ie $\hat{f} = f$, we recover the Perelman functional.  In other words, crudely speaking,

\emph{Normalised Ricci Flow can be viewed as the steepest descent flow of the "Physical" Information}.

Perhaps more intuitively, if we consider the Fisher Information Functional for a sharp Riemannian manifold, we get

\begin{center}
$I = \int_{M}R(g)d\mu(g)$
\end{center}

In particular, for an arbitrary variation of the metric,

\begin{center}
$\frac{dI}{dt} = \int_{M}R_{ij}(g)\frac{\partial}{\partial t}g^{ij}d\mu(g)$,
\end{center}

from which it is easily deduced that the steepest descent flow is

\begin{center}
$\frac{\partial}{\partial t}g_{ij} = -R_{ij}(g)$,
\end{center}

So we may conclude:

\emph{Ricci Flow is the steepest descent flow of the Fisher Information for a sharp Riemannian manifold}.

This makes sense- for it is clear that, under the Ricci flow, we are losing information about the manifold.  Moreover, later when I discuss Ricci flow with surgery, it is obvious that we are losing information, because we are losing topology.

\section{Proof of Hamilton's Theorem (1982)}

This is a typeset version of a proof of Hamilton's theorem given by Nick Sheridan at the AMSI winter school.

\subsection{The theorem, and sketch of its proof}

We have the main result of this section:

\begin{thm} (Hamilton). Let $M^{3}$ be a compact 3 manifold which admits a Riemannian metric of positive Ricci curvature.  Then $M^{3}$ also admits a metric of positive sectional curvature. \end{thm}

\begin{proof} (sketch).  We shall follow the following procedure to prove this:

\begin{itemize}
\item[(1)] Prove short time existence of the Ricci Flow on $M^{3}$ starting with a metric $g_{0}$ with $Ric(g_{0}) = 0$.
\item[(2)] Show curvature blows up as $t \rightarrow T$.
\item[(3)] Show sectional curvatures get close together as the curvature gets large.
\item[(4)] Rescale time and the metric to get a solution to the equation $\frac{\partial g}{\partial t} = -2R_{ij} + \frac{2}{n}rg$, where $r = \frac{\int_{M}RdV}{\int_{M}dV}$.  Show furthermore that this solution exists for all time and converges to a metric of constant curvature.
\end{itemize}
\end{proof}

\subsection{A maximum principle}

We have the following maximum principle for sections of an arbitrary vector bundle which will prove most useful:

\begin{prop} Let $\Pi : \xi \rightarrow M^{n}$ be a vector bundle, with bundle metric $h$.  Let $\bar{\nabla}(t)$ be a family of connections on $\xi$ compact with respect to $h$.  Furthermore, suppose

\begin{center}
$F : \xi \times [0,T) \rightarrow \xi$
\end{center}

is a continuous fibre preserving Lipshitz map on each fibre.  Let $\kappa$ be a closed subspace of $\xi$ which is invariant under parallel translation by $\bar{\nabla}(t)$, and such that $\kappa_{x} = \kappa \cap \Pi^{-1}(x)$ is closed and convex in $\xi_{x}$ for all $x \in M^{n}$.  Finally, suppose $\alpha$ is a solution of

\begin{center}
$\frac{\partial}{\partial t}\alpha = \hat{\Delta}\alpha + F(\alpha)$
\end{center}

such that $\alpha(0) \in \kappa$.  Then if for each fibre $\xi_{x}$ every solution of

\begin{center}
$\frac{da}{dt} = F(a)$
\end{center}

with $a(0) \in \kappa_{x}$, remains in $\kappa_{x}$, we may conclude that the solution $\alpha(t)$ of the PDE remains in $\kappa$. \end{prop}

\subsection{The Uhlenbeck Trick}

(This trick is used to simplify the evolution equation for $R_{ijkl}$.)

So suppose $g(t)$ is a solution to Ricci flow on $M^{n}$.  Let $(V, h)$ be a vector bundle over $M^{n}$ with $h$ as metric such that

\begin{center}
$U_{0} : (V, h) \rightarrow (TM^{n},g_{0})$
\end{center}

is a bundle isometry.  Evolve $U(t)$ by

\begin{center}
$\frac{\partial}{\partial t}U^{i}_{a} = R_{l}^{i}U^{l}_{a}$
\end{center}

where $U_{a}^{i}$ are components of the isometry with respect to some bases for $V$ and $TM^{n}$.  Then I claim that $U(t)$ remains an isometry from $(V, h)$ to $(T(M),g_{t})$ and we can consider the behaviour of the pullback of the Riemann tensor $Rm$ on $M$ to $V$, $U^{\star}Rm$, rather than $Rm$.

\emph{Proof of Claim}.

\begin{align}
\frac{\partial}{\partial t}(U^{\star}g_{t})_{ab} &= \frac{\partial}{\partial t}(g_{ij}U^{i}_{a}U^{j}_{b}) \nonumber \\
&= \frac{\partial}{\partial t}g_{ij}U^{i}_{a}U^{j}_{b} + g_{ij}\frac{\partial}{\partial t}(U^{i}_{a})U^{j}_{b} + g_{ij}U^{i}_{a}\frac{\partial}{\partial t}U^{j}_{b} \nonumber \\
&= -2R_{ij}U^{i}_{a}U^{j}_{b} + g_{ij}R_{p}^{i}U^{p}_{a}U^{j}_{b} + g_{ij}U^{i}_{a}R_{q}^{j}U^{q}_{b} \nonumber \\
&= 0
\end{align}

In other words, $U$ remains an isometry.

Why do we want to do this? Well, first of all, the evolution equation for $U^{\star}Rm$ is much nicer.  Computing:

\begin{align}
\frac{\partial}{\partial t}R_{abcd} &= \Delta R_{abcd} + (R \star R)_{abcd} \nonumber \\
&= \Delta R_{abcd} + 2(B_{abcd} - B_{abdc} + B_{acbd} - B_{adbc}) \end{align}

where $B_{abcd}(Q) = -Q^{e}_{fab}Q^{f}_{edc}$ for a 4-tensor $Q$.

We can in fact view $Rm$ as a section of the bundle $\xi = \Lambda^{2}T^{\star}M^{n} \otimes_{S} \Lambda^{2}T^{\star}M^{n}$, since in 3 dimensions each fibre is the same as the vector space of 3 by 3 symmetric matrices.

Our relevant ODE for this bundle is

\begin{center}
$\frac{d}{dt}Q_{abcd} = 2(B_{abcd}(Q) - B_{abdc}(Q) + B_{acbd}(Q) - B_{adbc}(Q)) = (R \star R)_{abcd}(Q)$
\end{center}

In 3 dimensions, this ODE can be written with respect to a basis $\{e_{i}\}$ of $T^{\star}M^{3}_{x}$ such that $Q$ is diagonal- in particular, since $E$ is the curvature operator (essentially) in dimension 3, we can write everything in terms of the eigenvalues of $E$ and we get

\begin{center}
\[ \frac{d}{dt}E  = \frac{d}{dt}\left( \begin{array}{ccc}
\lambda & 0 & 0 \\
0 & \mu & 0 \\
0 & 0 & \nu \end{array} \right) = \left( \begin{array}{ccc}
\lambda^{2} + \mu \nu & 0 & 0 \\
0 & \mu^{2} + \nu \lambda & 0 \\
0 & 0 & \nu^{2} + \lambda \mu \end{array} \right) \]
\end{center}

Initial values tell us all about the initial Ricci and scalar curvatures

\begin{center}
\[ Ric = \frac{1}{2}\left( \begin{array}{ccc}
\mu + \nu & 0 & 0 \\
0 & \nu + \lambda & 0 \\
0 & 0 & \lambda + \mu \end{array} \right) \]
\end{center}

and $R = \lambda + \mu + \nu$.

From now on we assume $\lambda(0) \geq \mu(0) \geq \nu(0)$.

\subsection{Core of the Argument}

We have the following

\begin{lem}  If $\exists C < \infty$ such that $\lambda < C(\mu + \nu)$ at $t = 0$, then this condition is preserved under the Ricci flow.
\end{lem}

\begin{proof}  We will use the maximum principle from before, for the set

\begin{center}
$\kappa = \{ Q \in \xi : \lambda(Q) - C(\mu(Q) + \nu(Q)) \leq 0 \}$
\end{center}

Note that this set is invariant under parallel translation.  To see that $\kappa$ is convex, note that

\begin{center}
$\lambda(Q) - C(\nu(Q) + \mu(Q))$

$= max_{\norm{u} = 1}Q(u,u) + max_{(V,W | <V,W> = 0)}\{ - C(Q(V,V) + Q(W,W))\}$ (from the definition of $E$)
\end{center}

But $max_{\norm{u} = 1}(\alpha Q_{1} + \beta Q_{2})(u,u) \leq \alpha max_{\norm{u} = 1}Q_{1}(u,u) + \beta max_{\norm{u} = 1}Q_{2}(u,u)$.  This establishes convexity.

It just remains to show that the solution of the associated ODE stays in $\kappa$.  This follows from the fact that

\begin{center}
$\frac{d}{dt}log(\frac{\lambda}{\mu + \nu}) = \frac{\mu^{2}(\nu - \lambda) + \nu^{2}(\mu - \lambda)}{\lambda(\mu + \nu)} \leq 0$
\end{center}

as can easily be proved from the evolution equation for $\lambda, \mu, \nu$ from before.  Thus by the maximum principle, the solution to the PDE remains in $\kappa$.
\end{proof}

The next result is

\begin{thm} If $(M^{3},g_{0})$ is a closed Riemannian 3 manifold with $Ric(g_{0}) > 0$ then $\exists 0 < \delta < 1$ and $\bar{C} > 0$ (depending only on $g_{0}$) such that

\begin{center}
$\frac{\lambda - \nu}{\lambda + \mu + \nu} \leq \frac{C}{(\lambda + \mu + \nu)^{\delta}}$
\end{center}
\end{thm}

\emph{Remark}.  Note that this theorem implies that

\begin{center}
$\frac{(\lambda - \mu)^{2} + (\mu - \nu)^{2} + (\nu - \lambda)^{2}}{(\lambda + \mu + \nu)^{2}} \leq C(\lambda + \mu + \nu)^{-\delta}$
\end{center}

In particular, as the curvature gets large, as it will under the Ricci flow (recall $R = \lambda + \mu + \nu$), this shows that eigenvalues get "pinched" together.  But the left hand side is scale invariant, so it follows that the limit of normalised Ricci flow, $g_{\infty}$, has constant sectional curvature, and we are done.

\begin{proof} Note that it suffices to show that

\begin{center}
$\frac{\lambda - \nu}{\mu + \nu} \leq \frac{C}{\mu + \nu}^{\delta}$
\end{center}

Since $\mu + \nu > 0$ everywhere at time $0$ by the $Ric > 0$ condition, we may by compactness choose such a $\bar{C}$ and a $\delta$.  We now demonstrate that this condition is preserved using the maximum principle.

Define

\begin{center}
$\kappa = \{ Q \in \xi : (\lambda(Q) - \nu(Q)) - \bar{C}(\mu(Q) - \nu(Q))^{1- \delta} \leq 0\}$
\end{center}

To show that the solution to the ODE stays in $\kappa$ compute

\begin{center}
$\frac{d}{dt}log(\frac{\lambda - \nu}{(\mu + \nu)^{1-\delta}}) \leq \delta \lambda - \frac{1}{2}(1 - \delta)(\nu + \mu)$
\end{center}

By the lemma, it is possible to choose $\delta = \delta(C)$ small so that this always is non positive, so that $\frac{\lambda - \nu}{(\mu + \nu)^{1 - \delta}}$ is non increasing, so that the inequality is preserved by the ODE.

Thus, by the maximum principle, it is also preserved by the PDE.  This completes the proof of Hamilton's theorem. \end{proof}

\section{The Huisken-Sinestrari theorem; surgery and classification of canonical singularities for MCF}

This is a typeset version, with some modifications, of a particularly instructive lecture given by Gerard Huisken at the University of Melbourne, in which he motivated much of the surgery procedure and the general considerations of the classification program for 3DRF by considering MCF.  Much of the techniques and results he and his collaborator Sinestrari used (and proved) are completely analogous to related results for the full blown Ricci Flow.  In this section I shall outline his argument, providing clarification wherever I view it appropriate.

\subsection{Preliminaries, including statement of the theorem}

Let $M$ be a hypersurface evolving according to MCF.  Let $A =\{h_{ij}\}$ be its second fundamental form, with eigenvalues (principal curvatures) $\lambda_{1} \leq \lambda_{2} \leq \cdots \leq \lambda_{n}$.  Then we say that $M^{n}$ is $2$-convex if $\lambda_{1} + \lambda_{2} > 0$ everywhere on $M^{n}$ (this is a weaker criterion than convexity, which would require $\lambda_{1} \geq 0$ everywhere).

For a couple of examples, note that $S^{n-1} \times R$ is 2-convex, but $S^{n-2} \times R^{2}$ has $\lambda_{1} + \lambda_{2} = 0$.  For instance, a thin neck with cross section $S^{n-1}$ with two bulbs on each end is allowed (see the neckpinch in Canonical Singularities) as a hypersurface, but the analogous picture with cross section $S^{n-2}$ is not.

This definition might seem somewhat arbitrary, but note that for $n = 3$, If $R = (\lambda_{1}\lambda_{2} + \cdots + \lambda_{n-1}\lambda_{n}) = \frac{1}{2}(H^{2} - \norm{A}^{2}) > 0$, this this implies 2 convexity.  If $n = 4$, then if the manifold $M$ has the positive isotropic curvature of Hamilton, this once again implies 2 convexity.  So this suggests that this might be a reasonable notion to study.  In fact, we have the following result:

\begin{thm} (Huisken-Sinestrari). If $M^{n}$, $n \geq 3$ is 2-convex, then either $M^{n}$ is diffeomorphic to $S^{n}$ or diffeomorphic to a finite connected sum of copies of $S^{n-1} \times S^{1}$.  If $M = \partial \Omega$, $\Omega \subset R^{n+1}$, then either $\Omega$ is diffeomorphic to $B^{n+1}_{1}(0)$ or to a finite connected sum of $B^{n}_{1}(0) \times S^{1}$. \end{thm}

\begin{cor} If $n \geq 3$, $M^{n}$ 2-convex and simply connected, then $M^{n} \equiv S^{n}$ and $\bar{\Omega} \equiv \bar{B}^{n+1}_{1}(0)$. \end{cor}


\begin{proof} (sketch). The idea as to how to prove this is to use the mean curvature flow.  So, given $F_{0} : M^{n} \rightarrow R^{n+1}$, solve

\begin{center}
$\frac{d}{dt}F(p,t) = -H \cdot \nu(p,t) = \Delta_{t}(F(p,t))$
\end{center}

subject to the initial condition $F(p,0) = F_{0}(p)$.  This is a weakly parabolic system.

We also have importantly that 2-convexity is preserved by MCF.
\end{proof}

\subsection{Canonical Singularities}

If we perform MCF, the flow will continue until it reaches a singularity.  For instance, if $M^{n} = S^{n}_{R_{0}}$, then $R(t) = \sqrt{R_{0}^{2} - 2nt}$, with the solution degenerating to a point at $t = T = \frac{R_{0}^{2}}{2n}$.

Similarly, for the situation of starting with the manifold $S^{n-m}_{R(t)} \times R^{m}$, $R(t) = \sqrt{R_{0}^{2} - 2mt}$.  In particular, we have the following result, due to Huisken:

\emph{Positive Case}. If $M^{n}$ has $\lambda_{1} > 0$ (convex) then $M^{n}_{t}$ contracts smoothly to a round point.  Compare this with the Ricci flow result due to Hamilton: If $n = 3$ and $Ric(g) > 0$, then $g_{t}$ contracts to a round metric on $S^{3}/\Gamma$.

In an ideal world, we might hope that this is always the case.  But we run into the following

\emph{Problem}.  If the positivity condition is relaxed, then other singularities will develop.

In particular, for mean curvature flow, if we initially have 2-convexity, it can be shown that these singularities fit into one of the following categories, known as the canonical singularities:

\emph{(i) The shrinking sphere}.  Under MCF this solution will collapse to a point.

\begin{center}\scalebox{0.3}{\includegraphics{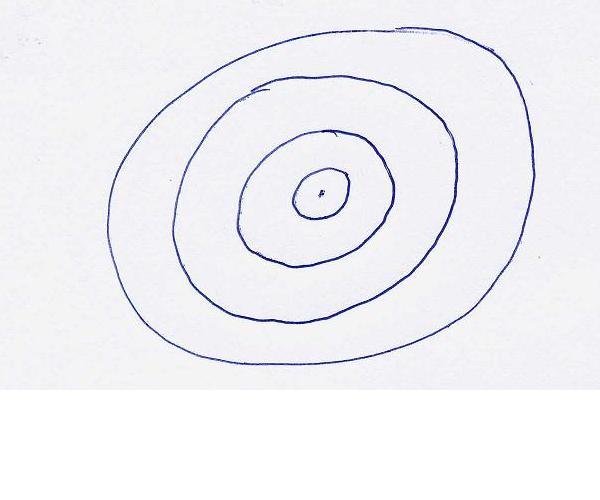}}\end{center}

\emph{(ii) The neckpinch}.  The neck will continue to become longer and thinner under the flow.  Under a microscope the singularity looks like the infinite cylinder $S^{n - 1} \times R$.

\begin{center}\scalebox{0.3}{\includegraphics{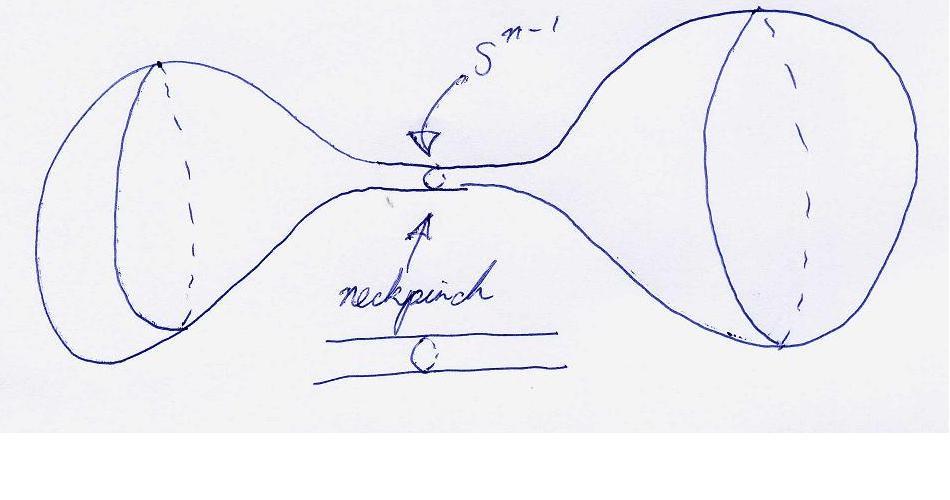}}\end{center}

\emph{(iii) The cusp, or degenerate neckpinch}.  This is a translating solution of MCF.  In general, if one has something looking like a neckpinch where one of the spheres is sufficiently large relative to the other, the smaller sphere will be "eaten" by the larger one and case (i) will apply; if both spheres decay at about the same speed a neckpinch will develop.  Hence there must be a critical point in between at which a "cusp" is formed, which is this singularity.  Under a microscope this will look like a horn.

\begin{center}\scalebox{0.3}{\includegraphics{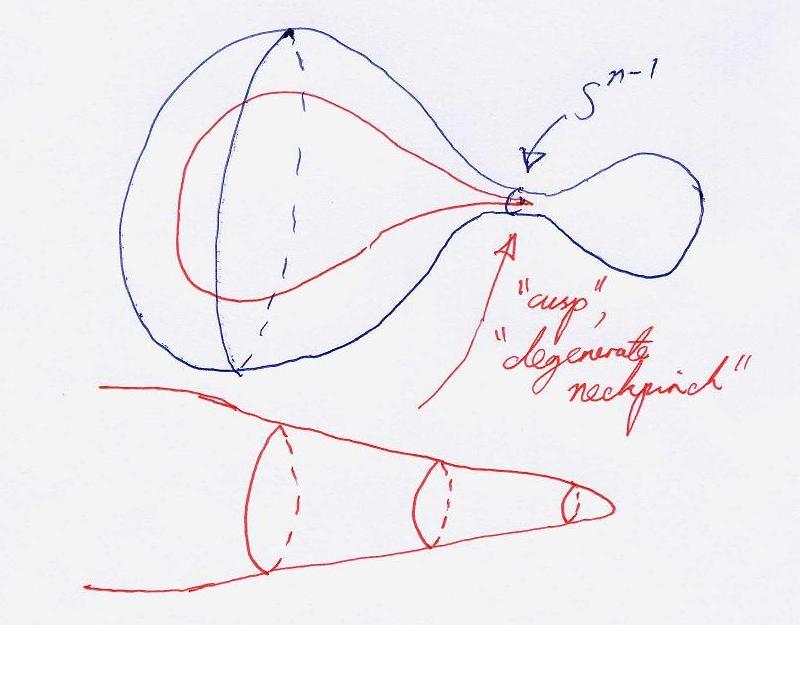}}\end{center}

\subsection{A priori Estimates}

We have the following a priori estimates for MCF.

\begin{itemize}
\item[(i)] $\lambda_{1} + \lambda_{2} > 0$ is preserved.
\item[(ii)] $\lambda_{1} + \lambda_{2} \geq \epsilon H$ is preserved, for some $\epsilon > 0$.
\item[(iii)] $\lambda_{1} \geq -\eta H - C_{\eta}$ $\forall \eta > 0$.  (Rescaling of singularities is weakly convex)
\item[(iv)] If $\lambda_{1} \leq \eta H$ (eg for a neckpinch or part of a horn \emph{not} at the tip), then $\norm{\lambda_{n} - \lambda_{2}} \leq 5 \eta H + \hat{C}_{\eta}$ ($n \geq 3$). (The cylindrical or "roundness" estimate).  Roughly what this is used for is to conclude that the curvatures $\lambda_{2}, \cdots, \lambda_{n}$ are all arbitrarily close in this situation, ie the manifold looks locally like $S^{n-1} \times R$.
\item[(v)] (Gradient Estimate)

\begin{center}
$\norm{\nabla A}^{2} \leq \eta_{0}H^{4} + C_{\eta_{0}}$
\end{center}

where $\eta_{0} = \eta_{0}(n)$, and $C_{\eta_{0}} = C(\eta_{0}, M_{0}^{n})$.  $\eta_{0}$ depends \emph{only} on the dimension, and not the initial data.
\end{itemize}

\emph{Remarks}.

\begin{itemize}
\item[(1)] (iii) is analogous to the Hamilton-Ivey estimates for the Ricci Flow, for $n = 3$, which control the structural tensor $E$:  $E_{ij} \geq - \eta R - C_{\eta}$.
\item[(2)] Note that (iii) and (iv) are only useful if $H$ is huge in the neighbourhood of a point, so that the terms with dependance on $H$ will dominate.  In particular, we would like to be able to show that if $H$ is large at a point, it is large in a neighbourhood of that point.  This motivates the last estimate,
\item[(3)] There is a natural analogue to (v) in the RF: This was established a posteriori by contradiction arguments due to Perelman, using his non-collapsing estimate, which he derived from his entropy.
\end{itemize}

\emph{Remark}.  Note that if we are not guaranteed an estimate like (v), we may get, in the RF, the following type of singularity, described as a "sheet of cigars" by Hamilton.

\begin{center}\scalebox{0.3}{\includegraphics{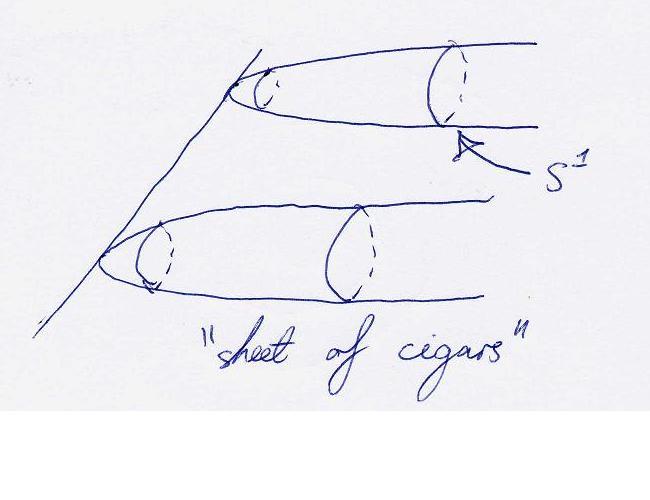}}\end{center}

If we get this type of singularity cropping up in the flow it creates problems, because the surgery procedure (which I will get around to describing for MCF) breaks down for this object.  In other words, we cannot perform surgery on this and then proceed to continue the flow, as we would like to do.  However, using the gradient estimate it is possible to eliminate this from the running, so to speak.  We then get an analogous classification of the singularities which develop under the Ricci flow as "canonical ones", namely, getting a neckpinch $S^{2} \times R$ and a horn with cross-section $S^{2}$, as well as the collapsing 3-sphere.

\subsection{Surgery}

Now that we have established that they will always develop, the idea is now to perform surgery on canonical singularities of type (ii) and type (iii) (if type (i) occurs we are done).

For the neckpinch, we cut the neck at both ends and glue in two $n$-balls, then continue the flow on the individual pieces:

\begin{center}\scalebox{0.3}{\includegraphics{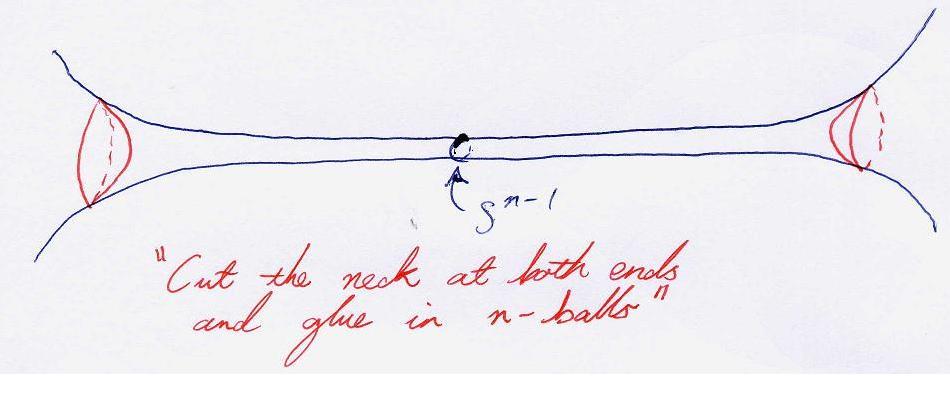}}\end{center}

For the cusp, we merely slice off the end and glue in an $n$-ball:

\begin{center}\scalebox{0.3}{\includegraphics{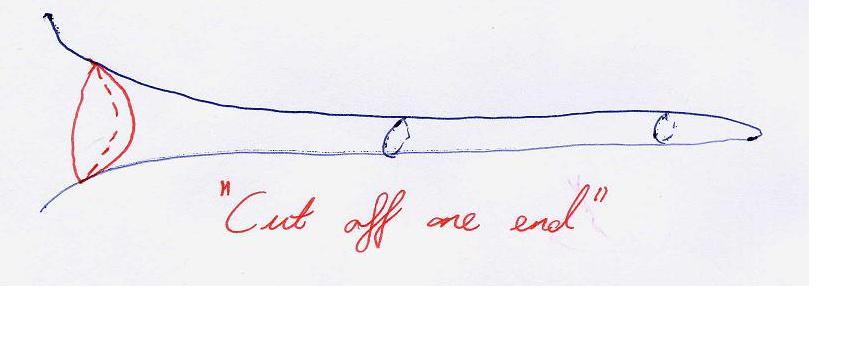}}\end{center}

This can all be made very precise.  In particular, the cusp and the neckpinch can be viewed as $\epsilon$-thin, ie with a cross-section of metric diameter $\epsilon$.  We can show that for $\epsilon$ sufficiently small, all estimates are preserved.  A big issue is whether things are actually improved after surgery; obviously if things do not improve we are not guaranteed convergence.  We also need to prove that only a finite number of surgeries are needed, ie. we do not have to perform an infinite number in an arbitrarily small interval.  Finally, we need to show that the solution "heals" sufficiently after a surgery before the next for us to be justified in using the same estimates (after surgery the solution is no longer $C^{\infty}$, so we need to be careful; we have to use the smoothing property of heat-type equations).

To make this all precise, we have the following useful result, which allows us to think of our necks as embedded in $R^{3}$.

\begin{prop} (MCF).  The interior of each neck has a canonical parametrisation close to a standard filled cylinder, ie $D^{n - 1} \times I$. \end{prop}

\begin{proof} (idea).  One shows that one can foliate the neck by minimal surfaces via harmonic maps.  The maps then give the required parametrisation.
\end{proof}

Finally, we have

\begin{prop} If $M^{n}$ is 2-convex, $n \geq 3$, then $M^{n} \cap \Pi$ consists of finitely many 2-convex components not nested, for any plane $\Pi$.  In other words, the surgery preserves embeddedness.  (Obviously we don't want the surface passing through the balls we are gluuing on during surgery- this result tells us that it is possible to choose the balls so this does not occur). \end{prop}

So this is what one does- one continues to perform MCF with surgery on the manifold until all the pieces one ends up with are canonical singularities of type (i), ie. shrinking spheres.  The result then follows.

\section{Outline of the Classification Program}

Even though this section does contain some remarks from Ben's lectures, it mostly continues from where he left off.  In this section I give a relatively detailed outline of how one links the Hamilton-Perelman program to Thurston's geometrisation conjecture, demonstrating how classification of gradient solitons is sufficient.  This section of this survey draws most strongly from John Morgan and Gang Tian's notes on the Ricci Flow ([MT]).

\subsection{Injectivity Radius and Collapsing of Balls}


\begin{dfn} (Injectivity Radius). The injectivity radius at $x \in M$ is defined as the maximal diameter of a ball in the tangent space at $x$ such that it is diffeomorphic to its image via the exponential map. \end{dfn}

\begin{dfn} ($\kappa$-non collapsed).  Let $B(x,r)$ be the ball of metric radius $r$ centred at $x$.  Then if $\norm{Rm} \leq r^{-2}$ within $B(x,r)$, then we say that $B(x,r)$ is $\kappa$-non collapsed if $vol(B(x,r)) \geq \kappa r^{n}$, where $n$ is the dimension of $B(x,r)$. \end{dfn}

The importance of non-collapsing is emphasised in the following result:

\emph{Result}. If $\norm{Rm} \leq r^{-2}$ on $B(x,r)$ and $B(x,r)$ is $\kappa$-non collapsed then $inj_{x}(M) \geq C(r,\kappa)$.

In particular, this type of control over the injectivity radius will be important, for the following reason:

\begin{thm} (Cheeger-Gromov).  Any sequence $(M_{i}^{n},g_{i},p_{i})$ with $(M_{i}^{n},g_{i})$ complete, $\norm{Rm(g_{i})} \leq C$, $\norm{\nabla^{k}Rm(g_{i})} \leq C_{k}$ and $inj(g_{i}) \geq \epsilon > 0$ has a subsequence which converges in the pointed sense. \end{thm}

I should of course mention:

\begin{dfn} (Pointed Convergence).  We say that a sequence $(M_{i},g_{i},p_{i})$ (where $M_{i}$ are manifolds, $g_{i}$ are metrics and the $p_{i}$ are points in the $M_{i}$) converges to $(M,g,p)$ in the \emph{pointed} sense if there exists

\begin{itemize}
\item[(i)] an increasing sequence of compact $\Omega_{i} \subset M$ with $\cup_{i}\Omega_{i} = M$ $(p_{i} \in \Omega_{i} \forall i)$, and
\item[(ii)] a sequence of maps $\phi_{i} : \Omega_{i} \rightarrow M_{i}$ which are diffeomorphisms onto their image, with $\phi_{i}(p) = p_{i}$ such that

\begin{center} $\phi_{i}^{\star}g_{i} \rightarrow_{C^{\infty}} g$ \end{center}
\end{itemize} \end{dfn}

What Perelman was able to show was the following result:

\begin{thm} (Perel'man). Let $g(t)$ be a solution of Ricci flow on a compact $M^{n} \times [0,T]$. If $p \in M$, $r > 0$ s.t. $\norm{R} \leq \frac{1}{r^{2}}$ on $B_{r}(p)$ at $t = T$ then $\exists \eta = \eta(n, g(0),T)$ such that $inj(g(T),p) \geq \eta r$. \end{thm}

In particular, this means that, for the Ricci flow on compact manifolds, the Cheeger-Gromov result applies.  In other words, we have a compactness result for such flows- which allows one to identify the original manifold as a connected sum of its limiting pieces after finitely many surgeries, which essentially proves Thurston's Geometrisation Conjecture.  I will repeat many of these statements later.

\subsection{Canonical Neighbourhoods}

Canonical neighbourhoods are of paramount importance in the Ricci flow, particularly in Ricci flow with surgery.  They reflect what happens around singularities of the flow as the manifold gets "stretched", as in MCF with surgery from before.

\emph{The $\epsilon$-neck}.  An $\epsilon$-neck about a point $x \in (M,g)$ is a submanifold $N \subset M$ and a diffeomorphism $\psi : S^{2} \times (-\epsilon^{-1},\epsilon^{-1}) \rightarrow N$ such that $x \in \psi(S^{2} \times \{ 0\})$ and such that the pullback of the rescaled metric, $\psi^{\star}(R(x)g)$ is within $\epsilon$ in the $C^{\epsilon^{-1}}$ topology to the product of the round metric of scalar curvature $1$ on $S^{2}$ with the usual metric on $(-\epsilon^{-1},\epsilon^{-1})$.  This is really a roundabout way of saying that we have the following picture:

\begin{center}\scalebox{0.8}{\includegraphics{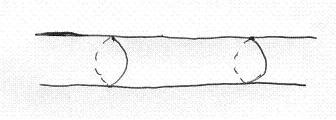}}\end{center}

\emph{Sidenote}: The $C^{\epsilon^{-1}}$ topology is the topology of maps which are $\epsilon^{-1}$ times differentiable.

\emph{The $\epsilon$-cap}.  An $\epsilon$-cap is intuitively a $\epsilon$-neck which is "rounded off" at one end, as shown:

\begin{center}\scalebox{0.8}{\includegraphics{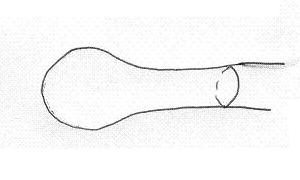}}\end{center}

These are really the only essential canonical neighbourhoods.  However, for a proper analysis, we need to include two more:

\emph{$C$-components}. A $C$-component is a compact, connected Riemannian manifold of positive sectional curvature diffeomorphic to $S^{3}$ or $RP^{3}$ such that if the metric is rescaled so that $R(x) = 1$ at some point $x$ in the component, then diam($C$-component) $\leq C$, $C^{-1} \leq Sect(x,v,w) \leq C$ any point $x$, any nonparallel vectors $v,w$, and $C^{-1} \leq vol(C-component) \leq C$.

\emph{$\epsilon$-round components}.  An $\epsilon$-round component satisfies the property that if the metric about any $x$ in it is rescaled by $R(x)$, then the metric is within $\epsilon^{-1}$ in the $C^{\epsilon^{-1}}$ topology of a round metric of scalar curvature one.

\begin{thm} (Classification of $3$-manifolds which are unions of $\epsilon$-caps and $\epsilon$-necks).  If a $3$ manifold is a union of $\epsilon$-caps and $\epsilon$-necks, and $\epsilon$ is sufficiently small, then it must be on the following list:

\begin{center}\scalebox{0.8}{\includegraphics{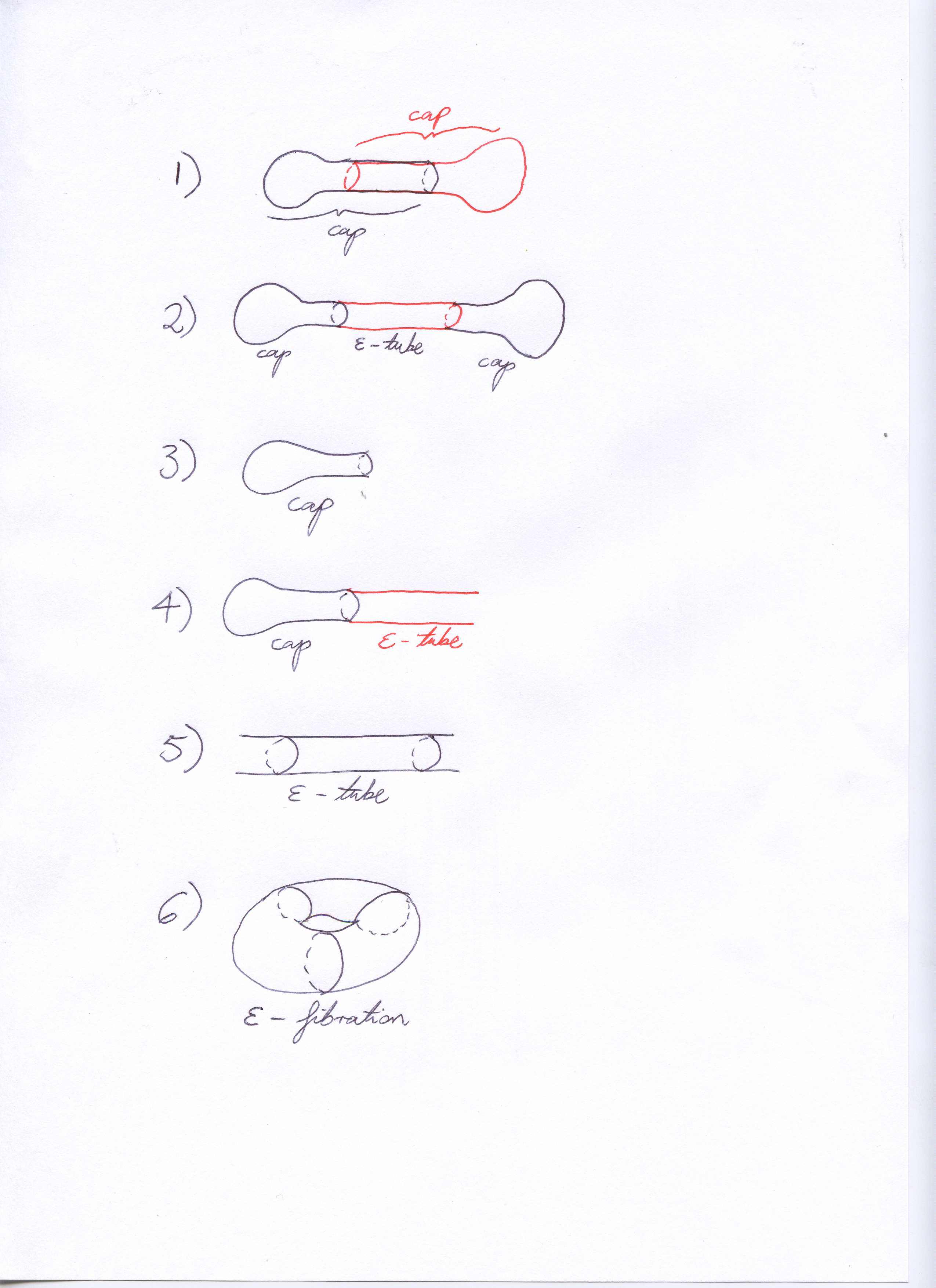}}\end{center}

Furthermore, if it is of the form (1) or (2), it is one of $S^{3}, RP^{3}$, or $RP^{3} \# RP^{3}$; if it is of the form (3) or (4), it is either $R^{3}$ or $RP^{3} - \{pt\}$; if it is of the form (5) it is $S^{2} \times R$; and finally if it is of the form (6) it is either the orientable or non-orientable $S^{2}$ bundle over $S^{1}$. \end{thm}

The proof of this theorem is not entirely trivial, but neither is it particularly novel.  However, it is, in some sense, one of the core results of the Hamilton-Perelman program.  It essentially involves a lot of careful analysis and reasoning.  The interested reader is advised to refer to the appendix of F. Morgan and G. Tian's extensive manuscript \cite{[MT]} for the details.

So we have eight possibilities for such manifolds: $S^{3}, RP^{3}, RP^{3} \# RP^{3}, R^{3}, RP^{3} - \{pt\}, S^{2} \times R, S^{2} \times S^{1}$ and the non-orientable $S^{2}$ bundle over $S^{1}$.  It will turn out that in fact the following result is true:

\begin{thm} (Thurston Geometrisation Conjecture).  Any $3$-manifold is diffeomorphic to a connected sum of the above eight so-called "model" geometries. \end{thm}

Sketching the proof of this theorem is what I will concern myself for the rest of this section.

\subsection{Connection to the Ricci flow with surgery}

The key to proving the above result is to start off with an arbitrary manifold $(M,g_{0})$, and then flow it by Ricci flow until it looks nicer.  Singularities will develop; these can be controlled and understood in terms of the model geometries.  Precise maps can be made as $\epsilon$-necks and $\epsilon$-caps get very long and thin ($\epsilon$ gets very very small) that parametrise the manifold around these singularities.  Exactly analogous to MCF with surgery, we may then perform surgery on our manifold, and then continue the flow on the new pieces.  This process is to be performed inductively.

Issues (just as with the MCF) are:

\begin{itemize}
\item[(i)] Do things improve after the surgery?
\item[(ii)] What is to stop us from having infinitely many surgeries in a finite amount of "time"?
\item[(iii)] How can we be sure the flow is "smooth enough" after previous surgeries when we wish to perform future surgeries?
\item[(iv)] Will we always get sensible limiting behaviour as a result of this process?
\end{itemize}

The standard way to address (i) is to have a priori estimates and show that one can perform surgery in such a way that these estimates are preserved.  To show that we only have a finite number of surgeries in a finite amount of "time", the idea is, crudely speaking, to observe that the process of surgery can be arranged to remove a fixed amount of volume from the manifold, and, furthermore, that each further surgery in a bounded interval must remove a fixed amount of volume \emph{bounded below} by some constant.  We may hence conclude that an infinite number of surgeries in a finite "time" interval for a manifold starting with finite volume is impossible.  (iii) may be overcome by making careful arguments using the tools of geometric measure theory and the smoothing property of parabolic equations of heat type.  Really the core trouble is with (iv).  In fact, we have the following result, due to Cheeger, Gromov and others, which I mentioned before:

\emph{Result}.  A Geometric Limit of a sequence of manifolds $(M_{n},g_{n}(t), x_{n})$ exists if

\begin{itemize}
\item[(i)] we have uniform non-collapsing at $x_{n}$ in the time zero metric, and
\item[(ii)] for each $A < \infty$ we have uniformly bounded curvature for restriction of the flow to metric balls of radius $A$ about base points.
\end{itemize}

If the above result holds, then we can conclude that the limit of a sequence of manifolds really does look like the solution of its renormalised Ricci flow, and the question of classification of these limits comes down to classifying gradient solitons.  But we have the final result that gradient solitons are unions of canonical neighbourhoods, in particular $\epsilon$-necks and $\epsilon$-caps, and we are done.

So evidently the key is to establish that (i) and (ii) of the above result hold for geometric flows, otherwise the whole process of surgery will not work.  In particular, horrible solutions like Hamilton's "sheet of cigars" might crop up and throw a spanner in the works.  It took Perelman with his notion of Length or "Entropy" to show by contradiction arguments that (i) in fact holds (relation to the Gradient estimate from before).  (ii) is essentially a consequence of the standard a priori estimates due to Hamilton.

Before I go into some more detail about Perelman's  length and so forth, I make the following remark- this whole argument could be simplified considerably if we observe that the Ricci flow is steepest descent for the Fisher Information of a sharp Riemannian $3$ manifold \emph{without masses}, and the corresponding renormalised Ricci flow is steepest descent flow for the Physical Information of the same.  Then we get automatically that the classification of the limits of such flows is the classification of gradient solitons.  But perhaps some subtlety escapes me, most particularly because we are not only flowing manifolds smoothly, we are also performing \emph{surgery}; so I shall continue with sketching the standard treatment.

\subsection{Perelman's Length or "Entropy"}

The $W$-entropy of Perelman is defined as

\begin{center}
$W(g,f,\tau) = \int (\tau(R + \norm{\nabla f}^{2}) + f - n)(4\pi \tau)^{-\frac{n}{2}}e^{-f}d\mu(g)$
\end{center}

\begin{prop} If $\frac{\partial}{\partial t}g_{ij} = -2R_{ij}$, $\frac{\partial \tau}{\partial t} = -1$ (ie $\tau \sim (T - t)$), and $\frac{\partial f}{\partial t} = - \Delta f + \norm{\nabla f}^{2} - R + \frac{n}{2\tau}$, then

\begin{center}
$\frac{\partial}{\partial t}W(g(t),f(t),\tau(t)) = 2\tau \int \norm{R_{ij} + \nabla_{i}\nabla_{j}f - \frac{1}{2\tau}g_{ij}}^{2}(4\pi \tau)^{-\frac{n}{2}}e^{-f}d\mu(g) \geq 0$
\end{center}

Furthermore, equality is acheived if and only if one has a gradient Ricci soliton. \end{prop}

\emph{Remarks}.
\begin{itemize}
\item[(1)] This is precisely the monotonicity formula for the rescaled Ricci flow.  Furthermore, the $W$ entropy is essentially a generalisation to rescaled Ricci flow of the Green's function result

\begin{center}
$R(x_{0},t_{0}) = \int_{M_{t}}(t - t_{0})^{-n/2}exp(-\frac{\norm{x - x_{0}}^{2}}{4(t_{0} - t)})R(x,t)d\mu(g_{t})$
\end{center}
\item[(2)] Note that $\tau(R + \norm{\nabla f}^{2}) + f - n$ is constant on solitons for the above choices for $f$ and $\tau$.
\end{itemize}

Alternatively, we may define the related $\emph{Length}$ of Perelman:

\begin{center}
$L(\gamma) = \int_{0}^{\bar{\tau}}\sqrt{\tau}(R(\gamma(\tau)) + \norm{\gamma'(\tau)}^{2})d\tau$
\end{center}

defined for a "timelike" path $\gamma : I \rightarrow M \times I$, such that $\gamma(\tau) \in M \times \{ t - \tau \}$, and $\gamma(0) = x$.  We say such a path is "parametrised by backwards time".

\begin{rmk} We immediately see a similarity between this quantity and the physical information.  I suppose one could roughly interpret $L(\gamma)$ as "the physical information of a path $\gamma$ through the manifold $M \times I$". \end{rmk}

A whole theory of $L$-geodesics can then be developed.  In particular, we can show using this functional that the Result in the previous section is true, by constructing an a posteriori gradient estimate for Ricci flow with surgery.  In particular, Perel'man used these quantities to deduce his theorem which I mentioned earlier:

\begin{thm} (Perel'man). Let $g(t)$ be a solution of Ricci flow on a compact $M^{n} \times [0,T]$. If $p \in M$, $r > 0$ s.t. $\norm{R} \leq \frac{1}{r^{2}}$ on $B_{r}(p)$ at $t = T$ then $\exists \eta = \eta(n, g(0),T)$ such that $inj(g(T),p) \geq \eta r$. \end{thm}

\begin{proof} (sketch). The key step is to show that $vol(B_{r}(p)) \geq \xi r^{n}$ for some $\xi = \xi(n, g(0), T)$.  We define $\mu(g,t) = inf \{W(g,f,\tau) \vert \int (4\pi \tau)^{-n/2}e^{-f}d\mu = 1 \}$ where $W$ is the $W$-entropy as defined above.  Note that $\mu(g,t)$ will increase strictly with $t$ unless $(M,g)$ is a soliton.

To establish the volume estimate, we construct an argument by contradiction- so assume that $vol(B_{r}(p))$ is very small.  Then we can find $f_{i}$ such that $\int (4 \pi \tau)^{-n/2}e^{-f}d\mu = 1$, but $W(g,f_{i},\tau) \rightarrow - \infty$.  But this gives a contradiction to the above observation, which completes the argument.
\end{proof}


\chapter{Statistical Geometry}

\section{Preliminary Definitions}

\subsection{Motivation}

Semi-Riemannian Geometry has an issue in dealing with particles, since it fails to treat them as objects with some degree of uncertainty as to their position; it treats them as points.  To deal with this issue, I introduce the structure on a differential manifold of a space of metrics, with a weighting function defined at each point signifying precisely how important each metric is there.

We are first and foremost interested in generalising the notion of a particle path.  For instance, consider the following picture:

\scalebox{0.8}{\includegraphics{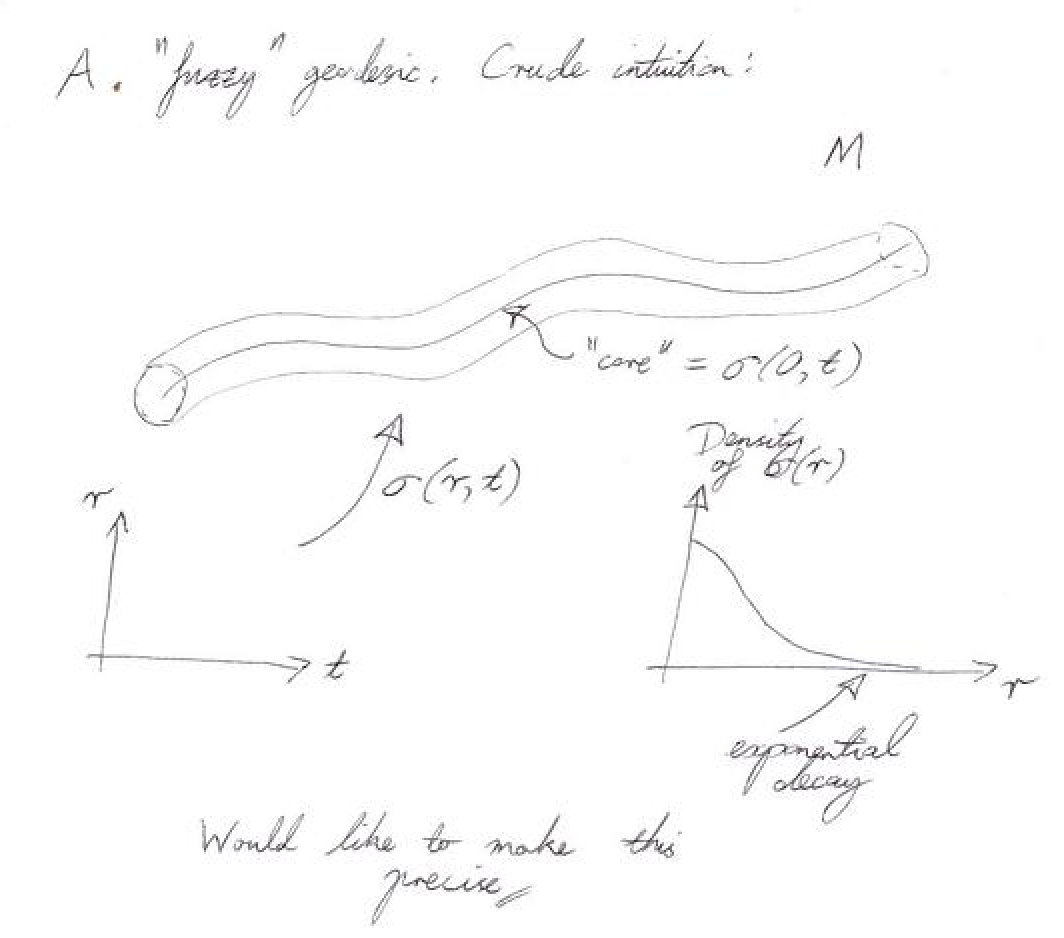}}

We would like somehow to be able to model as particles objects that have measure normalised to one across each spacelike hypersurface of our manifold, with the bulk of the measure concentrated about a world tube.  I will be more specific - we would like the measure to maximise at the centre of this tube, ie about some optimal geodesic path, but also somehow model uncertainty in position by having exponential falloff relative to the core geodesic.

It is necessary to construct some sort of machinery to handle this problem.  In order to do this, we consider the notion of a discrete statistical manifold, which is a discrete space $M$ together with a discrete distribution space $A$ such that at each $m = m_{0} \in M$ there are a finite set $a_{1},...,a_{k}$ such that there are paths from $m$ to $m_{1}(a_{1}),...,m_{1}(a_{k}) \in M$.  From these points there are paths to points $m_{2}(m_{1}(a_{1}),a_{1}), m_{2}(m_{1}(a_{1}),a_{2}),...$ etc.  The signal function now is a function from $M$ to $R$ such that, for fixed $m$, $\int_{A}f(m,a)da = 1$.  It assigns to each path from $m_{0}$ to the $m_{1}(a_{i})$ the weight, or probability, of that path being selected.

Diagrammatically:

\scalebox{0.8}{\includegraphics{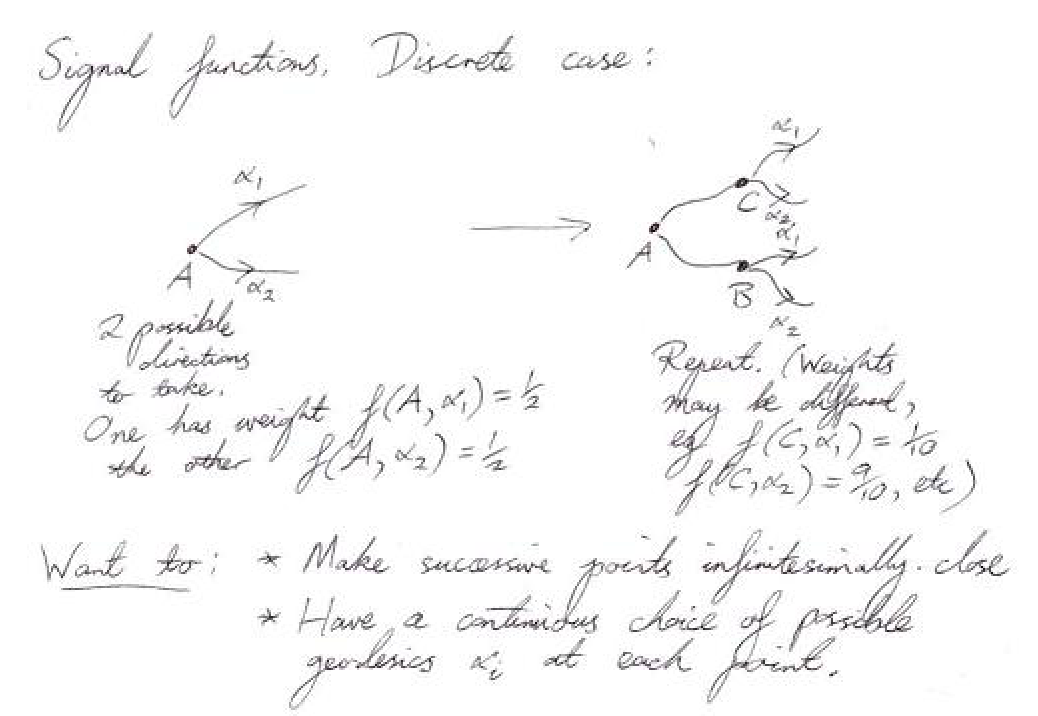}}

The trick will be to find a way to move from discrete spaces $M$ and $A$ to natural smooth manifolds, and reduce the distance between consecutive steps $m_{0}$, $m_{1}$ etc along each path to zero.  This will be the focus of the next subsection.

Also we would like to define $M$ not only to be a metric space, but a measure-metric space, with measure $\psi$.  Conservation of probability demands that the signal function and the measure be related in the following way:

\scalebox{0.8}{\includegraphics{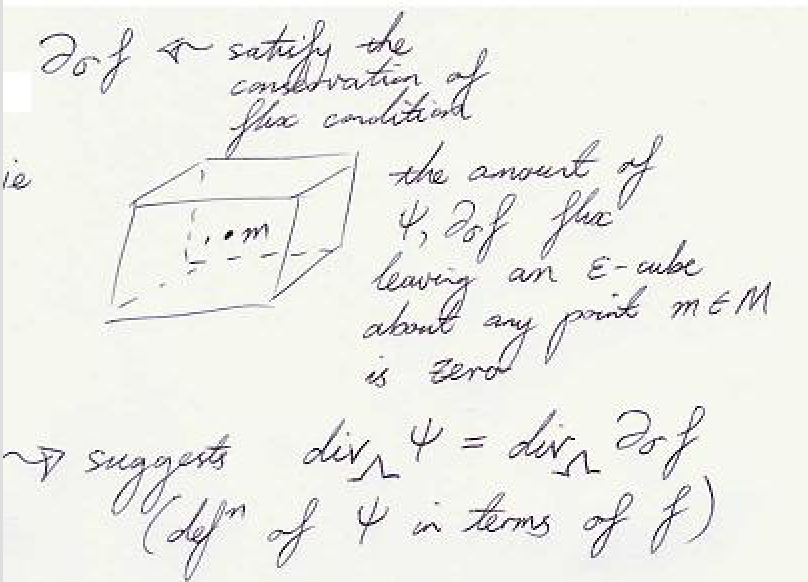}}

Ultimately we will like to be able to represent particles in a sensible way as geometric objects, intuitively as follows:

\scalebox{0.8}{\includegraphics{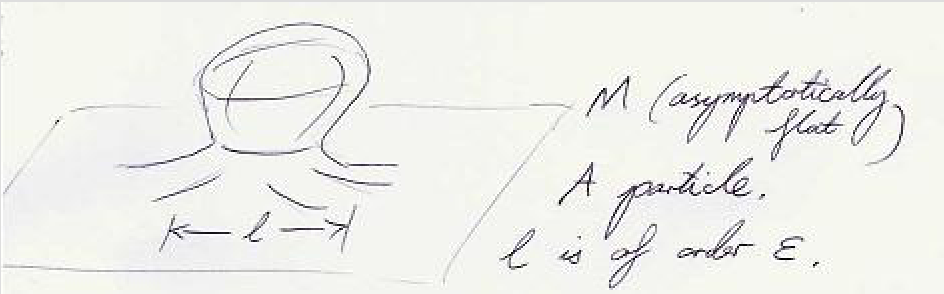}}

I shall now make this all more precise.

\subsection{Introduction}

\begin{dfn}
A fuzzy Riemannian manifold $(M,A,f)$ is a differential manifold $M$ together with a Riemannian manifold $A$, together with a smooth non-negative function $f : M \times A \rightarrow R$.

Each point in $A$ is an $n$ by $n$ matrix corresponding to an inner product on $R^{n}$.  We might make further assumptions on the structure of $A$.  For instance we will assume for now that each inner product is symmetric, and that all points in $A$ have the same index.  Finally we will assume that as we vary $m \in M$ the inner product $g(m,a)$ corresponding to $a \in A$ will vary smoothly; ie $g(.,a)$ will be a metric on $M$.  This correspondence must be \emph{smooth}, and all the metrics must be of constant index throughout $M$. 

Furthermore, the subspace of the space of metrics with fixed index corresponding to the set $A$ must be connected (this is to avoid issues in applications with spacelike directions becoming timelike under different metrics, ie to avoid pathological problems with causality).  $grad_{a}$, $div_{a}$, $curl_{a}$ and $\Delta_{a}$ shall from now on represent the usual differential operators with respect to the metric corresponding to $a$.  Evidently we will need to integrate over $A$ in order to take all of these operators into account when we are using them, and they will be weighted in importance by our distribution function $f$.

$f$ is defined so as to have the following two properties:

\textbf{The Normalisation Condition}
\begin{center}
 $\int_{A}f(m,a)\sqrt{det(h(a))}da = 1$, for all $m \in M$
\end{center}

\textbf{Divergence Free Flow Condition}
\begin{center} $\int_{A \times A \times A}f(m,b)div_{b}(f(m,c)grad_{c}(f(m,a)))\sqrt{det(h(a))det(h(b))det(h(c))}dbdcda = 0$\\ for all $m \in M$
where $h$ is the metric on $A$
\end{center}

However, more generally the second property can be relaxed.  When I go on to define the coupling condition for a fuzzy manifold the significance of this will become more apparent.

This function shall from now on be referred to as the \emph{signal function}.  It is by no means unique- there is nothing to stop one from choosing several different signal functions for the same choice of $M$ and $A$.

The role of the signal function (provided it satisfies the conditions above) is fairly straightforward. $f(m,a)$ quantifies the importance of metric $g(.,a)$ at $m \in M$ by assigning a number to it between $0$ and $1$ at that point.

\end{dfn}

For instance, a standard Riemannian manifold is a fuzzy Riemannian manifold with $A$ being a point space, $A \mapsto g$ where $g$ is a particular metric on $M$, and $f(m,A) = 1$ for all $m \in M$.

We define the differential operators on a fuzzy semi-Riemannian manifold in the following way.

\begin{dfn}

If $K : M \times A \rightarrow \mathcal{R}$ and $\psi : M \times A \rightarrow TM$ are general functions of the given type on the fuzzy semi-Riemannian manifold $\Lambda = (M,A,f)$, then

\begin{center}
$grad_{\Lambda}K(m,a) = \int_{b \in A}f(m,b)grad_{b}K(m,a)db$ \\
$div_{\Lambda}\psi(m,a) = \int_{b \in A}f(m,b)div_{b}\psi(m,a)db$ \\
$curl_{\Lambda}\psi(m,a) = \int_{b \in A}f(m,b)curl_{b}\psi(m,a)db$
\end{center}

\end{dfn}

Note now that the divergence free flow condition for the signal function is equivalent under this new notation to

\begin{center}
$\int_{A}div_{\Lambda}grad_{\Lambda}f(m,a)da = 0$ for all $m \in M$
\end{center}

\subsection{Particles and mass flux}

Clearly we want to be able to define some notion of particle on this structure and map out its behaviour.  We want the notion of particle to somehow be compatible with our signal function, in that we want its motion to also be divergence free (we want stuff to have divergence free flow because if it didn't, that would mean stuff is entering a region and not exiting or emerging from a region without stuff going in in the first place).  We also want, given some specification of values for our mass distribution on the boundary of a compact neighbourhood, a precise idea of how the mass distribution filters through that neighbourhood.  In other words, we want to define some sort of generalisation of geodesic flow/geodesics.

This leads us to define

\begin{dfn} A mass distribution $\psi : M \times A \rightarrow TM$ on a fuzzy riemannian manifold $(M,A,f)$ is a distribution that is \emph{compatible} with the signal function $f$, that is, for all $(m,a,b,c) \in M \times A \times A \times A$:

\begin{equation}
\label{fuzzy1}
f(m,c)div_{b}\psi(m,a) = div_{b}(f(m,c)grad_{c}(f(m,a)))
\end{equation}

This makes sense because we obviously want $\psi$ to also satisfy a more general divergence free condition:

\begin{center}
$\int_{A \times A}f(m,b)div_{b}\psi(m,a)\sqrt{det(h(a))det(h(b))}d(a \otimes b) = 0$ for all $m \in M$
\end{center}

which is notationally equivalent to (ignoring determinant factors) the following conservation equation

\begin{equation}
\label{conservation}
\int_{A}div_{\Lambda}\psi(m,a)da = 0,
\end{equation}
\begin{center}
for all $m \in M$.
\end{center}

In fact, as it will turn out later, this notion of compatibility is precisely what is required for $\psi$ to correspond to a notion of generalised momentum (see the section on interesting examples).

Also, a mass distribution must satisfy the property that

\begin{center}
$\norm{\psi(m,a)}^{2}_{(m,a)} \geq 0$ for all $(m,a) \in M \times A$.
\end{center}

as well as

\begin{center}
$\norm{\psi(m,a)}_{(m,a)} = 0$ iff $\psi(m,a) = 0$
\end{center}

One could think of the first of these as a "reality" or causality condition on $\psi$.  For if $\norm{\psi(m,a)}^{2}_{(m,a)} < 0$ for some $(m,a)$, then $\norm{\psi(m,a)}_{(m,a)}$ is purely imaginary, so cannot be interpreted in any physically meaningful way.  Note that in general this could happen since the space of metrics we are dealing with might have nonzero index.  This will become important later on (rather soon actually) when I define length.  The second condition is a positivity condition.
\end{dfn}

Note the conservation equation is still an extraordinarily strong condition.  Intuitively it amounts to choosing a particular choice for $\psi$ so as to make the solution unique.  Also note that, once again, the notion of compatibility can be generalised.  For instance, the compatibility equation could be written as:

\begin{equation}
\label{coupling}
div_{b}\psi(m,a) = \Delta_{b}f(m,a)
\end{equation}

For the purposes of ease of calculation I will in fact use this equation first and use this as a toy example to demonstrate how one derives the secondary coupling condition as a consequence.   I will then proceed to use the previous, more physically motivated equation and proceed along similar lines to derive its secondary coupling condition.

Let $\mathcal{F}$ be the space of functions $\{ f | f : M \times A \rightarrow TM \}$.

\subsection{Fuzzy geodesics}

Define length as follows:

\begin{dfn}

The length $L : \mathcal{F} \times \{ U | U \subset \subset M \} \rightarrow \mathcal{R}$ is defined by
\begin{center}
$L(\psi,U) = \int_{m \in U}\int_{a \in A}\parallel\psi(m,a)\parallel_{(m,a)}\sqrt{det(g(m,a))det(h(a))}dadm$
\end{center}
where $\parallel.\parallel_{(m,a)}$ denotes the norm with respect to metric $g(.,a)$ at the point $m \in M$.

\end{dfn}

It turns out that, in order for the flow to have locally extremal length, $\psi$ and $f$ must satisfy the additional coupling condition for all $(m,a) \in M \times A$:

\begin{equation}
\label{fuzzy2}
\nabla_{(\psi(m,a) - grad_{b}f(m,a);b)}\frac{\psi(m,a)}{\norm{\psi(m,a)}_{m,a}} = 0
\end{equation}

Here $\nabla_{(.,a)}$ is the Levi-Civita connection with respect to metric $g(.,a)$.

Proof.  Let $L(\phi,U)_{s} = \int_{A}\int_{U}\norm{\phi(m,a,s)}_{(m,a)}dmda$ be a variation of $L(\phi,U)$.

Then

\begin{align}
\frac{\partial}{\partial s}L(\phi,U)_{s}
&= \int_{A}\int_{U}\frac{\partial}{\partial s}\norm{\phi}_{(m,a)}dmda \nonumber \\
&= \int_{A}\int_{U}1/2 \frac{ \frac{\partial}{\partial s}\norm{\phi}^{2} } {\norm{\phi}}dmda \nonumber \\
&= \int_{A \times A}\int_{U} <(\nabla_{\frac{\partial}{\partial s};b)}\phi(m,a),\phi(m,a)>/\norm{\phi(m,a)}dmda \end{align}

The integral is over $A \times A$ now since information transfer throughout $A$ is instantaneous and so we must take into account all the different covariant derivatives for each metric.  Now since $div_{b}\phi(m,a,s) = \Delta_{b} f(m,a)$ we can write $\phi(m,a,s) = grad_{b}f(m,a) + curl_{b}K(m,a,s)$ for some vector field $K$.  Note that $f$ is fixed in the variation; only $K$ is a function of $s$.  Therefore $\frac{\partial}{\partial s}$ is a vector only in the direction of the curl of $\phi$.

So...

\begin{align}
<\nabla_{(\frac{\partial}{\partial s};b)}\phi,\phi/\norm{\phi}> &= <\nabla_{(v;b)}(\phi - grad_{b}f), \phi/\norm{\phi}> \text{for arbitrary } v \nonumber \\
&= <\nabla_{(\phi - grad_{b}f;b)}v, \phi/\norm{\phi}> \text{since the connection is symmetric } \nonumber \\
&= -< v, \nabla_{(\phi - grad_{b}f;b)}\phi/\norm{\phi} > + (\phi - grad_{b}f)< v , \phi/\norm{\phi}> \end{align}

Let us consider the second term.  Set $G = curl_{b}K$.

\begin{align}
\int_{A}\int_{U}\nabla< v, \phi/\norm{\phi} > . (\phi - grad_{b}f) &= \int_{A}\int_{U}\nabla<v , \phi/\norm{\phi}> . G \nonumber \\
&= \int_{A}\int_{U}\nabla . (G< v, \phi/\norm{\phi}>) - \int_{A}\int_{U}< v , \phi/\norm{\phi} >\nabla . G \nonumber \\
& \text{(by a vector identity).} \end{align}

Now the first term vanishes by the divergence theorem and the fact that $v$ vanishes on the boundary.  The second term vanishes since div curl is the zero operator.

Hence

\begin{center}
 $\frac{\partial}{\partial s} L (\phi(m,a,s),U)|_{s = 0} = -\int_{A}\int_{U} < v, \nabla_{(\phi - grad_{b}f;b)}\frac{\phi}{\norm{\phi}} >dmda = 0$
 \end{center}
 
 And since $v$ is arbitrary, by the fundamental theorem of the calculus of variations, $\nabla_{(\phi(m,a) - grad_{b}f(m,a);b)}\frac{\phi(m,a)}{\norm{\phi(m,a)}}_{m,a} = 0$.  Q.E.D.

Now, what about the more physically motivated compatibility equation from before, \begin{latexonly} $\eqref{fuzzy1}$? \end{latexonly}

Well certainly we must have $div_{b}\psi(m,a,s) = \frac{1}{f(m,c)}div_{b}(f(m,c)grad_{c}f(m,a))$ for some variation $\psi(m,a,s)$ of $\psi(m,a)$.  Since the right hand side is not a function of $s$, we must have $\psi(m,a,s) = curl_{b}A(m,a,s) + F(m,a)$ for some good choice of $F$ (Note also that the left hand side is not a function of $c$, so the right hand side must exhibit some form of cancellation).  Such a good choice of $F$ is any solution of the equation $div_{b}F(m,a) = \frac{1}{f(m,c)}div_{b}(f(m,c)grad_{c}f(m,a))$.  Then if we define $\gamma(m,a,b) = curl_{b}A(m,a,0)$, the resulting condition for the length being minimal that results is

\begin{equation}
\label{fuzzy3}
\nabla_{\gamma(m,a,b)}\frac{\psi(m,a)}{\norm{\psi(m,a)}_{(m,a)}} = 0
\end{equation}

as can be easily verified by following the same procedure as for the other example.
\begin{latexonly}
A pair $(f,\psi)$ satisfying  $\eqref{fuzzy1}$ and $\eqref{fuzzy3}$ I shall from now on refer to as a fuzzy geodesic.  Note that $\eqref{fuzzy1}$ and $\eqref{fuzzy3}$ can be interpreted as describing a coupling between the curvature and mass distributions on the manifold $M$.
\end{latexonly}
\section{Existence of solutions}

\subsection{Stability}

It is necessary to be careful, since we are only extremising length, not necessarily minimising it.  For a solution to be minimal, we hence require the additional condition that for a two parameter variation of $\psi$, $\Psi : [-\epsilon,\epsilon] \times [-\delta,\delta] \times M \times A \rightarrow TM$ (where $\Psi(0,0) = \psi$), that

\begin{center}
$\frac{\partial^{2}L}{\partial u \partial v}(\Psi(u,v),U)|_{u = 0, v = 0} \geq 0$.
\end{center}

But even this is beset with problems.  For if any of the second variations are in fact $0$ we cannot conclude anything about stability and must examine higher derivatives.  In particular we would need all the third order variations to be $\geq 0$.  But then if a single third order variation is $0$, we have to look further, etc.  However, since the domains we are dealing with are compact, it is inherently reasonable to expect that we should be able to stop this process after looking at a finite number of derivatives.  So we have:

\begin{conj}  In order to ascertain whether a solution to this process is stable over a compact domain we only need to look at a finite number of derivatives. \end{conj}

\begin{proof} (Idea).  Intuitively if all the variational derivatives up to a certain order $n$ vanish the solution is flat to a certain extent and in order to vary sufficiently to satisfy the boundary conditions it becomes apparent that the diameter of the domain must exceed some certain parameter $C(n)$.  It is obvious that $C(n)$ must increase strictly with $n$.  Hence there will be some natural number $N$ such that for all $n \geq N$, $C(n)$ is greater than the diameter of the domain.  Hence it is only necessary to look at the variations of degree less than $N$.   \end{proof}

\begin{rmk} In order for the above proof to work it may need to be established that all the induced P.D.E.s from looking at $n^{th}$ order variations have analytic solutions.  This is something that we should expect from this problem, however, since it is motivated by physical considerations. \end{rmk}

\begin{dfn}
A domain $\Omega$ is called \emph{fundamental} if it suffices to look at only the first and second order derivatives to determine stability, i.e. if $C(3) > diam(\Omega)$.
\end{dfn}

Let me make all this more precise.

\begin{dfn}
The diameter of a set $U \subset M$ is defined as

\begin{center}
$diam(U) = sup_{a \in A}(diam_{a}(U))$
\end{center}

where $diam_{a}(U)$ is the standard diameter from semi-Riemannian geometry with respect to the metric indexed by $a$.
\end{dfn}

\begin{thm}  $U \subset M$ is topologically compact iff $diam(U)$ is finite. \end{thm}

\begin{rmk}  This could be more or less taken as a definition for the topology of $\Lambda = (M,A,f)$. \end{rmk}

\begin{conj}  Suppose $\lambda = sup_{(m,a) \in M \times A}\int_{A}f(m,b)grad_{b}f(m,a)db$ is finite, and dim(M) = n.  Then, for each $R > 0$, there is a constant $K(\lambda,n,R)$ such that, for any $U$ with $diam(U) < R$, it is sufficient to determine stability to the local length minimisation problem in $U$ by looking at $q$-parameter variations of order up to $q = K(\lambda,n,R)$. \end{conj}

\emph{Remarks}.

\begin{itemize}
\item[(i)] This is in accordance with intuition, because if it was necessary to look at an infinite number of variations within a compact domain, somehow our solution would possess an infinite amount of information in a bounded domain, which, needless to say, is unphysical.
\item[(ii)] Note that $K(\lambda,n,R)$ need not be an integer.  In fact, in analogy with fractional differentiation as defined before, it makes perfect sense to be able to make variations of noninteger order.  Note that this level of sophistication is probably unnecessary, since physical functions tend to be $C^{\infty}$ anyway.
\item[(iii)] The global bound on the variability of $f$ is essentially a global bound on the curvature of $\Lambda = (M,A,f)$.  In other words, I am roughly saying, "provided $\Lambda$ does not curve too much, we can put a bound on the number of variations we need to make".
\end{itemize}

\subsection{Characterisation of the solution space}

There will be different locally minimal solutions for a given set $U$ with signal function $f$.  These solutions for $\psi$ will be discrete in $\mathcal{F}$ since $U$ is compact.  We could loosely think of these in analogy with eigenfunctions.  It is possible that there may be families of these solutions; if so, we could loosely think of the parameters indexing the families as eigenvalues.

We have the following conjecture, which is in two parts:
\begin{latexonly}
\begin{conj} Given a set $U$, complete knowledge of $f$ and knowledge of the boundary conditions $\psi|_{\partial U}$, $\psi$ is uniquely determined in $U$ by $\eqref{fuzzy1}$ and $\eqref{fuzzy2}$.  Furthermore, the solution $\psi$ can be written in terms of the minimal solutions for the length function $L$ on $U$. \end{conj}

Though it is possible that both parts of this conjecture are incorrect, and we really should be looking at the following problem:

\emph{Problem}. Given a set $U$ and complete knowledge of $f$ in $U$, compute all possible solutions $\psi$ can take in $U$ such that the length function $L(\psi,U)$ with fixed boundary conditions $\psi|_{\partial U}(m,a) = g(m,a)$ is locally minimised in $\mathcal{F}$.

The influence of boundary conditions mean that $\eqref{fuzzy1}$ and $\eqref{fuzzy2}$ will no longer suffice to determine the correct solutions inside $U$ for $\psi$, and we will have to derive new relations that $\psi$, $g$ and $f$ must jointly satisfy.
\end{latexonly}
If solutions are nonunique, I posit the further

\begin{conj} Suppose we have the conditions of the conjecture immediately above.  Suppose furthermore that we know $\psi$ precisely in some neighbourhood of $U$ in $M - U^{int}$.  Then $\psi$ is determined precisely inside $U$. \end{conj}

If this conjecture is correct, this leads one on to being able to define the relative probability of a solution occuring inside $U$ in terms of random noise outside.

\begin{dfn}
The "probability that a solution $S$ will occur" is interpreted as the relative influence of random noise outside a countably infinite number of copies of $(U,f)$ in determining it.  Furthermore, if we have two solutions, $A$ and $B$, with $length(A) < length(B)$, we should reasonably expect $probability(A) > probability(B)$.  How one would go about computing these probabilities as a function of $U$, $f$ and state $S$ might be as follows.

Define $N_{U}$ to be the space of solutions for $\psi$ in $U$.

Define $L(\psi|_{U}, K)_{avg}$ to be the average of the length induced from all possible extensions of $f$ to $K \subset M$ from $U \subset K$ (since we are dealing with copies of $U$ only, so $f$ may change outside $U$ and hence be effectively random), given a particular solution $\psi|_{U}$.  (Note that if the above conjecture is true then $\psi$ will in fact extend uniquely to all of $K$ once we know $f$.) Then define the normalised length of solution $\psi|_{U}$ as

\begin{center}
 $\hat{L}(\psi|_{U}) = lim_{\alpha \rightarrow \infty}L(\psi|_{U},K_{\alpha})_{avg}/\sum_{S \in N_{U}}L(S,K_{\alpha})_{avg}$ 
\end{center}

 where $K_{\alpha}$ are a series of expanding sets that envelop all of $M$ in the limit (we are of course making the assumption that $M$ is paracompact).
 
We now want to define a notion of probability such that if $\hat{L}(S_{1}) = n\hat{L}(S_{2})$, then $P(S_{2}) = P(S_{1})/n$.  Certainly $1/\hat{L}(S_{2}) = n/\hat{L}(S_{1})$, so define the probability of solution $\psi|_{U}$ occuring as being proportional to $1/\hat{L}(\psi|_{U})$.  In particular,

\begin{center}
$P(\psi|_{U}) = \frac{1}{\hat{L}(\psi|_{U})}/\sum_{S \in N_{U}}\frac{1}{\hat{L}(S)}$
\end{center}
\end{dfn}

\begin{conj}  This corresponds to the real probability that a solution will occur. \end{conj}

In order to understand the types of solutions that may occur, we have the following conjecture:

\begin{conj}  Stable solutions (ie first variation of length is zero, second variation of length is positive) are characterised by geometrical symmetries.  In other words, there is a group of symmetries acting on the space that preserves the solution. \end{conj}

Which leads us to

\begin{conj}  There is a 1-1 correspondence between symmetry groups acting on $(M,A,f)$ and stable solutions for $\psi$ in $(M,A,f)$. \end{conj}

Other examples of the variation problem might be in modeling the excitation of an object by an incoming packet of energy, or modeling the possible states leading up to a particular distribution of mass/energy.  This motivates the more general idea of considering the mixed problem where part of the boundary conditions are fixed and the other values for $\psi$ on the boundary are allowed to vary, then trying to find stable minimal solutions for the length over the compact domain.

Yet another generalisation that can be made to this variational problem is by first associating to each $a \in A$ a $U(a) \subset M$ such that the correspondence $a \mapsto U(a)$ is smooth.  Denote this particular distribution of sets by $\mathcal{D}(A)$.  Then define the generalised length to be

\begin{center}
$L(\psi,\mathcal{D}(A)) = \int_{a \in A}\int_{m \in U(a)}\parallel\psi(m,a)\parallel_{(m,a)}\sqrt{det(g(m,a))det(h(a))}dmda$
\end{center}

Then solve the same problems as before while extremising this notion of length.

\section{Two classes of Signal Functions and related results}

\subsection{"Sharp" signal functions}

There are several interesting choices we can make for the signal function $f$.  For instance, if the physical behaviour of the manifold is sharp, we have $f(m,a) = \delta(\sigma(m) - a)$ for some cross section $\sigma : M \rightarrow A$.  In other words we have one, and only one choice of inner product defined on $T_{m}M$ for each $m \in M$.  Of course we then have $\psi(m,a) = \overline{\psi}(m,a)\delta(\sigma(m) - a)$ for some function $\overline{\psi}$.
\begin{latexonly}
Now, if we consider the compatibility condition $\eqref{fuzzy1}$ and substitute the above functional forms for $\psi$ and $f$, we get nothing other than
\end{latexonly}
\begin{center}
$div_{\sigma(m)}\overline{\psi}(m,\sigma(m)) = 0$,
\end{center}

precisely what we would want if $\overline{\psi}$ was to correspond to momentum!

Now, if we solve for $F$ in $\delta(\sigma(m) - c)div_{b}F(m,a) = div_{b}(\delta(\sigma(m) - c)grad_{c}\delta(\sigma(m) - a))$, we get that $F = 0$ is a solution.  So we get finally that

\begin{center}
$\nabla_{(\overline{\psi}(m,\sigma(m));\sigma(m))}\frac{\overline{\psi}(m,\sigma(m))}{\norm{\overline{\psi}(m,\sigma(m))}} = 0$
\end{center}

as our secondary equation, which is precisely the standard geodesic equation.  This is, once again, in agreement with classical physics- we have correspondence between the particle paths and the geodesics of the manifold.

One might now ask what sort of equations $T = \psi \otimes \psi$ might satisfy in this case.  Here of course we are thinking of $T$ as the stress energy tensor, inspired by [MTW].  Well certainly $div_{\sigma(m)}T = 0$ follows easily from $div_{\sigma(m)}\psi = 0$.  What about a secondary equation for $T$?  For this we need additional information about what it means to be a physical manifold- see the next section.

\subsection{Almost sharp signal functions}

Another interesting case is when the space of metrics is nontrivial but, in some sense, very narrow at each point in $M$.  To model this, first identify $T_{m}M$ with $M$, assuming, of course, that $M$ is complete.

Then consider

\begin{center}
$f(m,a) = \frac{exp_{m}(-\vert(\sigma_{ij}(m) - a_{ij})m_{i}m_{j}\vert/h)}{\int_{a' \in A}exp_{m}(-\vert(\sigma_{ij}(m)-a'_{ij})m_{i}m_{j}\vert/h)da'}$
\end{center}

So once again we have a choice of metric $\sigma(m)$ defined effectively at each point $m \in M$, except this time we have a distribution of approximate distance $h$ about it in the space of metrics, loosely speaking.

Since we are dealing with a narrow distribution, any ambiguity with our choice of compatibility condition is accurate up to order $(h/(\text{min eigenvalue of }  \sigma(m)))^{dim M}$ which is very small anyway, provided $h << 1$.

Interesting subcases are when $\sigma_{ij}$ is diagonal with signature $(1,1,1,-c^{2})$ or diagonal with signature $(1,1,1,1)$.

To solve for $\psi$ in this case we might just blindly hack away and try to solve, or we might try

$\psi(m,a) = \overline{\psi}(m)exp_{m}(-\vert(\sigma_{ij}(m)-a_{ij})m_{i}m_{j}\vert/h)$ for some choice of $\overline{\psi}$.

We might want to look at the case where the length scale, $L$, varies with position.  A natural choice is to modify the usual choice, $h$, by an invariant depending on the inner product at that point.  So it is sensible to look at a length scale $L(m) = h\sqrt{det(\sigma(m))}$, and look at signal functions $f$ proportional to $exp_{m}(-\frac{\vert(\sigma_{ij}(m) - a_{ij})m_{i}m_{j}\vert}{L(m)})$.

A very important signal function to look at is $f(m,a) = (1 + \epsilon \Delta_{\sigma})\delta(\sigma(m) - a)$, for $\epsilon << 1$, as shall become apparent later.  I shall now proceed to derive the conservation and geodesic equations for this particular object.

In particular, the conservation equation will be

\begin{center}
$(1 + \epsilon \Delta_{\sigma})div_{\sigma}(1 + \epsilon \Delta_{\sigma})\bar{\psi} = 0$
\end{center}

which reduces to

\begin{equation}
\label{1storderconservation}
(1 + 2\epsilon \Delta_{\sigma})div_{\sigma}\bar{\psi} = 0
\end{equation}

The geodesic equation will be

\begin{center}
$(1 + \epsilon \Delta_{\sigma})\nabla_{((1 + \epsilon \Delta_{\sigma})\bar{\psi},\sigma)}\frac{(1 + \epsilon \Delta_{\sigma})\bar{\psi}}{\norm{(1 + \epsilon \Delta_{\sigma})\bar{\psi}}} = 0$
\end{center}

After discarding terms of order $\epsilon^{2}$ or higher, this reduces to

\begin{center}
$\nabla_{(\bar{\psi},\sigma)}\frac{\bar{\psi}}{\norm{\bar{\psi}}} + \epsilon\{ \Delta_{\sigma}\nabla_{(\bar{\psi},\sigma)}\frac{\bar{\psi}}{\norm{\bar{\psi}}} + \nabla_{(\Delta_{\sigma}\bar{\psi},\sigma)}\frac{\bar{\psi}}{\norm{\bar{\psi}}} + \nabla_{(\bar{\psi},\sigma)}(\frac{\Delta_{\sigma}\bar{\psi}}{\norm{\bar{\psi}}} - \frac{\bar{\psi}\Delta_{\sigma}\norm{\bar{\psi}}}{\norm{\bar{\psi}}^{2}})\} = 0$
\end{center}

But this expression may be simplified, for we observe that the last four terms can be written as

\begin{center}
$2\Delta_{\sigma}\nabla_{(\bar{\psi},\sigma)}\frac{\bar{\psi}}{\norm{\bar{\psi}}}$
\end{center}

Hence, we have the expression for our geodesic equation for this choice of signal function:

\begin{equation}
\label{1stordergeodesic}
(1 + 2\epsilon \Delta_{\sigma})\nabla_{(\bar{\psi},\sigma)}\frac{\bar{\psi}}{\norm{\bar{\psi}}} = 0
\end{equation}

\subsection{Example of a Fuzzy Geodesic}
\begin{latexonly}
In this section I examine solutions of $\eqref{1stordergeodesic}$ subject to $\eqref{1storderconservation}$.
\end{latexonly}
Let $\sigma$ be the metric of standard Euclidean space.  Then I make the following
\begin{latexonly}
\emph{Claim}.  The solution to $\eqref{1stordergeodesic}$ for $\sigma = Id$ is
\end{latexonly}
\begin{center}
$\frac{\psi}{\norm{\psi}} = Aexp(i<k,x>\epsilon^{-1/2}) + \alpha$
\end{center}

where by abuse of notation I am embedding $R^{n}$ in $\mathcal{C}^{n}$.

\emph{Proof}.

\begin{center}
$(1 + 2\epsilon \Delta)\psi(\frac{\psi}{\norm{\psi}}) = 0$
\end{center}

is equivalent to 

\begin{center}
$(1 + 2\epsilon \Delta)\frac{d}{dt}\{\frac{\psi}{\norm{\psi}}(x + t \psi)\}|_{t = 0}$
\end{center}

which is the same as

\begin{center}
$(1 + 2\epsilon \Delta)\frac{\partial}{\partial x^{j}}(\hat{\psi}^{i}) \hat{\psi}^{j} \norm{\psi} = 0$
\end{center}

where $\hat{\psi} = \frac{\psi}{\norm{\psi}}$.

Rewriting we get

\begin{center}
$(1 + 2\epsilon \Delta)(\frac{\partial}{\partial x^{j}}(\hat{\psi}^{i}\hat{\psi}^{j}) - \hat{\psi}^{i}div \hat{\psi}) = 0$
\end{center}
\begin{latexonly}
Using $\eqref{1storderconservation}$ we get
\end{latexonly}
\begin{center}
$(1 + 2\epsilon \Delta)\frac{\partial}{\partial x^{j}}(\hat{\psi}^{i}\hat{\psi}^{j})= 0$
\end{center}

So we are interested in solving an equation of the form

\begin{center} $(1 + 2 \epsilon \Delta) \nabla f = 0$ \end{center}

\emph{Note}: If $\epsilon = 0$, $f$ is a constant, which is synonymous with the classical case of geodesics in flat space just being straight lines, and hence the flow being constant.

So $(1 + 2 \epsilon \Delta)g_{i} = 0$.  If we try $g_{i} = Ae^{k \cdot x}$, $\Delta g_{i} = A\norm{k}^{2}e^{k \cdot x}$ and hence $A(1 + 2\epsilon \norm{k}^{2})e^{k \cdot x} = 0$.  Hence $\norm{k} = \frac{i}{2\sqrt{\epsilon}}$.

So

\begin{center} $g_{j} = e^{i\frac{u_{j} \cdot x}{2\sqrt{\epsilon}}}$, \end{center}

where $\norm{u} = 1$.

Hence $(\nabla f)^{\mu}_{\alpha} = A_{\alpha} exp(i\frac{u^{\mu}_{nu}x_{nu}}{2\sqrt{\epsilon}})$, and therefore $f^{kl} = \sqrt{\epsilon} \hat{A}^{kl} exp(i\frac{u^{kl}_{p}x_{p}}{2\sqrt{\epsilon}}) + a$

But using symmetry $\hat{\psi}^{k}\hat{\psi}^{l} = f^{kl} = \sqrt{\epsilon}\hat{B}^{k}\hat{B}^{l} exp(i(\tau^{k} + \tau^{l})\cdot x\frac{1}{\sqrt{\epsilon}}) + \alpha^{k}\alpha^{l}$

Hence taking square roots we get that $\hat{\psi}^{k} = \sqrt{\epsilon}C^{k} exp(\frac{i\lambda^{k} \cdot x }{\sqrt{\epsilon}}) + \gamma^{k}$ for some new constants $C$, $\lambda$, and $\gamma$.

\begin{rmk} For the above solution, the overall direction of flux is in direction $\gamma^{k}$, since the oscilliatory term changes phase so rapidly with small epsilon that all contributions over an integrated region will cancel each other out.  However, we do expect to have some correction term of order $\epsilon$ or higher, which will be the focus of the next investigatory exercise. \end{rmk}

\begin{prop}  For the above solution to the fuzzy geodesic equation, consider a rectangular region $R$ with coordinates $\hat{\gamma}$,$\hat{\gamma}^{\perp}$.  Suppose further that $R$ is uniformly not small (with dimensions of order $\epsilon^{n}$, for $n < 1/2$).  Then if the volume of the region is $V$, the amount of flux through the region, $\Phi$, is

\begin{center} $\Phi(R) = V\norm{\gamma} + \sqrt{\epsilon}E(\epsilon)$ \end{center}

where the error $E(\epsilon)$ is bounded in magnitude by

\begin{center} $K\norm{\int_{N}C \cdot n exp(\frac{i\lambda \cdot x}{\sqrt{\epsilon}})}$ \end{center}

where $n$ is the unit normal to the boundary of $V$, $N$ is an $\epsilon^{n}$ neighbourhood of $\partial R$ for some $n > \frac{1}{2}$ and $K$ is some positive constant depending on $R$. \end{prop}

\begin{rmk}  It is clear to see from the above proposition that the corrections to the flux are of order $\epsilon^{1 + n}$, and can hence be neglected.  This is in accordance with "classical intuition" -- fuzzy effects should only be local, and not have global consequences. \end{rmk}


\section{Stokes Theorem for Statistical Manifolds}




\subsection{Naive formulation}

Many times in what is to follow I will require the following result, the analogue of Stokes theorem applied to statistical manifolds:

\begin{thm} (Stokes). \begin{center} $\int_{M}\int_{A}d_{\Lambda}\omega = \int_{\partial M}\int_{A}\omega$ \end{center}
\end{thm}

\begin{proof} In fact, though daunting at first, this is relatively trivial to prove, and follows easily from Stokes theorem in the sharp case.  If we absorb the distribution $f$ into the derivatives and write the form $\omega$ as $\int_{b \in A}F(m,a)dK(m,b)$, for $F$ a function and the $dK$ normalised unit forms, then we observe that, pointwise

\begin{center} $\int_{M}\int_{A}\epsilon_{ijk}\partial_{j}(m,b)F_{k}(m,a)dx_{j}(m,b)dK_{i}(m,c) = \int_{M_{bc}}\int_{A}\epsilon_{ijk}\partial_{j}(\phi_{bc}(m),c)F_{k}(\phi_{bc}(m),a)dx_{j}(\phi_{bc}(m),c)dK_{i}(\phi_{bc}(m),c)$ \end{center} 

for a suitable change of metric on $M$, and $\phi_{bc} : M \rightarrow M_{bc}$.  Since $\phi$ is a diffeomorphism, we then see that

\begin{center} $\int_{M}\int_{A}\epsilon_{ijk}\partial_{j}(m,b)F_{k}(m,a)dx_{j}(m,b)dK_{i}(m,c) = \int_{\partial M_{bc}}\int_{A}F_{k}(\phi_{bc}(m),a)dK_{i}(\phi_{bc}(m),c) = \int_{\partial M}\int_{A}F_{k}(m,a)dK_{i}(m,c)$ \end{center}

and since then this holds for all points $m,a,b,c$ it follows that the statistical stokes theorem is true.
\end{proof}

\subsection{The general statistical derivative}

However there was some display of naivete in our proof of the previous result, since we were assuming our statistical derivative had properties that it may well not in full generality.  So evidently before we can proceed further and fully nail Stokes theorem for statistical spaces it is necessary to flesh out precisely what I mean by a general statistical derivative in a geometrical context.  


I will show that the appropriate derivative to use for a statistical manifold is in fact

\begin{center} $\partial_{\Lambda ; f}(m,b) = \int_{A}F(m,c)(\frac{\partial}{\partial m} + \frac{\partial \sigma_{b}^{(ij)}}{\partial c}\frac{\partial}{\partial \sigma_{b}^{(ij)}})dc$, \end{center}

where $f(m,a) = \int_{A}F(m,b)\delta(\sigma_{b}(m) - a)db$.  (Note that, if $\sigma_{b}$ is just a number, as in our number theory example, the statistical derivative reduces to $\int_{A}F(m,c)(\frac{\partial}{\partial m} + \frac{\partial}{\partial c})dc$, since we don't have matrix multiplication twisting things up, or rather, matrix multiplication is trivial ).  


Recall \textbf{Theorem 2.5.2} from chapter 2:

\begin{center} $\int_{M}\int_{A}d_{(M;a)}\omega(m,a) = \int_{\partial M}\int_{A}\omega(m,a)$ \end{center}

But this is clearly equivalent to what we want to show, for since $d_{(M;a)}$ is a statistical derivative, we can simply let it be the statistical derivative with respect to $f$, and we are done.  In other words, stokes theorem holds for general statistical manifolds.


\chapter{Fisher Information and application to the theory of Physical Manifolds}

\section{The Shannon Entropy}

Clearly we need more information about what it means to be a physical manifold in order to generalise the classical results.  I am motivated by Perelman's notion of entropy for the Ricci flow \cite{[P]}, since a flow is another name for a 1-parameter variation, and the Ricci tensor occurs in the classical Einstein equation.  So we are led to consider entropy, and what it means in the context of a physical manifold.

I follow the treatment given in Khinchin \cite{[K]} in what follows.

\begin{dfn}
A \emph{finite scheme} is a collection of events $\{n_{1},...,n_{k}\}$ with associated probabilities $\{p_{1},...,p_{k}\}$.
\end{dfn}

For a particular scheme, we want to measure the degree of uncertainty in it- for example, if we have 2 events with probabilities $0.5$ each, the relative degree of uncertainty is higher than if we had probability $0.01$ for one and $0.99$ for the other.

It turns out that a very good measure of uncertainty for a given scheme is its \emph{entropy} $H$, given by $H = -\sum_{i}p_{i}ln (p_{i})$.

Suppose we have two schemes $A$ and $B$, and form the product scheme $AB$.  If events in $A$ are entirely unrelated to those in $B$, we have that $H(AB) = H(A) + H(B)$.  In general, however, these two schemes will depend on one another, and we get that $H(AB) = H(A) + \sum_{i}p_{i}H_{i}(B)$, where $H_{i}(B)$ is the conditional entropy of the scheme $B$ given that event $A_{i}$ has occured.  To be more precise, let $q_{ij}$ be the probability that the event $B_{j}$ in the scheme $B$ occurs given that the event $A_{i}$ of scheme $A$ occurred, and let $p_{i}$ be the associated probabilities to the events in $A$.

Hence 

\begin{align} H(AB) &= - \sum_{i,j}p_{i}q_{ij}(ln (p_{i}) + ln (q_{ij})) \\
 &= - \sum_{i}p_{i}ln (p_{i}) \sum_{j}q_{ij} - \sum_{i}p_{i}\sum_{j}q_{ij}ln (q_{ij}) \\
 &= H(A) - \sum_{i}p_{i}H_{i}(B)
\end{align}

We then define $H_{A}(B) = \sum_{i}p_{i}H_{i}(B)$.

We have the following result which I give without proof.  The interested reader is advised to refer to Khinchin \cite{[K]} if they wish for a detailed argument.

\begin{thm} There is only one notion of entropy, up to some proportionality factor, which satisfies the following three properties:
\begin{itemize}
\item[(i)] For any given $k$ and $\sum_{i = 1}^{k}p_{k} = 1$, the function $H(p_{1},...,p_{k})$ is largest for $p_{i} = 1/k$, $i = 1,...,k$.

\item[(ii)] $H(AB) = H(A) + H_{A}(B)$, where $H_{A}(B)$ is as defined above.

\item[(iii)] $H(p_{1},...,p_{k},0)$ = $H(p_{1},...,p_{k})$
\end{itemize}
\end{thm}

\emph{Remarks}. It can be shown that (i) holds for the notion of entropy that we have defined.  Clearly we would want this to be true, since intuitively things are most uncertain in a scheme when all events have equal probability.  (iii) basically is the statement that the entropy of a scheme should not change if we add impossible events.   Most importantly, all three of these properties hold for the notion of entropy defined above.

\emph{Remark}. There is also the notion of the information of a system.  Intuitively speaking, if we perform an experiment on a scheme we gain some information (by finding out which event occurs), and we eliminate the uncertainty of the scheme.  Hence this notion of information should be an increasing function of the entropy.  It is sensible to choose this function to be proportional to the entropy, so we can really think of entropy as latent information.

Evidently we might think of using this notion of entropy for our purposes by suitably generalising it and requiring it to be critical.  However this particular form of entropy is not entirely suited to our purposes, for a couple of reasons:

\begin{itemize}
\item[(i)] Shannon entropy is more or less a global measure of uncertainty and does not allow finer examination of phenomena.  In other words if we were to shuffle the values in a finite or even continuous scheme around we would get the same value for the entropy.  This can be seen most clearly in the one dimensional case by taking a discrete sum of values and observing that there is no interdependence between them.
\item[(ii)] This entropy does not incorporate the act of observation or measurement of information into the system.  Shannon entropy is a measure of uncertainty from the point of view of an observer from outside the system;  for instance, someone observing the outcome of the roll of a die.  It is not immediately apparent that this should matter, but, in general, the act of making a measurement from within a system will actually locally perturb the system, and hence, perturb the measurement.  In other words, internal information (which we want to quantify) $\neq$ external information (which is the quantity measured by Shannon entropy).
\end{itemize}

There is in fact a notion of information/entropy known as \emph{Fisher Information} that avoids these problems, which suggests it is more natural to use.  This will be the focus of the next section, and the remainder of this dissertation.

\section{Fisher Information Theory and the Principle of Extreme Physical Information}

\subsection{A few definitions from statistics}

Before I can talk about what the Fisher information is, I will have to guide the reader (and myself) through a crash course in statistics in order to equip properly for the development and motivation of the Fisher information as a useful invariant for getting useful data from a physical manifold.  In the course of initially investigating this area I found Wikipedia to be most useful.  However, due to the dubious and changeable nature of this source, it behooves me not to provide references to the wikipedia entries where I originally found this data, but rather a more conventional source, such as a modern introduction to statistics.  An example of such a book containing the material I will need is the recent one by Coolidge \cite{[Coo]}.


After developing the statistical tools required, I shall refer to Frieden's pioneering work \cite{[Fr]}, describing in some detail why I have decided to use Fisher information in the development that follows.

\begin{dfn} A \emph{probability distribution} with statistical variables $\theta$ is a nonnegative function $f : A \times M \rightarrow R$ satisfying the property that

$\int_{A}f(x,\theta)d\theta = 1$ for each $x \in M$
\end{dfn}

For example, the signal function is an example of such a function, with $A$ being the indexing space for events and $M$ being a manifold upon which the events occur pointwise.

\begin{dfn} A \emph{random variable} is a function $X : A \times M \rightarrow R$ for a probability indexing space $A$.
\end{dfn}

Now choose a probability indexing space $A$, and let $f$ be a probability distribution defined on it, and let $X : A \times M \rightarrow R$ be a random variable.

\begin{dfn} Two events $\alpha, \beta \in A$ are \emph{independent} if and only if $f(x, \alpha \cap \beta) = f(x,\alpha)f(x,\beta)$ for all $x \in M$.  Note that this is certainly true for signal functions, if we regard the events as individual metrics.
\end{dfn}

\begin{dfn} Two random variables $X,Y$ are \emph{independent} if and only if the events $[X \leq a]$ and $[Y \leq b]$ are independent for any numbers $a,b$.
\end{dfn}

\begin{dfn} The \emph{expected value} of a random variable $X$ for a probability distribution $f$ is given by the relation

\begin{center}
$E(X) = \int_{A}f(x,\theta)X(x,\theta)d\theta$
\end{center}
\end{dfn}

\begin{dfn} The \emph{variance} of a random variable $X$ for a probability distribution $f$ as above is given by the relation

\begin{center}
$var(X) = E((X - \mu)^{2})$
\end{center}

where $\mu = E(X)$ is the expected value of $X$.
\end{dfn}

It is clear why we would expect $E(X)$ to be a useful invariant.  What is not clear is why we would ever want to consider $var(X)$.

Well, certainly $var(X)$ gives us some measure of the variability of the random variable $X$ about its mean.  It can easily be shown that $var(X) = E(X^{2}) - (E(X))^{2}$.  Furthermore, for independent random variables $X, Y$ we have that $var(X + Y) = var(X) + var(Y)$.

\begin{dfn} The \emph{score} for a probability distribution with one hidden variable $\theta$ for a random variable $X$ that does not depend on position in $M$ is defined as

\begin{center} $V(X,\theta) = \frac{\partial}{\partial x}(ln (X(\theta)f(x, \theta)))$ \end{center}
\end{dfn}

An important property of the score is that if we view it as a random variable then $E(V) = 0$, provided certain integrability conditions are met (see J. Koliha's notes on analysis \cite{[Ko]}, p.209), since

\begin{center}
$E(V) = \int_{A}\frac{f'(x,\theta)}{f(x,\theta)}f(x,\theta)d\theta = \frac{\partial}{\partial x}\int_{A}f(x,\theta)d\theta = \frac{\partial}{\partial x}(1) = 0$
\end{center}

\subsection{The Fisher Information and the EPI principle}

We are now ready to define the Fisher information.

\begin{dfn} The \emph{Fisher Information} $I(x,X)$ corresponding to a probability distribution $f$ and random variable $X$ not depending on position $x \in M$ is defined as the variance of the score of $f$.  Since the expectation of the score is zero, we have that
\begin{center}
$I(x,X) = E[[\frac{\partial}{\partial x}ln (X(\theta)f(x, \theta))]^{2}]$
\end{center}
\end{dfn}

It is then natural to define the \emph{Total Fisher Information} as

\begin{center}
$I(X) = \int_{M}I(x,X)dx$
\end{center}

This is the quantity that I will be most interested in, and will be the primary focus of the analysis to follow.

Note that for the intents and purposes of what I plan to use this for, $X$ will predominantly be the identity variable, ie $X(\theta) = 1$ for all $\theta \in A$, since the quantity that I am predominantly interested in measuring is only the signal function.  Also note that the Fisher information \emph{is linear in its second argument for independent random variables}.  In other words, $I(\theta,X + Y) = I(\theta,X) + I(\theta,Y)$.  This follows from the result that the variance of the ``sum'' of independent events is the sum of their variances.  (Note $(X + Y)(x)f(\theta, x) = X(x)f(\theta,x)Y(x)f(\theta,x)$ for independent random variables $X,Y$; addition is not the same in $\chi$ as one might expect.  Equivalently we could notate $X + Y$ as $X \cap Y$.)  Intuitively this is in accordance with the fact that we should be able to add the information from unrelated experiments together in a linear fashion.

\begin{dfn} A random observation is a random variable that is equivalent to the identity on the probability space over a manifold.
\end{dfn}

This of course leads us to the question of what we mean when we talk about a measure of information, ie what properties it should have.  We might like to try to prove something like the following:

\begin{thm} \emph{(Uniqueness of the Information Measure)}. Let $\chi$ be the space of random variables on $A$.  Then, given a probability distribution $f$ defined pointwise over the manifold $M$ as above there is one, and only one function $K : M \times \chi \rightarrow \mathcal{R}_{\geq 0}$, unique up to a nonzero scalar constant, that satisfies the following properties:
\begin{itemize}
\item[(i)] Linearity in its second argument for independent random observations.
\item[(ii)] It is the maximal measure of variability in the probability distribution over $M$, subject to the previous condition.
\end{itemize}
\end{thm}

Certainly the Fisher Information as I defined it above satisfies these conditions.  And indeed, for most intents and purposes, something like this will suffice.  However, it behooves us to seek a more fundamental and convincing derivation of the principle I am going to invoke, since otherwise I would essentially have to motivate everything physically and wave my hands over the parts where something more precise might be required.  For now, however, I will continue with this particular approach and see where it leads us.

So anyway, uniqueness only up to some scalar constant will not matter, as it will turn out, because at the end of the day we are aiming to extremise this measure of information, and $\delta K = 0$ is completely equivalent to $\delta (\lambda K) = 0$, for $\lambda$ a fixed constant.

We are now at the point where we can define the bound information, $J$, and describe the principle of extreme physical information, as described in Frieden's book.

I define the \emph{channel information} to be the total Fisher information as defined above.  This is the total amount of information our manifold, can, in some sense, ``carry''.  The bound information, $J$, is in some sense the unavoidable information contained within the system.  It will be derived in our case by our conservation equation, as described and motivated in the previous section.  For a physical manifold then, we expect $I = J$.

Now, the principle of extreme physical information (The EPI principle) states that the \emph{physical information}, $K = I - J$, must be critical with respect to the natural Fisher variables of the system.  In our case, the Fisher variables will be the metric components $\sigma^{ij}$ in some local chart.  This is how the EPI principle is described by Frieden.  But I in fact go one step further.

\emph{Refined EPI principle}: For physical manifolds, the physical information, $K$, must be locally minimal.  In other words we roughly need the following two things to hold:

\begin{center}
$\delta K = 0$, and

$\delta^{2}K \geq 0$.
\end{center}

I say roughly, of course, because the second criterion does not guarantee stability.

\subsection{Interpretation}


One could view this principle as being a maximal efficiency principle.  The necessity of local minimality is required, of course, from the point of view of stability: the system must be stable, otherwise perturbations will throw things out of wack.  Even if one has $K = 0$, this is no guarantee of stability; obviously it makes no sense if $K$ is negative, but it could happen that an unphysical manifold might have $K = 0$, $\delta K = 0$ for all variations, but $K''(t) < 0$ for some variation.  Then there will be a natural flow to negative information $K$ and the manifold will become truly unphysical.

Perhaps a better way to interpret the EPI principle, particularly in the statistical setting, is that the universe wants to run as efficiently as possible, so each and every point, in an ideal world, would know exactly what was going on everywhere else and adjust itself to realise a minimum.  (This is actually what happens in the degenerate or sharp case, where one eliminates all probabilistic behaviour).  However, allowing more complex statistical models (which are more realistic) entail that there is an increasing cost to a point "knowing" what occurs metrically further and further away from it; in other words, the universe must organise a trade off between knowing exactly what happens everywhere (which would require a massive investment of energy in pinning down the geometry) and spending the least energy possible so as to be slightly uncertain about the geometry, but still realise a lower energy state.  This is precisely what I will be pinning down with my discussion of almost sharp geometries, which are closest in behaviour to standard quantum mechanical models, but of course the mathematics of information geometry does allow quite significant generalisation beyond this.

As for why the universe would "want" to run as efficiently as possible, perhaps a better way of putting it would be to observe that since all other configurations are unstable it has little choice in the matter.  The inevitable noise due to statistical fluctuations will cause any suboptimal choice will be quickly dropped in favour of one more locally optimal.  So while noise might prevent the universe from ever attaining a global minimum for information, the fact that the information is stable there means that it will get very close.  So to a reasonable approximation we can think of the universe as actually sitting at that point in solution space, in order to draw various cartoon models of increasing sophistication.  Naturally, of course, we would like to understand the noise - and indeed elimination in parts will lead to more complex formulations of Cramer-Rao and EPI.  However the act of elimination will always leave higher order noise yet to be quantified and understood.  Hence the best we can hope for is to produce models that work under a range of conditions that we can experimentally measure and understand.



\subsection{Concluding Remarks}

An equivalent representation of the channel information, $I$, for a fuzzy riemannian manifold $(M,A,f)$ may now be defined to be

\begin{center}$I = \int_{M}\int_{A}\norm{\partial_{\Lambda}q(m,a)}^{2}_{(m,a)}dadm$\end{center}

for $\psi$ a mass distribution and $U \subset M$, where $q^{2} = f$.  This makes sense because by construction $f$ is really a sort of probability distribution for the curvature.  Here by $\partial_{\Lambda}$ I mean the fuzzy gradient operator acting on $q$ as a function, and not as a tensor.

We now need to determine the bound information $J$.  Using the conservation equation relating $\psi$ and $f$,

\begin{center}$div_{\Lambda}\psi(m,a) = div_{\Lambda}grad_{\Lambda}q^{2}(m,a)$,\end{center}

and

\begin{center}$grad_{\Lambda}q^{2}(m,a) = 2q(m,a)grad_{\Lambda}q(m,a)$,\end{center}

and assuming perfect transfer of information, we get that

\begin{center}$I = \frac{1}{4}\int_{M}\int_{A}\frac{\norm{\psi(m,a) - B(m,a)}^{2}}{f(m,a)}dadm = J$,\end{center}

where $\psi(m,a) = \partial_{\Lambda}q^{2}(m,a) + B(m,a)$ for some $B$ such that $div_{\Lambda}B(m,a) = 0$.

The physical information, $K$, is defined to be $I - J$.  To apply the EPI principle one now need only solve the variational equation $\delta K = 0$.  Hopefully this will prove fruitful in deriving general relations that a physical manifold must then satisfy.

\emph{Philosophical Digression}:

It might seem that in describing the characteristics a solution must satisfy on a manifold that we are completely defining the state the solution must take.  This is not the case on complete manifolds because it seems to me against the very principles of Fisher Information Theory to be able to make an infinite measurement on a space (recall the EPI principle basically asserts physical behaviour arises as a consequence of measurement), since this would require an experimental apparatus of infinite capacity!  So only finite domains can be measured, and, as before in my discussion about extremising length, the solution may well not be unique.  Of these possible solutions, the one that the domain ultimately takes will depend on the behaviour of the manifold outside.


\section{Unbiased Estimators and the Cramer-Rao Inequality}



I now return to basics and attempt to defend once more the principle I invoked, notably that the quantity $K$ should be locally minimal in the space of signal functions, $f$ on a fuzzy Riemannian manifold.  In this section I will demonstrate how this property can be thought of as the condition of attaining equality, or, at the very least, local minimality, in what I shall call the \emph{Weak Cramer-Rao Inequality}.

To put things into the terminology of Murray and Rice, a fuzzy Riemannian manifold $(M,A,f)$ which is complete with respect to each metric in its corresponding metric indexing space $A$ is a complete exponential family; the sample space of events is diffeomorphic to $M \times A$; the space in which events occur is $M$; and the likelihood function is $f : M \times A \rightarrow R$.

From now on I shall assume that we are dealing with a fixed fuzzy Riemannian manifold $\Lambda = (M,A,f)$ such that $M$ is complete with respect to each metric $\sigma(.,a)$, any $a \in A$.

\subsection{Estimators and strong Cramer-Rao}

\begin{dfn} An estimator $u$ is a map from the sample space $M \times A$ to $M$.  An unbiased estimator is an estimator that satisfies the condition that

\begin{center} $E_{m}(\theta \circ u) = \theta(m)$ \end{center}

where $\theta : M \rightarrow R^{n}$ is a choice of coordinates for $M$ and $E_{m}(\lambda) = \int_{A}\lambda(m,a)f(m,a)da$, $m \in M$.
\end{dfn}

\begin{lem}. An unbiased estimator $u$ satisfies the following relation:

\begin{center}
$E_{m}((\theta^{i} \circ u)\frac{\partial ln(f \circ u)}{\partial \theta^{j}}) = 0$
\end{center}
\end{lem}

\begin{proof} Differentiating the condition with respect to $\theta^{j}$ on both sides, we get:

\begin{center}
$\int_{A}\frac{\partial}{\partial \theta^{j}}(\theta^{i}f)da = \delta^{i}_{j} + \int_{A}((\theta^{i} \circ u)\frac{\partial (f \circ u)}{\partial \theta^{j}}) = \delta^{i}_{j}$
\end{center}

If we use the fact that $\frac{\partial f}{\partial \theta^{j}} = \frac{\partial ln(f)}{\partial \theta^{j}}f$, we see that the relation follows.
\end{proof}

\begin{dfn}
The Fisher Information Matrix $g^{ij}(f)$ (corresponding to the likelihood function $f$), is defined to be

\begin{center}
$g^{ij}(f) = E_{m}(\frac{\partial ln(f)}{\partial \theta^{i}}\frac{\partial ln(f)}{\partial \theta^{j}})$
\end{center}
\end{dfn}

We have the following inequality for unbiased estimators on $\Lambda$, known as the Cramer-Rao inequality:

\begin{equation}
cov(\theta^{i} \circ u, \theta^{j} \circ u) - g^{ij}(f(u)) \geq 0
 \end{equation}

where $cov(v^{i},w^{j})(m) = E_{m}(v^{i}w^{j})$ is the covariance of $v$ and $w$, and by $g^{ij}(f(u))(m,a)$ I mean $g^{ij}(f(u(m,a),a))$.

\begin{proof} (Cramer-Rao).

Note that

\begin{center}
$E_{m}((\theta^{i} \circ u - E_{m}(\theta^{i} \circ u) - \frac{\partial ln(f(u))}{\partial \theta^{k}}g^{ik})(\theta^{j} \circ u - E_{m}(\theta^{j} \circ u) - \frac{\partial ln(f(u))}{\partial \theta^{l}}g^{jl}))$
\end{center}

is positive semi-definite.  Expanding this, we get

\begin{center}
$cov(\theta^{i} \circ u,\theta^{j} \circ u) - 2E_{m}((\theta^{i} \circ u - E_{m}(\theta^{i} \circ u))\frac{\partial ln(f(u))}{\partial \theta^{l}}g^{jl}) + g^{ij}$
\end{center}

Upon applying our lemma for unbiased estimators on the second term, then integrating by parts, we get

\begin{center}
$E_{m}((\theta^{i} \circ u - E_{m}(\theta^{i} \circ u))\frac{\partial ln(f(u))}{\partial \theta^{l}}g^{jl}) = -E_{m}(\theta^{i}\frac{\partial ln(f(u))}{\partial \theta^{l}}g^{jl}) = g^{ij}$
\end{center}

and we are done.
\end{proof}

\begin{dfn} A maximum likelihood estimator (mle) of a statistical manifold like $\Lambda$ is defined to be an unbiased estimator $v$ realising a local maximum of $ln(f(v(m,a),a))$, $m \in M$ (this is equivalent to realising a local maximum of the likelihood function since ln is an increasing function, hence the name maximum likelihood estimator).  Amongst other things we certainly must have

\begin{center}
$\frac{\partial ln(f(v(m,a),a))}{\partial \theta^{i}} = 0$
\end{center}

\end{dfn}

\begin{lem} Any mle of an exponential family realises the Cramer-Rao lower bound.

\begin{proof} See Murray and Rice's book. \end{proof} \end{lem}

\subsection{Physical estimators and weak Cramer-Rao (for Riemannian-metric-measure spaces)}

\begin{dfn} An estimator $u$ is said to be \emph{weakly unbiased} if

\begin{center}
$E(\theta \circ u) = \int_{M}E_{m}(\theta \circ u)dm = E(\theta \circ Id)$
\end{center}
\end{dfn}

\begin{lem} Weakly unbiased estimators satisfy the additional relation that

\begin{center}
$E(\theta^{i}\frac{\partial ln(f)}{\partial \theta^{j}}) = 0$.
\end{center}

\begin{proof} Similar to that for unbiased estimators. \end{proof} \end{lem}

\begin{cor} Note that if our likelihood function $f$ decays sufficiently fast towards infinity from the center of our coordinates $\theta$ that $E(\theta \circ Id) = 0$.

\begin{proof} Write $\theta = d\phi$.

Certainly

\begin{center}
$E(\theta \circ Id) = \int_{A}\int_{M}\theta f dmda = \int_{A}\{\int_{\partial M}\phi f dS - \int_{M}\phi \frac{\partial f}{\partial \theta}dm\}da$
\end{center}

Now the first term on the right disappears since $f$ decays sufficiently fast.  The second term on the right disappears by the lemma.
\end{proof} \end{cor}

\begin{lem}  Any estimator $\Gamma$ of the form $\Gamma = exp(-\frac{\psi}{f})\partial_{\sigma}f$ is a weakly unbiased estimator, where we define $\psi$ by the conservation equation

\begin{center} $div_{\Lambda}\psi = \Delta_{\Lambda}f$ \end{center}

where $div_{\Lambda}$ and $\Delta_{\Lambda}$ are the standard fuzzy differential operators.  I shall call such estimators physical estimators.

\begin{proof} In a physical situation, which is the situation towards this is to apply, $f$ will decay sufficiently quickly that we can use our remark.  So all we need to show is that:

\begin{center} $E(\theta \circ \Gamma) = 0$ \end{center}

Now $E(\theta \circ \Gamma) = \int_{M}\int_{A}ln(exp(\frac{ - \psi}{f})\partial_{\sigma}f)f$ for $\theta = ln = exp^{-1}$ as our choice of coordinate chart.

So

\begin{center}
$\int_{M}\int_{A}ln(exp(\frac{ - \psi}{f})\partial_{\sigma}f)f = \int_{M}\int_{A}ln(\partial_{\sigma}f)f - \psi$
\end{center}

Now $ln \circ \partial_{\sigma} = \partial_{\sigma} \circ ln$, so this becomes

\begin{center}
$\int_{M}\int_{A}\partial_{\sigma}f - \psi$
\end{center}

but since $\psi = \partial_{\sigma}f + curl_{\Lambda}B$, this is just

\begin{center}
$-\int_{M}\int_{A}curl_{\Lambda}B$
\end{center}

which vanishes by Stoke's theorem provided $B$ decays sufficiently rapidly towards infinity, which proves the lemma.
\end{proof} \end{lem}

\begin{lem}
For physical estimators we have a corresponding \emph{weak} Cramer-Rao inequality:

\begin{equation}
\int_{M}\sigma_{ij} \{cov((\theta^{i} \circ u),(\theta^{j} \circ u)) - g^{ij}\}dm \geq 0
\end{equation}
\end{lem}

\begin{proof}  The above inequality is certainly not true for general weakly unbiased estimators since the metric $\sigma$ is not positive definite.  So we need to show that

\begin{center}
$\sigma_{ij}cov((\theta^{i} \circ u),(\theta^{j} \circ u)) \geq 0$
\end{center}

and

\begin{center}
$\sigma_{ij}g^{ij}(f(u)) \geq 0$.
\end{center}

In fact, what I will end up doing is proving that the second term vanishes for physical estimators and that the following equality holds for the first:

\begin{center}
$\int_{M}\sigma_{ij}cov((\theta^{i} \circ u),(\theta^{j} \circ u)) = \int_{M}\int_{A}(\frac{\norm{\partial_{\sigma}f}^{2}}{f} - \frac{\norm{\psi}^{2}}{f})$
\end{center}

But by the definition of a fuzzy Riemannian manifold $\partial_{\sigma}f$ and $\psi$ must be timelike vectors, so indeed $\int_{M}\sigma_{ij}cov((\theta^{i} \circ u)(\theta^{j} \circ u))$ is positive.  The lemma then follows from the strong Cramer-Rao inequality.  \end{proof} 

The previous lemma (which we just proved) is equivalent to the following:

\begin{lem} For the choice of $\Gamma$ for an estimator, the Cramer-Rao inequality becomes

\begin{equation}
\int_{M}\int_{A}\frac{\norm{\partial_{\sigma}f}^{2}}{f}dadm - \int_{M}\int_{A}\frac{\norm{\psi}^{2}}{f} \geq 0
\end{equation}
\end{lem}

\emph{Remarks}.  Note that the first term of the above inequality is the channel information as I defined it before, and the second term is the bound information.  So what this inequality is saying is that the difference of these, which I likewise defined as the physical information, must be greater than or equal to zero.  In particular, since $\Lambda$ is an exponential family, any mle of it will realise the inequality \emph{critically}.  To choose such a critical mle amounts to computing the variation of $K$ with respect to the signal function $f$.  This provides the sort of backing we were looking for before to motivate the variational analysis to follow.

\begin{proof} (Completion of Proof).  I prove that $\sigma_{ij}cov((\theta^{i} \circ \Gamma),(\theta^{j} \circ \Gamma)) = \frac{\norm{\partial_{\sigma}f}^{2}}{f} - \frac{\norm{\psi}^{2}}{f}$.

So

\begin{center}
$cov((\theta^{i} \circ \Gamma),(\theta^{j} \circ \Gamma)) = \int_{A}(\partial_{\sigma}ln(f) - \frac{\psi}{f})(\partial_{\sigma}ln(f) - \frac{\psi}{f})f$
\end{center}

Expanding, we get

\begin{center}
$\int_A(\frac{\partial_{i}f\partial_{j}f}{f} + \frac{\psi_{i}\psi_{j}}{f}) - 2\int_{A}\frac{\partial_{i}f}{f}\frac{\psi_{j}}{f}f$
\end{center}

Now

\begin{center}
$\int_{M}\int_{A}\sigma_{ij}\frac{\partial_{i}f}{f}\frac{\psi_{j}}{f}f = \int_{M}\int_{A}\sigma_{ij}\psi_{i}\psi_{j}\frac{1}{f}$
\end{center}

by Stoke's theorem.

It remains to show that $\int_{M}\sigma_{ij}g^{ij}(f(\Gamma))dm = 0$.

Now

\begin{center}
$g^{ij}(f(\Gamma)) = \int_{A}\frac{\partial ln(f(\Gamma))}{\partial \Gamma^{k}}\frac{\partial \Gamma^{k}}{\partial \theta^{i}}\frac{\partial ln(f(\Gamma))}{\partial \Gamma^{l}}\frac{\partial \Gamma^{l}}{\partial \theta^{j}}f$

$= \int_{A}\frac{\partial ln(f(\Gamma))}{\partial \Gamma^{k}}\frac{\partial (curl V)^{k}}{\partial \theta^{i}}\frac{\partial ln(f(\Gamma))}{\partial \Gamma^{l}}\frac{\partial (curl V)^{l}}{\partial \theta^{j}}f$
\end{center}

But this is $0$ since $grad curl V = 0$.  This completes the proof.
\end{proof}

\begin{lem} Any weak mle of an exponential family realises the weak Cramer-Rao lower bound. \end{lem}

\emph{Further Remarks}.  Now it may seem that it is a bit restrictive to only consider physical estimators, but I claim that we can transform any map $u : M \times A \rightarrow M$ into a physical estimator.  All one needs to do is choose an appropriate likelihood function $f$ s.t. $E(\theta \circ u) = 0$ and $u = exp(-\frac{\psi}{f})(\partial_{\sigma}f)$ for some appropriate $\psi$ that will itself depend on $f$.  This transforms our problem into one of varying likelihood functions, which makes good physical sense.

Furthermore, the analysis is this section has rested often on the fact that various objects vanish at infinity.  In certain situations of interest (for instance, modelling the event horizons of black holes, or of the behaviour of charged plasmas contained in a finite vessel), this is not enough, and it is necessary to consider finite domains, or at least domains with boundary.  One is led to the following conjecture, motivated by similar work done by physicists working on the physics of black holes:

\begin{conj} The physical information, $K_{M}$ of the manifold and the physical information, $K_{\partial M}$ of its boundary are related by the following inequality:

\begin{center}
$K_{M} - K_{\partial M} \geq 0$
\end{center}

which becomes an equality for a critical choice of signal function $f$.  Here by $K_{\partial M}$ I mean of course

\begin{center}
$K_{\partial M} = \int_{\partial M}\int_{A} \{\frac{\norm{\partial_{\tau}f}^{2}_{\tau}}{f} - \frac{\norm{\psi}^{2}_{\tau}}{f}\}$
\end{center}

where $\tau$ is the restriction of $\sigma$ to the boundary of $M$.
\end{conj}

We can think of this as a kind of "holographic principle" for the physical information $K$.  Note that our earlier result can be considered a special case of the above, for $\partial M = \phi$.  Note that if this conjecture is true, as a particular consequence it suffices in applications to discard all boundary terms that crop up in a calculation, or, in other words, to consider problems with Dirichlet boundary conditions.

\emph{Final Remark}.  All of this suggests that we define ourselves an information functional

\begin{center}
$K(f) = \int_{M}\int_{A}\frac{\norm{\partial_{\sigma}f}^{2}}{f}dadm - \int_{M}\int_{A}\frac{\norm{\psi}^{2}}{f}$
\end{center}

and play the following game- choose a family of signal functions $f$ and minimise this functional within this family.  The best we could hope for of course would be to zero the above functional.  However, in most circumstances, all we can hope for is to make it as small as possible.  This is the game that I shall play in what is to follow.

\subsection{Physical estimators and weak Cramer-Rao (for Riemann-Cartan manifolds)}

The notation in the previous subsection is unfortunately a bit too cumbersome and too particular for our tastes.  Ultimately we would like to express things more succinctly and in full generality.  More precisely we would like to deal with the case of our metrics being general nondegenerate bilinear forms instead of merely symmetric.

As before, recall the following definition and its initial consequences:

\begin{dfn} An estimator $u$ is said to be \emph{weakly unbiased} if

\begin{center}
$E(\theta \circ u) = \int_{M}E_{m}(\theta \circ u)dm = E(\theta \circ Id)$
\end{center}
\end{dfn}

\begin{lem} Weakly unbiased estimators satisfy the additional relation that

\begin{center}
$E(\theta^{i}\frac{\partial ln(f)}{\partial \theta^{j}}) = 0$.
\end{center}

\begin{proof} Similar to that for unbiased estimators. \end{proof} \end{lem}

\begin{cor} Note that if our likelihood function $f$ decays sufficiently fast towards infinity from the center of our coordinates $\theta$ that $E(\theta \circ Id) = 0$.

\begin{proof} Write $\theta = d\phi$.

Certainly

\begin{center}
$E(\theta \circ Id) = \int_{A}\int_{M}\theta f dmda = \int_{A}\{\int_{\partial M}\phi f dS - \int_{M}\phi \frac{\partial f}{\partial \theta}dm\}da$
\end{center}

Now the first term on the right disappears since $f$ decays sufficiently fast.  The second term on the right disappears by the lemma.
\end{proof} \end{cor}

\begin{lem}  Any estimator $\Gamma$ of the form $\Gamma = \partial_{\Lambda}f$ is a weakly unbiased estimator, where $f$ satisfies the conservation equation

\begin{center} $0 = \Delta_{\Lambda}f$ \end{center}

where $\Delta_{\Lambda}$ is the standard fuzzy differential operator for the Laplacian, and $\partial_{\Lambda}$ is the gradient operator associated to the space.  I shall call such estimators physical estimators.

\begin{rmk} Note that a physical estimator corresponds precisely to a geometry with signal function where probability flux is conserved.  So physics occurs in spaces that "make sense". \end{rmk}

\begin{proof} In a physical situation, which is the situation towards this is to apply, $f$ will decay sufficiently quickly that we can use our remark.  So all we need to show is that:

\begin{center} $E(\theta \circ \Gamma) = 0$ \end{center}

Now $E(\theta \circ \Gamma) = \int_{M}\int_{A}ln(\partial_{\sigma}f)f$ for $\theta = ln = exp^{-1}$ as our choice of coordinate chart.

So

\begin{center}
$\int_{M}\int_{A}ln(\partial_{\sigma}f)f = \int_{M}\int_{A}ln(\partial_{\sigma}f)f - \partial_{\sigma}f$
\end{center}

Now $ln \circ \partial_{\sigma} = \partial_{\sigma} \circ ln$, so this becomes

\begin{center}
$\int_{M}\int_{A}\partial_{\sigma}f $
\end{center}

But this is zero by conservation of probability.

\end{proof} \end{lem}

\begin{lem}
For physical estimators we have a corresponding \emph{weak} Cramer-Rao inequality:

\begin{equation}
\int_{M}\sigma_{ij} \{cov((\theta^{i} \circ u),(\theta^{j} \circ u)) - g^{ij}\}dm \geq 0
\end{equation}
\end{lem}

\begin{proof}  Observing that as before we have that 

\begin{center}
$\int_{M}\sigma_{ij}E_{m}((\theta^{i} \circ u - E_{m}(\theta^{i} \circ u) - \frac{\partial ln(f(u))}{\partial \theta^{k}}g^{ik})(\theta^{j} \circ u - E_{m}(\theta^{j} \circ u) - \frac{\partial ln(f(u))}{\partial \theta^{l}}g^{jl}))$
\end{center}

is positive semi-definite, since $u := \partial_{\Lambda}f$ is timelike, we conclude the lemma.  
\end{proof}

\begin{lem} Any weak mle of an exponential family realises the weak Cramer-Rao lower bound. \end{lem}

\begin{thm} (The EPI principle (for Riemann-Cartan manifolds)).  With the understanding that a physical estimator is to be used, the Cramer-Rao inequality is equivalent to the nonnegativity of the Fisher Information. In other words, given a distribution $f$ over a space of $n$-dimensional nondegenerate bilinear forms of constant index $\Lambda$, we have that

\begin{center} $I(f) := \int_{M}\int_{A}\frac{\norm{\partial_{\Lambda}f}^{2}}{f} \geq 0$ \end{center} \end{thm}

\begin{proof} First of all, it is clear that 

\begin{center} $\int_{M}\sigma_{ij}cov_{m}(\theta^{i} \circ \partial_{\Lambda}f,\theta^{j} \circ \partial_{\Lambda}f) = \int_{M}\int_{A}\frac{\norm{\partial_{\Lambda}f}^{2}}{f}$ \end{center}

so it remains to show that $\int_{M}\sigma_{ij}g^{ij}(f(\partial_{\Lambda}f)) = 0$.

Now

\begin{align} \int_{M}\sigma_{ij}g^{ij}(f(u)) &= \int_{M}\int_{A}\sigma_{ij}\frac{\partial ln(f(u))}{\partial u^{k}}\frac{\partial u^{k}}{\partial \theta^{i}}\frac{\partial ln (f(u))}{\partial u^{l}}\frac{\partial u^{l}}{\partial \theta^{j}}f \nonumber \\
&=  \int_{M}\int_{A}\sigma_{ij}\frac{\partial ln(f(u))}{\partial u^{k}}\frac{\partial curl(V)^{k}}{\partial \theta^{i}}\frac{\partial ln (f(u))}{\partial u^{l}}\frac{\partial curl(V)^{l}}{\partial \theta^{j}}f \nonumber \\
& \text{ (for some V, since the divergence of u is zero)} \nonumber \\
&= 0 \nonumber \\
& \text{ (since grad(curl) is the zero operator)} \end{align}

This proves the EPI principle in the form we would like to use it.

\end{proof}

\section{Fisher Information is the optimal sharp information measure}

\subsection{Introduction}

There were a number of criticisms leveled at Frieden's program of using Fisher information to recover physical principles.  But there are two in particular that require to be addressed before I go any further:

\begin{itemize}
\item[(i)] What is the mathematical justification for using the EPI principle?
\item[(ii)] Why the Fisher information? Why not some other arbitrary information measure?
\end{itemize}

I believe that I have provided a rather thorough explanation of (i), grounding it on the fact that the Fisher information satisfies the Cramer-Rao inequality for statistical manifolds.  (ii) however is not so clear cut.  In particular, there are many, many other information measures that also satisfy the Cramer-Rao inequality.  Many, many more.  To have an idea of just how many, I will point out one general result:

The \emph{L - information}, 

\begin{center}
$I(L) = \int_{M}\int_{A}\frac{\norm{Lf}^{2}}{f}$
\end{center}

where $L$ is a differential operator, satisfies the Cramer-Rao inequality.

For example, $L_{0} = (\partial_{1},\partial_{2},...,\partial_{n})$ is the standard Fisher operator.  Restricting to the case of a $3$-dimensional chart, note that we could also have used $(\partial_{3},\partial_{1},\partial_{2})$, $(\partial_{2}^{2},\partial_{1},\partial_{1}^{3})$, $(\partial_{1} + \partial_{2} + 4\partial_{3},\partial_{2},\partial_{3} + 2\partial_{2}^{2})$, $(x_{2}\partial_{1} + x_{1}\partial_{2},\partial_{2},\partial_{3})$ etc, to have an idea of how many degrees of freedom there are here.

So I shall now state the main result of this section.

\begin{thm}  Over all possible information measures, the $L_{0}$-information, or Fisher Information, is \emph{critical}.  In other words, it is an \emph{optimal} measure of information. \end{thm}

\begin{rmk}  This does not preclude the fact that there may be other (locally) optimal measures.  However it is a good starting point for further philosophical discussion. \end{rmk}

For now I will restrict myself to the case that the signal function is sharp.

Let $S(M)$ be the space of functionals $F : (M \rightarrow R) \rightarrow R$ on our manifold $M$.  We would like to come up with some natural parametrisation of these.


We expect such a parametrisation to be compatible with the metric such that our functionals $F$ are locally diagonalisable with signature identical to the metric.  So locally our space has the dimension of the space of positive matrices on $R^{n}$.

\subsection{Variation of the meta information}

To gather the appropriate tools to prove the theorem, we need to define a new functional which I shall call the \emph{meta-information}.  First I shall assemble the ingredients:

Essentially we wish to associate a probability distribution of information operators $L(m,a)$ to each pair $(m,a)$ in our statistical manifold.  Such information operators are characterised by the way they act on vectors.  In particular, we define the information $h_{ijk}$ to be the coefficient of the $i$th component of a general operator $L$, being the coefficient of the $j$th directional vector field in that component raised to the $k$th power.  In particular,

\begin{center}
$L_{ijk}f = h_{ijk}(\partial_{j})^{k}f$
\end{center}

Just as with our inner products for statistical manifolds as before, we play the same game and define a probability density functional $\lambda(m,a,h)$ such that for a fixed $(m,a)$, $\int_{B}\lambda(m,a,h)dh = 1$, and also, as before, $\int_{B}\Lambda(\lambda(m,a,h))dh = 0$ (with the appropriate conventions for fuzzy derivatives).

Then, we may define the meta-information to be

\begin{center}
$I = \int_{M}\int_{A}\int_{B}\lambda(m,a,b)\frac{\norm{L_{(m,a,b)}
f(m,a)}^{2}_{(m,a)}}{f(m,a)}$
\end{center}

Make the assumption now that the information distribution functional is \emph{sharp}, that is, a particular information is preferred. Then $\lambda(m,a,b) = \delta(b - \kappa(m,a))$, and

\begin{center}
$I = \int_{M}\int_{A}\frac{\norm{\kappa_{jk}(m,a)(\partial_{j})^{k}f(m,a)}^{2}}{f(m,a)}$
\end{center}

Take the first variation with respect to $\kappa$ and require that it be zero; we get immediately that $\kappa$ is constant in $m$ and $a$ if it is critical.  Furthermore, it is constant in $k$, if we use the notion of generalised derivative (to real numbers).  So $\kappa$ is essentially just $\kappa_{ij}(\partial_{j})^{k}$ for some fixed $k$ where $\kappa_{ij}$ is a constant matrix. One may then make the observation that it is possible to find a new metric $\bar{\sigma}$ such that $\norm{\kappa_{j}(\partial_{j})^{k}f}^{2}_{\sigma} = \norm{\partial^{k}f}^{2}_{\bar{\sigma}}$, ie it is possible to absorb $\kappa$ completely into the metric $\sigma$.

So we have reduced the problem to one of examining the so called $k$-information.  Why then should the $1$-information be preferred?

I guess then the ultimate conclusion is this:

\begin{thm}  For each $k \in R$, the $k$-information is optimal.  \end{thm}

However it shall turn out that this is relatively irrelevant when it comes to matters of physics, leaving one with a (potentially infinite) number of integration constants.  I suppose if there was any way to say that the one information is preferred it would be simply that there is no guide as to how to determine these integration constants; therefore assuming they exist is a nonsense (overspecification), which rules out criticality for all informations except one or less.  But if one considers the $k$-information for $k < 1$ one gets nothing at all (underspecification).  So $k = 1$ is the only possibility.

This shall all hopefully become clearer to the reader after perusing the section to come.

\subsection{Alternative perspectives}

However, we can make the following important observation, that what we have been doing has been slightly unnecessary and perhaps a little silly.

Recall the idea of maximal likelihood estimator (mle) from before.

From Murray and Rice's book on the subject \cite{[Mu]} we have that

\begin{thm} For any exponential family (statistical structure) the associated mle is unique. \end{thm}

\begin{rmk} In particular via my construction this means that for every geometric structure there is a unique mle. \end{rmk}

Also from \cite{[Mu]} we have that

\begin{thm} The unique mle realises the Cramer-Rao lower bound for the Fisher Information. \end{thm}

So the Fisher information is the best information in the regard that it is the unique functional which is zeroed by the estimator of maximal likelihood.  This is in accord with our common sense as to how an information should behave.

For yet another perspective, consider the Cencov Uniqueness Theorem \cite{[Pe1]},\cite{[Pe2]},\cite{[Ce]}.  This essentially states that for any coarse graining (averaging over the geometry of a space), the Fisher information of the new geometry is less than that of the old geometry.  This means that the Fisher information is a special information, in that it directly captures what we would expect from the 2nd law of thermodynamics - as the disorder of the system decreases, so does the Fisher information.

In other words, the information of our system is in monotone correspondence with the Fisher information.  Which gives us another reason to prefer the Fisher information as a measure of such.






\chapter{Physics from Fisher Information}

We are now finally ready to put theory to application.  This brings me back to the signal functions that I introduced before in my discussion of statistical manifolds.  There was in fact a very good reason for studying these things, since it is these particular classes of objects (sharp and almost sharp manifolds) that will be the basis of my study in this section.  It will turn out that by assuming that a manifold is sharp and applying our variational principle that we shall recover the equations of general relativity; furthermore, if we relax our restrictions a bit and assume that our space is slightly fuzzy, then rinse and repeat our variational calculation, we shall be able to derive the underlying equations on which one can base the standard model.

\section{Physics for Sharp Manifolds}

\subsection{Introduction}
Let us now suppose $f(m,a) = \delta(\sigma(m) - a)$ again.  Our aim is to derive the classical equations for a physical manifold.  The first step is to reexpress $I$ and $J$ using this signal function.  Clearly for this choice, $\psi(m,a) = \overline{\psi}(m,a)\delta(\sigma(m) - a)$.

\emph{Remark}: Note that, when all is said and done, we will get a value for the physical information $K \geq 0$.  In fact it is extremely unlikely that we will be able to get $K(\sigma) = 0$ for any metric $\sigma$.  However the point of this exercise is to find the conditions on $\sigma$ for it to minimise, or optimise, our information functional as defined before.

Substituting into $I$, and observing that $grad_{\Lambda}\delta(\sigma(m) - a) = grad_{\sigma}\sigma_{jk}\frac{\partial}{\partial \sigma_{jk}}\delta(\sigma(m) - a)$ we get that

\begin{center}
$I = I(\sigma) = \int_{M}\int_{A}\frac{1}{4\delta}<\partial_{\sigma}\sigma_{jk}\frac{\partial}{\partial \sigma_{jk}}\delta(\sigma(m) - a),\partial_{\sigma}\sigma_{mn}\frac{\partial}{\partial \sigma_{mn}}\delta(\sigma(m) - a)>(det \sigma)^{1/2}dadm$

$= \int_{M}\int_{A}\frac{1}{4\delta}\frac{\partial \delta}{\partial \sigma_{jk}}\frac{\partial \delta}{\partial \sigma_{mn}}<\partial_{\sigma}\sigma_{jk},\partial_{\sigma}\sigma_{mn}>(det\sigma)^{1/2}dadm$
\end{center}

Now, before I proceed any further, I shall assume that in the general expression for $\psi$ deduced from the conservation equation, $\psi(m,a) = grad_{\Lambda}f(m,a) + B(m,a)$ that $B = 0$.  Later I shall relax this restriction, and show that, at least in the classical case, $B$ corresponds to the presence of electromagnetic effects.

\subsection{Opening statements}

I now make a series of claims:

\emph{Claim 1}:

\begin{center}
$\int_{M}\int_{A}\frac{1}{\delta}\partial_{i}(\delta)\partial_{j}(\delta)\Sigma dadm = \int_{M}\int_{A}\partial_{i}\partial_{j}(\Sigma) \delta dadm$,

provided that $\Sigma(1/\sqrt{a}) \sim 1/a^{p}$ for some $p > 1$, $a >> 1$.
\end{center}

\emph{Claim 2}:

$<\partial_{\sigma}x_{jk},\partial_{\sigma}x_{mn}>$ satisfies the asymptotic property in Claim 1, for very small values in the entries of the matrix $y = (x_{jk} - \sigma_{jk})_{j,k = 1}^{n}$.

\emph{Claim 3}:

\begin{center}
$I = I(\sigma) = \frac{1}{4}\int_{M}R(det\sigma)^{1/2}dm$
\end{center}

Examining $J$, we find that, upon assuming $B = 0$, that

\begin{center}
$J = J(\sigma) = \int_{M \times A}\frac{\norm{\psi(m,a)}^{2}}{4f(m,a)}dadm$

$ = \int_{M}\frac{\norm{\overline{\psi}}^{2}}{4}dm$
\end{center}

\emph{Claim 4}: 

Assuming boundary terms are negligible, we get that $\delta(4I) = \int_{M}(R_{ij} - \frac{1}{2}\sigma_{ij}R)(det\sigma)^{1/2}\delta \sigma^{-1}$ and $\delta (4J) = \int_{M}(\bar{\psi}_{i}\bar{\psi}_{j} - \frac{1}{2}\sigma_{ij}\norm{\bar{\psi}}^{2})(det\sigma)^{1/2}\delta \sigma^{-1}$, where $\delta$ is now the variational derivative. 

Substituting the expressions from the final claim into the equation $\delta K = 0$ we get

\begin{center}
$G_{ij} = \bar{\psi}_{i}\bar{\psi}_{j} - \frac{1}{2}\sigma_{ij}\norm{\bar{\psi}}^{2}$
\end{center}

This is remarkably similar to the Einstein equation, but with the additional term $\frac{1}{2}\sigma_{ij}\norm{\bar{\psi}}^{2}$.  This term might explain why cosmologists using the standard Einstein equation to describe large objects like galaxies get puzzling results, and are forced to invoke the presence of large amounts of (as yet unobserved) dark matter/energy.  It might also explain the Pioneer anomaly (the discrepancy between the predicted flight path of the now famous probe out of the solar system and its actual journey).

\emph{Note:} It can be shown that $\bar{\psi}$ corresponds to the four momentum if we consider toy examples like Minkowski space.

Also observe that this equation can be simplified, for if we contract it on both sides with respect to the metric we obtain

\begin{center}
$R - \frac{n}{2}R = \norm{\bar{\psi}}^{2} - \frac{n}{2}\norm{\bar{\psi}}^{2}$
\end{center}

where we are assuming $n =$ dim$M > 2$ (in applications, $n$ will usually be $4$).

As a consequence of the above, the equation of state simplifies to

\begin{center}
$R_{ij} = \bar{\psi}_{i}\bar{\psi}_{j}$
\end{center}

or, in coordinate invariant form

\begin{equation}
Ric - \bar{\psi} \otimes \bar{\psi} = 0
\end{equation}

Together with the geodesic condition

\begin{center}
$\nabla_{(\bar{\psi},\sigma)}\frac{\bar{\psi}}{\norm{\bar{\psi}}} = 0$
\end{center}

and the conservation equation

\begin{center}
$div_{\sigma}\bar{\psi} = 0$
\end{center}

we have now completely described the dynamics of a physical manifold without boundary for the zeroth order perturbation $f(m,a) = \delta(\sigma(m) - a)$.

\subsection{Proof of the first Claim}

I prove this for the one dimensional case, but this example is easily extended to the general case.  First observe that $\delta(x) = lim_{a \rightarrow \infty}f(a,x)$, where $f(a,x) = \sqrt{\frac{a}{\pi}}exp(-ax^{2})$.  Of course this definition is not unique, but by a deep theorem in functional analysis it does not matter what limiting functions we use, as long as they are smooth, since I am going to take derivatives.

Then


\begin{align} \int (\delta'(x))^{2}(1/\delta(x)) \Sigma(x) dx &= lim_{a \rightarrow \infty}\int (\frac{\partial f(a,x)}{\partial x})^{2}/f(a,x) \Sigma dx  \nonumber \\
&= lim_{a}\int \sqrt{\frac{a}{\pi}}4a^{2}x^{2}e^{-ax^{2}}\Sigma(x)dx \nonumber \\
&= lim_{a}4a^{5/2}/\sqrt{\pi} \frac{\partial}{\partial a} (a^{-1/2}\int e^{-u^{2}}\Sigma(u/\sqrt{a}) du), u = \sqrt{a}x \nonumber \\
&= lim_{a}4a^{5/2}/\sqrt{\pi}(-\frac{1}{2a^{3/2}}\int e^{-u^{2}}\Sigma(u/\sqrt{a})du \nonumber \\
&+ a^{-1/2}\int \frac{-1}{2a^{3/2}}ue^{-u^{2}}\Sigma'(u/\sqrt{a})du)  \end{align}

The first term vanishes by our asymptotic assumption on $\Sigma$.  We are left with

\begin{align} lim_{a}\frac{4a^{2}}{2\sqrt{\pi}a^{3/2}}\int e^{-v}\Sigma''(\sqrt{v/a})\frac{1}{2\sqrt{va}}dv, (v = u^{2}), \nonumber \\
&= lim_{a}\frac{4}{2\sqrt{\pi}2}\int e^{-u^{2}}\frac{u}{u}\Sigma''(u/\sqrt{a})du \nonumber \\
&= lim_{a}\frac{1}{\sqrt{\pi}}\int e^{-u^{2}}\Sigma''(u/\sqrt{a})du \nonumber \\
&= \Sigma''(0) \nonumber \\
&= \int \delta''(x) \Sigma(x) dx, \end{align}

which proves the claim.

\subsection{Proof of the second Claim}

From the conservation equation for a subset $U$ of $M$, we have that

\begin{align}
0 &= \int_{U \times A}div_{\sigma}(\partial_{\sigma}\sigma_{jk}\frac{\partial \delta}{\partial \sigma_{jk}})dadm \nonumber \\
&= \int_{\partial U \times A}<\partial_{\sigma}\sigma_{jk}\frac{\partial \delta}{\partial \sigma_{jk}},\hat{n}>dSda \end{align}

Now consider the one dimensional case for simplicity:

\begin{align}
0 = \int \delta'\Sigma \nonumber
\end{align}

and

\begin{center} $\delta = lim_{a \rightarrow \infty}\sqrt{\frac{a}{\pi}}e^{-ax^{2}} = lim_{a \rightarrow \infty}g(a,x)$ \end{center}

which implies that

\begin{center} $\delta' = -lim_{a \rightarrow \infty}2\sqrt{\frac{a}{\pi}}axe^{-ax^{2}}$. \end{center}

So

\begin{align}
0 &= -lim_{a \rightarrow \infty}\int_{-\infty}^{\infty}2\sqrt{\frac{a}{\pi}}axe^{-ax^{2}}\Sigma(x)dx \nonumber \\
&= - lim_{a \rightarrow \infty}2\frac{a}{\pi^{1/2}}\int_{-\infty}^{\infty}ue^{-u^{2}}\Sigma(\frac{u}{\sqrt{a}})du \nonumber \\
&= -lim_{a \rightarrow \infty}\frac{a}{\sqrt{\pi}}\int_{0}^{\infty}e^{-t}\Sigma(\sqrt{t/a})dt \nonumber \\
&= -lim_{a \rightarrow \infty}\frac{a}{\sqrt{\pi}} \Sigma(1/\sqrt{a}) \end{align}

and the only way this can be zero is if $\Sigma(1/\sqrt{a}) \sim \frac{1}{a^{p}}$, $p > 1$, for large $a$.

Translating back to our situation, this means that $<\partial_{\sigma}\sigma_{jk},\hat{n}> = 0$

Now, consider the level sets of $\sigma_{mn}$ in $M$.  These sets will be codimension one, and have a normal $\hat{n}$ that is parallel to $\partial_{\sigma}\sigma_{mn}$.  Hence on these sets $<\partial_{\sigma}\sigma_{jk},\partial_{\sigma}\sigma_{mn}> = 0$.  Now sweep out all of $M$ by following the vector field $\partial_{\sigma}\sigma_{mn}$ to create a family of sets covering all of $M$ with this property.

Hence the second Claim is proven.

\subsection{Proof of the third Claim}

From the first and second claims, we establish quickly that

\begin{align}
I(\sigma) &= \int_{M \times A}\frac{1}{4\delta}\frac{\partial \delta}{\partial \sigma_{jk}}\frac{\partial \delta}{\partial \sigma_{mn}}<\partial_{\sigma}\sigma_{jk},\partial_{\sigma}\sigma_{mn}>(det\sigma)^{1/2}dadm \nonumber \\
&= \int_{M \times A}\frac{1}{4}\delta(\sigma(m) - a)\frac{\partial}{\partial \sigma_{jk}}\frac{\partial}{\partial \sigma_{mn}}(<\partial_{\sigma}\sigma_{jk},\partial_{\sigma}\sigma_{mn}>(det\sigma)^{1/2})dadm \nonumber \\
&= \frac{1}{4}\int_{M}\frac{\partial}{\partial \sigma_{jk}}\frac{\partial}{\partial \sigma_{mn}}(<\partial_{\sigma}\sigma_{jk},\partial_{\sigma}\sigma_{mn}>(det\sigma)^{1/2})dm \end{align}

I now demonstrate that this new expression is equivalent to one quarter the integral of the scalar curvature.

First observe that the scalar curvature is the contraction of the Ricci curvature with respect to the metric:

\begin{center}
$R = \sigma^{ij}Ric_{ij}$.
\end{center}

The Ricci tensor has the following expression in coordinates:

\begin{center}
$Ric_{ij} = \partial_{l}\Gamma^{l}_{ik} - \partial_{k}\Gamma^{l}_{il} + \Gamma^{l}_{ik}\Gamma^{m}_{lm} - \Gamma^{m}_{il}\Gamma^{l}_{km}$
\end{center}

\begin{lem}

I claim that the first three terms evaluate to zero after contraction, that is, that the scalar curvature can be written in terms of coordinates as:

\begin{center}
$R = -\sigma^{i\alpha}\Gamma^{m}_{il}\Gamma^{l}_{\alpha m}$
\end{center}

From this point on, without loss of generality, I will making the convenient assumption that the coordinate vectors $\partial_{i}$ have been chosen so that $[\partial_{i},\partial_{j}] = 0$.  This has the natural consequence that the Christoffel symbols are antisymmetric in their lower indices, that is,

\begin{center}
$\Gamma^{k}_{ij} = -\Gamma^{k}_{ji}$
\end{center}

I will also make use of the following two identities:

\begin{equation}
\label{identityone}
\Gamma_{ij}^{m} = \frac{1}{2}(\partial_{i}\sigma_{jk} + \partial_{j}\sigma_{ki} - \partial_{k}\sigma_{ij})\sigma^{km}
\end{equation}

\begin{equation}
\label{identitytwo}
\partial_{i}\sigma_{jk} = \Gamma_{ki}^{l}\sigma_{lj} + \Gamma_{ij}^{l}\sigma_{lk}
\end{equation}
\end{lem}

\begin{proof} (of lemma)
\begin{latexonly}
By $\eqref{identityone}$,
\end{latexonly}

\begin{align}
(\partial_{l}\Gamma_{i\alpha}^{l} - \partial_{\alpha}\Gamma_{il}^{l})\sigma^{i\alpha}
&= \frac{1}{2}\{ \partial_{l}[(\partial_{i}\sigma_{\alpha k} + \partial_{\alpha}\sigma_{ki} - \partial_{k}\sigma_{i\alpha})\sigma^{kl}] - \partial_{\alpha}[(\partial_{i}\sigma_{lk} + \partial_{l}\sigma_{ki} - \partial_{k}\sigma_{li})\sigma^{kl}]\}\sigma^{i\alpha} \nonumber \\
&= \frac{1}{2}\{(\partial_{i}\sigma_{\alpha k} + \partial_{\alpha}\sigma_{ki} - \partial_{k}\sigma_{i\alpha})\sigma^{kl}(-\sigma^{kl}\partial_{l}\sigma_{kl}) \nonumber \\
&- (\partial_{i}\sigma_{lk} + \partial_{l}\sigma_{ki} - \partial_{k}\sigma_{li})\sigma^{kl}(-\sigma^{kl}\partial_{\alpha}\sigma_{kl})\}\sigma^{i\alpha} \nonumber \\
&+ \frac{1}{2}\{\partial_{l}(\partial_{i}\sigma_{\alpha k} + \partial_{\alpha}\sigma_{ki} - \partial_{k}\sigma_{i\alpha}) - \partial_{\alpha}(\partial_{i}\sigma_{lk} + \partial_{l}\sigma_{ki} - \partial_{k}\sigma_{li})\}\sigma^{kl}\sigma^{i\alpha} \nonumber \\
&= -\frac{1}{2}[(\partial_{i}\sigma_{\alpha k})(\partial_{l}\sigma_{kl})(\sigma^{kl})^{2}\sigma^{i\alpha} + (\partial_{\alpha}\sigma_{ki})(\partial_{l}\sigma_{kl})(\sigma^{kl})^{2}\sigma^{i\alpha} \nonumber \\
&- (\partial_{k}\sigma_{i\alpha})(\partial_{l}\sigma_{kl})(\sigma^{kl})^{2}\sigma^{i\alpha}] -\frac{1}{2}[(\partial_{l}\sigma_{ik})(\partial_{\alpha}\sigma_{kl})(\sigma^{kl})^{2}\sigma^{i\alpha}  \nonumber \\
&+ (\partial_{i}\sigma_{kl})(\partial_{\alpha}\sigma_{kl})(\sigma^{kl})^{2}\sigma^{i\alpha} - (\partial_{k}\sigma_{li})(\partial_{\alpha}\sigma_{li})(\sigma^{kl})^{2}\sigma^{i\alpha}] \end{align}

Now
\begin{align}
\frac{1}{2}\{\partial_{l}\partial_{i}\sigma_{\alpha k} + \partial_{l}\partial_{\alpha}\sigma_{ki} - \partial_{l}\partial_{k}\sigma_{i\alpha} - \partial_{\alpha}\partial_{l}\sigma_{ik} - \partial_{\alpha}\partial_{i}\sigma_{kl} + \partial_{\alpha}\partial_{k}\sigma_{li}\}\sigma^{kl}\sigma^{i\alpha} = 0 \end{align}

after one observes that $\sigma^{kl}\sigma_{i\alpha} = \delta_{ki}\delta_{l\alpha}$ and plugs things through.

The remaining problem term, $\sigma^{ik}\Gamma^{l}_{ik}\Gamma^{m}_{lm}$, is zero because $\sigma^{ik} = \sigma^{ki}$ and $\Gamma^{l}_{ik} = -\Gamma^{l}_{ki}$.  A simple argument using these facts plus the judicious replacement of dummy indices will then do the job.

This proves the lemma.
\end{proof}

So it remains to demonstrate that

\begin{center}
$\frac{\partial}{\partial \sigma_{jk}}\frac{\partial}{\partial \sigma_{mn}}(<\partial_{\sigma}\sigma_{jk},\partial_{\sigma}\sigma_{mn}>(det\sigma)^{1/2}) = -\sigma^{i\alpha}\Gamma^{m}_{il}\Gamma^{l}_{\alpha m}(det\sigma)^{1/2}$
\end{center}

Now, under application of the second claim to remove a few unnecessary terms, we get that

\begin{align}
\label{mainsharp}
\frac{\partial}{\partial \sigma_{jk}}\frac{\partial}{\partial \sigma_{mn}}(<\partial_{\sigma}\sigma_{jk},\partial_{\sigma}\sigma_{mn}>(det\sigma)^{1/2}) \nonumber \\
&= \sigma^{i\alpha}[\frac{\partial}{\partial \sigma_{jk}}(\partial_{i}\sigma_{jk})\frac{\partial}{\partial \sigma_{mn}}(\partial_{\alpha}\sigma_{mn}) \nonumber \\
&+ \frac{\partial}{\partial \sigma_{mn}}(\partial_{i}\sigma_{jk})\frac{\partial}{\partial \sigma_{jk}}(\partial_{\alpha}\sigma_{mn})](det \sigma)^{1/2} \nonumber \\
&+ \frac{\partial}{\partial \sigma_{jk}}(\sigma^{i\alpha})[\frac{\partial}{\partial \sigma_{mn}}(\partial_{i}\sigma_{jk})\partial_{\alpha}\sigma_{mn} \nonumber \\
&+ \partial_{i}\sigma_{jk}\frac{\partial}{\partial \sigma_{mn}}(\partial_{\alpha}\sigma_{mn})](det \sigma)^{1/2} \nonumber \\
&+ \frac{\partial}{\partial \sigma_{mn}}(\sigma^{i\alpha})[\frac{\partial}{\partial \sigma_{jk}}(\partial_{i}\sigma_{jk})\partial_{\alpha}\sigma_{mn} \nonumber \\
&+ \partial_{i}\sigma_{jk}\frac{\partial}{\partial \sigma_{jk}}(\partial_{\alpha}\sigma_{mn})](det \sigma)^{1/2} \nonumber \\
&+ \sigma^{i\alpha}\frac{\partial}{\partial \sigma_{mn}}[(\partial_{i}\sigma_{jk})(\partial_{\alpha}\sigma_{mn})]\frac{\partial}{\partial \sigma_{jk}}(det \sigma)^{1/2} \nonumber \\
&+ \sigma^{i\alpha}\frac{\partial}{\partial \sigma_{jk}}[(\partial_{i}\sigma_{jk})(\partial_{\alpha}\sigma_{mn})]\frac{\partial}{\partial \sigma_{mn}}(det \sigma)^{1/2} \end{align}

\begin{latexonly}
Consider the first term.  By repeated application of $\eqref{identitytwo}$ we get:
\end{latexonly}

\begin{align}
\sigma^{i\alpha}[\frac{\partial}{\partial \sigma_{jk}}(\partial_{i}\sigma_{jk})\frac{\partial}{\partial \sigma_{mn}}(\partial_{\alpha}\sigma_{mn})
&+ \frac{\partial}{\partial \sigma_{mn}}(\partial_{i}\sigma_{jk})\frac{\partial}{\partial \sigma_{jk}}(\partial_{\alpha}\sigma_{mn})](det \sigma)^{1/2} \nonumber \\
&= (det \sigma)^{1/2}\sigma^{i\alpha}[(\Gamma_{ki}^{l}\delta_{jl}\delta_{kj} + \Gamma_{ij}^{l}\delta_{jl}\delta_{kk})(\Gamma_{n\alpha}^{\bar{l}}\delta_{m\bar{l}}\delta_{nm} \nonumber \\
&+ \Gamma_{\alpha m}^{\bar{l}}\delta_{m \bar{l}}\delta_{nn}) + (\Gamma_{ki}^{l}\delta_{ml}\delta_{nj} + \Gamma_{ij}^{l}\delta_{ml}\delta_{nk})(\Gamma_{n\alpha}^{\bar{l}}\delta_{j\bar{l}}\delta_{km} + \Gamma_{\alpha m}^{\bar{l}}\delta_{j \bar{l}}\delta_{kn})] \nonumber \\
&= (det \sigma)^{1/2}\sigma^{i\alpha}[\Gamma_{ki}^{k}\Gamma_{n\alpha}^{n} + \Gamma_{ji}^{j}\Gamma_{m\alpha}^{m} + \Gamma_{il}^{l}\Gamma_{m\alpha}^{m} \nonumber \\
&+ \Gamma_{il}^{l}\Gamma_{\alpha m}^{m} + \Gamma_{mi}^{m}\Gamma_{n\alpha}^{n} + \Gamma_{ji}^{m}\Gamma_{\alpha m}^{j} + \Gamma_{ij}^{m}\Gamma_{m\alpha}^{j} + \Gamma_{ij}^{m}\Gamma_{\alpha m}^{j}] \nonumber \\
&= \sigma^{i\alpha}(\Gamma_{mi}^{m}\Gamma_{n\alpha}^{n} - \Gamma_{in}^{m}\Gamma_{\alpha m}^{n})(det \sigma)^{1/2} \end{align}

the last equality obtained using the fact that $\Gamma_{ij}^{k} = -\Gamma_{ji}^{k}$ several times.

Now it might appear that we have neglected terms like $\frac{\partial}{\partial \sigma_{jk}}\Gamma_{ki}^{l}$ but these all vanish, because, for instance,

\begin{align}
\frac{\partial}{\partial \sigma_{jk}}(\partial_{i}\sigma_{jk})
&= \Gamma_{ki}^{l}\delta_{jl}\delta_{kj} + \Gamma_{ij}^{l}\delta_{jl}\delta_{kk} + \frac{\partial}{\partial \sigma_{jk}}(\Gamma_{ki}^{l})\sigma_{lj} + \frac{\partial}{\partial \sigma_{jk}}(\Gamma_{ij}^{l})\sigma_{lk} \nonumber \\
&= \Gamma_{ki}^{l}\delta_{jl}\delta_{kj} + \Gamma_{ij}^{l}\delta_{jl}\delta_{kk} \end{align}

 after we observe that the last two terms cancel after the relabelling of $j \mapsto k$, $k \mapsto j$ in the second last term and using the antisymmetry of the Christoffel symbols in their lower indices.  Similar results hold for all the other terms.

Examining the second and third terms, we see that they are the same after relabelling of dummy indices.  Similarly for the fourth and fifth terms.  Furthermore, we observe that the second and fourth terms are related by the relation (second term) = -2(fourth term), since we know that $\frac{\partial}{\partial \sigma_{jk}}(det \sigma)^{1/2} = \frac{1}{2}\sigma^{jk}(det \sigma)^{1/2}$.

So the last three terms of \begin{latexonly} \eqref{mainsharp} \end{latexonly} cancel, leaving us with only the second term to evaluate:

\begin{align}
\frac{\partial}{\partial \sigma_{jk}}(\sigma^{i\alpha})[\frac{\partial}{\partial \sigma_{mn}}(\partial_{i}\sigma_{jk})\partial_{\alpha}\sigma_{mn} &+ \partial_{i}\sigma_{jk}\frac{\partial}{\partial \sigma_{mn}}(\partial_{\alpha}\sigma_{mn})](det \sigma)^{1/2} \nonumber \\
&= -(\sigma^{i\alpha})^{2}\delta_{ji}\delta_{k\alpha}[(\Gamma_{ki}^{l}\delta_{ml}\delta_{nj} + \Gamma_{ij}^{l}\delta_{ml}\delta_{nk})\partial_{\alpha}\sigma_{mn} \nonumber \\
&+ (\partial_{i}\sigma_{jk})(\Gamma_{n\alpha}^{l}\delta_{ml}\delta_{nm} + \Gamma_{\alpha m}^{l}\delta_{ml}\delta_{nn})](det \sigma)^{1/2} \nonumber \\
&= -(\sigma^{jk})^{2}[(\Gamma_{kj}^{l}\delta_{ml}\delta_{nj} + \Gamma_{jj}^{l}\delta_{ml}\delta_{nk})\partial_{k}\sigma_{mn} + (\Gamma_{nk}^{l}\delta_{mn}\delta_{nm} \nonumber \\
&+ \Gamma_{km}^{l}\delta_{ml})\partial_{j}\sigma_{jk}](det \sigma)^{1/2} \nonumber \\
&= -(\sigma^{jk})^{2}[(\Gamma_{kj}^{m})\partial_{k}\sigma_{mj} + (\Gamma_{mk}^{m} + \Gamma_{km}^{m})\partial_{j}\sigma_{jk}](det \sigma)^{1/2} \nonumber \\
&= -(\sigma^{jk})^{2}\Gamma_{kj}^{m}\partial_{k}\sigma_{mj}(det \sigma)^{1/2}  \end{align}

by antisymmetry of $\Gamma$ in its lower indices.

Continuing:

\begin{align}
-(\sigma^{jk})^{2}\Gamma_{kj}^{m}\partial_{k}\sigma_{mj}(det \sigma)^{1/2} 
&= -(\sigma^{i\alpha})^{2}\Gamma_{\alpha i}^{m}(\Gamma_{i \alpha}^{l}\sigma_{lm} + \Gamma_{\alpha m}^{l}\sigma_{li})(det \sigma)^{1/2} \nonumber \\
&= -\sigma^{i\alpha}(\Gamma_{\alpha i}^{m}\Gamma_{i\alpha}^{l}\delta_{il}\delta_{\alpha m} + \Gamma_{\alpha m}^{l}\Gamma_{\alpha i}^{m}\delta_{il}\delta_{\alpha i})(det \sigma)^{1/2} \nonumber \\
&= -\sigma^{i\alpha}(\Gamma_{mi}^{m}\Gamma_{l\alpha}^{l} + \Gamma_{\alpha m}^{\alpha}\Gamma_{\alpha \alpha}^{m}\delta_{\alpha i})(det \sigma)^{1/2} \nonumber \\
&= -\sigma^{i\alpha}\Gamma_{mi}^{m}\Gamma_{l\alpha}^{l}(det \sigma)^{1/2} \end{align}

Combining these computations together, we get that

\begin{align}
\frac{\partial}{\partial \sigma_{jk}}\frac{\partial}{\partial \sigma_{mn}}(<\partial_{\sigma}\sigma_{jk},\partial_{\sigma}\sigma_{mn}>(det\sigma)^{1/2}) &= \sigma^{i\alpha}(\Gamma_{mi}^{m}\Gamma_{n\alpha}^{n} - \Gamma_{in}^{m}\Gamma_{\alpha m}^{n})(det \sigma)^{1/2} - \sigma^{i\alpha}\Gamma_{mi}^{m}\Gamma_{l \alpha}^{l}(det \sigma)^{1/2} \nonumber \\
&= -\sigma^{i\alpha}\Gamma_{in}^{m}\Gamma_{\alpha m}^{n}(det \sigma)^{1/2} \nonumber \\
&= R(det \sigma)^{1/2} \end{align}

which completes the proof of Claim 3.

\begin{rmk} Note that it is possible to avoid this painful calculation, if we observe that the metric expansion about a point $m \in M$ is

\begin{center} $\sigma_{ij}(m) = \delta_{ij} + R_{(1)ijkl}(m)x_{k}x_{l} + R_{(2)ijklts}(m)x_{k}x_{l}x_{t}x_{s} + O(x^{6})$ \end{center}

and

\begin{center} $\Delta_{\sigma}(m)\sqrt(det(\sigma)) = \sigma_{ij}\sigma_{kl}(m)\partial_{k}\partial_{l}(\delta_{ij} + R_{(1)ijkl}(m)x_{k}x_{l} + ...)\vert_{x = 0}(det(\sigma))^{1/2} = R_{(1)}(det(\sigma))^{1/2}$ \end{center}

Hence from claims 1 and 2 we have 

\begin{center} $I(\sigma) = \int_{M}\Delta_{\sigma}(m)(det(\sigma))^{1/2}dm = \int_{M}R_{(1)}(m)(det(\sigma))^{1/2}dm$ \end{center}

which again completes the proof of Claim 3.

\end{rmk}

\begin{rmk} Note that this may be generalised.  Suppose the $2n^{th}$ term of the metric expansion for $\sigma$ is $R_{(n)i_{1}j_{1}...i_{n}j_{n}}x_{i_{1}}x_{j_{1}}...x_{i_{n}}x_{j_{n}}$.  Then we can similarly show that $\Delta^{n}_{\sigma}(det(\sigma))^{1/2} = R_{(n)}(det(\sigma))^{1/2}$.  We will in fact need this later when we start to look at almost sharp geometries.  \end{rmk}

\subsection{Proof of Claim 4}

$I = \int_{M}R det(\sigma)^{1/2}dm$ is the standard Einstein Hilbert action.  Under our assumption that boundary terms are negligible, it is well known that the variation of $I$ is as claimed.  So it remains to establish the expression for $J$.

Now,

\begin{center} $\delta J = \int_{M}\frac{\partial}{\partial \sigma^{kl}}(\sigma^{ij}\psi_{i}\psi_{j}(det\sigma)^{1/2})\frac{\partial \sigma^{kl}(s)}{\partial s}dp$ \end{center}

for a variation $\sigma^{kl}(s)$ of $\sigma^{kl}$.

Hence

\begin{center} $\delta J = \int_{M}(\bar{\psi}_{i}\bar{\psi}_{j} - \frac{1}{2}\sigma_{ij}\norm{\bar{\psi}}^{2})(det\sigma)^{1/2}\frac{\partial \sigma^{ij}(s)}{\partial s}dp + 2\int_{M}\int_{A}\sigma^{kl}\frac{\partial}{\partial \sigma^{ij}}(\bar{\psi}_{k})\bar{\psi}_{l}\delta (det \sigma)^{1/2}\frac{\partial \sigma^{ij}(s)}{\partial s}dp$ \end{center}

It remains to show that the last term is zero.  Denote $\frac{\partial \sigma^{ij}(s)}{\partial s}\vert_{s = 0}$ as $F$.  

So, with the integral over the manifold $M$ understood,

\begin{align}
\label{mainsharpbound}
2\sigma^{ij}\bar{\psi}_{i}\frac{\partial \bar{\psi}_{j}}{\partial \sigma^{kl}}(det \sigma)^{1/2}F &= -2\frac{\partial}{\partial \sigma_{mn}}(\sigma^{ij}\frac{\partial \bar{\psi}_{j}}{\partial \sigma^{kl}}\partial_{i}\sigma_{mn}(det \sigma)^{1/2}F) \nonumber \\
&\propto \delta_{i \alpha}\delta_{j \beta}(-\sigma^{ij})^{2}\frac{\partial \bar{\psi}_{j}}{\partial \sigma^{kl}}\partial_{i}\sigma_{\alpha \beta}(det \sigma)^{1/2}F \nonumber \\
&+ \frac{\partial^{2}\bar{\psi}_{j}}{\partial \sigma_{\alpha \beta}\partial \sigma^{kl}}\partial_{i}\sigma_{\alpha \beta}(det \sigma)^{1/2}F\sigma^{ij} \nonumber \\
&+ \frac{\partial}{\partial \sigma_{\alpha \beta}}(\partial_{i}\sigma_{\alpha \beta})\sigma^{ij}\frac{\partial \bar{\psi}_{j}}{\partial \sigma^{kl}}(det \sigma)^{1/2}F \nonumber \\
&+ \frac{1}{2}(-\sigma^{\alpha \beta}(det \sigma)^{1/2})\sigma^{ij}(\partial_{i}\sigma_{\alpha \beta})\frac{\partial \bar{\psi}_{j}}{\partial \sigma^{kl}}F \nonumber \\
&+ \frac{\partial F}{\partial \sigma_{\alpha \beta}}\sigma^{ij}\frac{\partial \bar{\psi}_{j}}{\partial \sigma^{kl}}\partial_{i}\sigma_{\alpha \beta}(det \sigma)^{1/2}. \end{align}

We immediately notice that the second term vanishes since

\begin{center} $0 = div_{\sigma}\bar{\psi} = \frac{\partial \bar{\psi}_{j}}{\partial \sigma_{\alpha \beta}}\partial_{i}\sigma_{\alpha \beta}\sigma^{ij} = \sigma^{ij}\partial_{i}\bar{\psi}_{j}$. \end{center}

Also observe that since $\partial_{i}\sigma_{\alpha \beta}\frac{\partial F}{\partial \sigma_{\alpha \beta}} = \partial_{i}F$, we may integrate the last term by parts to obtain

\begin{center} $- F \partial_{i}(\sigma^{ij}\frac{\partial \bar{\psi}_{j}}{\partial \sigma^{kl}}\partial_{i}\sigma_{\alpha \beta}(det \sigma)^{1/2})$ \end{center} 
  
(neglecting boundary terms, of course, since we are assuming that they are negligible anyway)

\begin{align}
&= -F(\sigma^{ij}\frac{\partial \bar{\psi}_{j}}{\partial \sigma^{kl}}(-\sigma^{\alpha \beta}\partial_{i}\sigma_{\alpha \beta}(det \sigma)^{1/2}\frac{1}{2}) \nonumber 
\end{align}

(since $\partial_{i}(det \sigma)^{1/2}$)

\begin{align}
&= \frac{- \sigma^{\alpha \beta}\partial_{i}\sigma_{\alpha \beta}}{2}(det \sigma)^{1/2}) \nonumber \\
&-F\partial_{i}\sigma^{ij}\frac{\partial \bar{\psi}_{j}}{\partial \sigma^{kl}}(det \sigma)^{1/2} -F\sigma^{ij}\partial_{i}\frac{\partial \bar{\psi}_{j}}{\partial \sigma^{kl}}(det \sigma)^{1/2} \end{align}

Now it can be seen that the first and second terms here cancel with the fourth and first terms in our previous expression \begin{latexonly} \eqref{mainsharpbound} \end{latexonly} respectively.

So our original expression reduces to

\begin{align} \{\frac{\partial}{\partial \sigma_{\alpha \beta}}(\partial_{i}\sigma_{\alpha \beta})\sigma^{ij}\frac{\partial \bar{\psi}_{j}}{\partial \sigma^{kl}}(det \sigma)^{1/2} &- \sigma^{ij}\partial_{i}\frac{\partial \bar{\psi}_{j}}{\partial \sigma^{kl}}(det \sigma)^{1/2}\}F \nonumber \\
&= \{\frac{\partial}{\partial \sigma_{\alpha \beta}}(\Gamma_{\beta i}^{l}\sigma_{l \alpha} + \Gamma_{i \alpha}^{l}\sigma_{l \beta})\frac{\partial \bar{\psi}_{j}}{\partial \sigma^{kl}} - \partial_{i}\frac{\partial \bar{\psi}_{j}}{\partial \sigma^{kl}}\}\sigma^{ij}(det \sigma)^{1/2}F \nonumber \\
&=\{(\Gamma_{\beta i}^{l}\delta_{\alpha l}\delta_{\beta \alpha} + \Gamma_{i \alpha}^{l}\delta_{\alpha l}\delta_{\beta \beta})\frac{\partial \bar{\psi}_{j}}{\partial \sigma^{kl}}\sigma^{ij} \nonumber \\
&- \frac{\partial}{\partial \sigma^{kl}}(\sigma^{ij}\partial_{i}\bar{\psi}_{j}) + \frac{\partial \sigma^{ij}}{\partial \sigma^{kl}}\partial_{i}\bar{\psi}_{j}\}(det \sigma)^{1/2}F \nonumber \\
&= 0, \end{align}

 since the first term is zero by the antisymmetry of $\Gamma$ in its lower indices, and the remaining two terms vanish since $div_{\sigma}\bar{\psi} = 0$.

But this was precisely what we needed to show, completing the proof of Claim 4.

\subsection{Further Work}

To deal with the zeroth order case properly I still need to look at a couple of things:

(i) Suppose $M$ has boundary.  What is the impact of the boundary on the resulting equations, ie what boundary terms does one get?

(ii) What happens if $B \neq 0$?  Moreover, can we determine from the resulting equations how $B$ is related to classical electromagnetic effects?

Note that if our conjecture from before, our "holographic principle" is true, (i) is an irrelevant question, i.e. it was completely general for us to consider this variational problem with Dirichlet boundary conditions in order to derive the relevant equations.

Point (ii) will be addressed in a later section.

\subsection{Addendum - simplification of the equations of geometrodynamics}

Thinking a little more about these ideas for deriving physics from sharp manifolds leads one to realise that they lack full generality.  Namely, one is unable to use all natural degrees of freedom.  In fact, it would appear that we have been led astray by the misleading coincidence that the number of
degrees of freedom for this simpler approach is sufficient to describe physics in a four geometry, and also that it is closer intuitively to standard notions of physics, eg 4-vectors and so on.

However, the crucial objects of study are not vectors, they are $(1,1)$ tensors.  So this suggests we should drop the assumption of symmetry in our choice of bilinear form for describing geometrical dynamics.  Then one can split the form into its symmetric and antisymmetric components.  The symmetric
component would correspond to the channel information - and would give rise to the action for the curvature.  The antisymmetric component corresponds to the bound information.  This is clearly the correct picture to bear in mind since it is well known that when one converts the equations of electrodynamics into tensor form one obtains an antisymmetric $(1,1)$ tensor, called the field strength tensor $F^{\alpha\beta}$ (Jackson, "Classical Electrodynamics", page 556 \cite{[J]}).

For the former the tools of Riemannian geometry are valid.  However, for the latter, which I might as well call "Cartan geometry", since all that is occuring is a minor change of axiom, all the theory for Riemannian geometry goes through - so for instance we have the analogue of a unique antisymmetric Levi-Civita connection - and we get a natural curvature tensor, with a corresponding scalar curvature induced for the bound information.

But the Levi-Civita theorem also applies if we have a metric with both symmetric and antisymmetric components, so as to produce a combined connection subject not to assumption of symmetry or antisymmetry.  This in turn induces a natural notion of curvature for our space.

So instead of considering $grad_{\sigma}f = \psi - curl_{\sigma}B$ as we did before in the sharp case for a symmetric $\sigma$, absorb $\psi$ into the dynamics of $\sigma$ and rewrite the above as $grad_{\hat{\sigma}}\delta(\hat{\sigma}) = -curl_{\hat{\sigma}}B$.  Since the right hand side lacks generality (as I have pointed out), replace it naturally with $- grad_{\tau}\delta(\tau)$.  I claim this is a valid choice for any antisymmetric $\tau$.  Then we have the rather simple expression $grad_{\hat{\sigma} + \tau}\delta(\hat{\sigma})\delta(\tau) = 0$, or, if we
use the fact that $\hat{\sigma}$ and $\tau$ are arbitrary,

\begin{center} $grad_{\sigma}\delta(\sigma) = 0$ for any arbitrary $\sigma : M \rightarrow GL(n)$. \end{center}

This leads naturally of course, after mirroring the above arguments and calculations, to a generalised Einstein-Hilbert action

\begin{center} $\int_{M}R_{\sigma}dm = 0$ \end{center}

which we would equate to zero to determine physical solutions (criticality of the information).  Here of course I stress $\sigma : M \rightarrow GL(n)$ may not be symmetric; if it is of course the solutions are the standard Ricci flat or Einstein metrics.  Certainly there is a fair bit to check here, but, as you can see, the advantage of this extra work is a considerable simplification of the theory, particularly in its later incarnations and generalisations.





\section{Physics for almost sharp manifolds}

\subsection{Introduction}

In this section I will attempt to develop, in an analogous manner to before, the equations describing the behaviour of the curvature and mass distributions within a physical manifold for a slightly more sophisticated choice of signal function.  One could think of this as the next term in a perturbative expansion about an ideal metric, so whereas before I was looking at the "zeroth order case", I am now looking at the "first order case".  In particular, our physical information $K$ will now depend on two parameters; the metric parameter $\sigma$ and the expansion parameter $\epsilon$.  Hence it will turn out that in order to optimise $K$ we will have to solve two coupled partial differential equations for these parameters.

Throughout this section I will assume that $B = 0$.

First I will demonstrate that a good choice of signal function is

\begin{center}
$(1 + \Delta_{\sigma}\epsilon + \Delta^{2}_{\sigma}\frac{\epsilon^{2}}{2!})\delta(\sigma(m) - a)$
\end{center}


where $\Delta_{\sigma} = \sigma^{kl}\nabla_{k}\partial_{l}$.


This signal function is motivated by considering the very reasonable signal function

\begin{center}
$f(m,a) = (\frac{1}{k \sqrt{\pi}})^{dim(M)}exp(-\vert (\sigma^{ij}(m) - a^{ij})m_{i}m_{j}\vert/k^{2})$
\end{center}

that I constructed previously.

For $k << 1$ this can be expanded in terms of $\delta$-functions to get the expression which I provided above up to quadratic order in $k$.  I shall now provide a proof of this claim, working, as usual, with the one dimensional case.

Let $f : R \rightarrow R$ be an arbitrary smooth function.

\begin{align}
\frac{1}{k\sqrt{\pi}}\int e^{-(\frac{x}{k})^{2}}f(x)dx &= \frac{1}{\sqrt{\pi}}\int e^{-v^{2}}f(kv)dv \nonumber \\
&= \frac{1}{\sqrt{\pi}}\int e^{-v^{2}}(f(0) + kvf'(0) + \frac{k^{2}}{2}v^{2}f''(0))dv + \text{terms of order } k^{3} \nonumber \\
&= f(0) + 0 + \frac{k^{2}}{2\sqrt{\pi}}\frac{d}{dt}(\sqrt{\frac{\pi}{t}})|_{t = 1}f''(0) + \Theta(k^{3}) \nonumber \\
&= f(0) + (\frac{k}{2})^{2}f''(0) + \Theta(k^{3}) \nonumber \\
&= \int(\delta(x) + (\frac{k}{2})^{2}\delta''(x))f(x)dx + \Theta(k^{3}) \end{align}

But $f$ was arbitrary, so $\frac{1}{k\sqrt{\pi}}e^{-(\frac{x}{k})^{2}} \sim \delta(x) + (\frac{k}{2})^{2}\delta''(x)$.

This naturally extends to the general case as

\begin{center}
$(\frac{1}{k \sqrt{\pi}})^{dim(M)}exp(-\vert (\sigma^{ij}(m) - a^{ij})m_{i}m_{j}\vert/k^{2}) \sim (1 + (\frac{k}{2})^{2}\Delta_{\sigma})\delta(\sigma(m) - a)$
\end{center}

Now $k$ is a dimensionless parameter, albeit very small, so we may define $\epsilon = (\frac{k}{2})^{2}$ as a new dimensionless parameter.



Note from my concluding remarks from the previous section that we may make the simplifying assumption of working with a Riemann-Cartan metric, in which case we do not need worry about torsion or a mass distribution, since they are absorbed into the metric tensor.




A few further preliminary remarks.  Let $k : \sigma \mapsto \Delta_{\sigma}$ be the Laplacian map.  Then observe that $\Delta_{\sigma} \circ \epsilon = k(\sigma \times k^{-1} \circ \epsilon)$, where we are using the property $\Delta_{\sigma \times \tau} = \Delta_{\sigma}\Delta_{\tau}$ of $k$.  Then we get that

\begin{center} $\int_{M}\Delta_{(\sigma ; \epsilon)}(det (\sigma ; \epsilon))^{1/2} = \int_{M}\Delta_{\sigma}(det\sigma)^{1/2}$ \end{center}

for small $\epsilon$, with error $O(\epsilon^{2})$.  So in what follows it makes sense to consider the deformed metric $\bar{\sigma} := \sigma \times k^{-1} \circ \epsilon$ instead of $\sigma$.  In order to make sense of the argument to follow, however, for transparency write $\bar{\sigma} := \hat{\epsilon}\frac{\hat{\bar{\sigma}}}{\hat{\epsilon}}$, where $\hat{\epsilon} = sup_{(x \in M)}(\norm{k^{-1} \circ \epsilon}_{\sigma})$ is the maximal variation and hence constant.

Finally for the argument to follow map $\hat{\epsilon} \mapsto \epsilon$, and $\hat{\bar{\sigma}} \mapsto \sigma$.  The point of doing all of this is essentially so we can treat our perturbative parameter as constant and concentrate wholly on the metric.  We will unravel afterwards according to this key.

\subsection{Result outline}

Now that these fundamentals have been covered we are ready for the main results.

\emph{Claim 1}: Under the assumption that boundary terms are negligible, for the above choice of signal function the total information can be written as

\begin{center}
$4K = \int_{M}(\Delta_{\sigma} + \Delta^{2}_{\sigma}\epsilon + \Delta^{3}_{\sigma}\epsilon^{2})(det(\sigma))^{1/2}dm$.
\end{center}



Note that the condition that boundary terms have negligible contribution seems eminently reasonable in light of our assumptions if the boundary is at infinity, since we assumed a finite mass distribution $E < \infty$ (the equations trivialise to the previous case if $E \rightarrow \infty$ (admittedly with boundary terms)).

\emph{Claim 2}: The total information $K$ can be rewritten as

\begin{center}
$4K = \int_{M}(R + R_{(2)}\epsilon + R_{(3)}\epsilon^{2})dm$
\end{center}



where $R_{(2)}$ is the fourth order geometric invariant defined by the contraction of $CurvRiem$ three times, where $Curv_{ij} = \nabla_{i}\nabla_{j} - \nabla_{j}\nabla_{i} - \nabla_{[i,j]}$ is the curvature operator.  Similarly $R_{(3)}$ is the sixth order geometric invariant.  More generally, we define $R_{(k)}$ to be the $2k^{th}$ order geometric invariant defined by the appropriate number of contractions of $Curv^{k - 1}Riem$.

\emph{Corollary}: Untangling the above via the key, we obtain that the information is

\begin{center} $4K(\sigma,\epsilon) = \int_{M}(R_{\hat{(\sigma;\epsilon)}}\hat{\epsilon} + R_{(2)\hat{(\sigma ; \epsilon)}}\hat{\epsilon}^{2} + R_{(3)\hat{(\sigma ; \epsilon)}}\hat{\epsilon}^{3} + ... )(det \hat{(\sigma;\epsilon)})^{1/2}dm$ \end{center}





\emph{Remarks}.

\begin{itemize}
\item[(1)] Requiring now that $\delta K = 0$ allows us to solve simultaneously not only for the optimal metric $\sigma$, but for the optimal expansion parameter $\epsilon$.  In particular this tells us that $\epsilon$ may vary wildly depending on the nature of our solution.  Furthermore, the requirement that two equations must hold - $\frac{\partial K}{\partial \sigma^{-1}} = 0$ and $\frac{\partial K}{\partial \epsilon} = 0$ - serves as an obstruction to the space of solutions one might expect given one or the other to hold by itself (see the later section on classification).


\item[(2)]  We have neglected the contribution to the information from boundary terms, which we assumed to be negligible.  In fact, due to the nature of the model we are using, this is a necessity.  However it is clear that there is physics that does not fall under this general umbrella.  This will be taken up again in a later chapter, when I talk about statistical stacks and condensed matter physics.
\end{itemize}














\subsection{Proof of Claims 1 and 2}


By definition,

\begin{center}
$4I = \int_{M}\int_{A}\frac{1}{(1 +\Delta_{\sigma}\epsilon + ...)\delta}\norm{(1 + \Delta_{\sigma} \epsilon + \Delta^{2}\frac{\epsilon^{2}}{2})\partial(1 + \Delta_{\sigma}\epsilon + ...)\delta}^{2}$
\end{center}

Expanding this expression and throwing away terms of order $\epsilon^{3}$ or higher, we are left with

\begin{center}
$4I = \int_{M}\int_{A}(\delta^{-1}\norm{\partial_{\sigma}\delta}^{2} + \epsilon\{-\frac{\Delta_{\sigma}\delta}{\delta}\delta^{-1} \norm{\partial_{\sigma}\delta}^{2} + 2\delta^{-1}\Delta\norm{\partial \delta}^{2}\}) + \epsilon^{2}\{ \Delta^{2}(\delta^{-1}\norm{\partial_{\sigma}\delta}^{2}) + 2\delta^{-1}\Delta^{2}\norm{\partial \delta}^{2} - 2\Delta_{\sigma}(\delta^{-1}\Delta\norm{\partial \delta}^{2}\}  + B(\epsilon)$
\end{center}

where $B(\epsilon)$ are boundary terms depending on derivatives of $\epsilon$, which can therefore be neglected.  We are also using the identity

\begin{center} $\frac{1}{1 + \epsilon f + \epsilon^{2}f^{2}/2} = 1 - \epsilon f + \epsilon^{2}f^{2}/2 + ...$ \end{center}

Once more throwing away terms of order $\epsilon^{3}$ and above, we rewrite the above expression as

\begin{center}
$4I = \int_{M}\int_{A}\delta^{-1}\norm{\partial_{\sigma}\delta}^{2} + \int_{M}\int_{A}\epsilon\{(\sigma^{\alpha \beta}\partial_{\alpha}\sigma_{jk}\partial_{\beta}\sigma_{mn} (det \sigma)^{1/2}) \frac{\partial^{2}}{\partial \sigma_{jk} \partial \sigma_{mn}}\Delta_{\sigma}\delta\} + \int_{M}\int_{A}\epsilon^{2}\Delta^{3}(det \sigma)^{1/2}\delta$
\end{center}

This proves the first claim.

But $\Delta^{n}(det \sigma)^{1/2} = R_{(n)}(det \sigma)^{1/2}$, so this is nothing other than

\begin{center} $\int_{M}\int_{A}(R_{(1)} + \epsilon R_{(2)} + \epsilon^{2} R_{(3)} + ...)\delta$ \end{center}

which proves the second claim.

\section{Recovery of standard results}

Here I will demonstrate how standard results in classical physics follow from the equations developed above.  In particular, I will show that Maxwell's equations arise as a special case for the physics of sharp manifolds, and I will show how the Klein-Gordan equation follows as a consequence of the physics of almost sharp manifolds.

First, write the stress energy tensor in terms of the metric as $T(\sigma) = curl_{\sigma}\sigma = curl_{\sigma_{s}}\sigma_{s} + curl_{\sigma_{as}}\sigma_{as}$ where $\sigma_{s}$ is the symmetric part of the Riemann-Cartan metric $\sigma$ and $\sigma_{as}$ is the antisymmetric part.

Then $div_{\sigma}T(\sigma)$ is certainly zero, as required.

We can also rewrite the action

\begin{center} $K = \int_{M}\int_{A}\norm{\partial \delta(\sigma)}^{2}/\delta(\sigma)$ \end{center}

as

\begin{center} $K = I(\sigma) - J(\sigma)$ \end{center}

where

\begin{center} $I(\sigma) = K(\sigma_{s})$ \end{center}

and

\begin{center} $J(\sigma) = \int_{M}\int_{A}\norm{T(\sigma_{s})}^{2}\delta$ \end{center}

In particular there exists a vector field $\phi$ such that $\norm{T}^{2} = \norm{\phi}^{2}$ over $M$.

\emph{Claim 1}.  $\partial f = \phi$.  So if we define a new vector field $\psi$, with $div \psi = 0$, then $\phi = \psi - curl_{\sigma}B$, for some vector field $B$.

\emph{Claim 2}. 
\begin{center} $J = \int_{M}\norm{\bar{\psi} - curl_{\sigma}\bar{B}}^{2}$, \end{center}

where $\psi = \bar{\psi}\delta$ and $curl_{\sigma}\bar{B}\delta = B$.

\emph{Claim 3}. 

\begin{center} $J = \int_{M}(\norm{\bar{\psi}}^{2} - \norm{curl_{\sigma}\bar{B}}^{2})$ \end{center}

So since we already understand the variation of the first term above from before, we need only focus on the second term.

\emph{Claim 4}. We have:

\begin{center} $\int_{M}\norm{curl_{\sigma}\bar{B}}^{2} = \int_{M}(\norm{\nabla \bar{B}}^{2} + (div_{\sigma}\bar{B})^{2})$
\end{center}

provided that we are assuming Dirichlet boundary conditions.

\emph{Claim 5}. There exists $\hat{f}$ such that

\begin{center} $\norm{\nabla \hat{f}}^{2} = \norm{\nabla \bar{B}}^{2} + (div_{\sigma}\bar{B})^{2}$ \end{center}

\emph{Claim 6}. The physical information is

\begin{center} $K(\sigma) = \int_{M}(R - \norm{\bar{\psi}}^{2} + \norm{\partial \hat{f}}^{2})$ \end{center}

\emph{Remark}. If $\bar{\psi} = 0$, then the physical information for a sharp manifold corresponds to the Perel'man functional after a simple change of coordinates.

\emph{Claim 7}. Assuming Dirichlet boundary conditions, the constitutive equation for a sharp manifold with nonzero electromagnetic term is then

\begin{equation}
R_{ij} = \Gamma_{ij} \end{equation}

\emph{Remark}. In the case that $\bar{\psi} = 0$, the first equation is very closely related to the equation for a gradient soliton.  In particular there is a close relation between the information functional for a sharp manifold and the Perelman functional for the Ricci flow.

The first claim is self-evident; it follows from the definition of signal function.  I prove Claims 2 through to 7 for the sharp case.  The perturbative case is analogous.

\subsection{Proof of Claim 2}

Write $B = \hat{B}\delta$.  Then we know that

\begin{center}
$div_{\sigma}(\hat{B}\delta) = 0$
\end{center}

since $B$ is the curl of some fuzzy vector field on $\Lambda$.

So

\begin{align}0 &= (div \hat{B})\delta + <\hat{B},\partial_{\sigma}\delta>(det \sigma)^{1/2} \nonumber \\
&= \{(det \sigma)^{1/2}div \hat{B} - \frac{\partial}{\partial \sigma_{jk}}(\sigma^{i \alpha}\hat{B}_{i}\partial_{\alpha}\sigma_{jk}(det \sigma)^{1/2})\}\delta \nonumber \\
&= \{(det \sigma)^{1/2}div \hat{B} - \hat{B}_{i}\frac{\partial}{\partial \sigma_{jk}}(\sigma^{i \alpha}\partial_{\alpha}\sigma_{jk})(det \sigma)^{1/2}\}\delta \nonumber \\
&= \{div \hat{B} - \hat{B}_{i}\frac{\partial}{\partial \sigma_{jk}}(\sigma^{i \alpha}(\Gamma^{l}_{k \alpha}\sigma_{lj} + \Gamma^{l}_{\alpha j}\sigma_{lk}))\}(det \sigma)^{1/2}\delta \nonumber \\
&= div \hat{B} (det \sigma)^{1/2}\delta
\end{align}

Hence

\begin{center}
$div_{\sigma}\hat{B} = 0$
\end{center}

So $\hat{B}$ is the curl of some vector field, say $\bar{B}$.

This concludes the proof of Claim 2.

\subsection{Proof of Claim 3}

Now, $J = \int_{M}\norm{\psi - B}^{2}/f = \int_{M}(\norm{\psi}^{2} + \norm{B}^{2} - 2<\psi,B>)/f = \int_{M}(\norm{\psi}^{2} + \norm{B}^{2} - 2\norm{B}^{2})/f$

So Claim 3 follows immediately from Claim 2.

\subsection{Proof of Claims 4 to 7}

\begin{proof} (Claim 4).

This part follows from integration by parts:

\begin{center}
$\int_{M}\norm{curl_{\sigma}\bar{B}}^{2} = -\int_{M}<\bar{B},curl_{\sigma}curl_{\sigma}\bar{B}>$
\end{center}

using the assumption of Dirichlet boundary conditions, followed by application of the identity $\epsilon_{ijk}\epsilon_{klm} = \delta_{il}\delta_{jm} - \delta_{im}\delta_{jl}$.
\end{proof}

\begin{proof} (Claim 5).

Now first observe that $\norm{\nabla \bar{B}}^{2} + (div_{\sigma}\bar{B})^{2}$ is positive, since as $\bar{\psi}$ is timelike we must have that $curl_{\sigma}\bar{B}$ is timelike, which implies that $\bar{B}$ must be timelike.  Hence there is a vector $V$ such that $\norm{V}^{2}$ is the above quantity.

Consider $\norm{curl_{\sigma}V}^{2} = \norm{\nabla curl_{\sigma}\bar{B}}^{2} + (div_{\sigma}curl_{\sigma}\bar{B})^{2}$.  This quantity is zero since $\nabla curl_{\sigma}$ and $div_{\sigma}curl_{\sigma}$ are zero operators.  Hence $curl_{\sigma}V = 0$, which implies $V = \nabla \hat{f}$ for some function $\hat{f}$.

This completes the proof of claim 5.
\end{proof}

\begin{proof} (Claim 6).  This claim is an easy consequence of the previous claims and the results of previous sections. \end{proof}

\begin{proof} (Claim 7).  Follows by a simple first variation argument. \end{proof}

\subsection{Classical electrodynamics}

Now, starting with the dynamical equation for a critical sharp manifold, assume $R_{(i)}$ is $O(\epsilon^{i})$, $\norm{\bar{\psi}}^{2}$ is $O(\epsilon)$.  Then $\partial_{i}\hat{f}\partial_{j}\hat{f} = 0$, in other words, $\hat{f}$ is constant.  We interpret $\hat{f}$ then to be the "frequency".

Also recall $0 = \norm{\partial \hat{f}}^{2} = \norm{\nabla \bar{B}}^{2} + (div \bar{B})^{2}$.  Then $div \bar{B} = 0$ and $\norm{\nabla \bar{B}}^{2} = 0$ since $\bar{B}$ is timelike.  So if we interpret $\bar{B}$ as the four-potential, $\bar{B} = (A,\phi)$ we get that

\begin{center} $\nabla \cdot A - \frac{\partial \phi}{\partial t} = 0$ \end{center}

using the Minkowski metric.  This is nothing other than the Coulomb Gauge.

I also claim that $\norm{\nabla \bar{B}}^{2} = 0$ is equivalent to the statement $\Delta \bar{B} = 0$, since we observe then that

\begin{center} $\norm{\nabla curl_{\sigma}\bar{B}}^{2} = \norm{\nabla^{2} \hat{f}}^{2} = \norm{\nabla^{2}\bar{B}}^{2}$ \end{center}

But the expression on the left hand side is zero by a basic operator identity and hence since $\bar{B}$ is timelike the conclusion readily follows.  So $\Delta \bar{B} = 0$, or

\begin{center} $(\nabla^{2} - \frac{\partial^{2}}{\partial t^{2}})A_{x} = 0$

$(\nabla^{2} - \frac{\partial^{2} }{\partial t^{2}})A_{y} = 0$
 
$(\nabla^{2} - \frac{\partial^{2} }{\partial t^{2}})A_{z} = 0$
  
$(\nabla^{2} - \frac{\partial^{2}}{\partial t^{2}})\phi = 0$
\end{center}

But these are the homogeneous electromagnetic field equations.

\subsection{(Classical) quantum mechanics}

If we do a bit of fiddling with the results from the beginning, it is clear that the action for the perturbative case should look something like

\begin{center} $K(\sigma) = \int_{M}(1 + \hat{\epsilon} \Delta + \hat{\epsilon}^{2}\Delta^{2} + ...)(R - \norm{\bar{\psi}}^{2} + \norm{\partial \hat{f}}^{2})$ \end{center}

Let $\partial \hat{f} = 0$.  Interpret $\bar{\psi}$ as the wave function, as intended.  Then a simple variation argument gives the constitutive equations

\begin{center} $(1 + \hat{\epsilon} \Delta + \hat{\epsilon}^{2}\Delta^{2} + ... )R = (1 + \hat{\epsilon} \Delta + \hat{\epsilon}^{2}\Delta^{2} + ...)\norm{\bar{\psi}}^{2}$ \end{center}

\begin{center} $(\Delta + 2\hat{\epsilon}\Delta^{2} + ...)R = (\Delta + 2\hat{\epsilon}\Delta^{2} + ...)\norm{\bar{\psi}}^{2}$ \end{center}

Make the further simplifying assumption that curvature is negligible.  Then to first order we have

\begin{center} $(1 + 2\hat{\epsilon}\Delta)\bar{\psi} = 0$ \end{center}

and

\begin{center} $\Delta(1 + 4\hat{\epsilon}\Delta + ...)\bar{\psi} = 0$ \end{center}

This second can be rewritten as

\begin{center} $(1 + 4\hat{\epsilon}\Delta + ...)\bar{\psi} = V$ \end{center}

for some potential function $V$, from which we conclude after substitution the Klein-Gordon equation for the wave function:

\begin{center} $2\hat{\epsilon}\Delta \bar{\psi}(m) = V(m)$ \end{center}

For flat Lorentz manifolds, in the non-relativistic limit, where $v/c << 1$, we may rescale time in terms of $v/c$ to remove $\hat{\epsilon}$ and the equation degenerates to a parabolic PDE, the Schr\"odinger equation:

\begin{center} $(2\hat{\epsilon}\Delta - \partial_{t})\bar{\psi}(x,t) = V(x,t)$ \end{center}

\section{Classification}

\subsection{Motivation}

Now the classification of all $n$-manifolds for $n \geq 4$ is well known to be impossible, since the word problem on the fundamental group of such objects is known to be in general insoluble.  However, since we have some control of the full Riemann curvature tensor for stable physical manifolds, and since, by their very nature, these manifolds are specially constructed, we might expect some form of classification for them.  Hence we are led to the following

\emph{Conjecture}: There is a classification of all stable, physical Lorentz $n$-manifolds for each $n$ for both the zeroth and first order cases treated previously.  In other words, there are a finite number of families of eigensolutions to these equations, which may be described precisely.

These fundamental solutions may well correspond, at least in relation to the first order equations, to well known ``particles'' which have been observed in experiment.  This leads us to the further

\emph{Conjecture}: There is a one to one correspondence between eigenfamilies of stable, physical Lorentz $4$-manifolds with Dirichlet boundary conditions satisfying the full blown 1st order equation with nonzero electromagnetic term, and the ``particles'' of the Standard model of particle physics.

Clearly in order to analyse the solutions of our PDEs we need to have a general set of tools at our disposal to identify symmetries in the equations.  So one might expect that Lie Group theory would play a pivotal role.

\emph{Motivating Example}: Consider the Dirac equation (with all constants set to one- except for $h$ and $m$):

\begin{center}
$\sigma^{ij}(\frac{h}{m}\frac{\partial}{\partial x_{j}} + A_{j})Q = iQ$
\end{center}

where $\sigma$ is an almost flat Lorentz metric, $Q$ is the probability amplitude and $A$ is the four-current.

This equation has two families of eigensolutions, one corresponding to the electron of particle physics, the other to the positron.  $\mathcal{Z}_{2}$ acts naturally on these solutions as a symmetry of the Dirac equation, suggesting that we might expect more generally for more complicated Lie Groups to allow us to classify the solution space of more sophisticated PDEs.

\subsection{Speculation}

Certainly for small particles, we would expect a natural splitting in local charts of the Lorentz metric into approximately a Riemannian metric for the space coordinates, and the Riemannian metric for the time coordinate.  To be more precise, we expect there to be a coordinate system such that the off block entries in the expression for the metric in the chart are of order $\epsilon^{2}$ and hence can be neglected.  Then, in other words, fundamental solutions would be classified by

\begin{center}
$\{$Prime 3-manifolds$\} \times \{$Prime 1-manifolds$\}$\end{center}


Then, we may observe that from the geometrisation conjecture it is now well known that there are only eight possible choices for the prime components of a three manifold.  One manifolds have, of course, a trivial classification: they can be either $S^{1}$ or $R$.

To get fermions, we need to write the PDEs $L_{1}R_{(1)} := R_{(1)} + \epsilon R_{(2)} = 0$, $L_{2}R_{(2)} := R_{(2)} + 2\epsilon R_{(3)} = 0$ as $K_{i}R_{(i)} = 0$ where "$\bar{K_{i}}K_{i} = L_{i}$" or, more precisely, $<K_{i},K_{i}>_{\hat{\sigma}} = L_{i}$  where $\hat{\sigma}$ is the complexification of $\sigma$ (the complex splitting of our original PDE).  To be more precise, we should write $R_{(i)}(\sigma) = \norm{R_{(i)}(\hat{\sigma})}^{2}_{\hat{\sigma}}$ for the corresponding complexifications of $R_{(i)}$.  Then we would say we get fermions if $K_{i}(R_{(i)}) = 0$.

Anyway, we then get eight fundamental solutions, according to our classification above, times either a circle or a line.  Solutions of the form $X \times S^{1}$ correspond to \emph{virtual} fermions; solutions of the form $X \times R$ we expect to be \emph{real} fermions.  This fits in nicely with the standard model, which lists off eight different types of fermions per ``generation''.

But there is also the other complex factorisation of our orginal PDE, where $L_{i} = \norm{\bar{K_{i}}}^{2}_{\hat{\sigma}}$, $\bar{K_{i}}$ being the complex conjugate to the operator $K_{i}$.  Then we get corresponding solutions to the equation $\bar{K_{i}}\Gamma = 0$, where $\Gamma = Ric - \phi \otimes \phi$.  This leads us to introduce an additional quantum number, which we shall call \emph{spin}.  Fermions with spin \emph{up} are solutions to the equation $K_{i}\Gamma = 0$; fermions with spin \emph{down} are solutions to the equation $\bar{K_{i}}\Gamma = 0$.

What about bosons?  This is a slightly tricky thing to get precisely right, but we can roughly say that we would expect bosons to be the pairing of a spin up fermion with a spin down that otherwise have very similar quantum numbers, via the factorisation of the equation we wrote above.  Virtual bosons will correspond to gluons, of course, of which there will be eight - and the real bosons will correspond to the set $\{$photon, W, Z, Higgs$\}$.  There are of course particle and antiparticle equivalents of these last four (under the symmetry of charge conjugation, $\lambda \mapsto \lambda^{T}$ where $\lambda$ is the antisymmetric component of $\sigma$.


Now, we do not have nearly enough information to talk about the different "generations" of the standard model fermions, but within the scope of this model we might guess that they correspond to the ground and excited states of each of the eight fundamental families of eigensolutions to the PDEs $K_{i}R_{(i)} = 0$ (up to spin).  It does not seem immediately obvious as to why there should be only three eigensolutions per family, since the PDE suggests there should be an infinite quantity.

One can make handwavy arguments saying that the perturbation expansion is no longer valid beyond three but I suspect that the deeper reason is that it is the particle physics version of the Lorentz problem.  To recall, the Lorentz problem is the mystery of why Lorentzian geometry should be the preferred geometry of low energy physics.  The reason that this is suggestive is because one can pair the triple of the three generations of a type of fermion (spacelike directions) with the corresponding force carrier, or boson (timelike direction).  I do attempt to sketch a solution to the Lorentz problem in the next chapter - perhaps the proof can be adapted to shed some light upon this particular mystery too.

As a final throwaway remark, we might naively guess that maybe fermions are somehow more fundamental than bosons (although this is in direct contradiction to what I have just said); that bosons can be decomposed into fermion "conjugate pairs".  Maybe there is a connection here to twistor theory?

In fact, what we are doing is performing a kind of matched asymptotic expansion with gluuing- we are taking a complete model of the particle ($S^{3} \times R$, for instance), and then gluuing it to a very small region in $R^{4}_{1}$ which of course is modelled after a space in which there is no matter (see diagram).  So in this instance, these particles I am describing are a perturbation on an otherwise boring landscape (there are more metrics than simply that for $R^{4}_{1}$ allowable on the outer solution, however; in fact, once again, we have eight alternatives).

I must stress that this whole approximation is only valid when

\begin{itemize}
\item[(a)] Our particles are small, i.e. we are dealing with the lower eigenvalues in this model of the governing equation, and
\item[(b)] $\epsilon$ is very small.
\end{itemize}

So evidently if we are finding solutions of our equations with values of $\epsilon$ which are not $<< 1$, then these solutions are unphysical, because they contradict the assumptions made to get to them.

Later on in the next chapter I will develop tools which might be usable to overcome these limitations, and create better even better models for particle physics.  However assumption (a) still stands by necessity.



\subsection{Extensions}

\begin{rmk}
Here are some thoughts that I had in 2007 which serve to motivate the following chapter.  They should not be taken too seriously, but rather as an informal discussion from which one can launch oneself as a point of departure.
\end{rmk}

Various extensions that could be made might be to consider an "expansion metric" $\tau$, which corresponds to the signal function $(1 + (\tau \cdot \nabla) \cdot (\tau \cdot \nabla))\delta(\sigma(m) - a)$, in order to take into account a thickening or thinning of tubes in the space of metrics approximating a sharp manifold.  Other work that could be done might be to try to examine more closely the link between the particular PDEs that crop up in these problems, the existence of local but possibly not global solutions to them, and the connection to chaos.  A journey into chaos, with its connections to fractal geometry, is worthy of another treatise in its own right.  Various people such as Smale have made great contributions to this area.  Perhaps there is a connection to Morse Theory here?

Furthermore, one might ask, why Lorentzian geometry?  Why should reality seem to be most plausibly described by a four-geometry with three spacelike and one timelike dimensions?  This leads naturally to the idea of considering a potential geometry with possibly infinite dimensions, and performing the same sort of analysis on this new geometry as I have done here.  The potential for "fractional" dimensions or a Cantor dust on various levels should be also allowable.  Then I make the following conjecture:

\emph{Conjecture}: On such an infinite dimensional space, an optimal distribution of information should have support almost everywhere on a finite dimensional submanifold- with three spacelike and one timelike dimensions. In other words, a minimal and stable configuration for the information is for reality structures with these dimensional characteristics.  The remainder of the information for a stable (or almost stable) system should inhabit a set of measure zero.

Why should this be true?  Well, there is a natural ring structure on such a space, as one can see from the observation that there are local charts into the quaternions

\begin{center} $a + bi + cj + dk$ \end{center}

where the real part is associated to the timelike dimension and the "imaginary" part to the spacelike dimensions.  

Of course this is all very artificial and may not really mean anything at all.  Still it suggests that quaternions might have a certain significance.  This does slightly fly in the face of intuition, however, because it is actually the complex numbers that have the most structure and one might therefore expect them to be somewhat more natural.

What perhaps is a more plausible way of looking at why perhaps the support of information in the universe is almost entirely contained on a Lorentz 4-manifold can perhaps be seen by examining the first order perturbative PDE for the equations of physics.  For it to be a nontrivial generalisation of the nonperturbative case, in other words, for there to be any "quantum mechanical effects", one requires that the second degree curvature term, $R_{(2)}$, be nonzero.  But this is only possible if there are at $\emph{least}$ 4-dimensions in which to work, because otherwise this term will vanish.  So then one can again invoke the principle of maximal laziness, or Occam's razor, or your favourite equivalent principle, and say then, well, $\emph{clearly}$, four dimensions is the simplest structure in which to work at low perturbations, so that is how the information in the universe will naturally be structured.

That does not mean that there will not be small (though usually negligible) traces of information in other, "higher" dimensions.

Why would the index be one?  It may be that the best way to organise a four dimensional system is for it to have as much structure as possible.  So we come back to quaternions again, which is clearly equivalent to a certain class of Lorentz manifolds.

The natural things to consider for such statistical manifolds are of course infinite dimensional random matrices, or, alternately, the continuous geometries of von Neumann could be adapted to help flesh out an appropriate model here to prove this conjecture properly.

Interestingly, this conjecture may not be true in certain circumstances, which could lead to some quite weird behaviour indeed.

Another research direction to consider is to look at "higher statistical moments", that is, to build a statistical geometry on top of an existing statistical geometry.  To make myself a bit clearer, we have been considering, so far, a distribution on the space of metrics, $A_{1}$, of a given differentiable manifold, $M$.  It is then a simple process of abstraction to consider different metrics over the space $A_{1}$, so we are considering a distribution of inner products in a space $A_{2}$ over the space of inner products $A_{1}$ of a manifold $M^{n}$.  We could even associate a differential structure to $A_{1}$ by building and gluing charts of appropriate dimension together from the underlying Euclidean structure to make $A_{1}$ essentially a general $n(n - 1)/2$ manifold, ie make it only locally euclidean for the purposes of construction of metrics over it.

In this way, one can potentially build an infinite ladder of spaces,
 $(M = A_{0},A_{1},A_{2},...,A_{i},...)$.  Relatively speculative though the thought is, I believe that for assuming that all distributions are almost sharp that the associated Banach or Hilbert spaces corresponding to the "Quantum Mechanics" of such systems form the pattern $B_{0}$, $B_{1} =$ dual of $B_{0}$, $B_{2} =$ dual of $B_{1}$, etc, if we view much of modern analysis as the study of the structure of almost sharp statistical geometries.

Perhaps the construction of higher statistical moments, together with ideas from continuous geometry, consideration of more general signal functions and also taking into account non-equilibrium dynamics, that is, dynamics that does not involve realising the lower bound of the Cramer-Rao inequality (the last a very strong maybe), would lead to interesting and significant physics beyond the standard model.  As we shall see in the material to follow, this is in a certain way indeed the case.


\chapter{Turbulent Geometry}

Fractals are generally wily and unruly beasts.  There are many unresolved questions about the ones that are known, and even the simplest (such as the Mandelbrot set) as notorious for their complexity and depth of structure.  Is it useful to model the atoms of reality to some degree as fractal structures?  Certainly there are many phenomena in nature that exhibit scale free behaviour, so it would appear that the answer to this question might be yes.  Certainly in cosmology, it is well known that the large scale structure of galaxies and interstellar dust appears to exhibit characteristic properties of being fractal, at least between certain distance scales - a self-similarity across many orders of magnitude.

Turbulence is another mystery that defies explanation.  From waves crashing down onto a beach and foaming onward, to even fairly simple systems where water is pushed through a pipe at high pressure which contains discrete jumps in tube diameter, turbulence is ubiquitous, but is not very well understood.  It is often confused as a chaotic phenomenon, but it does not have to be.  For instance, a flow can converge to one with turbulent, stable eddies that have predictable and reproducible properties.  The flow over the surface of a dimpled golf ball is also turbulent, but, if anything, it allows an experienced golfer more control over its trajectory.  Obviously, it would be marvelous to have some model that at least gave us some method of understanding such phenomena, at least while they are stable!

In this chapter I attempt to develop a mathematics that may go towards helping to understand these objects and processes in more detail.

\section{Preliminaries}

\subsection{Fractional Calculus}

Before I go any further, I will need to describe the fractional calculus.

A fairly natural question to ask is, is there a way, for any real number $\alpha > 0$, to take the $\alpha^{th}$ derivative of a function, or the $\alpha^{th}$ integral?

It turns out that there is such a way, however, whereas for integral $\alpha$ differentiation is purely a local operation, in general boundary information is required.

For some motivation, define $(Jf)(x) = \int_{0}^{x}f(t)dt$, $(J^{2}f) = \int_{0}^{x}(Jf)(t)dt$ etc.

Then clearly $(J^{n}f)(x) = \frac{1}{(n-1)!}\int_{0}^{x}(x-t)^{n-1}f(t)dt$ for natural numbers $n$

and this definition extends naturally to all positive real numbers $\alpha$:

\begin{center}
$(J^{\alpha}f)(x) = \frac{1}{\Gamma(\alpha)}\int_{0}^{x}(x-t)^{\alpha - 1}f(t)dt$
\end{center}

We then define the $\beta^{th}$ fractional derivative of a function $f$ as

\begin{center}
$(D^{\beta}f)(x) = ((D^{n} \circ J^{n-\beta})f)(x)$
\end{center}

for any integer $n > \beta$.

These definitions of course extend naturally to multivariable calculus and calculus on manifolds.

\subsection{Basic tools and motivation}

Here I shall define the necessary tools and the objects of my study henceforth.  The aim is to be able to model spaces of generalised fractal dimension, in a sense which I will soon make more precise.  The motivation was to attempt to fabricate a machinery that would, once and for all, provide the means to validate rigorously our previous educated guess as to why a stable physical geometry should be Lorentzian.  In other words, we should be able to say why physics as we usually understand it should take place in Lorentzian space.  Furthermore, I am motivated by the need to remove myself of dimensional dependence, since if we can model our base space as an infinite dimensional manifold with a certain specific fractal measure, we will have that if we take a statistical manifold on top of our statistical manifold (for the base space) instead of having a quadratic growth in the size of the dimensions, we will go from $\infty$ to $\infty$.  So, with such a theory, we should be able to treat such stacked spaces in a more streamlined, natural, and elegant fashion.

Of course, it seems quite logical that once we have determined and fleshed out the necessary concepts, it should then be possible to use them to determine new physics, after hitting the relevant information with the Cramer-Rao inequality and evaluating the critical points.  This shall become the ultimate aim of this section.

Before I mention the main ideas here, I should talk first of fractal geometry. By this I do not mean fractal geometry in the standard sense, where one deals with Cantor dust like objects.  Rather, I am interested in considering a variety of a continuum of objects where the dimension at each point in the set takes a value on the real line and varies from point to point.  More generally, in turbulent geometry, I am interested in extending the notion of the dimension of an infinitesimal piece of a set to $R^{n}$, or, more precisely, to a Riemannian $n$-manifold $N$, that depends somehow on the base space $M$.   

The key idea is to observe that, if $\psi$ is some distribution over the real line, we can measure the $r^{th}$ fractal dimension of $\psi$ by computing $\int_{R}\int_{R}\psi(x)\partial^{r}_{(a)}\delta(\sigma(x) - a)dadx$, where $\sigma$ is an inner product on the tangent space of $R$, and $\partial^{r}$ is the usual extension of the derivative to the real line.  This can easily be checked. We immediately observe that there is a key connection to statistical geometry here.  Pursuing this analogy, we would like to be able to sensibly define $\partial^{r}$ where $r = (r_{1},...,r_{n})$ is the local expression of an element of some manifold $N$.

First define $g$ to be a point in $R^{r}$ if it is a map

\begin{center} $ g : (...((R^{r_{1}} \rightarrow R^{r_{2}}) \rightarrow R^{r_{3}}) \rightarrow ... ) \rightarrow R^{r_{n}}$ \end{center}

Now, define $\partial^{r}(f)$ to be the function $g$ in $R^{r}$ which maps $\partial^{r_{1}}f \mapsto \partial^{r_{2}}f \mapsto ... \mapsto \partial^{r_{n}}f$.

Then we can make sense of the $r$-measure of a $\psi$ distribution over a manifold $M$ in an analogous fashion:

\begin{center}
$\int_{M}\int_{A}\psi(m)\partial^{r}_{(a)}\delta(\sigma(m) - a)dadm$
\end{center}

However certain issues remain unresolved, such as how to develop dynamics from this language.  That will be the focus of the next section.

\subsection{Construction of the information}

We seek to compute the information of a statistical manifold with a completely general signal function, subject to almost sharp turbulence.  This should serve as a useful model for at least certain physical phenomena that lie beyond the scope of the standard model.

First I will have to determine how to make sense of the information for this problem.  In fact, a two tier approach is necessary.  Suppose now instead of having one definite dimension of each point in our space we have a statistical distribution.  Furthermore, suppose we define the turbulent information as follows:

\begin{center} $I_{(turb)} = \int_{M}\int_{A}\rho_{(\vec{g})}dadm$ \end{center}

where $\rho_{(\vec{g})}$ is the density of the Fisher information with respect to a vector of signal functions $\vec{g}$.  Then this is a natural interpretation of the fractal dimension, given in terms of the standard language of statistical geometry from before, such that the physical information is:

\begin{center} $I_{(phys)} = \int_{M}\int_{A}(\partial f^{*}\rho_{(f)}; \rho_{(\vec{g})})dadm$ \end{center}

where $\partial f^{*}$ is the derivative with respect to the signal function $f$.

I will need to go into some more detail about $\partial f^{*}$.  If we look at the motivational example of the signal function $\delta (\sigma(m) - a)$, we would like to somehow have some sensible notion compatible with the idea that $(\frac{\partial}{\partial a} \delta ; \kappa)$ for some function $\kappa$ will give an effective dimension of $\kappa(m,a)$ at $(m,a)$, but does not interfere with our previous notion of statistical derivative over the domain $M \times A$ for plain statistical geometries.

In particular we are interested in a derivative with respect to the domain $\{ M \times A \} \rightarrow R$, the space of signal functions.  Consider

\begin{center} $\frac{\partial}{\partial \delta^{-1}} = \frac{\partial}{\partial \sigma^{-1}}\frac{\partial \sigma^{-1}}{\partial \delta^{-1}} = \frac{\partial}{\partial \sigma^{-1}}\frac{\partial \delta}{\partial \sigma} = \frac{\partial \delta}{\partial \sigma}\frac{\partial}{\partial \sigma^{-1}}$ \end{center}

with the last equality following because $\frac{\partial}{\partial \sigma}\frac{\partial}{\partial \sigma^{-1}} = \frac{\partial}{\partial Id}$ is the zero operator.

If we now use this as a prototype for our idea of turbulent derivative, we wish to look at

\begin{center} $(\frac{\partial}{\partial \delta}\delta ; \kappa)^{1/2} = ((\frac{\partial}{\partial \sigma}\delta ; \kappa)(\frac{\partial}{\partial \sigma^{-1}}\delta ; \kappa))^{1/2}$ \end{center}

which we may notate shorthand as

\begin{center} $(\partial \delta^{*} \delta ; \kappa) := (\frac{\partial}{\partial \delta}\delta ; \kappa)^{1/2}$ \end{center}

This will provide the desired effect of providing dimension $\kappa(m,a)$ at the location $(m,a)$.  This notation also naturally extends to more general signal functions.

We need to check that this is well defined, of course - namely that $I_{(phys)}$ is invariant under alteration of $\rho_{(g)}$ by a boundary term.  This may require a further condition to be imposed, such as criticality or some form of conservation law, more likely.

Regardless, it now seems clear that the natural space to use as our "dimension manifold" at $p \in M$ for the $a^{th}$ inner product is $T_{(p,a)}M$.

A special case of the above, when $f$ and $g$ are both sharp, with $f(m,a) = \delta(\sigma(m) - a)$ and $\vec{g}(m,a) = \delta(\vec{\tau}(m) - a)$, causes the information $I_{(phys)}$ to reduce to the rather elegant expression

\begin{center} $I_{(phys)} = \int_{M}\partial^{R_{\vec{\tau}}}R_{\sigma}$ \end{center}

and of course the turbulent information to be

\begin{center} $I_{(turb)} = \int_{M}R_{\vec{\tau}}$ \end{center}

What still is not clear, however, is whether we should merely require that $I_{(phys)}$ be critical - the most likely bet - or whether both informations should be critical.

The information that we are interested in studying is for almost sharp turbulence:

\begin{center} $I_{(turb)} = \int_{M}(1 + \Delta_{\vec{\tau}}\epsilon + \Delta^{2}_{\vec{\tau}}\frac{\epsilon^{2}}{2!} + ... )R_{\vec{\tau}}$ \end{center}

and a completely general signal function for the base space.  If we evaluate things by brute force this leads to expressions which are bulky and difficult to deal with, so I will not do so here.  Instead it will turn out that with a bit more work and development of appropriate notation these expressions can be reduced to something more manageable.

Note that a completely general signal function is of the form $f(m,a) = \int_{A}F(m,b)\delta(\sigma_{b}(m) - a)db$, and, without turbulence, the information would be $\int_{M}\int_{A}e^{h}\Delta_{\sigma_{b}}(h)dbdm$ where $F = e^{h}$.


\subsection{The Correspondence Principle}

In the previous discussion, I was interested in signal functions of the form

\begin{center} $f(m,a) = \int_{A}F(m,b)\delta(\sigma_{b}(m) - a)db$ \end{center}

with sharp, or almost sharp turbulence defined on them, given by a signal function of the form

\begin{center} $\vec{g}(m,a) = (1 + \epsilon \Delta_{\vec{\tau}})\delta(\vec{\tau}(m) - a)$ \end{center}

in order to perhaps gain an understanding of entanglement physics.

The base signal function is an example of the most general signal function one can get over a statistical manifold without additional structure.  However, there is a key problem with looking at objects of this type, because it is not at all clear what their connection to real world physics is; there are few obvious clues as to how to connect the statistical dynamics predicted by these things with experiment, and certainly no clear way to connect to practical application.

However, we might make the observation that sharp turbulent manifolds share similar properties to general signal functions over statistical geometries.  In particular there is integration twice over the space of inner products, and once over the manifold itself.  This suggests in turn the following question:

\begin{quest} Is there a correspondence between sharp turbulent manifolds and general statistical geometries? \end{quest}

It turns out that the answer to this question is \emph{yes}:

\begin{thm} (Correspondence principle).  Let $f(m,a) = \delta(\sigma(m) - a)$ and $g(m,a) = \delta(\vec{\tau}(m) - a)$ be sharp signal functions.  Then there is a correspondence $(f,g) \mapsto (F,\sigma)$ such that the signal function volume element associated to $(f,g)$, $(f;g)dV = (\partial f^{*}f ; g)dV(f;g)$, corresponds to the signal function volume element associated to $(F,\sigma)$, $\bar{f}(m,a)dm = \int_{A}F(m,b)\delta(\sigma_{b}(m) - a)(det(\sigma_{b}))^{1/2}dbdm$. \end{thm}

\begin{proof}
First of all,

\begin{center} $(f;g)(m,a,b) = (\partial f^{*}f(m,a) ; g(m,b))$ \end{center}





Integrate by parts so that $(\partial f^{*} ; \delta(\tau(m) - b))$ is acting on the volume element.  By general nonsense we may then realise this as an eigenfunction $\rho(m,b)$ of a volume element induced by a new family of metrics $\sigma_{b}$.  Hence

\begin{center} $(f;g) = \int_{A}\rho(m,b)\delta(\sigma(m) - a)(det(\sigma_{b}))^{1/2}db$ \end{center}

By absorbing $\rho$ partially into $\delta$ we find a new function $F(m,b)$ such that

\begin{center} $(f;g) = \int_{A}F(m,b)\delta(\sigma_{b}(m) - a)(det(\sigma_{b}))^{1/2}db$ \end{center}

which completes the proof.
\end{proof}

\begin{rmk} It is clear as corollary to this result then that entanglement physics may well be described to a first approximation by second order turbulence.  That is, geometries with actions of the form

\begin{center} $I(\sigma,\tau,\phi) = \int_{M}(\partial \sigma^{*} R(\sigma) ; \partial \tau^{*} R(\tau) ; R(\phi))dm$  \end{center}

would seem to be the key models to look at.
\end{rmk}

\begin{rmk} Another key point to mention is that we are not looking at turbulent geometry in full generality here to realise the correspondence; in particular we are looking only at fractal geometry.  Our signal functions are both sharp.  $g$ is not a vector; it is merely a scalar functional. \end{rmk}

\subsection{Further generalisations}

It is clear that merely looking at a vector of signal functions is not enough for full generality.  For instance,  we might be interested in transformations of this vector, such as if it were only a representation from a coordinate chart of an enveloping space $N$.  In particular, we are interested in signal functions of the form

\begin{center}
$h(m,n,a,b) = \partial^{g(m,n,b)}f(m,a)$.
\end{center}

where $g((m,n),b)$ is a signal function on $M \times N$, for instance of the form $\delta(\tau(m,n) - b)$ if $g$ is sharp.  Of course $M \times N$ induces a new topology on $M$, since $\tau$ induces a new metric for each $n \in N$, $\bar{\tau}_{n}(m) := \tau(m,n_{1}) \rightarrow \tau(m,n_{2}) \rightarrow ... \rightarrow \tau(m,n_{k})$.

But this is not entirely general enough; we would like to have a signal function of the form

\begin{center} $h(m,a,b,c) = (f(m,a) ; g(m,b,c))$ \end{center}

where $g(m,b,c) = G(m,c)\delta(\sigma_{c}(m) - b)$.  But by the correspondence principle this is equivalent to

\begin{center} $(f(m,a) ; \alpha(m,b) ; \beta(m,c))$ \end{center}

where $\alpha$ and $\beta$ are sharp.  So we can see that turbulent geometry is just a special case of fractal fractal geometry, which we are interested in for its potential applications to entanglement physics.  By abuse of terminology I will also call the latter \emph{second order turbulent geometry}, and geometry with fractal measure \emph{first order turbulent geometry}.




\section{Cramer-Rao for turbulent statistical manifolds}

\subsection{Introduction}

What is absolutely necessary for us to show before we can proceed, in order to justify the calculation in our final section, is to show that the physical information satisfies some form of generalised Cramer-Rao inequality for the class of objects of this section which I choose to call turbulent statistical manifolds.

In other words, we would like $I_{(phys)} = I_{(f,\vec{g})} := \int_{M}\int_{A}\partial_{f}(\rho_{f} ; \rho_{\vec{g}})dadm \geq 0$ for all $f$ and $\vec{g}$, subject possibly to some sort of proviso - a "conservation relation" linking the two.  Here of course $\rho_{k} = \frac{\norm{\partial_{k}k}^{2}}{k}$ is the usual information density for a standard statistical manifold.

I shall hazard a guess and suggest that the condition which we need is that $I_{(turb)} = I_{g} := \int_{M}\int_{A}\rho_{\vec{g}}dadm$ is critical; that is, it is zero and $\delta I_{\vec{g}} = 0$.

In fact, it turns out we do not need this additional condition.  In fact, it turns out that merely the following is true:

\begin{thm} (Turbulent EPI Principle).  Let $\Lambda = (M,A,f,\vec{g})$ be a statistical turbulent manifold.  Then

\begin{center} $I_{(phys)} = \int_{M}\int_{A}\partial_{f}(\rho_{f} ; \rho_{\vec{g}}) \geq 0$ \end{center}

where $\rho_{f}$, $\rho_{\vec{g}}$ are the Fisher information densities associated to the signal functions $f$ and $\vec{g}$.
\end{thm}

To prove this result I will need to play the same game as before and develop an appropriate idea of weak and strong turbulent estimators.  However, due to the complexity of the issues involved, my treatment will be at best a sketch of what a properly rigourous treatment would require.


\subsection{Turbulent estimators and the Cramer-Rao inequality}

\begin{dfn} Let $\Lambda = (M,A,f,\vec{g})$ be a statistical turbulent manifold, where $f$ is the signal function on the base and $\vec{g}$ is the signal function vector on the turbulent superstructure.  (In this case signal functions take the form $h(m,a,b) = (\partial f^{*} f(m,a) ; \vec{g}(m,b))$.)  A turbulent estimator on this structure is a tuple $(u,\vec{v})$ where $u$, $v_{i}$ are maps from the sample space $M \times A$ to $M$.  An unbiased turbulent estimator is an estimator that satisfies the condition that

\begin{center} $E_{m}^{f \times \vec{g}}(\partial f^{*} \theta \circ u ; \vec{\phi} \circ \vec{v}) = (\partial f^{*} \theta(m) ; \vec{\phi}(m))$ \end{center}


where $\partial f^{*}$ is the turbulent statistical derivative with respect to $f$,  $E_{m}^{f \times \vec{g}}(\kappa) = \int_{A}\int_{A}(\partial f^{*}f(m,a) ; \vec{g}(m,b))\kappa(m,a,b)dbda$, and $\phi_{i}, \theta : M \rightarrow R^{n}$ are coordinate charts on $M$.
\end{dfn}


\begin{lem} (Factorisation). $E_{m}^{(f \times \vec{g})}(\partial f^{*} \tau ; \sigma) = \int_{A}(\partial f^{*}( f\tau );  (\vec{g} \sigma))$.  Equivalently, $(\partial f^{*}f ; \vec{g})(\partial f^{*} \tau ;  \sigma) = (\partial f^{*} (f\tau) ; ( \vec{g} \sigma))$.  \end{lem}


\begin{lem}  Unbiased turbulent estimators satisfy the following relation:

\begin{center} $E_{m}^{f \times \vec{g}}\{ (\partial f^{*} \theta^{i} \circ u ; \vec{\phi}^{k} \circ \vec{v})ln_{f(u)}\partial_{f}ln_{\vec{g}(\vec{v})}\partial_{\vec{g}}(\partial f^{*} f(u) ; \vec{g}(\vec{v})) \} = 0 $ \end{center} \end{lem}


\begin{proof} The result follows by differentiating the definition of an unbiased turbulent estimator on both sides and then integrating by parts, as with the simpler case.
\end{proof}

\begin{dfn} The Fisher information tensor $h^{kl}_{ij}(f,g)$ (corresponding to the likelihood functions $f,g$), is defined to be

\begin{center} $h^{kl}_{ij}(f,\vec{g}) := E_{m}^{f \times \vec{g}} \{ \Gamma^{ik}(f,\vec{g})\Gamma^{jl}(f,\vec{g}) \}$ \end{center}

where

\begin{center} $\Gamma^{ik}(f,\vec{g}) = (ln_{f} \circ \partial_{f}^{i})(ln_{g} \circ \partial_{g}^{k})(\partial f^{*} f ; g)$ \end{center}


\end{dfn}

\begin{thm} (Turbulent Cramer-Rao Inequality).

\begin{center} $cov^{f \times \vec{g}}\{ (\partial f^{*} \theta^{i} \circ u ; \vec{\phi}^{k} \circ \vec{v}), (\partial f^{*} \theta^{j} \circ u ; \vec{\phi}^{l} \circ \vec{v}) \} - h^{kl}_{ij}(f(u),\vec{g}(\vec{v})) \geq 0$
\end{center}


where $cov^{f \times \vec{g}}(v^{ik},w^{jl}) := E_{m}^{f \times \vec{g}}(v^{ik}w^{jl})$.
\end{thm}

\begin{proof}  Note that $E_{m}^{f \times \vec{g}}(T^{ik}_{\alpha \beta}T^{jl}_{\gamma \delta})$ is certainly nonnegative, for

\begin{center} $T^{ik}_{\alpha \beta} := (\partial f^{*}(\theta^{i} \circ u);(\vec{\phi}^{\alpha} \circ \vec{v})) - E_{m}^{f \times \vec{g}}\{ \partial f^{*}(\theta^{i} \circ u) ; (\vec{\phi}^{\alpha} \circ \vec{v}) \} - \{ \partial f^{*}\frac{\partial}{\partial \theta^{k}}ln(f(u)) ; \frac{\partial}{\partial \vec{\phi}^{\beta}}ln(g(v)) \} h_{ik}^{\alpha \beta}(f(u),\vec{g}(\vec{v}))$ \end{center}

Expanding, we get

\begin{center} $cov^{(f \times \vec{g})}( \{\partial f^{*}(\theta^{i} \circ u); (\vec{\phi}^{k} \circ \vec{v})\},\{ \partial f^{*}(\theta^{j} \circ u) ; (\vec{\phi}^{l} \circ \vec{v})\}) - 2E_{m}^{(f \times \vec{g})}( \{ (\partial f^{*}(\theta^{i} \circ u) ; (\vec{\phi}^{k} \circ \vec{v})) - E_{m}^{(f \times \vec{g})}( \partial f^{*}(\theta^{i} \circ u); (\vec{\phi}^{k} \circ \vec{v})) (\partial f^{*} \frac{\partial}{\partial \theta^{\alpha}}ln(f(u)) ; \frac{\partial}{\partial \vec{\phi}^{\beta}}ln(\vec{g}(\vec{v})))\}h_{j \alpha}^{l \beta}(f(u),\vec{g}(\vec{v}))) + h_{ij}^{kl}(f(u),\vec{g}(\vec{v})) \geq 0$ \end{center}

But, observing first of all that $h$ is statistically independent, then by the lemma the second term reduces to 

\begin{center} $2E_{m}^{(f \times \vec{g})}( \{ - E_{m}^{(f \times \vec{g})}( \partial f^{*}(\theta^{i} \circ u); (\vec{\phi}^{k} \circ \vec{v}))(\partial f^{*} \frac{\partial}{\partial \theta^{\alpha}}ln(f(u)) ; \frac{\partial}{\partial \vec{\phi}^{\beta}}ln(\vec{g}(\vec{v}))) \}h_{j \alpha}^{l \beta}(f(u),\vec{g}(\vec{v})))$ \end{center}

using $E_{m}(- v) = -E_{m}(v)$, together with the definition of unbiased turbulent estimator, we get

\begin{center} $-2E_{m}^{(f \times \vec{g})}(  (\partial f^{*}(\theta^{i} \circ u); (\vec{\phi}^{k} \circ \vec{v}))(\partial f^{*} \frac{\partial}{\partial \theta^{\alpha}}ln(f(u)) ; \frac{\partial}{\partial \vec{\phi}^{\beta}}ln(\vec{g}(\vec{v}))) h_{j \alpha}^{l \beta}(f(u),\vec{g}(\vec{v})))$ \end{center}

Note that $\frac{\partial f}{\partial \theta^{j}} = \frac{\partial ln(f)}{\partial \theta^{j}}f$, and also $(\partial f^{*} \frac{\partial ln(f)}{\partial \theta^{i}}f; \frac{\partial ln(\vec{g})}{\partial \vec{\phi}^{j}}\vec{g}) = (\partial f^{*} \frac{\partial ln(f)}{\partial \theta^{i}}; \frac{\partial ln(\vec{g})}{\partial \vec{\phi}^{j}})(\partial f^{*}f; \vec{g})$ by the factorisation lemma.

One concludes that

\begin{center}  $(\partial f^{*} \frac{\partial ln(f)}{\partial \theta^{i}}; \frac{\partial ln(\vec{g})}{\partial \vec{\phi}^{j}})(\partial f^{*}f; \vec{g}) = (\partial f^{*} \frac{\partial f}{\partial \theta^{i}} ; \frac{\partial \vec{g}}{\partial \vec{\phi}^{j}})$ \end{center}

from which follows, after integration by parts that the second term in the original expansion simplifies to

\begin{center} $-2E_{m}^{(f \times \vec{g})}(h_{ij}^{kl}) = -2h_{ij}^{kl}$ \end{center}

completing the proof.

\end{proof}

\begin{dfn} A turbulent maximum likelihood estimator (tmle) is an unbiased estimator $(u,v)$ that realises a local maximum of $(\partial f^{*}ln(f \circ u); ln(\vec{g} \circ \vec{v}))$. \end{dfn}

\begin{lem} Any tmle realises the Cramer-Rao lower bound. \end{lem}

\subsection{Application to physical estimators}

Similarly to our previous experience with statistical manifolds, we have a notion of weakly unbiased turbulent estimator, and also of weak Cramer-Rao inequality, where in this case we integrate over $M$ in order to obtain our result.

In particular we have

\begin{thm} (Weak turbulent Cramer-Rao inequality).  Write $f(m,a) = \int_{A} F(m,b)\delta(\sigma_{b}(m) - a)db$, $\vec{g}(m,a) = \int_{A} G(m,b)\delta(\tau_{b}(m) - a)db$.  Let $f_{ij} = \int_{A}F(m,b)\sigma_{ij}(m,b)db$, similarly for $\vec{g}$.  Note that we need not worry about the role of derivatives of $F$ and $\vec{G}$ since as $f$,$\vec{g}$ are probability distributions such contributions will vanish at infinity.  Then

\begin{center} $\int_{M}\int_{A}(\partial f^{*} f_{ij} ; \vec{g}_{kl})(cov_{a}^{(f \times \vec{g})}( \{\partial f^{*}(\theta^{i} \circ u) ; (\vec{\phi}^{k} \circ \vec{v}) \} , \{\partial f^{*}(\theta^{j} \circ u) ; (\vec{\phi}^{l} \circ v) \}) - h_{ij}^{kl}(a)(f(u),\vec{g}(\vec{v})))dadm \geq 0$ \end{center}

where $cov_{a}$ is the covariance density function and $h_{ij}^{kl}(a)$ is the Fisher information tensor density.

\end{thm}

which follows from appropriate generalisations of the prior results.

\begin{lem} Any weak turbulent unbiased estimator may be written as

\begin{center} $(u,v) = (\partial f^{*}f , \vec{g})$ \end{center}

\end{lem}

\begin{proof} To prove this we need the turbulent version of Stoke's theorem. However after this the rest is easy.  \end{proof}

\begin{thm} (Turbulent Stokes).  Let $\omega$ be an $n-1$ forms on an $n$ manifold $M$, and $\kappa$ a vector valued function on $M$.  Then

\begin{center}
$\int_{M}d(\partial \omega^{*} \omega ; \kappa) = \int_{\partial M}(\partial \omega^{*} \omega; \kappa)$ 
\end{center}

\begin{proof}
The key observation is that the problem may be converted into a repeated application of Stoke's theorem in the standard case, from which the result quickly follows.
\end{proof}

\end{thm}

\begin{rmk}
The statistical version of the above follows by merging this result with the statistical stokes theorem (see chapter 2 or chapter 7).  I will not go into this here.
\end{rmk}

Suppose now $(u,\vec{v})$ is a weak tmle.  Then for this choice, the above inequality is an equality, and, furthermore, we have that $cov^{(f \times \vec{g})}( \{\partial f^{*}(\theta^{i} \circ u) ; (\vec{\phi}^{k} \circ \vec{v}) \} , \{\partial f^{*}(\theta^{j} \circ u) ; (\vec{\phi}^{l} \circ \vec{v}) \}) = 0$, since we may use $\theta = ln = exp^{-1}$ as our coordinate chart, and then

\begin{center} $cov^{(f \times \vec{g})}( \{\partial f^{*}(\theta^{i} \circ u) ; (\vec{\phi}^{k} \circ \vec{v}) \} , \{\partial f^{*}(\theta^{j} \circ u) ; (\vec{\phi}^{l} \circ \vec{v}) \}) = \int_{A}ln((\partial f^{*}f;\vec{g}))(\partial f^{*}f; \vec{g})$ \end{center}

But $ln \circ f^{*} = f^{*} \circ ln$, so this becomes

\begin{center} $\int_{A}(\partial f^{*}ln(f);ln(\vec{g}))(\partial f^{*}f;\vec{g})$ \end{center}

but by the definition of weak tmle this is zero.

\begin{cor} For a weak tmle $(u,v)$, we have that $\int_{M}h^{kl}_{ij}(f(u),\vec{g}(\vec{v})) = 0$. \end{cor}

\begin{lem} (Turbulent EPI principle).  Write $h(f(u),\vec{g}(\vec{v})) = (\partial f^{*}f_{ij} ; \vec{g}_{kl})h_{ij}^{kl}(a)(f(u),\vec{g}(\vec{v}))$.  Then we have that

\begin{center} $0 \leq \int_{M}\int_{A}h(f(u),\vec{g}(\vec{v})) = \int_{M}\int_{A}(\partial f^{*}\rho_{f} ; \rho_{\vec{g}})$ \end{center}

\begin{proof} Follows by the factorisation lemma. \end{proof} \end{lem}




\section{Statistical Stacks}

\subsection{Motivation}

We would like to sometimes be able to describe the critical behaviour of systems where properties of the solution in the base space do not necessarily vanish at infinity.  For instance, we might like to study periodic physical behaviour in our systems (for instance, as in condensed matter physics) rather than rapid asymptotic decay of the salient properties of our solution (as in particle physics).  This suggests that we need to derive an extremisation principle that naturally has a particular generalisation of the holographic principle that applies to some sort of "lifted" geometry that is related to, but not corresponding, to the base.

One such way of doing this as I shall describe is to consider a so called statistical stack.  The basic idea here is fairly simple - we would like to construct a statistical structure of a statistical structure, define a suitable generalisation of the fisher information, and then play the standard game, by examining critical behaviour corresponding to particular points in the solution space of this functional.

What makes this game slightly harder to play in full generality is that we do not merely want to be able to construct the one-fold stack corresponding to an underlying manifold with signal function (the statistical manifold), or a two-fold, three-fold etc.  Rather we need to somehow develop a language that allows us to describe for instance $1/2$-fold stacks, or $(n,m)$-stacks; in other words we would like to construct an appropriate language for turbulent stacks.

\subsection{Basic tools}

I shall now describe the construction.  As a first observation, we will have a sequence of signal functions for each level, so this complicates things.  Hence we need some way of combining these into a single induced signal function.  I shall write this meta signal function as

\begin{center} $f^{k}(m,a_{0},a_{1},...)$ \end{center}

where the $a_{i}$ corresponds to the $i$th statistical level, and $k \in R^{\infty}$ is the turbulent variable.  But this is rather bad notation, since after all we might have less than one $a_{i}$; so we might instead write

\begin{center} $f^{k}(m,a^{(k)})$ \end{center}

with some sort of induced probabilistic distribution being understood over $R^{R^{\infty}}$, in particular $k = \partial\rho_{g_{(stack)}}(m,b)$ where $\rho_{g}(m,b)$ is the information density of a signal function for a standard statistical distribution over our manifold $M$ with signal function $g$ and statistical variable $b$.

In particular, I will define $a^{(k)} \in R^{k}$.  Also, I claim that we may interpret without loss of generality $f^{k}$ literally as $f^{k}$, for some base signal function $f(m,a)$.

Then, in order to compute the information, we write $f(m,a)^{g(m,n,b)} =: \gamma(m,n,a,b)$, and then compute

\begin{center} $I(\gamma) = \int_{M}\int_{N}\int_{A}\int_{B}\rho_{\gamma}$ \end{center}


As a motivating example for all this, I claim that the information of a sharp turbulent stack is

\begin{center} $\int_{M}\int_{N}R^{R_{\bar{\tau}_{n}(m)})}_{\sigma}(m)dndm$ \end{center}

where here of course $g(m,n,b) = \delta(\tau(m,n) - b)$, and $f(m,a) = \delta(\sigma(m) - a)$.  This can be verified by standard sorts of arguments to before and reference to the definitions.

A slightly less trivial example is the following:

\begin{center}
$\int_{M}\int_{A}(1 + \Delta_{\tau}\epsilon + \Delta^{2}_{\tau}\epsilon^{2} + ...)\Delta_{f(m,a)}F(m,a)dadm$\end{center}

which arises from the second statistical moment of a stack associated to the signal function $f(m,a) = \int_{A}F(m,b)\delta(\sigma_{b}(m) - a)db$ after "forgetting" some of the degrees of freedom at the first level.  Here evidently our turbulent signal function is almost sharp scalar turbulence, $g(m,a) = (1 + \Delta_{\tau}\epsilon + ... )\delta(\tau(m) - b)$ (since then $f^{\partial \rho_{g}} = (1 + \Delta_{\tau}\epsilon + ...)f$ as required).

\subsection{Dynamics}

I now give the extremisation theorem for this object, ie, the Cramer-Rao inequality, and give a sketch of its proof.  This will not nearly be as hard as it could be, since we have already done most of the work; it will mainly be an adaptation of the proof of the Cramer-Rao inequality for standard turbulent spaces.

\begin{thm} (Cramer-Rao Inequality for a turbulent stack).  Let us have a turbulent stack $(M,N,A,f,g)$ as above.  Then $I(f^{g}) \geq 0$.
\end{thm}

To prove this result we will need to play the standard game and build up an appropriate set of tools to tackle the problem.

\begin{dfn} Let $\Lambda = (M,N,A,f,g)$ be a turbulent stack, where $f$ is the signal function on the base and $g$ is the signal function on the turbulent superstructure.  A turbulent estimator on this structure is a tuple $(u,v)$ where $u$, $v$ are maps from the sample space $M \times A$ to $M$.  An unbiased turbulent estimator is an estimator that satisfies the condition that

\begin{center} $E_{(m)}^{f^{g}}((\theta^{i} \circ u) ;(\phi^{j} \circ v )) = (\theta^{i}(m) ; \phi^{j}(m,n))$ \end{center}

where

\begin{center} $E_{m}^{f^{g}}(\tau ; \sigma) := \int_{A}f^{g}(\tau^{\sigma})$ \end{center}

is the turbulent expectation with respect to the signal functions $f$ and $g$, and $\theta : M \rightarrow R^{n}$ is a coordinate chart on $M$, $\phi : M \times N \rightarrow R^{n}$ is a coordinate chart on $M \times N$.
\end{dfn}


\begin{rmk} From now on I will use the notation $f^{g}$ and $(f ; g)$ interchangeably. \end{rmk}

\begin{lem} (Factorisation). $E_{m}^{(f ; g)}( \tau ;  \sigma) := \int_{A}(( f\tau );(g \sigma)) = \int_{A}(f ;g)(\tau ; \sigma)$.
\end{lem}

\begin{proof} We use the fact that $f$ and $g$ are normalised distribution functions over $A$, which conserve probability.  \end{proof}

\begin{lem}  Unbiased turbulent estimators satisfy the following relation:

\begin{center} $E_{m}^{(f^{g})}(\{(\theta^{i} \circ u);(\phi^{k} \circ v)\}\{(\frac{\partial}{\partial \theta^{j}}ln(f \circ u)) ; (\frac{\partial}{\partial \phi^{l}}ln(g \circ v))\}) = 0$ \end{center} \end{lem}

\begin{proof} The result follows by differentiating the definition of an unbiased turbulent estimator on both sides and then integrating by parts, as with the simpler case.
\end{proof}

\begin{dfn} The Fisher information tensor $h^{kl}_{ij}(f,g)$ (corresponding to the likelihood functions $f,g$), is defined to be

\begin{center} $h^{kl}_{ij}(f,g) := E_{m}^{(f^{g})}\{ (\frac{\partial}{\partial \theta^{i}}ln (f) ; \frac{\partial}{\partial \phi^{k}}ln (g))(\frac{\partial}{\partial \theta^{j}}ln (f) ; \frac{\partial}{\partial \phi^{l}}ln (g))\}$ \end{center}

\end{dfn}

\begin{thm} (Turbulent Cramer-Rao Inequality).

\begin{center} $cov^{(f ; g)}( \{(\theta^{i} \circ u) ; (\phi^{k} \circ v) \} , \{(\theta^{j} \circ u) ; (\phi^{l} \circ v) \}) - h_{ij}^{kl}(f(u),g(v)) \geq 0$ \end{center}

where $cov^{(f ; g)}(v^{ik},w^{jl}) := E_{m}^{(f ; g)}(v^{ik}w^{jl})$.
\end{thm}

\begin{proof}  Note that $E_{m}^{(f ; g)}(\Gamma^{ik}_{\alpha \beta}\Gamma^{jl}_{\gamma \delta})$ is certainly nonnegative, for

\begin{center} $\Gamma^{ik}_{\alpha \beta} := ((\theta^{i} \circ u);(\phi^{\alpha} \circ v)) - E_{m}^{(f \times g)}\{ (\theta^{i} \circ u) ; (\phi^{\alpha} \circ v) \} - \{ \frac{\partial}{\partial \theta^{k}}ln(f(u)) ; \frac{\partial}{\partial \phi^{\beta}}ln(g(v)) \} h_{ik}^{\alpha \beta}(f(u),g(v))$ \end{center}

Expanding, we get

\begin{center} $cov^{(f ; g)}( \{(\theta^{i} \circ u); (\phi^{k} \circ v)\},\{ (\theta^{j} \circ u) ; (\phi^{l} \circ v)\}) - 2E_{m}^{(f ; g)}( \{ ((\theta^{i} \circ u) ; (\phi^{k} \circ v)) - E_{m}^{(f ; g)}( (\theta^{i} \circ u); (\phi^{k} \circ v)) ( \frac{\partial}{\partial \theta^{\alpha}}ln(f(u)) ; \frac{\partial}{\partial \phi^{\beta}}ln(g(v)))\}h_{j \alpha}^{l \beta}(f(u),g(v))) + h_{ij}^{kl}(f(u),g(v)) \geq 0$ \end{center}

But, observing first of all that $h$ is statistically independent, then by the lemma the second term reduces to 

\begin{center} $2E_{m}^{(f ; g)}( \{ - E_{m}^{(f ; g)}( (\theta^{i} \circ u); (\phi^{k} \circ v))( \frac{\partial}{\partial \theta^{\alpha}}ln(f(u)) ; \frac{\partial}{\partial \phi^{\beta}}ln(g(v))) \}h_{j \alpha}^{l \beta}(f(u),g(v)))$ \end{center}

using $E_{m}(- v) = -E_{m}(v)$, together with the definition of unbiased turbulent estimator, we get

\begin{center} $-2E_{m}^{(f ; g)}(  ((\theta^{i} \circ u); (\phi^{k} \circ v))( \frac{\partial}{\partial \theta^{\alpha}}ln(f(u)) ; \frac{\partial}{\partial \phi^{\beta}}ln(g(v))) h_{j \alpha}^{l \beta}(f(u),g(v)))$ \end{center}




from which follows, after integration by parts and use of the factorisation lemma that the second term in the original expansion simplifies to

\begin{center} $-2E_{m}^{(f ; g)}(h_{ij}^{kl}) = -2h_{ij}^{kl}$ \end{center}

completing the proof.

\end{proof}

\begin{dfn} A turbulent maximum likelihood estimator (tmle) is an unbiased estimator $(u,v)$ that realises a local maximum of $(ln(f \circ u) ; ln(g \circ v))$. \end{dfn}

\begin{lem} Any tmle realises the Cramer-Rao lower bound. \end{lem}

It is then possible, following the same methodology as before, to then sharpen these results to dealing with weak estimators and finally prove the Cramer-Rao inequality for turbulent stacks. In order to do this, of course, we need one more result - Stokes theorem for turbulent stacks.

\begin{thm} (Stokes for (sharp) turbulent stacks).  Let $\omega$ be a $n - 1$ form and $\tau$ an arbitrary vector valued function on $M$.  Then

\begin{center} $\int_{M}\int_{N}d(\omega^{\tau(m,n)}(m)) = \int_{\partial M}\int_{N}\omega^{\tau(m,n)}(m)$ \end{center}

\begin{proof} As opposed to stokes for sharp turbulent manifolds, we are not performing in effect a repeated application of the standard form of stokes.  Rather we are trying to show that stokes holds not only for forms that are function valued, but are functionals, or higher order objects.  But this is an eminently reasonable thing to expect.
\end{proof}
\end{thm}





\subsection{A few quick remarks on notation}

The notation that I have been using so far is perhaps not entirely transparent, so I will briefly digress and provide an alternative perspective on what I have been doing.  Essentially the difference between turbulence in the measure and turbulence in the stack is a matter of the order at which operators are applied.  Consider the curvature operator $R : \{$metrics on a manifold$\} \rightarrow \{$functions on the same manifold$\}$.  Similarly, consider the operator $\partial^{*}_{(\tau)} : F \rightarrow F$ where $F$ is the space of functions on a manifold, and $\tau$ is an auxiliary "turbulence" metric associated to the operator, whose action is defined as $\partial^{*}_{(\tau)}(f) = (\partial f^{*}f ;R(\tau))$.

Then for turbulence in the measure, we construct $\partial^{*}_{(\tau)} \circ R(\sigma)$.  Conversely, for turbulence in the stack, we act component wise on $\sigma$ with the operator $R \circ \partial^{*}_{(\tau)}$.  Of course in general these two operators will not commute.

We may iterate this process.  For consider

\begin{center}
$\partial^{*}_{(\partial^{*}_{(\gamma)} \circ R \circ \partial^{*}_{(\alpha)}(\tau))} \circ R(\sigma)$
\end{center}

It is clear that there is a natural "bifurcation" occurring in the operators here, and in fact this is exactly what does occur when we look at deeper levels of information within the structures I have defined above.

However I will continue to stick to the older notation for the remainder of this manuscript, and perhaps develop the above ideas further in a later discourse.

\section{Topics in turbulent statistical dynamics}

Here I motivate a number of areas where the tools in this chapter may be applied, without going into too much detail or development.

\subsection{Application to condensed matter physics}

The simplest nontrivial action incorporating both fractal behaviour in the stack as well as in the measure is the focus of this section.  This serves well as a model for systems where statistical effects are still small but not negligible, such as in condensed matter physics.  As an important remark, I should mention that some particularly exotic condensed matter systems will probably have more significant statistical effects, so this model will not be useful in such cases, since it is a perturbative expansion in 

\begin{center} $\norm{\epsilon}_{R^{3}}$, $\epsilon = (\epsilon_{1},\epsilon_{2},\epsilon_{3})$. \end{center}

First of all, ignore perturbative effects.  Assume in other words that the geometry is to some extent sharp, but exhibiting both types of fractal behaviour.  Then I claim we will have an action of the form

\begin{center} $\int_{M}(\partial (\sigma ; \tau)^{*}R_{(\sigma ; \tau)(m)} ; R_{\kappa}(m))dm$ \end{center}

where $f(m,a)^{g(m,n,b)} = \delta(\sigma(m) - a)^{\delta(\tau(m,n) - b)}$ induces the metric $(\sigma ; \tau)$ on $M$.  In particular the associated signal function $S$ is of the form

\begin{center} $S = ( (f ; g) ; h)$ \end{center}

where $h(m,a) = \delta(\kappa(m) - a)$.

Introducing small statistical fluctuations in $f$ only - we would have to have a higher order construction to justify fluctuations in $g$ and $h$  - creates a new action

\begin{center} $\int_{M} (\partial \sigma^{*}(R_{(R_{\tau})(1)\sigma} + R_{(R_{\tau})(2)\sigma}\epsilon + R_{(R_{\tau})(3)\sigma}\epsilon^{2} + ...) ; R_{(1)\kappa})$ \end{center}

where $R_{(R_{\tau})\sigma}(det(\sigma ; \tau))^{1/2} := \Delta_{\sigma}^{\Delta_{\tau}}(det( \sigma ; \tau))^{1/2}$ is the eigenfunction associated to the respective operator acting on the volume element, $(\sigma ; \tau) = \sigma_{ij}^{\tau_{kl}}$, and $det$ is the natural extension of determinant to the tensor product of two matrices.

So we have an action in terms of the scalar variable $\epsilon$, as well as three metrics- $\sigma, \tau$, and $\kappa$.  Differentiating with respect to these parameters we obtain a set of partial differential equations which control the behaviour of critical examples of these objects.

With respect to $\kappa$:

\begin{center}   $(\partial \sigma^{*}(R_{(R_{\tau})(1)\sigma} + R_{(R_{\tau})(2)\sigma}\epsilon + R_{(R_{\tau})(3)\sigma}\epsilon^{2} + ...) ; R_{(1)\kappa}) = 0$ \end{center}

With respect to $\tau$:

\begin{center}  $R_{\tau} = 0$ \end{center}

With respect to $\sigma$:

\begin{center}  $(\partial \sigma^{*}(R_{(R_{\tau})(1)\sigma} + R_{(R_{\tau})(2)\sigma}\epsilon + R_{(R_{\tau})(3)\sigma}\epsilon^{2} + ...) ; \partial \sigma^{*}R_{(1)\kappa}) = 0$ \end{center}

With respect to $\epsilon$:

\begin{center} $(\partial \sigma^{*}(R_{(R_{\tau})(2)\sigma} + 2R_{(R_{\tau})(3)\sigma}\epsilon + ...) ; R_{(1)\kappa}) = 0$ \end{center}

It is my hope that, giving certain simplifying assumptions, these coupled PDEs will reduce to give standard results in condensed matter theory.  To sharpen the models to a greater extent, a second order theory is the way to go.  Roughly speaking, one now allows second order turbulence in both the stack and the measure, as well as considering fractal measure with fractal stacking (cross term).  This allows a model as above to be constructed, but now with statistical fluctuations permissible at all parts of the first level.

However, I feel that the foundations need to be touched up some more before this is developed further.  This is a task for future research.


\subsection{Fluid dynamics}

On a slightly simpler note, we are interested in fluid flow with "turbulence", or at least first order turbulence.  This requires a sharp action as described in the discussion immediately above.  ie

\begin{center} $\int_{M}(\partial \sigma^{*}(R_{(R_{\tau})\sigma} ; R_{\kappa}(m))dm$ \end{center}

The geometric fluid flow of a physical system is described by the criticality of this functional.  Hence we must have

\begin{center} $(\partial \sigma^{*}R_{(R_{\tau})\sigma} ; R_{\kappa}) = 0$, \end{center}

\begin{center} $R_{\tau} = 0$ \end{center}

and

\begin{center} $(\partial \sigma^{*}R_{(R_{\tau})\sigma} ; \partial \sigma^{*}R_{\kappa}) = 0$ \end{center}

We might be interested in the case where the geometry is flat, as is usually the case where one is examining the behaviour of fluids on earth.  Hence all three metrics become purely antisymmetric.  The aim is to now rewrite the corresponding PDEs and compare with the standard form of the Navier-Stokes equations.

I should remark however that, before we can go any further, just like before we need to reexamine the fundamentals more closely to make sure we know what we are dealing with.  This is a task for the future.

\subsection{The Lorentz problem}

I would like to now sketch a potential resolution of the Lorentz problem, or the mystery of why the physics we experience should occur in a Lorentzian geometry.  Recall from our results before in 8.2 that we certainly need at least four dimensions to realise a perturbative model of physics, since $R_{(2)}$ requires four dimensions to be computably nonzero.

This of course leaves the questions of why the index should be one, and, furthermore, why also in a sharp model, why the framework is stabler the fewer dimensions one has to play with.  The latter turns out to be easier to show, so I will prove that first:

\begin{thm} (Minimisation of dimension).  In a sharp fractal manifold, or first order turbulent geometry, the optimal fractal dimension for the information is as small as possible; in other words, it is zero. \end{thm}

\begin{proof} (sketch). Recall that the action is of the form

\begin{center} $\int_{M}(\partial \sigma^{*}R_{\sigma} ; R_{\tau})$ \end{center}

For this to be critical we require both the derivatives wrt $\tau$ and $\sigma$ to vanish.  Recalling the definition of derivative, we have

\begin{center} $\partial^{f_{\tau}} := \Sigma_{n=0}^{\infty}\frac{f^{(n)}_{\tau}(x)\partial^{n}}{n!}$ \end{center}

hence

\begin{center} $\frac{\partial}{\partial \tau}\partial^{f_{\tau}} := \Sigma_{n=0}^{\infty}\frac{\frac{\partial f^{(n)}_{\tau}(x)}{\partial \tau}\partial^{n}}{n!} = \partial^{\frac{\partial f}{\partial \tau}}$ \end{center}

Consequently requiring the derivative wrt $\tau$ to vanish is equivalent to requiring

\begin{center} $(\partial \sigma^{*}R_{\sigma} ; R_{\tau}) = 0$ \end{center}

Differentiating wrt $\sigma$ is not entirely a good idea, but differentiating half with respect to $\sigma$, then half with respect to $\sigma^{-1}$, ie applying the operator $\partial \sigma^{*}$, is a better idea and clearly equivalent.

Via the same trick we have

\begin{center} $\partial \partial^{f}g = \Sigma \frac{f^{(n)}\partial^{n + 1}g}{n!} = \partial^{\partial f}g$ \end{center}

So differentiating wrt $\sigma^{*}$ and setting to zero produces the consequential relation

\begin{center} $(\partial \sigma^{*}R_{\sigma} ; \partial \sigma^{*}R_{\tau}) = 0$ \end{center}

But the fractional derivatives of a function can only both be zero if the orders of the derivative are both the same.  So

\begin{center} $\partial \sigma^{*} R_{\tau} = R_{\tau}$ \end{center}

But this is impossible unless $\partial \sigma^{*}$ is the identity operator, which will occur only if $\sigma = 0$.  So the optimal dimension is zero as required.
\end{proof}

Now, in a sharp first order turbulent geometry, it turns out that the concept of index generalises in a nice way:

\begin{lem} (Turbulent index). The index of a sharp turbulent geometry $(\partial \sigma^{*}\delta(\sigma(m) - a) ; \delta(\tau(m) - b))$ takes the form

\begin{center} $ind(\partial \sigma^{*}\sigma ; \tau) = ind(\sigma)^{ind(\tau)}$ \end{center}

where $ind(\sigma)$ is the weighted index of the metric $\sigma$.
\end{lem}

\begin{proof} (sketch). This basically follows from the fact that the weighted index of $A^{B}$, where $A$, $B$ are both matrices, is equal to the index of $A$ raised to the index of $B$.  For if we denote $A = ln \hat{A}$ to be the logarithm of $\hat{A}$, ie such that $\hat{A} = e^{A} = \Sigma \frac{A^{n}}{n!}$, we get $\hat{A}^{B} = e^{A \otimes B} = \Sigma \frac{(A \otimes B)^{n}}{n!}$.  Now the index of $A \otimes B$ is the index of $A$ times the index of $B$.  Write $\hat{ind} := log \circ ind \circ exp$. Then similarly $\hat{ind}(A \otimes B) = \hat{ind}(A)ind(B)$.  Hence the index of $e^{A \otimes B}$ will be the exponential of this index, ie $e^{\hat{ind}(A)ind(B)}$.  But the index of $\hat{A}$ is $e^{\hat{ind}(A)}$ ie $ind(e^{A}) = e^{\hat{ind}(A)}$; hence the index of $\hat{A}^{B}$ is $ind(\hat{A})^{ind(B)}$. \end{proof}

\begin{thm} (Index Lemma).  In a nontrivial turbulent geometry with almost sharp geometry on the base, and sharp geometry on the first and second levels, the Cramer-Rao inequality implies the turbulent index must be greater than or equal to one; in particular, when the information is critical, the index must be one. \end{thm}

\begin{proof} (sketch).  Via the proof of minimisation of dimension, in a sharp second order turbulent geometry we must have that the dimension for the base and first levels must be as small as possible.  In particular at the first level the geometry of the fractal measure will be trivial when optimised.

Now, by the lemma above, we then must have that the index of the geometry to first order is $ind(\sigma)^{ind(\tau)}$.  But by the above remarks $\tau = 0$, and so $ind(\tau) = 0$.  Since we are assuming that the geometry of the base is nontrivial (in particular, since we are assuming perturbative effects, which pose an obstruction to minimisation), we have that $\sigma \neq 0$.  Hence the index is well defined and must be one.  \end{proof}

It is then relatively easy to see that these results should generalise; in particular it seems clear that as a consequence of these two results that an optimal and stable almost sharp geometry must be Lorentzian.

\subsection{Entanglement physics}

Recall from before I mentioned that the simplest cartoon model for entanglement physics, or action at a distance in the base geometry $M$, is probably described by criticality of the action

\begin{center} $\int_{M}(\partial \sigma^{*}R_{\sigma} ; \partial \tau^{*}R_{\tau} ; R_{\kappa})$ \end{center}

Here I am assuming "scalar turbulence", or merely fractal dynamics.

Now, it is of interest to discuss how this all relates to the issues raised in the EPR paper of the mid 1930s \cite{[EPR]}, which were discussed further in Bell's paper of the mid 1960s \cite{[Be]}.  I will make an attempt to do this now.

In \cite{[EPR]}, it was demonstrated that the separate assumptions that quantum mechanics was a physically complete theory, together with the idea that separate noncommuting observables could not have simultaneous existence, were incompatible.  This was achieved by examining the states of two particles which interacted only from time $0$ to time $T$, and showing that, after this time, the observation of the state of one could arbitrarily determine the state of the other.  In particular this demonstrated that a contradiction could be constructed wherein it was possible to violate the second assumption, that two noncommuting observables could have independent reality in the second state.

This motivated the development of hidden variable models, such as the Bohmian mechanics of Aharonov and Bohm \cite{[Bo1]}, \cite{[Bo2]} in the 1950s, in an effort to restore to quantum mechanics causality and locality.  But these theories all suffered from various problems.  In \cite{[Be]}, Bell argued that it was impossible to construct a theory, consistent with the wave function interpretation of quantum mechanics that did not involve entanglement, or nonlocal effects.  In order to show this he constructed a famous inequality, now known (unsurprisingly) as the Bell Inequality.

Of course, the theory which I have laboriously constructed is, in and of itself, riddled through and through with hidden variables.  So in order to continue it becomes necessary to address Bell's arguments, and examine the assumptions upon which he reached his conclusions.

Certainly in the course of this chapter and indeed over the course of this treatise I have demonstrated that it is possible to construct higher order geometric objects, which transcend limiting oneself to examining the geometry of the base "physical" space.  This allows mechanisms via which information may transit smoothly between points which in a simpler geometry would be far apart.  For instance, consider naively the example of a manifold, and consider two points which with respect to one metric $\sigma_{0}$ are a distance $L$ apart.  But we may consider a family of metrics $\sigma_{t}$ such that the distance between these two points with respect to $\sigma_{t}$ is $Le^{-t}$.  Then clearly in the limit as $t$ becomes large and positive the distance will drop to zero.

We may now construct as in chapter 7 a distribution over this one parameter family of metrics.  So we have a probabilistic theory with hidden variables.  Via Bell's paper our theory is still nonlocal.  But we may make a further extension of the class of objects we are dealing with, to consider what I describe in this chapter as  turbulent geometry.  Then via the correspondence principle our statistical geometry becomes a sharp first order turbulent geometry, a hidden variable model which nonetheless provides us with a local description, and furthermore wherein at least part of the information - the part corresponding to the limit of the sequence of metrics $\sigma_{t}$ - is mediated between the two points instantaneously.

Furthermore, if one considers the geometry to be \emph{almost sharp}, in the sense I defined earlier in chapter 7, then one recovers standard quantum mechanical effects (as in chapter 8, section 3.5).

This certainly seems to provide a constructive contradiction with the thesis of \cite{[Be]}.  So some discussion is definitely warranted.

The crux of the matter, it seems to me, is what definition Bell used or had in mind for what he thought of as a hidden variable model.  Looking through his paper, and reading between the lines, I extract the following definition by implication:

\begin{dfn} A hidden variable model (in the sense of Bell) is a model that

\begin{itemize}
\item[(i)] is consistent with the wave function interpretation of quantum mechanics,
\item[(ii)] that makes use of variables that are "hidden" in the sense that although they impact on the process of measurement of physical parameters, they are not necessarily measurable themselves.
\end{itemize}

Furthermore it is a model that satisfies the standard wave function interpretation, ie, one has a wave function $\psi(m,n)$, where $m$ are the coordinates in a chart corresponding to the standard format of quantum mechanics, and $n$ are the coordinates in a chart corresponding to the hidden variables.

An observable in this model is an operator $A$ that acts on this wave function in the following manner:

\begin{center}
$A \vert \psi > := \int_{N}A(n,m)\psi(m,n)dn$
\end{center}

The expectation of $A$ is then

\begin{center} $E(A) := \int_{M}A \vert \psi >(m)$ \end{center}

There are no observables for the space $N$; it is $M$ that in this interpretation is the space in which physics occurs.
\end{dfn}

This is certainly quite restrictive!  And, indeed by inspection it does not capture the full generality of the mathematics which I have developed in the course of this project; indeed, it fails to even capture the spirit of the statistical geometry; it is similar to trying to understand the full generality of statistical geometry by merely looking at the behaviour of almost sharp signal functions.  Indeed, it is clear the following is true:

\begin{lem} A hidden variable model (in the sense of Bell) is equivalent to the standard product of an almost sharp geometry $K$ with some other form of topological space $T$.  Furthermore, this is a product of the simplest type: if $m \in K$ and $n \in T$, then the corresponding hidden variable model (in the sense of Bell) will have corresponding coordinates $(m,n)$.  \end{lem}

\begin{proof} (sketch).  This follows from the fact that one is still using a wave function description, and the definition of measurement of an observable in a hidden variable model.
\end{proof}

An example of such a model would be for instance a differentiable manifold $M$ producted with the space of all possible metrics $\sigma$ on $M$.  However, this is quite a different thing from what I have been advocating in terms of the statistical geometry; whereas I was dealing with a similar sort of product on a local level, the statistics was imbedded implicitly in the computation of geodesic paths induced by local descriptors I called signal functions, rather than explicitly via a naive probabilistic distribution akin to the idea of many "parallel universes".

This distinction is crucial, for it is absolutely necessary that different inner products should be able to interact locally, rather than artifically enjoying a disjoint and distinct existence from their peers.  Another way of putting this, is that the naive model I described above, of $M$ producted with the space of all possible metrics, has a \emph{trivial}, and \emph{maximally disconnected} topology, whereas the model I have been advocating has a \emph{maximally connected} topology.  The hidden variable models (in the sense of Bell) are I believe of the first type.

In particular, there is hence no reason why the Bell inequality should extend to the more general objects I have constructed - since they are not hidden variable theories in the sense above, and, in certain respects, are quite the opposite.

This (hopefully) suggests that the EPR paradox is satisfactorially resolved by the above theoretical considerations.  Furthermore, since the resolution is constructive, we now have an action (as above) which may be studied to provide further insights into the mediation of so-called entanglement effects.  This is a promising direction for further research.





\chapter{Number theory from information geometry}

In this chapter I will give a number of results from number theory which can be derived using the methods I have outlined in preceding chapters.  In a sense, many of the results I will sketch here serve as "toy examples" for the physics associated to the same mathematical models.  The first section requires nothing more than a detailed understanding of the statistical calculus for statistical manifolds.  The second needs a basic understanding of statistical stacks and of course of statistical manifolds.  As for the final result, it requires the tools of turbulent geometry.

As a minor remark, for the special case of restricting to criticality of subsets of the real numbers, and considering their analytic extensions, it turns out that the theory of turbulent stacks and the theory of turbulent geometry become more or less equivalent.  This is mainly because the statistical derivative becomes trivial through experiencing cancellation; the sections over the space of inner products do not contribute to the dynamics.  However if one were to consider more general results pertaining to tuples of complex numbers, these techniques split apart again and develop their own individualistic character.

Of course there are many, many more results and variations on such that can be derived using these tools, methods, and general philosophy.  Such a treatment is beyond the scope of this work - the main purpose of this chapter is not to catalogue all potential results that can be demonstrated using the machinery of information geometry, but rather to demonstrate its power and capacity.  Indeed, it is my suspicion that it would be impossible to succeed in the former task, since I am of the opinion that the well of truth accessible after this fashion to be inexhaustible.

\section{The first statistical moment of the complex numbers}

This is an adaptation of a paper which I wrote on this subject, titled "On the criticality of the prime numbers and a conjecture of Riemann".

\subsection{Introduction}

Broadly speaking, the prime numbers are defined to be the minimal set that contains maximal information about the natural numbers, ie the minimal set that generates the natural numbers via multiplication.

I have already introduced the concept of a statistical manifold and developed a variational principle which essentially amounts to computing critical points of an information like quantity, called the Fisher information, on the statistical manifold.

Motivated by these two similar properties, it seems natural that many famous conjectures and problems in number theory should be translatable into this language.  In this short note I indicate how it is possible to translate a famous conjecture of Riemann's into this context, and sketch how it might be possible to prove it using these new tools.

This conjecture is:

\begin{conj} (Conjecture A). The class of functions $g_{\{f_{n}\}_{n \in N}}(z) = \sum_{n \in N}\frac{f_{n}}{n^{z}}$, where $\{ f_{n}\} \in l^{\infty}(C)$  have no poles to the right of the critical line $Re(z) = 1/2$.  
\end{conj}

I will quickly illustrate now how this implies the more famous conjecture:

\begin{conj} (Riemann Hypothesis).  The zeroes of the functions $g_{\{f_{n}\}_{n \in N}}$ lie on the critical line $Re(z) = 1/2$.
\end{conj}

Essentially this implication rests on the following simple result:

\begin{lem}. For every sequence $f_{n} \in l^{\infty}(C)$, there exists a sequence $h_{n} \in l^{\infty}(C)$ such that $1/g_{\{f_{n}\}} = g_{\{h_{n}\}}$.
\end{lem}

From which follows, assuming the first conjecture holds:

\begin{cor}  The class of functions $g_{\{f_{n}\}_{n \in N}}$ have no zeroes to the right of the critical line $Re(z) = 1/2$.
\end{cor}

It is then well known that the RH follows from this statement for this class.  In particular, this follows from the symmetry of the zeroes for $L$-functions about the critical line $Re(z) = 1/2$, which is in turn a consequence of the functional equation for the same.

Now, let $\Pi(x,a,d)$ denote the number of prime numbers in an arithmetic progression $a, a+d, a+2d, ...$ which are less than or equal to $x$.  If the RH is true, then it is well known that for every coprime $a$ and $d$ and for every $\epsilon > 0$, we have that

\begin{center} $\Pi(x,a,d) = \frac{1}{\phi(d)}\int_{2}^{x}\frac{1}{ln(t)}dt + O(x^{1/2 + \epsilon})$ as $x \rightarrow \infty$, \end{center}

where $\phi : N \rightarrow N$ is the Euler phi function, that is, the function where $\phi(n)$ denotes the number of naturals less than or equal to $n$ which are coprime to $n$.

\subsection{The Argument (sketch)}


We consider $M = \mathcal{C}$, $A = \mathcal{C}$ (the space of metrics on the complex line).

The natural signal function is

\begin{center}
$f(m,a) = \int_{A}F(m,b)\delta(\sigma(m,b) - a)db$
\end{center}

We wish to choose $f$ such that the Fisher information is optimised; in particular, we wish to find $f$ such that the quantity 

\begin{center}
$K = \int_{M}\int_{A}\norm{\partial f}^{2} / f - \norm{\psi}^{2} / f$
\end{center}

is not only zero but that $\delta K$, the first variation of $K$, be zero too.  Here $\psi = grad_{\Lambda} f$.  We ignore torsion due to our assumption of analyticity which renders its contribution zero by Cauchy's theorem.  

In particular I will more explicitly write $\psi$ more carefully in terms of the signal function fairly shortly.  This will be to assist in generalisations.


\begin{center}
$I = K = \int_{M}\int_{A}\norm{\partial f}^{2} / f$
\end{center}

A computation of the Fisher information yields the fairly easy expression


\begin{center}
$I = \int_{M}\int_{A}(\partial_{\Lambda}^{2}F - \frac{(\partial_{\Lambda} F)^{2}}{F}) - \int_{M}\int_{A}\norm{F^{1/2}\frac{\partial}{\partial \sigma}\partial \sigma}^{2}$
\end{center}

Let $F = e^{h}$.  Then $F'' = (h'' + (h')^{2})e^{h}$ and $F' = h'e^{h}$, so that $F'' - \frac{(F')^{2}}{F} = h''e^{h}$.

Write $\psi = F^{1/2}\frac{\partial}{\partial \sigma}\partial \sigma$.

If we recall that $\partial_{\Lambda}(z,a) = \int_{A}F(z,b)(\frac{\partial}{\partial z} + \frac{\partial}{\partial b})db$ then roughly speaking $h'' = \frac{\partial^{2}h}{\partial z ^{2}} + \frac{\partial^{2}h}{\partial a^{2}} + 2\frac{\partial^{2}h}{\partial z \partial a}$.  But the second term vanishes due to the holographic principle, and so does the third, so in particular, we have

\begin{center} $h'' = h_{zz}$ \end{center}

as a distributional expression.  More precisely, we get

\begin{center} $h'' = \int_{A}F(z,b)h_{zz}db$ \end{center}

Note we have the following functional identity that will come in handy:

\begin{lem} $e^{H/\delta} = H$.
\begin{proof} This may be checked qualitatively by verifying that each side possesses the same general properties; alternatively, one may take a sequence of smooth functions converging to $H$ and check both sides, observing that $H' = \delta$.
\end{proof}
\end{lem}

\begin{lem} $h(z,a) = A(a)z + B(a) + G(z,a)\frac{H}{\delta}(\gamma(z,a))$, where $H$ is the Heaviside function and $\delta$ is the Dirac delta functional, and $G(z,a)$ is some random smooth function depending on $\sigma$. 
\begin{proof} Certainly $h_{zz} = \norm{\frac{\partial}{\partial \sigma}\partial \sigma}^{2}\delta$ follows easily from above.  Hence $h$ is the second antiderivative of this expression.

But $\delta$ may be written as $\frac{H''}{\delta}$.  Then certainly $h = G(z,a)\frac{H}{\delta}(\gamma(z,a)) + A(a)z + B(a)$, since $z$ is the active variable here, and we may integrate out the derivatives on $H$ by judicious use of the holographic principle.



\end{proof}
\end{lem}

\begin{lem} $e^{h} = H(\gamma)\hat{F}$, with $\hat{F} = e^{Az + B}$.
\begin{proof} Follows from the fact that $e^{G\frac{H}{\delta}} = e^{\frac{H}{\delta}} = H$ from a previous lemma.
\end{proof}
\end{lem}

\begin{lem} If we write $\psi = H\hat{\psi}$, then $H^{1/2}\hat{\psi}_{a} = H_{a}^{1/2}\hat{\psi}$.
\begin{proof} By symmetry $H(z,a) = -H(z,-a)$ and $H_{a}(z,a) = H_{a}(z,-a)$.  Then $H^{1/2}(z,a)\hat{\psi}_{a} = -H^{1/2}(z,-a)\hat{\psi}_{a} = H_{a}^{1/2}\hat{\psi}(z,-a) = H_{a}^{1/2}\hat{\psi}(z,a)$ follows as a consequence of these symmetries and integration by parts.
\end{proof}
\end{lem}

\begin{lem} The fact that the first variation of the information is zero implies that $\gamma = \gamma(2z - a)$.
\begin{proof}
Since the first variation of $\int_{M}\int_{A} h_{zz}e^{h}$ is zero already, we may restrict ourselves to the reduced information

\begin{center} $\hat{I} = \int_{M}\int_{A}H(\gamma)\norm{\hat{\psi}}^{2}dadz$ \end{center}

where $\hat{\psi} = H\psi$.

Note also by conservation of probability, we moreover have that $div_{\Lambda}\psi = 0$, and hence $div_{\Lambda}\hat{\psi} = 0$.

Requiring that $\delta \hat{I} = 0$ implies that $\frac{\partial \hat{I}}{\partial \gamma} = 0$, or,

\begin{center} $\int_{M}\int_{A}(\partial \norm{H(\gamma)^{1/2}\hat{\psi}}^{2}(\partial \gamma)^{-1})dadz = 0$ \end{center}

Recall that $\partial$ is essentially the operator $\frac{\partial}{\partial z} + \frac{\partial}{\partial a}$.  Now, we use the fact that

\begin{center} $H_{a}^{1/2}\hat{\psi} = H^{1/2}\hat{\psi}_{a}$ \end{center}

together with 

\begin{center} $\hat{\psi}_{z} = 0$ \end{center}

which is a consequence of conservation of probability, to deduce that our expression above is the same as

\begin{center} $\int_{M}\int_{A}(H_{z}^{1/2} + 2H_{a}^{1/2})(\partial  \gamma)^{-1}\psi(H^{1/2}\hat{\psi}) = 0$ \end{center}

which clearly implies

\begin{center} $\gamma_{z} + 2\gamma_{a} = 0$ \end{center}

proving the lemma.

\end{proof}
\end{lem}

\begin{lem} $\gamma$ is trivial; that is, $\gamma = C(2z - a)$ for some constant $C$.
\begin{proof}
Recall that the information is

\begin{center} $I = \int_{M}\int_{A}H(\gamma(z,a))\norm{\psi(z,a)}^{2}$ \end{center}

If we compute the first variation with respect to $\gamma$ once again and set it to zero, we have that

\begin{center}
$\delta(\gamma(t))\frac{d\gamma}{dt}\vert_{t = a - 2z}\norm{\hat{\psi}}^{2} = 0$
\end{center}

Then this implies

\begin{center} $0 = H(\gamma(t))\frac{d\gamma'}{d\gamma}\norm{\hat{\psi}}^{2}$ \end{center}

which in turn implies $\frac{d\gamma'}{d\gamma} = 0$ or $\gamma'$ is constant, hence $\gamma = Ct + D$ for constants $C$ and $D$.  But by the Riemann mapping theorem we can deform $\gamma$ to $Ct$ since the transformation $T : Ct + D \mapsto Ct$ is a conformal map.  Finally we observe that $H(Ct) = H(t)$ as $C$ is constant, which completes the proof.
\end{proof}
\end{lem}

\begin{proof} (of Conjecture A).
Note that $F$ is analytic (since our statistical manifold is analytic, require that the signal function be analytic). 

It then follows from the analyticity of $F$ that

\begin{center} $0 = \int_{M}\int_{A}F(z,a) = \int_{M}\int_{A}H(a - 2z)e^{A(a)z + B(a)}$

$= \int_{M}\int_{A}H(a - 2z)\bar{B}(a)a^{-z}$ for certain choices of $A$ and $B$. \end{center}

Now suppose $\bar{B}(a) = \sum_{n = 1}^{\infty}B_{n}\delta(n - a)$.  Then the above expression becomes

\begin{center}
$0 = \int_{Re(z) \geq \frac{1}{2}}\sum_{n = 1}^{\infty}B_{n}n^{-z}$
\end{center}

which completes the proof of Conjecture A.
\end{proof}

\subsection{Alternative approach using Riemann-Cartan geometry}

The previous approach was outlined by myself before I started to realise that the metric-measure formalism of Riemannian geometry with measure could be simplified and generalised to Riemann-Cartan geometry.  However I should emphasise that there is no loss of generality in applying the former methods as in the previous subsection, to the special case of the complex numbers.  Nonetheless one is motivated to find a cleaner and more purely geometric derivation of the above result using the more general methods.  So I will seek to quickly sketch one here for the benefit of myself and reader.

As before, the natural signal function is

\begin{center} $f(m,a) = \int_{A}F(m,b)\delta(\sigma(m,b) - a)db$ \end{center}

but now the total information is merely

\begin{center} $I(f) = \int_{M}\int_{A}\norm{\partial f}^{2} / f$ \end{center}

since we can ignore the measure terms.

This we can compute and find to be

\begin{center} $I(f) = \int_{M}\int_{A}\Delta f $\end{center}

which is

\begin{center} $\int_{M}\int_{A}\int_{A}((\Delta F)\delta + F \Delta \delta ) = \int_{M}\int_{A}(h'' + R(\sigma)\delta)e^{h}$ \end{center}

where $F(m,b) = e^{h(m,b)}$, since cross terms vanish as they can be rewritten as boundary terms.

If this is to be zero, then $h'' + R\delta = 0$, or, via a similar argument to the previous Lemma 10.1.4,

\begin{center} $h = A(a)z + B(a) + C(z,a)\frac{H}{\delta}(\gamma(z,a)) $ \end{center}

which gives the information as

\begin{center} $I(f) = \int_{M}\int_{A}H(\gamma)(h'' + R)e^{A(a)z + B(a)}dadz$ \end{center}

where we require that $\norm{\psi(a,z)}^{2} := (h'' + R)e^{A(a)z + B(a)}$ behaves nicely asymptotically (and in fact behaves like a generalised mass distribution).  Note that this may not necessarily be positive definite.

The rest of the argument is totally identical to that given previously, following Lemmas 10.1.7 and 10.1.8.

\subsection{Higher moments of the complex numbers}

It is possible to look at higher statistical moments of the complex numbers.  For instance, if we look at the second statistical moment, I claim we get the following result:

\begin{conj} Consider the class of functions $\zeta_{f}' = \sum_{i}f_{i}\frac{d}{dz}i^{-z}$, where $f \in l_{\infty}(N)$.  Then $\zeta_{f}'$ has nontrivial zeroes only on the critical line $Re(z) = 1$.
\end{conj}

In particular I claim this is a consequence of the Cramer-Rao inequality for the second statistical moment of the real line, extended analytically to the complex numbers.

One might naively expect, from these two pieces of data, that the critical line should move by multiples of $1/2$ indefinitely as we examine successively higher moments.  This is in fact not the case; in fact, the true story of what happens in general will be the focus of the next section.

\section{Turbulence and the criticality of the prime numbers}

\subsection{Introduction}

Next, I will demonstrate how the tools I have developed in the chapter on turbulent geometry can be used to significantly extend the (generalised) Riemann conjecture.  However before I do this I will demonstrate how the sketch I gave above can be streamlined, through use of the correspondence principle.

Naively we might then hope to get say optimal information about even say the singleton distribution of the primes, but I fear that this is unfortunately impossible to achieve.  The reason for this is because I have the opinion that there is a natural geometric hierachy of models of growing sophistication and generality, and that not only do models higher in the progression give better information, but that there are an infinite number of such.

So, following an initial extension of the Riemann conjecture to signal functions with fully general turbulence in the stack and the measure, I will indicate the beginnings of the idea of a new mathematics, the idea of geometric exponentiation, and suggest how this could be used to get even better results.

\subsection{A streamlined sketch of the Riemann conjecture}

Consider the following action:

\begin{center} $I(\sigma,\tau) = \int_{M}(\partial \sigma^{*}R_{\sigma} ; R_{\tau})dm$ \end{center}

Suppose $M$ is the complex plane, and $\sigma$,$\tau$ are analytic functions from $C$ to $C$, being metrics on $M$.  By the correspondence principle it is equivalent to the Fisher information given by the signal function

\begin{center} $f(m,a) = \int_{A}F(m,b)\delta(\bar{\sigma}(m,b) - a)db$ \end{center}

For the information to be critical we require that $\delta I(\sigma,\tau) = 0$.

Consequently, via similar methods to before, we get the following equations that $\sigma$ and $\tau$ should satisfy:

\begin{center} $(\partial \sigma^{*}R_{\sigma}; \partial \sigma^{*}R_{\tau}) = 0$ \end{center}

and

\begin{center} $(\partial \sigma^{*}R_{\sigma}; R_{\tau}) = 0$ \end{center}

If $\sigma$ is nontrivial this implies that $\tau$ is constant and $R_{\sigma} = 0$.  Note that we know that the constant $\tau$ must be nonzero since otherwise the corresponding metric $\tau(z) = \tau$ will become degenerate.

Recall now that an analytic metric is of the following general form:

\begin{center}
$ds^{2} = \sigma(z)^{2}dzd\bar{z}$
\end{center}

and the corresponding (Gaussian) curvature is

\begin{center}
$R_{\sigma} = K = -\frac{1}{\sigma(z)^{2}}\frac{\partial^{2}}{(\partial z)^{2}}ln(\sigma(z))$
\end{center}

So since $R_{\sigma} = 0$ we have that $\frac{\partial^{2}}{\partial z^{2}}ln(\sigma) = 0$.

Solving for $\sigma$, we get

\begin{center} $ln(\sigma) = Az + B$ \end{center}

or

\begin{center} $\sigma = e^{Az + B}$ \end{center}

where $A$ and $B$ are constants.

Note $f(z,a) = \delta(\sigma(z) - a)$ is analytic.

Consider now $g(z,a) = \delta(\tau - a)$ the analytic signal function associated to the constant metric $\tau$.  Via properties of the turbulent derivative and delta functions, I claim then that the total signal function is

\begin{center} $F(z,a) = (\partial f^{*}f ; g) = H(\kappa(z,a))\delta(\sigma(z) - a)$ \end{center}

(Note that the result is independent of the value of $\tau$, but of course we need $\tau \neq 0$ as demonstrated above.)

This follows since

\begin{itemize}

\item[(i)] $(\partial f^{*}f ; g)^{2} = (\frac{\partial}{\partial \sigma}(\sigma^{\delta(\tau - b)})\frac{\partial}{\partial \sigma}(\sigma^{(-1)\delta(\tau - b)}))\delta(\sigma - a)^{2}$, as $\frac{\partial^{\delta}}{\partial f^{\delta}} = f^{\delta}$.

\item[(ii)] $f(z,a)^{\delta(\tau - b)} = \delta ln f + H(\kappa(z,a))$ for some function $\kappa$, so the above reduces to 

\item[(iii)] $\frac{\partial}{\partial \sigma}(\delta ln \sigma + H)\frac{\partial}{\partial \sigma}(- \delta ln \sigma + H)\delta(\sigma - a)^{2} = H(\hat{\kappa}(z,a))^{2}\delta(\sigma - a)^{2}$, for a deformed function $\hat{\kappa}$ which proves my claim (if we approximate $\delta$ by a sequence of peaked Gaussians of the form $be^{-ax^{2}}$ with integral normalised to one, then $\delta^{2}$ will be by definition approximated by a sequence of the form $b^{2}e^{-2ax^{2}}$, and hence in the limit will have zero measure).

\end{itemize}

We then conclude as before that $\hat{\kappa}(z,a) = C(2z - a)$ for some constant $C$, using once more the fact that the information is critical.

Finally the conjecture then follows as before as a consequence of the analyticity of $F$. 


\subsection{Towards a slightly more sophisticated conjecture}

Consider now the signal function

\begin{center}
$(\partial F^{*}F(m,a,c) ; g(m,b))$, $F = f(m,a)^{h(m,c)}$.
\end{center}

We would like to extract optimal information about critical subsets of the real line using this model, given that $f$, $g$ and $h$ are all analytic.  We know that the information will be

\begin{center} $I(\sigma,\tau,\kappa) = \int_{M}(\partial \sigma^{*}R_{(\partial \kappa^{*}R_{\kappa}; R_{\gamma})\sigma};\partial \tau^{*}R_{(R_{\beta})\tau}; R_{\lambda})dm$ \end{center}

Extremising this, or requiring that $\delta I = 0$, gives the following six relations:

\begin{center}$(\partial \sigma^{*}R_{(\partial \kappa^{*}R_{\kappa}; R_{\gamma})\sigma};\partial \tau^{*}R_{(R_{\beta})\tau}; R_{\lambda}) = 0$ \end{center}

\begin{center} $(\partial \sigma^{*}R_{(\partial \kappa^{*}R_{\kappa}; R_{\gamma})\sigma};\partial \sigma^{*}\partial \tau^{*}R_{(R_{\beta})\tau}; R_{\lambda}) = 0$ \end{center}

\begin{center}$(\partial \sigma^{*}R_{(\partial \kappa^{*}R_{\kappa}; R_{\gamma})\sigma};\partial \tau^{*}R_{(R_{\beta})\tau}; \partial \tau^{*}R_{\lambda}) = 0$ \end{center}

\begin{center}$(\partial \kappa^{*}R_{\kappa}; \partial \kappa^{*}R_{\gamma}) = 0$ \end{center}

\begin{center}$(\partial \kappa^{*}R_{\kappa}; R_{\gamma}) = 0$ \end{center}

\begin{center}$R_{\beta} = 0$ \end{center}

Solving these simulatenously in the case that $M$ is the complex plane should provide us with a sharper understanding of critical subsets of the real numbers.

For some idea of what we might expect to come out of this, I provide the reader with a conjecture that may well be a partial consequence of analysis of the above.

\begin{conj} (Turbulent RH). Let $f,g \in l_{\infty}(N) \times N$ be functions from $N \times N$ to $R$.  Consider the generalised $L$-function

\begin{center} $\zeta_{(f,g)}(z) = \sum_{i,j}f_{ij}\frac{\partial^{g_{ij}j^{-z}}}{\partial z^{g_{ij}j^{-z}}}i^{-z}$ \end{center}

I claim that, for all such $f$ and $g$, the analytic extension of $\zeta_{(f,g)}$ to the complex numbers will only have (nontrivial) zeroes on the critical line 
\begin{center} $Re(z) = \frac{d}{dx}arctan(\sqrt{xtan(x)})\vert_{xtan(x) = \lambda(g)}$, \end{center}
 where $\lambda(g)$ is the average of the eigenvalues of the $g_{ij}$, weighted by multiplicity. \end{conj}
 

As a couple of remarks, note that this conjecture reduces to our statements for the first and second statistical moments of the complex numbers, respectively (as well it should).  Perhaps the most mysterious aspect is the relationship between $\lambda(g)$ and the location of the critical line - and in particular the appearance of the trigonometric function $tan$.  Essentially this means that as the average eigenvalue increases, the location of the critical line will occasionally "spike" to positive infinity, then reappear at negative infinity and quickly drop back close to zero again, and do so with a period of length $\pi$.  Indeed, this is a significant departure from the naive expectation that if we were to have taken successively higher statistical moments of the complex numbers, that the location of the critical line for the associated $L$ functions would have jumped by multiples of $1/2$ indefinitely.

It would be very good to even have a sketch of the above.  However, given the fact that I am still uncertain whether this is a precisely correct consequence of the criticality of the turbulent information (even given everything else, I am unsure about whether $\lambda$ should not also be a function of $f$), it would perhaps be wiser to instead try to see what follows from the criticality of the action I have given above instead.  Also it was never my intent to delve particularly deep into this area of mathematics - more I merely wished to indicate how the tools I have been developed can be used, and what results might follow if care is taken.

It also goes without saying that if the above conjecture is slightly off-beam and in fact not corollary to the criticality of our signal function above, a counterexample should be easy enough to find via numerical methods.

\section{Geometric exponentiation and deeper information}

\subsection{Motivation and definition} Recall from my previous work on turbulent geometry that there are two types of turbulence that admit straightforward modelling - turbulence in the measure and turbulence in the stack.  These results in actions of the form $R^{R}$ and $R_{R}$ respectively, where I am taking considerable liberties paraphrasing here.  Similarly we can consider measure-measure turbulence, measure-stack turbulence, stack-measure turbulence, and stack-stack turbulence.  These result in actions of the form $R^{R^{R}}$, $R^{R_{R}}$, $R_{R^{R}}$, and $R_{R_{R}}$ respectively.

It is natural to then ask what happens in general.  Well certainly at the $n$th iteration we will have $2^{n}$ different possibilities for an action.  So complexity of our models if we wish to encompass all possibility grows exponentially with further attention to detail.  This is evidently not desirable.  In particular we would ideally like to know what happens if we push $n$ off to infinity, to generate an infinite number of discrete geometric bifurcations in our models.  Then it is readily seen that the number of possibilities is $2^{\aleph_{0}}$, or $\aleph_{1}$, the cardinality of the real numbers (I am taking a further liberty here - for those who wish to believe in intermediate infinities I adopt the convention that they might take a continuous range of values $\aleph_{k}$ where $k$ is between $0$ and $1$).

So this leads one to ask, is there some formalism that would allow us to deal with an infinite number of discrete bifurcations?  Needless to say, if this is to be doable, some novel new idea or way at looking at things is essential to make progress so that calculations do not become unmanageable.  To cut things short, I believe that the answer is yes, and this is where the idea of geometric exponentiation enters the picture.

Set exponentiation is a fairly simple concept to understand.  Consider two sets $A$ and $B$.  The product $A \times B$ may readily be formed, and is understood to be the set of tuples $(a,b)$ where $a \in A$ and $b \in B$.  So this is the product of two sets.  How about raising one set to the power of another?  What exactly should this mean?  Obviously we expect the cardinality of $A^{B}$ to be necessarily of quite a different order to that of $A \times B$.  So this needs to follow from the definition.  

Briefly, the exponential of $A^{B}$ will be understood to mean a product $\times_{b \in B}A(b)$, where each $A(b)$ is a copy of $A$ indexed by an element $b$ of $B$.  So if $B$ is large, say infinite, as will often be the case, $A^{B}$ will be very large indeed.  In what is to follow, $A$ and $B$ will often be Riemannian manifolds.  Then $A^{B}$ will be a manifold itself of \emph{uncountably infinite dimension}.  However more information is required to specify and map out an appropriate amount of geometric structure for such spaces, and this leads directly to the notion of geometric exponention.

So we would like to exponentiate say one Riemannian manifold $(M,\sigma)$ by another $(N,\tau)$.  As before, by exponentiation of two sets $A^{B}$ I mean taking one copy of $A$ for each element in $B$, leading to a structure that potentially could have uncountable dimension.  If $A$ and $B$ have differentiable structure, we can similarly induce an indexing on the tangent space of $A^{B}$ via $v(w)$ for the vector $v(w)$ assigned to the $w$th copy of the tangent space of $A$. Furthermore it is in fact possible to induce a natural geometric structure on $M^{N}$ using the metrics $\sigma$ and $\tau$, via the fourth order nondegenerate tensor $\Lambda = \sigma \otimes \tau$, with geometry given by $\norm{(v(w),p(q))}_{\Lambda} := \sigma_{ij}\tau_{kl}v^{i}(w^{k})p^{j}(q^{l})$.

\subsection{Application} To reiterate, this idea was developed in an attempt to find a way to deal with an infinite number of geometric bifurcations, as is engendered by the two types of mathematical turbulence in the turbulent geometry discussed above (turbulence in the stack and the measure).  It is my hope that it might be possible to find a new correspondence principle for such spaces that allows the description of the previous geometry not only in full generality, but also in an elegant and finite way.

As to applications of geometric exponentiation, my thoughts on this matter are still quite vague.  However I suspect it could be used to improve the understanding of transcendental equations, as well as of aspects of the mathematics related to Galois theory.  I also have the intuition that geometric exponentiation would be useful to find still deeper results about critical subsets of the real line than those described previously in this chapter.

For instance, one might expect actions of the following general form, in the simplest case:

\begin{center} $\int_{M}S_{\Lambda(m)}dm$ \end{center}

Here $S_{\Lambda}$ is the generalised curvature due to the 4-tensor $\Lambda$ on $M$.  This of course is the action corresponding to the signal function

\begin{center} $f(m,a(b)) = \delta(\Lambda(m) - a(b))$ \end{center}

To extract the geometric exponential analogue of the Riemann conjecture, the following action needs to be considered:

\begin{center} $\int_{M}(\partial \Lambda^{*}S_{\Lambda} ; S_{\Gamma})dm$ \end{center}

This is of course the corresponding turbulent information of two signal functions as immediately above and I claim that it satisfies a corresponding Cramer-Rao inequality.

In particular I suspect that exploring the consequences of this action over the complex numbers would lead to preliminary criticality results in transcendental number theory.

\subsection{Concluding Remarks}

It is necessary to emphasise that all of the results that I have discussed up to this point are still quite primitive.  I mentioned not too long ago that I believe that there is a natural geometric hierachy of countably infinite order of structures of monotonically increasing complexity.  Throughout the latter half of this work I have endeavoured to climb as much of this edifice as I have been able.

The first few tentative steps up this mountain - differential geometry, riemannian geometry, statistical geometry, turbulent geometry - have been difficult.  One might hope that, with these minor successes, one could systematise the method of ascent, and indeed we are led to the idea of geometric exponentiation in this manner.  However I have come to increasingly suspect that this hierarchy is pathologically non-inductive, even bearing in mind that each step up brings something new and unexpected.

For instance, one might naively intuit that one can keep on doubling the order of our tensors - from 2nd degree in Riemannian geometry to 4th in exponential, to 8th, etc.  It is possible to do this and study the associated structures - in fact, I have already initiated a study of exponential, or "plastic" geometry, which has interesting connections to viscoplastic media - but we are ignoring something quite important.  That is, there is no reason we should not consider a distribution of information over the space of tensors of all possible ranks.  Furthermore, this rank need not be a natural number - one could reasonably imagine a sensible abstraction of the idea of tensor rank to the reals, or even to euclidean n-space.

We then would get a tensor rank manifold, $N$, forming a strange variety of topological product with our physical space, $M$.  We may also in turn, for simplicity, have a standard metric, or second rank tensor on $N$ that gives it the structure of a Riemann-Cartan manifold.  The associated geometric models would then look like quite exotic fibre bundles.

So further progress of a definite and fully general variety will certainly be difficult to make.  However, applications of such methods may have large potential payoff.  I suspect that even slight progress could have benefits to the understanding of materials science (in the case of plastic geometry) and improvement in sieving methods (with respect to the tensor rank manifolds outlined).  The reason to suspect the latter is because it is quite possible that the metric on the tensor rank space may be in correspondence with the bilinear form of the remainder term in the linear sieve (see for instance Greaves \cite{[Gr]}).  Related questions, such as the Geometric Langlands program, might also be addressable with deeper understanding.

Note that if we step back a bit, it is easy to observe that all the above considerations relate to having ideas from statistics driving the development of abstract geometric structures on manifolds.  It then becomes natural to ask whether we can use geometry to drive the development of statistical structures.  Motivating questions here are as to the rigorous formulation of Heisenberg's matrix mechanics, the understanding of Brownian motion, and doubtless various other processes, such as queuing theory, population dynamics, or analysis of markets and group psychology.

And in fact we can.  The prototypical sort of example to bear in mind here is that of a Markov chain, where one has a nondegenerate bilinear form driving the evolution of a probabilistic state vector.  It seems clear to me that this is a natural place to start in order to begin the process of abstraction and generalisation with respect to this particular viewpoint.  In fact in retrospect it seems to me that this and its continuous analogues are perhaps closer in spirit to the work of many of the more established practitioners of Information Geometry than that covered in the bulk of this treatise, being of a more statistical and less geometric flavour.

These and doubtless other considerations will hopefully be the subject of a future work.

\backmatter

\end{document}